\documentclass[11pt]{article}
\usepackage[margin=1in]{geometry}
\usepackage{yk, dsfont, mathrsfs, hyperref}
\usepackage{algorithm}
\usepackage{algpseudocode}
\usepackage{graphicx}
\usepackage{subcaption}
\usepackage{booktabs,multirow}

\RequirePackage{graphicx}

\usepackage[numbers,sort&compress]{natbib}

\algrenewcommand\algorithmicrequire{\textbf{Input:}}
\algrenewcommand\algorithmicensure{\textbf{Output:}}

\newcommand{\sfS}{\mathsf{S}}
\newcommand{\htdb}{\hat{\theta}^{\mathrm{db}}}
\newcommand{\sG}{\mathsf{G}}
\newcommand{\sS}{\mathsf{S}}

\title{PDE-constrained inverse problems at the $\sqrt{n}$ rate via debiased physics-informed neural networks}
\author{Yves Atchade and Debarghya Mukherjee \\ Department of Mathematics and Statistics, Boston University}
\date{}
\begin{document}
\maketitle
%\noindent
\begin{abstract}
    We study the problem of estimating unknown parameters in PDE-constrained inverse problems from noisy observations, where the PDE solution is approximated using Physics-Informed Neural Networks (PINNs). While PINNs have demonstrated remarkable empirical success, existing estimators often inherit the slow nonparametric convergence rate of the neural-network solution, leading to biased and statistically inefficient inference for the finite-dimensional parameters of interest. To address this, we propose a two-step debiased estimation procedure that combines neural-network-based nonparametric estimation with an influence-function-based bias correction. By eliminating the first-order sensitivity of the estimator to errors in the nuisance function, our procedure yields a $\sqrt{n}$-consistent and asymptotically normal estimator without requiring undersmoothing of the neural network component. We further extend this framework to Bayesian inference by replacing the original likelihood with a debiased quasi-likelihood and establish a Bernstein–von Mises theorem showing that the resulting posterior contracts at the $\sqrt{n}$-rate with an asymptotic covariance matching that of the frequentist estimator. As a by-product of our analysis, we establish near-minimax optimal convergence rates for estimating a nonparametric regression function and its derivatives in Sobolev spaces using neural networks. Extensive numerical experiments corroborate our theoretical findings and demonstrate the necessity of the proposed debiasing procedure for valid statistical inference in PDE-constrained inverse problems.
\end{abstract}
\section{Introduction}
Partial differential equations (PDEs) are widely used to describe the evolution of physical quantities over space and time, serving as the foundational modeling framework in many areas of science and engineering \citep{burgers1948mathematical,turing1990chemical,keller1971model,richards1931capillary,danabasoglu2020community,schneider2014migrations}. In numerous real-world applications, the general forms of the governing PDEs are known from domain-specific knowledge, but certain key parameters relevant to the underlying physical laws remain unknown. The estimation and inference on these unknown parameters from data is of paramount importance in many applications \citep{depina2022application,iglesias2014well,huang2024inverse,kadeethum2020physics}. %Accurate inference in this setting is of paramount importance, as errors or incorrect inferences in parameter estimation may lead to misleading scientific conclusions or poor predictions.
For concreteness, consider the steady-state heat-conduction equation
$$
-\nabla \cdot (a_{\theta_\star}(x) \nabla u_\star(x)) = f(x), \qquad x \in \Omega \,,
$$
along with appropriate boundary conditions. Here $u_\star$ denotes the temperature field, $f$ is the known source function, and $a_{\theta^\star}$ is the thermal conductivity of the material. For simplicity, assume that $\log a_{\theta_\star}(x) = \sum_{j = 1}^q \theta_{\star,j} \phi_j(x)$, i.e., it belongs to the span of $q$ (known) functions $\{\phi_j: 1\le j \le q\}$. Given $n$ noisy temperature measurements of the form $Y_i = u_\star(X_i) + \eps_i$, the goal is to recover the finite-dimensional conductivity parameter $\theta_\star$. The unknown temperature field $u_\star$ must also be estimated from the data, and hence acts as an infinite-dimensional nuisance parameter.  
%, potentially affecting subsequent engineering designs, safety assessments, and environmental models.

Recovering $\theta_*$ from noisy observations in problems of this kind falls under the purview of the inverse problem \cite{monard:etal:21,nickl2020convergence,kekkonen2022consistency,giordano2020consistency,nickl2024posterior,bissantz2004consistency,bissantz2007convergence,nickl2020bernstein,stuart2010inverse,calvetti2018inverse,benning2018modern,arridge2019solving,ghattas2021learning}. Methodologies developed to tackle this from both frequentist and Bayesian perspectives can be divided into two main categories: direct (one-step) methods and indirect (regularized or two-step) methods. The direct approach involves parameterizing the distribution of observations directly in terms of the unknown parameter and the \textit{forward operator} (\cite{bissantz2007convergence,nickl:etal:2020,nickl2024posterior,nickl2020bernstein,blanchard:etal:18}). While theoretically appealing, implementing this approach is notoriously challenging in practice due to the complexity of the forward operator. For example, computing the gradients needed for gradient-based optimization is computationally demanding, and constructing a surrogate model/optimization problem to bypass this often introduces additional approximation errors in the estimation process that are typically not well understood.

To bypass the need to query the forward operator or its derivatives, the indirect approach simplifies the optimization problem by estimating the infinite-dimensional PDE solution and the finite-dimensional parameter via regularization. Although the indirect approach has been studied extensively in the statistical literature (e.g., see \cite{xun2013parameter,qi2010asymptotic,ramsay2007parameter,dattner2015optimal,bellman1971use,varah1982spline,gugushvili2012sqrtn,liang2008parameter} and references therein), there has been a recent resurgence of interest in this formulation as researchers have begun employing deep neural networks to solve these PDE-constrained optimization problems under the name of Physics-Informed Neural Networks (PINNs) \citep{raissi:etal:19,yu2018deep,sirignano2018dgm,yang:bpinn:21,karniadakis2021physics,wang2021understanding,krishnapriyan:etal:2021}. Leveraging recent advances in deep learning and automatic differentiation, PINNs provide a flexible and computationally attractive tool for jointly estimating parameters and the PDE solution from noisy observations. Despite their empirical success across various fields, the theoretical understanding of the statistical properties of PINN-based estimators remains limited \citep{mishra:etal:22, sun2024estimation,de2022error}. %While they are highly effective for prediction, they are primarily black-box models that typically lack valid uncertainty quantification for inference on the underlying PDE parameters.

A critical limitation of the standard PINN-based approach is that the estimate of the finite-dimensional parameter of interest typically inherits the slower, non-parametric convergence rate from the neural network-based estimation of the PDE solution. As a consequence, the widely used PINN estimators for PDE parameters are typically biased and statistically inefficient. In classical semi-parametric inference, such biases are often mitigated by undersmoothing the non-parametric component (e.g., by selecting a smaller bandwidth in kernel regression or using more knots in spline-based estimators) \citep{van2023efficient,mcgrath2022nuisance,paninski2008undersmoothed,newey1998undersmoothing,gine2008simple,bickel2003nonparametric,goldstein1992optimal,gugushvili2012sqrtn}. However, when employing deep neural networks, it is unclear which parameters (e.g., depth, width, or learning rate) should be carefully tuned to achieve the necessary undersmoothing, especially when fine-tuning pre-trained networks (see \cite{yuundersmoothing} for a recent work in that direction). This paper aims to bridge this critical gap by addressing the following question: \textit{Is it possible to estimate the unknown parameters of the PDE at a parametric $\sqrt{n}$ rate while using PINNs to estimate the solution?}

To answer this, we study PINN-based inverse problems through the lens of the theory of efficient estimation of functionals in semiparametric models \cite{bickel1993efficient,tsiatis2006semiparametric}. Inspired by the idea of bias correction in double machine learning/Neyman orthogonalization \cite{neyman1959optimal,neyman1979c,chernozhukov2018double,chernozhukov2022locally,foster2023orthogonal}, we propose a novel debiasing technique that yields an estimator that does not require the initial neural network estimator to be undersmoothed. Specifically, we develop a simple two-step estimation procedure that combines neural-network-based nonparametric estimation of the solution with an influence-function-based debiasing strategy \citep{newey1994asymptotic,robins1994estimation,ichimura2022influence}. By performing a one-step Taylor expansion on the first-order condition of the PDE-based risk functional, we carefully construct a debiased score function whose first-order sensitivity to errors in the nuisance function is eliminated. This bias correction ensures that the effect of the first-stage estimation error enters only at a higher order, yielding a statistically valid, $\sqrt{n}$-consistent, and asymptotically normal (CAN) estimator of the finite-dimensional PDE parameters.

The estimation of the finite-dimensional parameters at the root-n rate in a semiparametric model is also an actively investigated problem in the Bayesian literature \citep{castillo:12, Bickel:kleijn:2012, Castillo:rousseau:2015,rousseau:16, Minwoo:etal:2019}. By replacing the likelihood with a suitably debiased quasi-likelihood function, we develop a Bayesian extension of our debiasing technique.  The resulting debiased posterior distribution bears similarities with two recent works on Bayesian posterior correction in semiparametric models. In \cite{yiu:etal:2025}, the authors leverage influence functions and the one-step update estimator to develop a posterior correction method for models with a probability measure nuisance parameter. More closely related to our work is the general Bayesian Neyman orthogonalization framework proposed by \cite{sabbagh:stephens:2026}. An important limitation of \citep{sabbagh:stephens:2026} is that the conditions under which their corrected posterior defines a valid probability measure must be verified on a case-by-case basis, and it is unclear whether their framework works in the current setting. By contrast, our debiased posterior distributions are always well-defined and can be readily explored via MCMC, and in instances involving linear PDEs with a Gaussian prior $\pi_0$, exact Monte Carlo sampling is even feasible. On the theory side, we establish a Bernstein–von Mises theorem with a limiting covariance matrix that matches that of the corresponding frequentist estimator, thereby ensuring asymptotic agreement between Bayesian and frequentist inference. In summary, the statistically valid inferential procedures developed in this work address a critical methodological need, enabling researchers across disciplines to draw reliable scientific conclusions when applying modern machine learning tools to PDE-constrained inverse problems. 
We summarize our contributions below: 
\begin{enumerate}

\item We develop a novel debiasing approach to obtain a $\sqrt{n}$-consistent and asymptotically normal (CAN) estimator of the unknown parameters in PDE-constrained models when the underlying solution is approximated using neural networks, such as PINNs. Our methodology is computationally efficient and straightforward to implement. To the best of our knowledge, this is the first work that provides rigorous $\sqrt{n}$-rate statistical guarantees for parameter estimation in PDE-constrained inverse problems using flexible machine learning models.

\item As a by-product, we establish near-minimax-optimal convergence rates for neural-network estimators of a nonparametric regression function and its derivatives simultaneously over the larger Sobolev class $W^{\beta,2}$, in contrast to much of the existing neural-network regression theory, which typically assumes more restrictive finite H\"{o}lder-type smoothness \citep{schmidt2020nonparametric,bauer2019deep,kohler2021rate,kohler2016nonparametric,fan2024factor,bhattacharya2024deep,belomestny2023simultaneous}.
Statistical theory for derivative estimation using neural networks is comparatively scarce.
One approach is to smooth a flexible first-stage regression estimator before differentiating it: \citep{fan2024conditional} estimate partial derivatives by convolving a deep neural network regression estimator with a smoothing kernel, while \citep{klyne2026average} propose a kernel-resmoothing scheme for arbitrary machine-learning regression estimators in the
estimation of average partial effects.
These approaches, however, do not directly establish simultaneous convergence rates for the derivatives of the same fitted neural-network estimator in Sobolev norms. 
Perhaps closest in spirit is \citep{ding2025semi}, which studies a semi-supervised deep Sobolev regressor based on ReQU networks and provides guarantees for estimating both the regression function and its gradient; however, their derivative-estimation rate is not minimax optimal.
Although simultaneous approximation results for neural networks have been established in H\"{o}lder spaces \cite{belomestny2023simultaneous,de2021approximation}, analogous results in Sobolev spaces, covering both approximation and stochastic errors to yield a statistical guarantee, appear to be missing in the literature. 
To the best of our knowledge, this is the first work to establish simultaneous minimax-optimal rates (up to a log factor) for estimating the regression function and its derivatives in Sobolev spaces using neural networks.
% \item As a by-product of our analysis, we establish the near-minimax optimal rate of convergence for estimating a nonparametric regression function and its derivatives using neural networks when the function belongs to a Sobolev space $W^{\beta,2}$, i.e., it has finite $L_2$ Sobolev norm. This setting contrasts with the finite Hölder norm assumptions that are commonly adopted in the literature \cite{schmidt2020nonparametric,bauer2019deep,kohler2021rate,kohler2016nonparametric,fan2024factor,bhattacharya2024deep,belomestny2023simultaneous}. Furthermore, to the best of our knowledge, this is the first work to establish near-minimax-optimal rates for estimating derivatives of the mean function in Sobolev spaces using neural networks. \DM{Add \cite{ding2025semi} carefully and check whether other papers are there.}

\item We further develop a novel Bayesian debiasing framework in which the likelihood function is suitably modified to incorporate the debiased score function. Under a mild log-concavity assumption on the prior distribution we show that the marginal posterior distribution of the parameter $\theta$ contracts at the $\sqrt{n}$ rate. Moreover, we establish a Bernstein--von Mises theorem demonstrating that the resulting Bayesian posterior distribution has the same asymptotic covariance as the corresponding frequentist estimator, thereby ensuring asymptotic agreement between Bayesian and frequentist inference. Our analysis provides a unified treatment of PDE-constrained parameter estimation when (i) the underlying mean function is estimated via neural networks, and (ii) a suitable debiasing procedure is incorporated to remove the bias arising from nonparametric estimation and achieve $\sqrt{n}$-consistent inference for the parameter of interest.

\item Finally, we complement our theoretical results with extensive numerical experiments. These studies corroborate our theoretical findings and, in particular, demonstrate the necessity of the proposed debiasing procedure to obtain valid statistical inference for the parameter of interest.

\end{enumerate}
The rest of the manuscript is organized as follows. In Section~\ref{sec:method}, we introduce the problem formulation and present our methodology in detail from both frequentist and Bayesian perspectives. In Section~\ref{sec:theory}, we establish the theoretical guarantees of our approach, including $\sqrt{n}$-consistency and asymptotic normality from the frequentist perspective (Section~\ref{sec:theory_freq}), as well as a Bernstein--von Mises theorem from the Bayesian perspective. In Section~\ref{sec:simulation}, we present numerical experiments that validate our theoretical results and demonstrate the necessity of the proposed debiasing procedure. Finally, in Section~\ref{sec:conclusion}, we conclude the paper and discuss several directions for future research. 

\section{Methodology}
\label{sec:method}
%In this section, we present our methodology for constructing a $\sqrt{n}$-consistent and asymptotically normal (CAN) estimator of the PDE parameter $\theta_\star$ (see Equation~\eqref{eq:pde_model_2}) from a frequentist perspective. As discussed in the Introduction, our goal is to efficiently estimate $\theta_\star$ based on noisy observations of the PDE solution $u_\star$ at random design points. While physics-informed neural networks (PINNs) provide flexible and computationally attractive tools for jointly estimating $(\theta_\star, u_\star)$, they are primarily designed for prediction and typically lack valid frequentist uncertainty guarantees for inference on $\theta_\star$.  To address this limitation, we develop a two-step debiased estimation procedure that combines neural network-based nonparametric estimation of $u_\star$ with a novel debiasing strategy. This approach yields a $\sqrt{n}$-CAN estimator of $\theta_\star$. We begin by describing the problem setup.

%\subsection{Problem setup}
\label{sec:setup}
Let $\Omega \subset \mathbb{R}^d$ be a bounded open domain with a sufficiently regular boundary. For ease of exposition, we assume throughout that $\Omega = (0,1)^d$, although our theoretical framework extends almost verbatim to any bounded domain with smooth (Lipschitz) boundary. Let $u_\star$ denote an unknown deterministic function that solves the parametrized partial differential equation
\begin{equation}
\label{eq:pde_model_2}
K_0 u_\star(x) + \Psi_{\theta_\star}\!\left(x,K_1 u_\star(x)\right) = 0, \quad \mathcal{B}(u_\star) = 0, \quad \forall x \in \Omega.
\end{equation}
Here, $K_0$ is a differential operator of order at most $\tau$ (see Section \ref{sec:theory} below for precise definition), and $K_1 = (K_{11}, \ldots, K_{1p})^\top$ is a vector of $p$ differential operators, each of order at most $\tau$. The operator $\mathcal{B}$ encodes the initial and boundary conditions\footnote{The initial and boundary conditions are required for $u_\star$ to be well-defined. However, since $\mathcal{B}$ does not depend on $\theta$ (the main quantity of interest in our setup), it will be ignored in our analysis. %In practice, including $\mathcal{B}$ in the loss function often yields qualitative improvements in the estimation of $u_\star$, particularly in the small $n$ regime. 
}. The parameter of interest $\theta_\star \in \Theta \subseteq \mathbb{R}^q$ is the unknown finite-dimensional vector that governs the underlying physical law. The function $\Psi_{\theta}(x,K_1u(x)) \equiv \Psi(x,K_1 u(x), \theta)$ with  $\Psi:\;\rset^d \times \mathbb{R}^p  \times \mathbb{R}^q\to\mathbb{R}$ that is assumed to be sufficiently smooth; see Assumption~\ref{assm:psi_smoothness} for precise regularity conditions. In a slight abuse of notation we will at times write 
\[ 
(\Psi_\theta\circ K_1 u)(x) \eqdef \Psi_\theta(x,K_1u(x)) = \Psi(x, K_1 u(x), \theta).
\]
Furthermore for any $m \in \{x, z, \theta\}$, we use the notation $\nabla_m (\Psi_\theta\circ K_1 u)(x)$ to denote the partial derivative of $\Psi$ with respect to $m$ evaluated at $(x, K_1(u)(x), \theta)$, i.e.
$$
\nabla_m (\Psi_\theta\circ K_1 u)(x) = \nabla_m\Psi(x, z, \theta) \vert_{(x, K_1 u(x), \theta)} \,.
$$
The higher-order derivatives are also defined analogously. 

An experimenter observes noisy measurements of the solution $u_\star$ at random design points. Specifically, the observed data consist of $\mathcal{D}_n = \{(X_i, Y_i)\}_{i=1}^n$, where $X_i \in \Omega$ are sampled independently from a known design distribution $p_0$ supported on $\Omega$ (that we take without loss of generality to be the uniform distribution), and $Y_i = u_\star(X_i) + \eps_i$, with $\{\eps\}_{i=1}^n$ being independent sub-Gaussian noise variables. Our setting is thus closer in spirit to experimental design rather than observational studies, as the design distribution $p_0$ is assumed to be known. While our framework can be extended to the case of unknown design distribution by estimating $p_0$ from the data, we leave this extension as an interesting future avenue to explore.

Our primary objective is to construct a $\sqrt{n}$-CAN estimator of $\theta_\star$ based on the observed data $\mathcal{D}_n$. The function $u_\star$ serves as an infinite-dimensional nuisance parameter. In our proposed method, we estimate it using flexible function classes, such as deep neural networks, while carefully correcting for the resulting bias in the estimation of $\theta_\star$. To motivate our proposed methodology, we first describe a standard PINN-based approach and explain why it may fail to achieve valid inference for $\theta_\star$. Define the PDE-based risk functional
\begin{equation}
\label{eq:pde_risk}
\mathcal{R}(\theta, u) 
= \frac{1}{2}\int_\Omega \left\{ K_0 u(x) + \left(\Psi_\theta\circ K_1 u\right)(x) \right\}^2 \ \omega(x) \, dx,
\end{equation}
where $\omega(x)$ is a strictly positive weight function on $(0,1)^d$ that smoothly decays to zero near the boundary; see Assumption~\ref{assm:w} for concrete examples. Under some minimal identifiability assumptions (see Assumption \ref{assm:curvature_pde_risk}) $\cR(\theta, u_\star)$ is uniquely minimized at $\theta = \theta_\star$, and in particular, $\mathcal{R}(\theta_\star, u_\star) = 0$ (immediately follows from Equation \eqref{eq:pde_model_2}). A standard PINN approach \citep{raissi:etal:19,yang:bpinn:21,karniadakis2021physics,sun2024estimation}  estimates $(\theta_\star, u_\star)$ jointly by solving the penalized optimization problem
\begin{equation}
\label{eq:PINN_org}
(\hat u_{\rm PINN}, \hat \theta_{\rm PINN}) = \argmin_{\substack{u \in \cF_{\rm NN} \\ \theta \in \Theta}} \ \frac1n \sum_{i = 1}^n (Y_i - u(X_i))^2 + \lambda_{n, 1} \cR(\theta, u) + \lambda_{n, 2} \rho_n(u) 
\end{equation}
where $\mathcal{F}_{\rm NN}$ denotes a class of neural networks with prescribed architecture (e.g., depth, width, sparsity), and $\rho_n(u)$ is a regularization term on the network parameters, such as weight decay. In essence, PINNs solve a nonparametric regression problem augmented with a PDE-based penalty, that encourages the estimated solution to satisfy the governing equation, while estimating $\theta$ and $u$ jointly via gradient-based optimization.

Although the standard PINN approach is widely used, largely because of its computational convenience, it has an important limitation for inference on $\theta_\star$. In particular, the resulting estimator $\hat \theta_{\rm PINN}$ typically inherits the nonparametric estimation error of the first-stage solution estimator $\hat u_{\rm PINN}$ \citep{sun2024estimation}. Consequently, even though $\hat \theta_{\rm PINN}$ is consistent, it generally converges more slowly than the parametric $\sqrt{n}$-rate and may exhibit substantial finite-sample bias. Moreover, its limiting distribution is typically unavailable, precluding the direct construction of theoretically valid analytical confidence intervals. To the best of our knowledge, there is also no general theory establishing the validity of bootstrap- or Monte Carlo-based uncertainty quantification procedures for standard PINN estimators in this setting. In many scientific and physical applications, however, reliable inference on $\theta_\star$ is the primary objective. This motivates the development of a principled debiasing procedure that removes the first-order effect of the estimation error in $\hat u_{\rm PINN}$ yields a $\sqrt{n}$-consistent and asymptotically normal estimate of $\theta_*$ while remains computationally straightforward to implement.
% Although the above approach is widely used, primarily because of its computational ease, the key limitation of this approach for inference on $\theta_\star$ is that the resulting estimator $\hat \theta_{\rm PINN}$ typically inherits the nonparametric estimation error of $\hat u_{\rm PINN}$ \citep{sun2024estimation}. As a consequence, $\hat \theta_{\rm PINN}$, though typically yields a consistent estimate of $\theta_\star$, generally converges at a rate slower than $\sqrt{n}$, leading to substantial bias and overly wide confidence intervals in finite samples. 
% Furthermore, the limiting distribution of $\hat \theta_{\rm PINN}$ is typically unknown and consequently it is not possible to construct valid analytical confidence intervals, and we do not have any established theory guarantee-ing consistency of monte-carlo/bootstrap based approach. 
% However, in many physical and scientific applications, accurate inference on $\theta_\star$ is the primary objective. This motivates the development of a principled debiasing strategy that effectively removes the impact of first-stage estimation error and yields a $\sqrt{n}$-CAN estimator of $\theta_\star$, while being computationally easy to implement.

\begin{remark}
In the preceding formulation, we defined the PDE penalty $\cR(\theta, u)$ at the \emph{population level}, involving integration over the spatial domain $\Omega$. In practice, this integral can be approximated using Monte Carlo integration by drawing samples over $\Omega$, or via other standard numerical integration techniques. However, under mild regularity conditions, the resulting numerical approximation error is typically negligible compared to the statistical error arising from the squared loss term, and hence does not affect the asymptotic rate of estimation of $(u_\star, \theta_\star)$. For clarity of exposition, we therefore work directly with the population-level penalty throughout the paper.
\end{remark}

\medskip
\begin{remark}\label{rem:source:f}
    It is worth pointing out that our framework and results extend verbatim to PDEs of the form $K_0 (u)(x) + (\Psi_{\theta} \circ K_1(u))(x) = F(x)$ for some known source function $F$, by redefining the function $\Psi_\theta(x,z)$ as  $\Psi_\theta(x,z)  - F(x)$. 
    We focus on the case $F\equiv 0$ for ease of exposition.
\end{remark}

\subsection{Debiased estimation of $\theta_\star$}
\label{sec:debias_method}
In this subsection, we present our methodology for constructing a $\sqrt{n}$-CAN estimator of the PDE parameter $\theta_\star$. 
As elaborated above, the PINN-based estimator $\hat \theta_{\rm PINN}$ typically inherits bias from the estimation of the non-parametric function $u_*$ and consequently converges at a slower rate, leading to a sub-optimal estimator of $\theta_\star$. 
Our proposed method overcomes this barrier by carefully constructing a \emph{debiased} estimating equation whose first-order sensitivity to the estimation error of $u_\star$ is (approximately) eliminated. Consequently, the first-stage error in estimating $u_\star$ affects inference on $\theta_\star$ only through higher-order terms, allowing the resulting estimator to attain the parametric $\sqrt{n}$ rate. 
The entire procedure is summarized in Algorithm~\ref{alg:two_step_freq}. At a high level, our method consists of two main steps. In the first step, we obtain a consistent preliminary estimator of $\theta_\star$ together with a ``good enough''(to be elaborated in Step I below) estimator of the nuisance function $u_\star$. In the second step, we construct and optimize an appropriate debiased score function, using these first-stage estimators, to obtain our final $\sqrt{n}$-CAN estimator of $\theta_\star$. We now describe each step in detail.
\\\\
\noindent
{\bf Stage I: }The objective of the first step is twofold. First, we aim to obtain an initial estimator $\hat\theta^{\rm PI}$ of $\theta_\star$ (defined precisely in Equation~\eqref{eq:theta_PI}) that is merely required to be consistent; its rate of convergence may be slow (e.g., nonparametric rates). Second, we require an estimator $\hat u$ of the PDE solution $u_\star$ that recovers both $u_\star$ and its derivatives at the minimax-optimal rate, up to logarithmic factors, with respect to the smoothness parameter $\beta$ and the dimension $d$ of the domain. One may opt to use the joint PINN estimator $(\hat u_{\rm PINN}, \hat\theta_{\rm PINN})$, as defined in Equation~\eqref{eq:PINN_org}. However, below we propose a simpler alternative that is \emph{not only computationally more efficient but also more amenable to theoretical analysis}. Specifically, we estimate $u_\star$ by regressing $Y$ on $X$ using a class of deep neural networks $\mathcal{F}_{\rm NN}$ (see Assumption \ref{assm:function_class} for details on the architecture). While existing theory for neural networks (e.g., \cite{schmidt2020nonparametric,kohler2021rate,bauer2019deep,kohler2016nonparametric}) establishes minimax-optimal convergence rates (up to logarithmic factors) for estimating smooth or compositional functions, it does not immediately yield corresponding guarantees for estimating derivatives of the regression function. To address this issue, we propose estimating $u_\star$ by solving the penalized optimization problem
\begin{equation}
\label{eq:u_est_NN}
\hat u = \argmin_{u \in \cF_{\rm NN}} \ \left\{\frac1n \sum_{i = 1}^n (Y_i - u(X_i))^2 + \lambda_n \|u\|^2_{W^{\beta, 2}}\right\},
\end{equation}
where $\|u\|^2_{W^{\beta, 2}}$ denotes the Sobolev norm of $u$ (with precise definition given below). In Theorem~\ref{thm:deep_sobolev}, we show that (i) $\hat u$ has finite $W^{\beta,2}$ norm, and (ii) for every multi-index $\balpha$ satisfying $0 \le |\balpha| < \beta$, the derivative $\partial^{\balpha} \hat u$ estimates $\partial^{\balpha} u_\star$ at the minimax-optimal rate (up to logarithmic factors). 
\emph{To the best of our knowledge, this is the first result establishing that a neural network-based estimator can simultaneously estimate a nonparametric regression function and its derivatives at minimax-optimal rates (up to logarithmic factors).} Having obtained $\hat u$ from Equation \eqref{eq:u_est_NN}, we define a preliminary plug-in estimator of $\theta_\star$ as
\begin{equation}
\label{eq:theta_PI}
\hat \theta^{\rm PI} = \argmin_{\theta \in \Theta} \ \cR(\theta, \hat u) \,,
\end{equation}
where $\cR(\theta, u)$ is as defined in Equation \eqref{eq:pde_risk}. 
In Lemma \ref{lem:consistency_PI}, we establish that $\hat\theta^{\rm PI}$ is a consistent estimator of $\theta_\star$. However, as discussed earlier, $\hat\theta^{\rm PI}$ generally converges at a rate slower than $\sqrt{n}$ due to its sensitivity to first-stage estimation error. In the next step, we address this issue by constructing a debiased score function, which leads to our final $\sqrt{n}$-CAN estimator of $\theta_\star$.

Thus far, we have constructed $\hat u$ and $\hat \theta^{\rm PI}$ as the minimizers of the penalized loss functions in Equations \eqref{eq:u_est_NN} and \eqref{eq:theta_PI}, respectively. This differs slightly from the original PINN estimator $(\hat u_{\rm PINN}, \hat \theta_{\rm PINN})$ defined in \eqref{eq:PINN_org}, where both quantities are obtained simultaneously in a single optimization step. Nevertheless, the two approaches are closely related; from Equation \eqref{eq:PINN_org}, one can equivalently characterize $(\hat u_{\rm PINN}, \hat \theta_{\rm PINN})$ as the solution to the following system of equations, which is essentially a profiled version of Equation \eqref{eq:PINN_org}: 
\begin{align}
\label{eq:u_est_NN_PINN}
\hat u_{\rm PINN} & = \argmin_{u \in \cF_{\rm NN}} \ \left\{\frac1n \sum_{i = 1}^n (Y_i - u(X_i))^2 + \lambda_n \|u\|^2_{W^{\beta, 2}} + \tilde \lambda_n \cR(u)\right\}, \\
\text{with } \ \ \cR(u) & = \min_{\theta \in \Theta} \cR(\theta, u), \notag \\
\label{eq:theta_est_NN_PINN}
\hat \theta_{\rm PINN} & = \argmin_{\theta \in \Theta} \cR(\theta, \hat u_{\rm PINN})
\end{align}
%as this two-step formulation is a profiled version of Equation \eqref{eq:PINN_org}. 
Therefore, the difference between $(\hat u, \hat \theta^{\rm PI})$ and $(\hat u_{\rm PINN}, \hat \theta_{\rm PINN})$ essentially lies in the first step: in the PINN formulation, the estimator $\hat u_{\rm PINN}$ is obtained via \eqref{eq:u_est_NN_PINN}, which includes the additional penalty term $\cR(u)$, whereas our approach used Equation \eqref{eq:u_est_NN}. Nonetheless, by adapting the arguments used in the proof of Theorem~\ref{thm:deep_sobolev}, one can show that $\hat u_{\rm PINN}$ enjoys similar theoretical guarantees as $\hat u$; we state this formally in Corollary~\ref{cor:deep_sobolev_PINN}. 
Consequently, one may equivalently use $\hat u_{\rm PINN}$ as an estimator of $u_\star$ and $\hat \theta_{\rm PINN}$ in place of $\hat \theta^{\rm PI}$ in the subsequent step, with all theoretical guarantees remaining essentially unchanged.
However, from a practical standpoint, our two-step approach is simpler to implement and facilitates a more transparent theoretical analysis. For this reason, we adopt it as our primary procedure. 
%However, one may equivalently use $\hat u_{\rm PINN}$ as an estimator of $u_\star$ and $\hat \theta_{\rm PINN}$ in place of $\hat \theta^{\rm PI}$, with all theoretical guarantees remaining essentially unchanged.

\begin{remark}
    \label{rem:minimaxity_pinn}
    At this point a natural question is whether the additional regularization $\cR(u)$ in the PINN estimator in Equation~\eqref{eq:u_est_NN_PINN} improves the rate of convergence. Observe that the true solution $u_\star$ satisfies $\cR(u_\star) = 0$. Indeed, the subset
    \[
    \cU_0 = \{u \in W^{\beta,2}(\Omega) : \cR(u) = 0\}
    \]
    consists precisely of those functions $u$ that solve the underlying PDE for some $\theta \in \Theta$. As shown in our previous work \cite{sun2024estimation}, under a mild injectivity condition on the forward map (which maps $\theta'$ to $u_{\theta'}$, the solution of the PDE with parameter $\theta = \theta'$), the set $\cU_0$ forms essentially a $q$-dimensional manifold of $W^{\beta, 2}(\Omega)$, where $q = \dim(\theta)$. Moreover, it was demonstrated in \cite{sun2024estimation} that, under suitable Hölder-type smoothness conditions on the forward map, one can estimate $u_\star$ at a near-parametric rate if the function class $\cF$ is restricted to $\cU_0$. 

    Although such an approach is computationally intractable in practice, this observation raises the natural question of whether the additional penalty $\cR(u)$ in Equation~\eqref{eq:u_est_NN_PINN} can improve the estimation rate of $\hat u_{\rm PINN}$ by encouraging solutions to concentrate near the low-dimensional manifold $\cU_0$. We conjecture that, under standard sieve-based implementations such as neural network classes, the inherent approximation error prevents the PINN penalty from improving the leading-order rate for estimating $u_\star$. Since our primary objective is inference on $\theta_\star$, we leave a sharper analysis of possible rate improvements for $\hat u$ as an interesting direction for future work.
 \end{remark}
 
\medskip
\noindent
{\bf Stage II: }A key concept underlying our debiasing strategy is \emph{Neyman orthogonality} \cite{neyman1959optimal,neyman1979c,chernozhukov2018double,foster2023orthogonal}, which we illustrate briefly here for the ease of the readers. Let $\ell(\theta,\eta)$ denote a loss function depending on a finite-dimensional parameter of interest $\theta$ and an infinite-dimensional nuisance parameter $\eta$ and $(\theta_\star, \eta_{\star})$ denotes the minimizer of the loss function. Let $s(\theta,\eta) = \nabla_\theta \ell(\theta,\eta)$ be the associated score function. We say that $s$ satisfies Neyman orthogonality at $(\theta_\star,\eta_{\star})$ if the Gateaux derivative of $s(\theta_\star,\eta)$ with respect to $\eta$ vanishes (or more generally, is of order $o(n^{-1/2})$) at $\eta=\eta_{\star}$, that is,
$$
\left.\frac{\partial}{\partial t}\, s\bigl(\theta_\star, \eta_{\star} + t h\bigr)\right|_{t=0} = 0,
$$
for all $h$ in a shrinking neighborhood around $\eta_{\star}$ with respect to $n$. Under standard regularity conditions, this orthogonality property ensures that the estimator of $\theta$ obtained by solving the empirical version of the equation $s(\theta,\hat\eta)=0$ is $\sqrt{n}$-consistent and asymptotically normal, provided the nuisance estimator $\hat\eta$ converges to $\eta_{\star}$ at a rate faster than $n^{-1/4}$. Intuitively, Neyman orthogonality essentially removes the first-order impact of nuisance estimation error, so that the leading asymptotic behavior of the estimator for $\theta$ is unaffected by replacing $\eta_{\star}$ with $\hat\eta$.

One of the main contribution of this work is the construction of an (approximately) Neyman-orthogonal score function associated with the PDE risk functional $\mathcal{R}(\theta,u)$ defined in Equation~\eqref{eq:pde_risk}. The score function corresponding to $\mathcal{R}(\theta,u)$ is given by
$$
(\nabla_\theta \cR)(\theta,u) =  \int_\Omega \left\{ K_0 u(x) + \left(\Psi_\theta\circ K_1 u\right)(x)\right\} \nabla_\theta \left(\Psi_\theta\circ K_1 u\right)(x)\omega(x)\rmd x.
$$
To achieve Neyman orthogonality, we require that the Gateaux derivative of $\nabla_\theta \mathcal{R}(\theta,u)$ with respect to the nuisance argument $u$, evaluated at the true parameter pair $(\theta_\star,u_\star)$, be (approximately) zero. Toward this goal, one of our key results (Proposition~\ref{prop:def_G}) establishes the existence of a vector-valued function $\mathsf{G}_{\theta,u} : \Omega \to \mathbb{R}^q$ such that
$$
\rmd (\nabla_\theta \cR)(\theta,u)[h] \eqdef \frac{\rmd }{\rmd t}(\nabla_\theta \cR)(\theta,u + t h )\vert_{t=0} = \int_{\Omega} h(x) \mathsf{G}_{\theta,u}(x)\rmd x, \quad \forall \ h \in W^{\beta, 2}(\Omega) \,.
$$
That is, $\mathsf{G}_{\theta,u}$ serves as a representer for the linear functional defined by the Gateaux derivative of $\nabla_\theta \mathcal{R}(\theta,u)$ with respect to $u$. Furthermore, $\sG_{\theta, u}$ is \emph{smooth} with respect to both $(\theta, u)$, as estbalished in Proposition \ref{prop:def_G}. 
%for details. %\Yves{I think we might need to work in $W^{\beta,q}$ for some $q\geq 2$. I came to that conclusion when I try to check our assumptions on Burger's equation. I see that you work around this issue by using Sobolev embedding theorems. Which forces $\beta$ to be slightly larger than $2\tau$. Perhaps we lose a bit there in terms of conditions on $\beta$, but the simplicity in the proof that we get is worth it, I think. } 
This representation immediately suggests the construction of a debiased score of the form
$$
\sS^{\rm db}(\theta, u) = \nabla_\theta \cR(\theta, u) + \int (u_\star(x) - u(x)) \ \sG_{\theta, u}(x) \ dx 
$$
for which the Gateaux derivative at $(\theta_\star,u_\star)$ vanishes in the direction $u-u_\star$. Leveraging this observation, together with the fact that $\mathbb{E}[Y\mid X=x]=u_\star(x)$, we define the empirical debiased score function as
\begin{equation}
\label{eq:debiased_score}
\sS_n^{\rm db}(\theta, u) = \nabla_\theta \cR(\theta, u) + \frac1n \sum_{i = 1}^n (Y_i - u(X_i)) \sG_{\theta, u}(X_i) \,,
\end{equation}
noting that $\mathbb{E}[Y\,\mathsf{G}_{\theta,u}(X)] = \mathbb{E}[u_\star(X)\,\mathsf{G}_{\theta,u}(X)]$. Our ultimate goal is to solve the estimating equation $\mathsf{S}_n^{\rm db}(\theta,\hat u)=0$ for $\theta$, where $\hat u$ is obtained from Equation~\eqref{eq:u_est_NN}. However, for finite samples, it is not guaranteed that such a root exists or is unique. We therefore define our final estimator $\htdb$ as the solution to the following optimization problem:
\begin{equation}
\label{eq:def_htdb}
\hat \theta^{\rm db} \in  \argmin_{\theta \in \Theta_n} \ \|\sS_n^{\rm db}(\theta, \hat u)\|_2^2 \,, \quad \Theta_n = \{\theta \in \Theta: \|\theta - \hat \theta^{\rm PI}\|_2 \le \delta_n\}
\end{equation}
where $\{\delta_n\}\downarrow 0$ is a deterministic sequence satisfying i) $\delta_n \gg n^{-\beta/(2\beta + d)}(\log{n})^{(3\beta + d)/(2\beta + d)}$ and $\bbP(\theta_\star \in \Theta_n) \to 1$. We refer the readers to Lemma \ref{lem:consistency_PI} and Lemma S.5.1 of the supplementary document for a range of choices of such $\{\delta_n\}$. Since $\Theta_n$ is compact and $\mathsf{S}_n^{\rm db}(\theta,\hat u)$ is continuous in $\theta$ under Assumption \ref{assm:psi_smoothness} and Proposition \ref{prop:def_G}, the minimizer in \eqref{eq:def_htdb} is guaranteed to exist. Moreover, the shrinking nature of $\Theta_n$ immediately implies consistency of $\hat\theta^{\rm db}$ for $\theta_\star$. In Section~\ref{sec:theory}, we prove that $\hat\theta^{\rm db}$ is $\sqrt{n}$-consistent and asymptotically normal.

\begin{algorithm}[htbp]
\caption{Two-step debiased estimation of $\theta_\star$ from frequentist perspective}
\label{alg:two_step_freq}
\begin{algorithmic}[1]
\Require Data $\{(X_i,Y_i)\}_{i=1}^{n}$, the PDE map, and a collection of DNNs.
\Ensure A rate optimal estimator $\hat u$ of $u_\star$ (along with derivatives), and a $\sqrt{n}$-CAN estimator $\htdb$ of $\theta_\star$. 
\State Obtain $\hat u$ by solving the minimization problem Equation \eqref{eq:u_est_NN} over a collection of DNNs to obtain $\hat u$. 
\State Compute $\hat \theta^{\rm PI} = \argmin_{\theta} \cR(\theta, \hat u)$ as in Equation \eqref{eq:theta_PI}. 
\State Define a shrinking neighborhood $\Theta_n = B_{\delta_n}(\hat \theta^{\rm PI})$, where $\delta_n$ is defined in Equation \eqref{eq:def_htdb}. 
\State Construct the \emph{debiased} score function $\sS_n^{\rm db}(\theta, u)$ as in Equation \eqref{eq:debiased_score}. 
\State Obtain the debiased estimator $\htdb$ as $\htdb = \argmin_{\theta \in \Theta_n} \|\sS_n^{\rm db}(\theta, \hat u)\|_2^2$. 
\State \Return $(\hat u, \htdb)$. 
\end{algorithmic}
\end{algorithm}

\subsection{Debiased Bayesian inference}
A Bayesian PINN formulation of the PDE parameter estimation learns $(\theta,u)$ jointly. Following  \cite{yang:bpinn:21} and the setup studied in \cite{sun2024estimation}, let $\bar{\Pi}_0$ be a given joint prior distribution on $\rset^q\times W^{\beta,2}(\Omega)$. First, we turn $\bar\Pi_0$ into the so-called PINN prior by considering the probability measure with density proportional to $(\theta,u)\mapsto e^{-\lambda\cR(\theta,u)/2}\bar\Pi(\rmd\theta,\rmd u)$, where $\lambda>0$ is a regularization parameter. Here, the term $\lambda \cR(\theta,u)/2$ is used to enforce the PDE structure. Using this PINN prior distribution and the data distribution, the joint PINN posterior distribution of $(\theta,u)$ on $\rset^q\times W^{\beta,2}(\Omega)$ is given by
\begin{equation}
\label{def:post:1}
\Pi(\rmd\theta,\rmd u\vert\D_n)\propto \exp\left(-\frac{1}{2\sigma^2}\sum_{i=1}^n(Y_i-u(X_i))^2 -\frac{\lambda}{2} \cR(\theta,u) \right)\bar\Pi_0(\rmd\theta,\rmd u).
\end{equation}
This posterior distribution has been recently studied in \cite{sun2024estimation} where the authors show that the marginal posterior distribution of $u$ contracts at the minimax nonparametric rate, and the marginal posterior of $\theta$ contracts at the corresponding plug-in rate.  
%As in the frequentist setting a parametric sieve family (e.g., a neural network function class) is used as a model for $u$. 
Let $\Pi^{(\theta)}(\cdot\vert\D_n)$ (resp. $\Pi^{(\theta\vert u)}(\cdot\vert\D_n)$, and $\Pi^{(u)}(\cdot\vert\D_n)$) denote the marginal posterior distribution of $\theta$ (resp. the conditional posterior distribution of $\theta$ given $u$, and the marginal posterior distribution of $u$) under the joint distribution in Equation \eqref{def:post:1}. We can re-express the marginal posterior distribution of $\theta$ as
\[\Pi^{(\theta)}(\rmd\theta\vert\D_n)=\int_{W^{\beta,2}(\Omega)}\Pi^{(\theta\vert u)}(\rmd\theta\vert u,\D_n)\;\Pi^{(u)}(\rmd u\vert\D_n),\]
which shows that indeed, as in the frequentist setting, Bayesian inference for $\theta$ is a plug-in method that suffers from the same slow convergence rate in estimating $u$. This raises the question: \emph{is it possible to correct the bias and achieve the $\sqrt{n}$-estimation rate in the Bayesian estimation of $\theta$?}

The estimation of finite-dimensional parameters at the $\sqrt{n}$-rate in semiparametric models has been actively investigated in the Bayesian literature \citep{castillo:12, Bickel:kleijn:2012,Castillo:rousseau:2015,rousseau:16,Minwoo:etal:2019}. As in the frequentist setting, a commonly used approach to achieve the $\sqrt{n}$-rate relies on the local asymptotic normality (LAN) expansion of the log-likelihood together with careful tuning of the prior distribution to control bias (see, e.g., \cite{rousseau:16} for a survey). And as pointed out above, it is unclear how this approach applies in settings such as ours, where the exact bias-variance trade-off of deep neural network-based models is not fully understood.

We extend our debiasing technique into a novel approach for debiasing posterior distributions. Let $\pi_0$ be a prior density on $\rset^q$ with respect to the Lebesgue measure (it need not coincide with the $\theta$-marginal of $\bar\Pi_0$ introduced above). Given a symmetric positive definite matrix $V_n\in\rset^{q\times q}$, and using the de-biased score function $\mathsf{S}_n^{\rm db}$ (Equation \eqref{eq:debiased_score}), we consider the following \textit{debiased conditional posterior distribution of $\theta$ given $u$} given by 
\begin{equation}\label{def:cond:post:db}
\Pi_{\rm db}^{(\theta\vert u)}(\rmd\theta\vert u,\D_n) \propto\exp\left(-\frac{n}{2} \sS^{\rm db}_n(\theta,u)^\top V_n^{-1} \sS^{\rm db}_n(\theta,u)\right)\pi_0(\theta)\rmd \theta.\end{equation}
Given an initial estimator $\hat u$ of $u_\star$, our proposed  debiased Bayesian approach consists in drawing inference on $\theta$ using the (plug-in) probability measure $\Pi_{\rm db}^{(\theta\vert u)}(\cdot\vert \hat u,\D_n)$. A fully Bayesian flavor of this procedure can also be defined by employing an initial marginal posterior distribution of $u$ (such as the $u$-marginal obtained from Equation \eqref{def:post:1}), leading to the alternative de-biased posterior
\begin{equation}\label{def:marg:post:db}
\Pi^{(\theta)}_{\rm db}(\rmd\theta\vert\D_n)=\int_{W^{\beta,2}(\Omega)}\Pi^{(\theta\vert u)}_{\rm db}(\rmd\theta\vert u,\D_n)\;\Pi^{(u)}(\rmd u\vert\D_n).\end{equation}
One important limitation of (\ref{def:marg:post:db}) is that it can be substantially more challenging to handle computationally than (\ref{def:cond:post:db}). Hence, out of computational considerations, we will focus our analysis on the plug-in posterior $\Pi_{\rm db}^{(\theta\vert u)}(\cdot\vert \hat u,\D_n)$. However, provided that the marginal posterior $\Pi^{(u)}(\cdot\vert\D_n)$ employed has contraction properties similar to those of $\hat u$, our results extend readily to (\ref{def:marg:post:db}). A summary of the resulting debiased Bayesian inference is given in Algorithm \ref{alg:bayesian}. Regarding $V_n$, the debiased posterior can be constructed with any positive definite matrix $V_n$, however a careful choice of $V_n$ is needed for the resulting posterior inference to match the frequentist inference, and we address this question in the theoretical discussions. 

%Our approach is related to two recent works on Bayesian posterior correction in semiparametric models. In \cite{yiu:etal:2025}, the authors leverage influence functions and the one-step update estimator to develop a posterior correction method for models in which part of the data distribution is the nuisance parameter. More closely related to our work is the Bayesian inference framework under Neyman orthogonality proposed by \cite{sabbagh:stephens:2026}. An important limitation of their framework, however, is that the conditions under which the corrected posterior defines a valid probability measure remain unclear and must be verified on a case-by-case basis. By contrast, our debiased posterior distributions are always well-defined and can be readily explored via MCMC; in some instances involving linear PDEs with a Gaussian prior $\pi_0$, exact Monte Carlo sampling is also feasible.

\begin{algorithm}[ht]
\caption{Bayesian two-step debiased estimation of $\theta_\star$}
\label{alg:bayesian}
\begin{algorithmic}[1]
\Require Data $\{(X_i,Y_i)\}_{i=1}^{n}$, the PDE map, a collection of DNNs $\{u,\;u\in\F_{\rm NN}\}$ and, if available, a symmetric positive definite matrix $V_n$.
\Ensure A rate optimal estimator $\hat u$ of $u_\star$, and posterior samples $\{\theta^{(0)},\ldots,\theta^{(K)}\}$. 
\State Obtain $\hat u$ by solving the minimization problem Equation \eqref{eq:u_est_NN} over $\F_{\rm NN}$ to obtain $\hat u$. 
\State Compute $\hat \theta^{\rm PI} = \argmin_{\theta} \cR(\theta, \hat u)$ as in Equation \eqref{eq:theta_PI}. 
\State If needed, use $(\hat u, \hat \theta^{\rm PI})$ to compute the weight matrix $V_n$.
\State Use MCMC to produce sample $\{\theta^{(0)},\ldots,\theta^{(K)}\}$ from the debiased posterior distribution $\Pi_{\rm db}^{(\theta\vert u)}(\cdot\vert \hat u,\D_n)$ as defined in (\ref{def:cond:post:db}).
\item Return $\hat u$ and $\{\theta^{(0)},\ldots,\theta^{(K)}\}$.
\end{algorithmic}
\end{algorithm}

\section{Theoretical results}
\label{sec:theory}
In this section, we present our main theoretical results from both the frequentist and Bayesian perspectives. We start with some notations and definitions. Throughout, we use the following standard notion of multi-index derivatives: given $\balpha =(\alpha_1,\ldots,\alpha_d)$, where $\alpha_i\geq 0$ is an integer, the $\balpha^{th}$ derivative operator $D^{\balpha}$ is defined on a function $u: \reals^d \to \reals$ as: 
\[D^{\balpha} u(x) \eqdef \frac{\partial^{|{\balpha}|}u(x)}{\partial x_1^{\alpha_1}\ldots \partial x_d^{\alpha_d}},\]
where $|{\balpha}| =\sum_i \alpha_i$ and by convention, $\partial x_j^{0}=1, D^{{\bf 0}} u =u$.
\begin{definition}[$L^p$ space]
    \label{def:L_p_space}
    Given any measurable set $\Omega \subseteq \reals^d$ and $p \in [1, \infty)$, the collection $L^p(\Omega, \reals)$ is the collection of measurable functions from $\Omega$ to $\reals$, which has finite $L_p$ norm with respect to the Lebesgue measure, i.e. 
    $$
    L^p(\Omega,\reals) \eqdef \left\{f: \Omega \mapsto \reals: \|f\|_{L^p(\Omega)}\eqdef \left(\int_\Omega |f(x)|^p \ \rmd x \right)^{1/p}< \infty \right\}, \,
    $$ 
    For $p = \infty$, we have $\|f\|_{L^\infty(\Omega)}\eqdef \sup_{x\in\Omega}|f(x)|$.
    When there is no risk of confusion with the standard Euclidean norm, we will use the shorthand $\|f\|_p$ instead of $\|f\|_{L^p(\Omega)}$. Furthermore, we use the notation $L^p(\Omega: \reals^j)$ to denote the set of all function $f = (f_1, \dots, f_j): \Omega \mapsto \reals^j$ such that $f_i \in L^p(\Omega, \reals)$ for $1 \le i \le j$.
\end{definition}
Given integer $\beta\geq 0$, $C^\beta(\Omega)$ denotes the space of $\beta$-times continuously differentiable functions on $\Omega$. For $u\in C^\beta(\Omega)$, we define 
\[ \|u\|_{C^\beta(\Omega)}\eqdef \max_{\bm{\alpha}:\;|\bm{\alpha}|\leq \beta}\; \|D^{\bm{\alpha}} u\|_\infty. \]
In the following $C^\infty_c(\Omega)$ denotes the space of infinitely differentiable functions on $\Omega$ with compact support in $\Omega$ (hence taking  value $0$ on the boundary of $\Omega$).

\begin{definition}[H\"{o}lder Space]
    \label{def:Holder_space}
    Given integer $\beta\geq 0$, and $\gamma\in (0,1]$ we define the  H\"{o}lder function space $C^{\beta, \gamma}(\Omega)$ as the collection of all measurable functions $u: \Omega \to \reals$, such that $D^{\balpha} u$ exists for all $\balpha: |\balpha| \le \beta$ and are uniformly bounded over $\Omega$ and   
    \[
    \|u\|_{C^{\beta,\gamma}(\Omega)}\eqdef \sum_{\bm{\alpha}:\;|\bm{\alpha}|\leq \beta}\; \|D^{\bm{\alpha}} u\|_\infty + \sum_{\bm{\alpha}:\;|\bm{\alpha}| = \beta}\; \sup_{x\neq y}\frac{|D^{\bm{\alpha}} u(y) - D^{\bm{\alpha}} u(x)|}{\|y-x\|_2^{\gamma}} <\infty.
    \]
 We call $\|u\|_{C^{\beta,\gamma}(\Omega)}$ the H\"{o}lder norm of $u$, and given $L>0$, the notation $C^{\beta,\gamma}(\Omega,L)$ denotes the sub-space of functions $u$ such that $\|u\|_{C^{\beta,\gamma}(\Omega)}\leq L$.
\end{definition}

%Given a multi-index ${\bf k}=(k_1,\ldots,k_d)$, and a locally integrable function $u:\Omega\to\rset$, the ${bf k}$-th weak derivative of $u$ is the unique function $g\in\rset^d\to\rset$ such that for all $\phi\in C_c^\infty(\Omega)$,
%\[\int_\Omega u(x) D^{{\bf k}}\phi(x) \rmd x = (-1)^{|{\bf k}|} \int_\Omega \phi(x) g(x) \rmd x.\]
%Clearly if $u$ has a ${bf k}$-th derivative in the classical sense, then its ${bf k}$-th weak derivative exists and the two functions are equal.

\begin{definition}[Sobolev Space]
    \label{def:sobolev}
    Given integer $\beta \in \bbN$ and $p \ge 1$, we define the Sobolev function space $W^{\beta, p}(\Omega)$ as the collection of all measurable functions $u: \Omega \to \reals$, such that for any multi-index $\balpha=(\alpha_1,\ldots,\alpha_d)$, with $|\balpha| \le \beta$, there exists $g_{\balpha} \in L^p(\Omega,\rset)$ that satisfies 
    \[ \int_\Omega u(x) D^{\balpha}\phi(x)\rmd x = (-1)^{|\balpha|} \int_\Omega g_{\balpha}(x)\phi(x)\rmd x,\;\;\mbox{ for all } \;\; \phi\in C_c^\infty(\Omega).\] 
    The function $g_\alpha$ is called the weak $\balpha$-index derivative of $u$, and denoted as $D^{\balpha} u$. The Sobolev norm of $u$ is
    $$
    \|u\|_{W^{\beta,p}(\Omega)} \eqdef \left(\sum_{{\bf k}:\;|{\bf k}|\leq \beta}\; \int_\Omega |D^{{\bf k}} u(x)|^p \rmd x\right)^{1/p} < \infty.
    $$
    In particular, we use the notation $W^{\beta, p}(\Omega, L)$ to be all functions $u$ such $ \|u\|_{W^{\beta,p}(\Omega)} \le L$. Finally we will write $W_0^{\beta,p}(\Omega)$ to denote the closure of $C_c^\infty(\Omega)$ in $W^{\beta,p}(\Omega)$.
\end{definition}

\begin{remark}
    It is well-known that if $u:\;\Omega\to\rset$ has a derivative $D^{\balpha} u$ in the classical sense, then it has a weak ${\balpha}$-index derivative $g_\balpha$ as defined above and moreover $g_\balpha = D^{\balpha} u$ almost everywhere. In this work, we will say that $u\in W^{\beta,p}(\Omega)$ in the classical sense if $D^{\balpha}u$ exists in the classical sense and $D^{\balpha}u\in L^p(\Omega,\rset)$, for all $|\balpha|\leq \beta$.
\end{remark}

We will say that an operator $K:\;L^2(\Omega,\rset)\to L^2(\Omega,\rset)$ is a differential operator of order $\tau$ if $K = \sum_{\balpha:|\balpha| \le \tau} a_{\bm{\alpha}} D^{\bm{\alpha}}$ where the coefficients $\{a_{\bm{\alpha}}: \reals^d \mapsto \reals,  |\balpha| \le \tau\}$ belongs to $C^{\beta}(\bar \Omega)$. Any such differential operator $K$ admits an adjoint $K^\dagger$ defined as
\begin{equation}
\label{eq:adjoint:def}  
K^\dagger u \eqdef \sum_{\bm{\alpha}:\;|\bm{\alpha}|\leq \tau} (-1)^{|\bm{\alpha}|}  D^{\bm{\alpha}}(a_{\bm{\alpha}} u),
\end{equation}
such that
\begin{equation}
\label{eq:adjoint} 
\pscal{Kf}{g} = \pscal{f}{K^\dagger g},\;\;\; f\in W^{\tau,2}(\Omega), \; g\in C_c^\infty(\Omega).
\end{equation}
Equation (\ref{eq:adjoint}) follows from the integration by parts formula (the boundary terms vanish as $g \in  C_c^\infty(\Omega)$). Moreover, it continues to hold for $g \in W_0^{\tau,2}(\Omega)$ by passing to the limit, since $W_0^{\tau,2}(\Omega)$ is the closure of $C_c^\infty(\Omega)$ in $W^{\tau,2}(\Omega)$. 
A special case is $K = D^{\balpha_*}$, where $a_{\balpha} \equiv 0$ for $\balpha \neq \balpha_*$ and $a_{\balpha_*} \equiv 1$. 
Let $K_0^\dagger,K_{1,1}^\dagger,\ldots,K_{1,p}^\dagger$ denote the adjoints of $K_0,K_{1,1},\ldots,K_{1,p}$ respectively as defined in (\ref{eq:adjoint:def}).  We set $K_1^\dagger\eqdef (K_{1,1}^\dagger,\cdots,K_{1,p}^\dagger)^\top$ as the adjoint of the operator $K_1$.  We will allow $K_1^\dagger$ to operate on matrix-valued functions as follows. Let  $A:\Omega \to\rset^{q\times p}$, where each component $A_{ij} \in W_0^{\tau,2}(\Omega)$, we define $K_1^\dagger A:\;\Omega\to\rset^q$ as $(K_1^\dagger A)_i \eqdef \sum_{k=1}^p K_{1,k}^\dagger A_{ik}$. With this definition, it is easy to check that the following generalization of (\ref{eq:adjoint}) holds:
\[ \int_\Omega A(x) K_1 u(x)\rmd x = \int_\Omega u(x) K_1^\dagger A(x)\rmd x.\] 
Before stating the main theorems, we first introduce the assumptions required for their proofs.

\begin{assumption}[Data generating process]
\label{assm:dgp}
    The design points $\{X_i\}_{1 \le i \le n}$ are assumed to be generated uniformly from $\Omega = (0, 1)^d$. The responses are generated as $Y_i = u_\star(X_i) + \eps_i$, where $u_\star \in W^{\beta, 2}(\Omega, L)$ in the classical sense and satisfy the PDE equation \eqref{eq:pde_model_2} for some $\theta = \theta_\star$ and the smoothness index satisfy $\beta - 2\tau > d$, where $\tau\geq 1$ is the maximum order of the derivatives in the PDE equation. The errors $\{\eps_i\}_{1 \le i \le n}$ are i.i.d. subgaussian random variables with mean $0$ and variance $\sigma^2_\eps$ and are independent of $X_i$. 
\end{assumption}
Assumption~\ref{assm:dgp} represents a standard setup in nonparametric 
statistical learning. The compact support condition on $\Omega$ and the 
sub-Gaussian assumption on $\eps$ can, in principle, be relaxed at the cost 
of additional technical arguments, without altering the fundamental 
statistical insights of the problem. We therefore adopt these conventional 
assumptions to streamline the presentation. 
% Moreover, to ensure that the PDE constraint is well-defined, we require 
% the smoothness parameter $\beta$ of $u$ to satisfy $\beta > \tau$, where 
% $\tau$ denotes the maximal order of derivatives appearing in the PDE 
% constraint \eqref{eq:pde_model_2}. 
% It is worth emphasizing that much of the existing literature on neural 
% network-based nonparametric estimation assumes the regression function 
% belongs to a H\"{o}lder class. In contrast, we work under a Sobolev 
% smoothness assumption, which introduces additional analytical challenges. 
% In particular, the neural network approximation results 
% for Sobolev functions provide optimal rates in the $L_2$ norm, rather than 
% in the $L_\infty$ norm that is frequently employed in the H\"{o}lder setting. 
\begin{remark}
\label{rem:smoothness_tau}
For the debiasing step and to establish asymptotic normality of the 
debiased estimator $\htdb$, we impose the condition $\beta - 2\tau > d$. If instead we assume $u_\star \in C^{\beta,\gamma}(\Omega, L)$, that is, 
$u_\star$ belongs to a H\"older class, then the corresponding requirement 
would reduce to $\beta - 2\tau > d/2$. 
At present, it is not clear whether these smoothness conditions are optimal. 
Investigating their sharpness remains an important direction for future work.
\end{remark}
Next, we describe the architecture of the neural network collection. We will use a standard RePU($\nu$) activation function (with $\nu \ge \beta$) with width $N$, depth $L$, and total number of trainable weights $W$. To keep the neural network outputs bounded as well as smooth, we apply a smooth truncation. In particular, let $T_m \in C^{\nu_*}$ ($\nu_* \ge \beta$) be a piecewise polynomial function that smoothly truncates $x$ at level $m$, i.e. $T_m(x) = x$ for all $|x| \le m-\delta$ and $T_m(x) = \sign(x)m$ for $|x| \ge m$, where $m$ is a large constant (typically one may take $m > 2\|u_{\star}\|_\infty$). See Lemma S.9.1 of the supplementary document for the exact definition. We slightly enlarge our collection $\cF_{\rm NN}$ as $\tilde \cF_{\rm NN} = \{T_m \circ u: \ u\}$, where $u$ is a RePU($\nu$)  network with width $N$ depth $L$ and total number of parameters $W$. 
% Note that any function in this new collection $\cF_{\rm NN}$ can also be realized a neural network with depth $L+1$, width $N$ and total number of weights $W$, and the architecture is essentially: 
% $$
% x \to \text{Linear} \to \text{RePU} \to \cdots \to \text{Linear} \to \text{RePU} \to \text{Linear} \to T_m \to \text{Output} \,.
% $$
% The only difference from standard architecture is that the activation function in the last layer is different, but both the activation functions (RePU($\nu$) and $T_m$) are piecewise polynomial functions. 
Assuming that neural network outputs are truncated or uniformly bounded is standard in the theory of neural networks (e.g., see Equation (4) of \cite{schmidt2020nonparametric}, or the definition of $T_\beta$ in \cite{kohler2021rate} or the definition of $\bar T_m$ in Definition 2 of \cite{fan2024factor}). Traditionally, this is achieved by the hard truncation map $T(x)=\sign(x)(|x|\wedge m)$. However, this map is non-smooth, and hence the derivatives of functions in $\cF_{\rm NN}$ are not well defined at the truncation points. Since our goal is to estimate both the target function and its derivatives simultaneously, we use a smooth truncation scheme instead. Because \(\nu \ge \beta\), the resulting truncated neural networks continue to belong to the Sobolev space $W^{\beta,2}$.

\begin{assumption}[Function class $\cF_{\rm NN}$]
\label{assm:function_class}
We choose $\cF_{\rm NN}$ as a collection of feedforward neural network functions $\Omega\to\rset$, with width of the order $N = c_1(n/\log{n})^{d/(2\beta + d)}$, depth of the order $c_2$, and total number of weights of the order $c_3(n/\log{n})^{d/(2\beta + d)}$, along with RePU($\nu$) activation function for some $\nu \ge \beta$, where $c_1,c_2,c_3$ are constants independent of $n$. Furthermore, we set $\tilde \cF_{\rm NN} = \{T_m \circ u: u \in \cF_{\rm NN}\}$.   
\end{assumption}

\begin{remark}
    For large values of $\nu$, the RePU($\nu$) activation function is known to be numerically unstable, and the modified rectified power unit (mRePU) has been proposed as an alternative \citep{mrepu:2026}. The mRePU($\nu$) is defined as ${\rm mRePU}_\nu(z) = z(z+1)^\nu$ if $z\geq -1$, and zero otherwise, which can be written as the difference of two RePU: ${\rm mRePU}_\nu(z) = {\rm RePU}_{\nu+1}(z+1) - {\rm RePU}_{\nu}(z+1)$. Therefore our results directly apply to mRePU activation functions as well.
\end{remark}

\begin{assumption}[Local identifiability and curvature]
\label{assm:curvature_pde_risk}
Assume that $\Theta\subset\mathbb R^q$ is compact and that
$\theta_\star\in \operatorname{int}(\Theta)$. Moreover, assume the following.
\begin{enumerate}
\item[(i)] \textbf{Forward well-posedness and local injectivity.}
There exists a neighborhood $\mathcal N_\theta$ of $\theta_\star$ such that, for every $\theta\in\mathcal N_\theta$, the PDE system $K_0(u) + \Psi_\theta \circ K_1(u)=0$ along with $\cB(u)=0$, admits a unique solution $u(\theta)\in W^{\beta,2}(\Omega)$.
The forward map $\theta\mapsto u(\theta)$ is locally injective at $\theta_\star$ and continuous at $\theta_\star$ in the $W^{2\tau,2}$ norm.

\item[(ii)] \textbf{Identification of the inverse problem.}
The true parameter $\theta_\star$ is the unique minimizer of $\theta\mapsto \cR(\theta,u_\star)$ over $\Theta$. Equivalently, for every $\eta>0$,
\[
\inf_{\theta\in\Theta:\ \|\theta-\theta_\star\|_2\ge \eta}
\cR(\theta,u_\star)>0 .
\]

\item[(iii)] \textbf{Local curvature.}
There exist constants $\delta_0,\delta_1,\lambda_->0$ such that $B(\theta_\star,\delta_1)\subset \Theta$ and
\[
\inf_{\|u-u_\star\|_{W^{2\tau,2}}\le \delta_0}
\inf_{\|\theta-\theta_\star\|_2\le \delta_1}
\lambda_{\min}\left\{\nabla_\theta^2 \cR(\theta,u)\right\}
\ge \lambda_- .
\]
\end{enumerate}
\end{assumption}
Assumption~\ref{assm:curvature_pde_risk} consists of three mild requirements: local well-posedness of the forward problem, identifiability of the inverse problem, and local curvature of the PDE risk functional. Recall that $\cB$ encodes the boundary and/or initial conditions needed to make the PDE solution well defined, while $\cR(\theta,u)$ denotes the weighted $L^2$ norm of the interior PDE residual. Since the weight function $w$ is strictly positive in the interior of $\Omega$ (see Assumption~\ref{assm:w}), the condition $\cR(\theta,u)=0$ implies that the PDE residual vanishes almost everywhere on $\Omega$. The first condition states that for every $\theta$ in a neighborhood of $\theta^\star$, the PDE admits a unique solution $u(\theta)\in W^{\beta,2}(\Omega)$. Thus the forward map $\theta\mapsto u(\theta)$ is locally well defined around $\theta_\star$. The additional local injectivity and continuity of this map around $\theta_\star$ are standard stability requirements in inverse problems: nearby parameters produce nearby solutions, and no two nearby parameter values generate the same solution \citep{stuart2010inverse,nickl2023bayesian}.

The second requirement is an identifiability condition for the inverse problem. Since our preliminary estimator $\hat \theta^{\rm PI}$ is defined as a minimizer of $\theta\mapsto \cR(\theta,\hat u)$ over $\Theta$, it is not enough to assume only local injectivity of the forward map. We also require that $\theta_\star$ be the unique minimizer of the population risk $\theta\mapsto \cR(\theta,u_\star)$ over $\Theta$. This condition rules out distant parameters that fit the true solution equally well and ensures consistency of the plug-in estimator $\hat\theta^{\rm PI}$ as long as $\hat u$ is a consistent estimate of $u_{\star}$.

The third requirement imposes a local curvature condition on the PDE risk functional. Specifically, it requires that the Hessian $\nabla_\theta^2 \cR(\theta,u)$ remain uniformly positive definite for $(\theta,u)$ in a neighborhood of $(\theta_\star,u_\star)$. This is a mild local nondegeneracy condition on the sensitivity of the PDE residual with respect to $\theta$, analogous to the positive-definiteness of the Fisher information matrix in classical parametric inference. It guarantees that the risk is locally strongly convex in $\theta$, so that perturbations of the solution $u_\star$ lead only to stable perturbations of the local minimizer in $\theta$. Assumption~\ref{assm:curvature_pde_risk} is satisfied by many standard
parametric PDE models (see Section \ref{sec:simulation} and Section S.1 of the supplementary document for some illustrative examples).
\begin{assumption}[Smoothness and local boundedness of $\Psi$]
\label{assm:psi_smoothness}
Let $\Theta_0 \subseteq \reals^q$ be an open convex set containing the parameter space $\Theta$, and let $\cO \subseteq \reals^d$ be an open set such that $\overline{\Omega} \subseteq \cO$. The nonlinear map
\[
\Psi:\cO\times \reals^p\times \Theta_0 \to \reals, \qquad (x,z,\theta)\mapsto \Psi(x,z,\theta) = \Psi_\theta(x, z),
\]
is sufficiently smooth in the spatial variable $x$, the feature variable $z$, and the parameter $\theta$. Specifically,, for multi-indices $a\in\mathbb{N}_0^q$, $b\in\mathbb{N}_0^d$, and $c\in\mathbb{N}_0^p$, with $|a|\leq 3$, $|b| + |c|\leq \beta$, we assume that the corresponding mixed derivative $(x,z,\theta)\mapsto \partial_\theta^{(a)} \partial_x^{(b)} \partial_z^{(c)} \Psi(x,z,\theta)$ is well-defined and continuous. % denote the corresponding mixed derivative of $\Psi$. Assume that for every $|a|\le 3$, the map $(x,z)\mapsto \partial_\theta^{(a)} \Psi(x,z,\theta)$ belongs to $C^\beta(K)$, for any compact set $K \subseteq \cO \times \reals^p$ and for every fixed $\theta\in\Theta_0$. 
Moreover, there exist continuous functions $ c:\Theta_0\to [0,\infty)$ and continuous nondecreasing functions $L_0:[0,\infty)\to [0,\infty)$ such that the following condition hold: for all $\theta\in\Theta_0$, $x\in\overline{\Omega}$, and $z\in\reals^p$,
\[
\max_{\substack{|a|\le 3\\ |b|+|c|\le \beta}} \left|\partial_\theta^{(a)} \partial_x^{(b)} \partial_z^{(c)} \Psi(x,z,\theta)\right| \le c(\theta)L_0(\|x\|_2 + \|z\|_2).
\]
% \begin{enumerate}
% \item 
% \item For all $r>0$, all $\theta\in\Theta_0$, all $x,x'\in\overline{\Omega}$, and all $z,z'\in\reals^p$ with $ \|z\|_2\vee \|z'\|_2 \le r$, we have
% \begin{align*}
% & \max_{\substack{|a|\le 2\\ |b|+|c|\le \beta-1}} \left| \partial_\theta^{(a)} \partial_x^{(b)} \partial_z^{(c)} \Psi(x,z,\theta)  - \partial_\theta^{(a)} \partial_x^{(b)} \partial_z^{(c)} \Psi(x',z',\theta)\right|
% \\
% & \qquad \qquad \qquad \le c(\theta)L(r) \left(\|x-x'\|_2+\|z-z'\|_2\right).
% \end{align*}
% \end{enumerate}
\end{assumption}
\begin{remark}
Since $\Theta_0$ is convex, the above assumption immediately implies that for any convex compact set $K\subseteq\Theta_0$ and every $r<\infty$, there exists a constant $C_{K,r}<\infty$ such that, for all $\theta,\theta'\in K$, $x\in\overline{\Omega}$, and $\|z\|_2\le r$,
$$
\max_{\substack{|a|\le 2\\ |b|+|c|\le \beta}} \left|\partial_\theta^{(a)}\partial_x^{(b)}\partial_z^{(c)}\Psi(x,z,\theta) - \partial_\theta^{(a)}\partial_x^{(b)}\partial_z^{(c)}\Psi(x,z,\theta')\right| \le C_{K,r}\|\theta-\theta'\|_2 \,.
$$
\end{remark}

\medskip 
Assumption~\ref{assm:psi_smoothness} imposes standard regularity and growth conditions on the nonlinear map $\Psi_\theta$, ensuring that the PDE residual and the associated risk functional are sufficiently smooth in both the parameter $\theta$ and the state variable. 
The requirement that $\theta \mapsto \Psi_\theta(x)$ be thrice continuously differentiable guarantees that the PDE risk $\mathcal{R}(\theta,u)$ admits well-defined second- and third-order derivatives in $\theta$, which are required for establishing local curvature, performing Taylor expansions, and deriving asymptotic normality of the debiased estimator. 
%The condition that $x \mapsto \nabla_\theta \Psi_\theta(x)$ belongs to $C^\tau(\Omega)$ ensures that $G_{\theta, u}$ the Gateux-derivative of $\nabla_\theta \cR(\theta, u)$ with respect to $u$, as defined in Stage II of Section \ref{sec:method}, is well-defined (see Proposition \ref{prop:def_G} for the exact definition of $\sG_{\theta, u}$). In particular, this implies, for any $u \in W^{\beta, 2}(\Omega, L)$, $\Psi_\theta \circ u, \partial_{\theta_j} \Psi_\theta \circ u$, $\partial_{\theta_i \theta_j}\Psi_\theta \circ u$ belongs to $W^{\tau, p}(\Omega)$, see Lemma \ref{lem:smoothness_comp_psi} for details. 
%The growth condition rules out excessively rapid growth (e.g., exponential growth) while still permitting a broad class of commonly used nonlinear functions.
Importantly, these conditions are mild and standard in nonlinear PDE-constrained estimation and semiparametric inference: they merely ensure controlled growth and local smoothness, rather than imposing restrictive structural assumptions on $\Psi_\theta$.

% \begin{assumption}[Compactness]
% \label{assm:compactness}
%     We assume that $X \in \Omega$, for $\Omega$ is a compact subset of $\reals^d$. Furthermore $u_\star \in \cH(\beta, L, \Omega)$, where $\cH(\beta, L, \Omega)$ is a collection of $\beta$-Holder functions from $\Omega$ to $\reals$, and the Holder norm is bounded by $L$. As a consequence, $K_1u: \Omega \mapsto [-L, L]^p$ for any $u \in \cH(\beta, L, \Omega)$. Furthermore, we assume $\theta \in \Theta$, where $\Theta$ is a compact subset of $\reals^q$. As consequence $\Psi_\theta(x) = \Psi(x, \theta)$ is a function from $\dom(\Psi) := [-L, L]^p \otimes \Theta$ to $\reals$. We assume that $\Phi$ is thrice continuously differentiable on $\dom(\Psi)$. As $\dom(\Psi)$ is compact, the derivatives of $\Psi$ are uniformly bounded on $\dom(\Phi)$. For notational simplicity, we use $M$ to denote a common uniform upper bound on $\Psi$, the norm of its derivatives, and the operator norm of its Hessian and 3rd order derivative.  
% \end{assumption}

\begin{assumption}[Weight function]
\label{assm:w}
The weight function $\omega$ belongs to $C^{\beta}_c(\bar \Omega)$, i.e.,
\[
\sup_{|\alpha|\le \beta} \|\partial^\alpha \omega\|_\infty < \infty
\quad \text{and} \quad
\partial^\alpha \omega(x) = 0
\quad \forall\, |\alpha|\le \beta,\ \forall\, x \in \partial\Omega \,,
\]
and $\omega(x) > 0$ for all $x \in \Omega$. Since $\Omega = (0,1)^d$, one such example is $\omega(x) = \Pi_{j = 1}^d [x_j(1-x_j)]^{(\beta + 1)}$. 
% \DM{Do we need $w \in C^\infty_c(\Omega_F)$ where $\Omega_F$ is a compact subset of $\Omega$ for Bayesian analysis?} \Yves{I don't think so.}
\end{assumption}
Assumption~\ref{assm:w} ensures that the weight function is sufficiently smooth and vanishes at the boundary with high enough order to eliminate boundary terms arising from integration by parts up to order $2\tau$. 
The positivity of $w$ in the interior guarantees that the PDE residual is meaningfully penalized throughout $\Omega$, preserving identifiability. 

\subsection{A representation theorem for the Gateaux derivative of $\nabla_\theta\cR$}
The construction of our debiased score rests upon the existence of an integral representation for the Gateaux derivative of $\nabla_\theta\cR(\theta, u)$ with respect to $u$. In the following Proposition, we show that under Assumption \ref{assm:psi_smoothness} and Assumption \ref{assm:w}, such a representation indeed exists. Recall that, for $\theta\in\rset^q$, $u,h\in W^{\beta,2}(\Omega)$, the Gateaux derivative of $\nabla_\theta\cR$ at $(\theta,u)$ in the direction $h$ is defined as: 
\[
\rmd (\nabla_\theta \cR)(\theta,u)[h] \eqdef \lim_{t \to 0}\; \frac{(\nabla_\theta \cR)(\theta,u + t h ) - (\nabla_\theta \cR)(\theta,u)}{t}.
\]
The following result establishes the existence of the integral representation of the Gateaux derivative, whose proof can be found in Section S.3 of the supplementary document: 
\begin{proposition}
\label{prop:def_G}
Assume Assumptions (\ref{assm:psi_smoothness}-\ref{assm:w}), and consider the PDE risk functional $\cR$ in (\ref{eq:pde_risk}). If $K_0$ and $K_1$ are differential operators of order $\tau$ and $\beta>2\tau + d/2$, then for all $L>0$, all $\theta\in \Theta$ and all $u\in W^{\beta,2}(\Omega,L)$,
\[ \rmd (\nabla_\theta \cR)(\theta,u)[h] = \int_{\Omega} h(x) \mathsf{G}_{\theta,u}(x)\rmd x, \quad \forall \ \; h \in W^{\beta, 2}(\Omega) \,\]
where $\sG_{\theta, u}\in W^{\beta-2\tau,2}(\Omega)$ takes the following form: 
\begin{align*}
\mathsf{G}_{\theta,u}(x) =& K_0^\dagger\left(\omega \ \nabla_\theta  (\Psi_\theta \circ K_1 u)\right)(x)  + K_1^\dagger\left(\omega \ \left\{\nabla_\theta  (\Psi_\theta \circ K_1 u) \nabla_z (\Psi_\theta \circ K_1 u)^\top\right\}\right)(x) \\
& \qquad + K_1^{\dagger}\left(\omega \ \left\{ K_0 u +  (\Psi_\theta \circ K_1 u)\right\} \nabla^{(2)}_{z\theta}  (\Psi_\theta \circ K_1 u)\right)(x),\;\;\;\; x\in\Omega,
\end{align*}
where $K_0^\dagger$ and $K_1^\dagger$ are the adjoint of $K_0$ and $K_1$ respectively. Furthermore, the representer $\sG_{\theta,u}$ satisfies the following smoothness properties.
\begin{enumerate}
\item For all $L>0,r>0$, there exists a constant $c_0= c_0(L,r)$ such that for all $u\in W^{\beta,2}(\Omega,L)$, and all $\theta\in\rset^q$ such that $\|\theta\|_2\leq r$, it holds $\|G_{\theta,u}\|_\infty \leq c_0$.
\item For all $L>0,r>0$, there exists a constant $c_1= c_1(L,r)$ such that for all $u,u'\in W^{\beta,2}(\Omega,L)$, all $\theta,\theta'\in\rset^q$ such that $\max(\|\theta\|_2,\|\theta'\|_2)\leq r$, and for all $h\in W^{\tau,2}(\Omega)$, we have
\begin{equation}\label{stab:G:2}
\max_{1 \le j \le q}\;\left |\int_\Omega h(x)\left(\mathsf{G}_{\theta,u,j}-\mathsf{G}_{\theta',u',j}\right)(x)\rmd x\right |  \leq c_1 \|h\|_{W^{\tau,2}}  \left( \|u - u'\|_{W^{\tau,2}} + \|\theta-\theta'\|_2\right)\,
\end{equation}
and
\begin{equation}
\label{stab:G:3}
\max_{1 \le j \le q}\; \left\|\mathsf{G}_{\theta,u, j}-\mathsf{G}_{\theta',u', j}\right\|_{L_\infty(\Omega)} \le c_1 \left(\|u - u'\|_\infty^{\alpha_*} + \|\theta-\theta'\|_2\right), \quad \alpha_* = \frac{\beta - 2\tau - d/2}{\beta - d/2} \,,
\end{equation}
where $\mathsf{G}_{\theta,u, j}$ denotes the $j$-th component of $\mathsf{G}_{\theta,u}$. \end{enumerate}
\end{proposition}
\medskip
The contribution of this proposition is twofold. First, it yields the exact form of $\mathsf{G}_{\theta,u}$, the integral representer of the Gateaux derivative of $\nabla_\theta \cR(\theta,u)$ with respect to $u$, which is necessary for constructing the debiased score function $\sS_n^{\rm db}(\theta, u)$ (Equation \eqref{eq:debiased_score}). Second, it establishes the requisite smoothness properties of $\mathsf{G}_{\theta,u}$, which play a crucial role in the proof of our main results. 
Having established this proposition, we are now in a position to present our main results from both the frequentist and Bayesian perspectives.
% \begin{remark}
%     In addition to the existence of $\mathsf{G}_{\theta,u}$, Proposition \ref{prop:def_G} also establishes some important smoothness properties of the map $u\mapsto \mathsf{G}_{\theta,u}$ that will play an important role in the analysis.
% \end{remark}

\subsection{Frequentist approach}
\label{sec:theory_freq}
Recall that in our frequentist framework, we tackle the problem in two stages. 
In the first stage, we estimate the function $u$ by solving Equation \eqref{eq:u_est_NN} using 
neural networks, and based on the resulting estimator $\hat u$, we construct 
an initial plug-in estimator $\hat \theta^{\rm PI}$ (see Equation \eqref{eq:theta_PI}). 
In the second stage, we introduce a debiased score (or loss) function 
$\sS_n^{\rm db}(\theta, \hat u)$ (Equation \eqref{eq:debiased_score}), and define our final estimator $\htdb$ 
as the minimizer of its norm over a shrinking parameter set $\Theta_n$, 
where $\Theta_n$ is chosen as a neighborhood of $\hat \theta^{\rm PI}$. 
% The overall procedure can be summarized schematically: 
% % \[
% % \hat u  \;\longrightarrow\;  \hat \theta^{\rm PI} \;\longrightarrow\;  \sS_n^{\rm db}(\theta, \hat u)  \;\longrightarrow\;  \htdb \,.
% % \]
% \[ \begin{tikzcd}[ row sep=.0em, column sep=5em ] & \sS_n^{\rm db} \arrow[dr] & \\[0.5em] \hat u \arrow[ur] \arrow[dr] && \htdb \\[0.1em] & \hat\theta^{\rm PI} \arrow[ur] & \end{tikzcd} \]
%Having stated our assumptions, we are now in a position to state our main results. 
As per the chronology, in our first result, we show that we can estimate $u_\star$ and its derivatives at a minimax optimal rate (up to a log factor) using RePU neural networks. The proof is given in Section S.4 of the supplementary document. 
\begin{theorem}
    \label{thm:deep_sobolev}
    Consider the estimation of $u$ via solving Equation \eqref{eq:u_est_NN} with $\cF_{\rm NN}$ replaced by $\tilde \cF_{\rm NN}$ (Assumption \ref{assm:function_class}) and $\lambda_n \asymp (n/\log{n})^{-2\beta/(2\beta + d)}\log{n}$. Then, under Assumption \ref{assm:dgp}, the estimator $\hat u$ satisfies with probability $\ge 1 - n^{-C}$ for some large $C > 0$: 
    \begin{align*}
        \|\hat u\|_{W^{\beta, 2}} & \le 8\|u_\star\|_{W^{\beta, 2}} + K_1 \le 8L + K_1, \quad \text{for some } K_1 > 0, \\
        \left\|\partial^{\balpha} \hat u- \partial^{\balpha} u_\star\right\|_{L_2(P_X)} & \le K_2 \left(\log{n}\right)^{\frac12\left(1 - \frac{|\balpha|}{\beta}\right)}\left(\frac{n}{\log{n}}\right)^{-\frac{\beta - |\balpha|}{2\beta + d}}, \quad \forall \ 0 \le |\balpha| < \beta \,.
    \end{align*}
\end{theorem}
\begin{remark}
The above theorem establishes the rate of convergence of a neural network–based estimator of $u_\star$ and its derivatives obtained by minimizing a penalized squared error loss (see \eqref{eq:u_est_NN}). 
We highlight several aspects of this result:

\begin{itemize}
    \item To the best of our knowledge, this is the first work that analyzes a neural network estimator constructed via a general Sobolev-penalized squared loss and proves that both the mean function and its derivatives can be estimated simultaneously at the minimax optimal rate \citep{stone1980optimal}. 
    The presence of a Sobolev penalty distinguishes our approach from standard DNN-based regression analyses (e.g., \cite{kohler2016nonparametric,bauer2019deep,kohler2021rate,schmidt2020nonparametric,bhattacharya2024deep,fan2024factor} to name a few) and requires different technical arguments. The most closely related work is \cite{ding2025semi}, which incorporates a first-order derivative penalty. 
    However, that work does not establish minimax optimal rates for derivative estimation. 
    Moreover, \cite{ding2025semi} focuses on the ReQU network framework as they only penalize the gradient, whereas we consider a general RePU($\nu$) network with $\nu \ge \beta$ to ensure that the estimator lies in $W^{\beta,2}(\Omega)$. Furthermore, in \cite{ding2025semi}, the authors assumed the mean function belongs to a H\"{o}lder class, whereas we assume they belong to a Sobolev class, which substantially broadens the scope of the result. 

    \item Our analysis relies on a new approximation theorem for Sobolev spaces using RePU($\nu$) networks (Theorem S.8.1 of the supplementary document). 
    Although approximation results for RePU networks exist in the literature (e.g., \cite{shen2023differentiable,li2019better}), those typically require the network depth to grow with $n$, which introduces additional logarithmic factors in the convergence rate. 
    In contrast, our approximation result allows the depth to remain fixed while only the width and the total number of active parameters increase with $n$ at the optimal rate.

    \item Finally, much of the existing neural network literature assumes the regression function belongs to a H\"{o}lder class. 
    Working under a Sobolev smoothness assumption presents additional analytical challenges. 
    In particular, available approximation results for Sobolev functions are typically established in the $L_2$ norm, rather than the $L_\infty$ norm commonly used in the H\"older setting, which further complicates the analysis.
\end{itemize}
\end{remark}

% \paragraph{Discussion on the assumptions.}
% Assumption \ref{assm:dgp} is a standard statistical setup used for non-parametric statistical learning. The assumption that $X$ is compactly supported and $\eps$ is sub-Gaussian can be relaxed with more technical bookkeeping without adding anything to the main understanding of the real problem. Therefore, we resort to these standard assumptions to make our presentation simple. Furthermore, we need the smoothness of $u$, namely $\beta$ to be larger than $\tau$ for the PDE constraint to be well-defined (recall that $\tau$ is the maximum order of the derivatives in our PDE constraint Equation \eqref{eq:pde_model_2}). Note that, in most of the papers in non-parametric estimation using neural networks, the underlying mean function is assumed to be H\"{o}lder, whereas here we are working on the Sobolev function class, which brings more challenge in the analysis. As a starter, we only have optimal approximation error of the Sobolev function using neural networks in $L_2$ norms instead of the $L_\infty$ norm frequently used in other papers. 
% For debiasing and achieving asymptotic normality of the debiased estimator $\htdb$, we require the condition $\beta - 2\tau > d$. If we instead assume $u_\star \in C^{\beta, \gamma}(\Omega, L)$, i.e. H\"{o}lder function, then we would need $\beta - 2\tau > d/2$. However, it is not immediately clear to us whether this condition is optimal and we leave it for an important future avenue to explore. 
The previous theorem establishes the convergence rate of $\hat u$, defined in Equation~\eqref{eq:u_est_NN}. 
%when $\cF_{\rm NN}$ is replaced by its truncated counterpart $\tilde \cF_{\rm NN}$. 
The following corollary provides an analogous result for $\hat u_{\rm PINN}$, as defined in Equation~\eqref{eq:u_est_NN_PINN}. The proof is provided in Section S.4 of the supplementary document.
\begin{corollary}
    \label{cor:deep_sobolev_PINN}
    % Let $\hat u_{\rm PINN}$ denote the estimator of $u_*$ obtained from Equation~\eqref{eq:u_est_NN_PINN}, with $\cF_{\rm NN}$ replaced by its truncated version $\tilde \cF_{\rm NN}$. Then, under Assumptions~\ref{assm:dgp} and~\ref{assm:function_class}, $\hat u_{\rm PINN}$ satisfies the same bounds as those obtained for $\hat u$ in Theorem~\ref{thm:deep_sobolev}.
    Let \(\hat u_{\rm PINN}\) denote the estimator of \(u_\star\) obtained from Equation~\eqref{eq:u_est_NN_PINN}, with \(\widetilde{\mathcal F}_{\rm NN}\). 
    Assume the conditions of Theorem~\ref{thm:deep_sobolev} along with  Assumptions~\ref{assm:psi_smoothness} and \ref{assm:w}. Suppose \\ \(\lambda_n \asymp (n/\log{n})^{-2\beta/(2\beta + d)}\log{n}\), and \(\tilde \lambda_n\) satisfies $\tilde\lambda_n N^{-2(\beta-\tau)/d} \lesssim \lambda_n,$ where \(N\asymp (n/\log n)^{d/(2\beta+d)}\) denotes the network width. 
    Then \(\hat u_{\rm PINN}\) satisfies the same bounds as those obtained for \(\hat u\) in Theorem~\ref{thm:deep_sobolev}. 
\end{corollary}
%Upon establishing the convergence result for $\hat u$, 
We now turn to our next theorem, which concerns the initial plug-in estimator $\hat \theta^{\rm PI}$ defined in Equation \eqref{eq:theta_PI}. 
We show that $\hat \theta^{\rm PI}$ is consistent and attains at least a nonparametric rate of convergence. 
%This preliminary rate result plays a crucial role in constructing the shrinking neighborhood $\Theta_n$ used in the subsequent debiasing step. 
The proof is given in Section S.4 of the supplementary document.

\begin{lemma}[Consistency of $\hat \theta^{\rm PI}$]
    \label{lem:consistency_PI}
    Let $\hat \theta^{\rm PI}$ be the initial estimator as defined in Equation \eqref{eq:theta_PI}. Then under Assumption \ref{assm:dgp}-\ref{assm:w} and $\Theta$ is a compact convex subset of $\reals^q$, we have:
    $$
    \|\hat \theta^{\rm PI} - \theta_\star\|_2 = o_p(r_n^{-1/2}), \quad \text{ for any } \ 
    r_n = o\left(n^{\frac{\beta - \tau}{2\beta + d}}(\log{n})^{-a(\beta, d, \tau)}\right) \,.
    $$
\end{lemma}
The above lemma guarantees that the plug-in estimator $\hat \theta^{\rm PI}$ consistently estimates the target parameter $\theta_\star$. 
Its convergence rate is at least of nonparametric order and, more precisely, is driven by the estimation rate of the $\tau$-th derivative of $u_\star$, since the PDE constraint linking $u$ and $\theta$ (Equation \eqref{eq:pde_model_2}) involves differential operators of order at most $\tau$. 
Notably, this rate essentially coincides with the posterior contraction rates established in \cite[Theorem 3.1]{sun2024estimation}. 
However, as elaborated in Section~\ref{sec:method}, we do not require rate optimality for $\hat \theta^{\rm PI}$ in the present analysis. 
Our objective is merely to construct a shrinking neighborhood around $\theta_\star$ to facilitate the subsequent debiasing step. In this regard, the (potentially slow) nonparametric rate is sufficient for our purposes. 

We now turn to the debiasing step. It is evident from the closed-form expression of the representer $\sG_{\theta, u}$ (Proposition \ref{prop:def_G}) that it depends on derivatives of $x \mapsto \Psi_\theta(x, K_1(u)(x))$ up to order $\tau$, and consequently involves derivatives of $u$ up to order $2\tau$. 
% Recall from Section~\ref{sec:method} that constructing the debiased estimator requires identifying the representer $\sG_{\theta,u}$ of the G\^ateaux derivative of $\nabla_\theta \cR(\theta,u)$. 
% This representer plays a central role, as it is used explicitly in forming the debiased score function $\sS_n^{\rm db}(\theta,u)$. 
% From the closed-form expression of $\sG_{\theta,u}$ (Proposition \ref{prop:def_G}), it is evident that $\sG_{\theta,u}$ depends on derivatives of $x \mapsto \Psi_\theta(x, K_1(u)(x))$ and, consequently, involves derivatives of $u$ up to order $2\tau$.  
This explains the requirement $\beta - 2\tau > d$ in our analysis, which ensures that these higher-order derivatives can be estimated consistently,  thereby allowing consistent estimation of the representer. As discussed in Remark~\ref{rem:smoothness_tau}, if instead we assume  $u_\star \in C^{\beta}(\Omega)$, then the weaker condition $\beta - 2\tau > d/2$ would suffice. 
This corresponds to a standard first-order condition in double/debiased machine learning (DML) \citep{chernozhukov2018double,chernozhukov2022locally,foster2023orthogonal}, ensuring that the $2\tau$-th derivative of $u$ belongs to a Donsker class.
% \DM{Do we need a bit more discussion?}
The following result is the first key result of the paper, which states that under the above-mentioned assumption, the debiased estimator $\htdb$ is indeed $\sqrt{n}$-CAN. The proof is given in Section \ref{appn}.

\begin{theorem}
    \label{thm:main_freq}
    Let $\htdb$ be the debiased estimator, obtained via Algorithm \ref{alg:two_step_freq}. Then under Assumption \ref{assm:dgp}-\ref{assm:w} and $\Theta$ is a compact convex subset of $\reals^q$, the estimator $\htdb$ satisfies: 
    $$
    \sqrt{n}(\htdb - \theta_\star) \overset{\mathscr{L}}{\implies} \cN(0, \sigma^2_\eps M^{-1}\Sigma M^{-1}), \qquad  \mbox{ where } $$
    $$ M = \nabla^2_\theta \cR(\theta_\star, u_\star), \;\;\mbox{ and }\;\; \Sigma = \bbE[\sG_{\theta_\star, u_\star}(X)\sG_{\theta_\star, u_\star}(X)^\top]  \,.
    $$

\end{theorem}
The limiting variance of $\htdb$ admits a sandwich form, involving both 
the Hessian of $\cR(\theta,u)$ with respect to $\theta$ and the variance 
of $\sG_{\theta_\star,u_\star}$.  
Such a variance structure is common in nuisance-adjusted estimation 
problems and in pseudo-likelihood-based procedures (e.g., see Theorems 3.1 and 3.3. of \cite{chernozhukov2018double} for a general result). 
However, we cannot immediately claim that this variance is semiparametrically efficient, even under Gaussian noise. 
To see this, note that one may alternatively consider a direct estimator 
based on the forward map operator $\Gamma:\Theta \to W^{\beta,2}(\Omega)$, which maps $\theta \to u(\theta)$, the solution of the PDE Equation \eqref{eq:pde_model_2}  \citep{nickl:etal:2020}. 
% by mapping
% $$
% \theta \longmapsto \text{the solution } u \text{ of } 
% \quad K_0(u) + \Psi_{\theta} \circ K_1(u) = 0, 
% \quad \mathcal{B}(u)=0,
% $$
% that is, the PDE solution corresponding to $\theta$. 
%This operator $\Gamma$ is referred to as the \emph{forward map} in the inverse problems literature (e.g., see \cite{nickl:etal:2020}). 
Under Gaussian noise, one may then define
$$
\hat\theta 
= \argmin_{\theta \in \Theta} 
\frac{1}{n}\sum_{i=1}^n \bigl(Y_i - \Gamma(\theta)(X_i)\bigr)^2,
$$
which coincides with the maximum likelihood estimator and, under suitable 
smoothness conditions on $\Gamma$, would be expected to achieve 
semiparametric efficiency. Nevertheless, this approach has two major limitations:  (i) the estimator is computationally intractable in practice, as it requires repeatedly solving a nonlinear PDE for each $\theta$, and (ii) the regularity conditions required on the forward map $\Gamma$ to 
establish asymptotic efficiency may be substantially stronger than those 
imposed in our analysis of $\htdb$. Consequently, it is not immediately clear whether the limiting variance 
of $\htdb$ is directly comparable to that of the likelihood-based estimator. We leave establishing such a comparison as an interesting direction for future work.
\\\\
\noindent
{\bf Consistent estimation of the limiting variance.}
To draw inference on $\theta_\star$, we need to construct a consistent estimator of the limiting variance of $\htdb$. From the smoothness of $\nabla^2 \cR(\theta, u)$ (see Assumption \ref{assm:psi_smoothness}), one may use $\hat M = \nabla^2_\theta \cR(\htdb, \hat u)$ as a consistent estimator of $M$. On the other hand, from the boundedness of $\sG_{\theta, u}$ on the compact domain (Proposition \ref{prop:def_G}), one may use the sample analog of $\Sigma$ as its estimator, i.e. 
$$
\hat \Sigma = \frac1n \sum_{i = 1}^n \sG_{\htdb, \hat u}(X_i)\sG_{\htdb, \hat u}(X_i)^\top. $$ % - \left(\frac1n \sum_{i = 1}^n \sG_{\htdb, \hat u}(X_i)\right)\left(\frac1n \sum_{i = 1}^n \sG_{\htdb, \hat u}(X_i)\right)^\top$$
%\Yves{I removed the centering as it is not needed. Unless it is shown to help in finite sample.}
Under the smoothness properties of $\sG_{\theta, u}$, as established in Proposition \ref{prop:def_G}, one can easily show that $\sigma_\epsilon^2\hat \Sigma$ is a consistent estimator of $\textsf{var}(\epsilon G_{\theta_\star,u_\star}(X))$. Finally, one may use the mean of squared residuals to obtain a consistent estimator of $\sigma^2_\eps$: 
$$
\hat \sigma^2_\eps = \frac1n\sum_i (Y_i - \hat u(X_i))^2 \,.
$$
Combining them, $\hat \sigma_\eps^2 \hat M^{-1} \hat \Sigma \hat M^{-1}$ is a consistent estimator of the limiting variance of $\htdb$. 

\subsection{Bayesian theoretical results}
\label{sec:theory_bayesian}
As with our frequentist estimator, we show here that by plug-in of the first stage estimator $\hat u$, the debiased posterior distribution $\Pi_{{\rm db}}^{(\theta\vert u)}(\cdot\vert\hat u,\D_n)$ contracts around $\theta_\star$ at the $\sqrt{n}$-rate under similar regularity assumptions. We assume that the prior $\pi_0$ is strongly log-concave. Specifically:
\begin{assumption}\label{assump:prior:pi0}
\begin{enumerate} 
\item $\log\pi_0$ is of class $C^1$ on $\rset^q$, and there exist constants $0<\underline{v}\leq\bar v$ such that for all $\theta,\theta'\in\rset^q$,
\begin{align*}
-\frac{\bar v\log(n)}{2}\|\theta-\theta'\|_2^2 & \leq \log\pi_0(\theta') -\log\pi_0(\theta) -\pscal{\nabla\log\pi_0(\theta)}{\theta'-\theta} \\
& \qquad \qquad \qquad \qquad \qquad \qquad \leq -\frac{\underline{v}\log(n)}{2}\|\theta' - \theta\|_2^2.
\end{align*}
Furthermore, we assume that $\underline{v}\log(n)\geq 1$, and $\|\nabla\log\pi_0(\theta_\star)\|_2\leq \underline{v}\log(n)\sqrt{q}$.
\item The matrix $V_n\in\rset^{q\times q}$ is symmetric and positive definite, and $\lambda_{\max}(V_n)$ is bounded in $n$ from above, and $\lambda_{\min}(V_n)$ is bounded in $n$ from below.
\end{enumerate}
\end{assumption}
\begin{remark} 
The main example that we consider is $\pi_0$ taken as the density of $\textbf{N}_q(0,(1/(v_0\log n))I_q)$ for some $v_0>0$, which satisfies Assumption \ref{assump:prior:pi0} with $\underline{v}=\bar v= v_0$. The shrinking variance $1/(v_0\log(n))$ plays a similar role as the shrinking neighborhood $\Theta_n$ used in the frequentist approach and focuses the posterior mass on compact neighborhoods of $\theta_\star$. In fact, we show in Lemma S.6.3 of the supplementary document that under Assumption \ref{assump:prior:pi0}, the debiased posterior distribution concentrates on balls of radius $c\sqrt{q}$ around $\theta_\star$ for some constant $c$. This result then means that to establish the desired concentration of the debiased posterior, we only need to assume that $\cR(\cdot,u_\star)$ is strongly convex on such balls. This is the purpose of the next set of conditions.
\end{remark}
For $b>0$, we denote $\mathbf{B}_b\eqdef\{\theta\in\rset^q:\;\|\theta-\theta_\star\|_2\leq b\}$. We set
\begin{equation}\label{def:M} 
M_{\theta}\eqdef (\nabla^{(2)}_\theta\cR)(\theta,u_\star), \;\;\;\mbox{ and }\;\;\; M\eqdef M_{\theta_\star}.\end{equation}
For some constant $b>0$, we define
\[\lambda_-(b) \eqdef \inf\left\{ z^\top\left(M_{\theta}^\top V_n^{-1} M_{\theta}\right) z,\;\;\theta\in\mathbf{B}_b,\;\; z\in\rset^q,\;\|z\|_2 =1\right\},\]
and 
\[\lambda_+(b) \eqdef \sup\left\{ z^\top\left(M_{\theta}^\top V_n^{-1} M_{\theta}\right) z,\;\;\theta\in\mathbf{B}_b,\;\; z\in\rset^q,\;\|z\|_2 =1\right\}.\]
We further define
\begin{equation}\label{def:taun}\tau_n  \eqdef  \frac{c}{\sqrt{\lambda_-(b)}}\sqrt{\frac{\lambda_+(b)}{\lambda_-(b)}} \sqrt{\frac{\lambda_{\max}(V_n^{-1})q\log(n)}{n}},
\end{equation}
for some positive constants $b,c$. 
The following theorem is our second key theorem, which establishes that the debiased posterior distribution contracts at the rate $\tau_n$, and the Bernstein-von Mises theorem holds. The proof can be found in Section S.6 of the supplementary document. 

\begin{theorem}\label{thm:cond:db:rate}
Suppose that Assumption \ref{assm:dgp}-\ref{assm:function_class}, Assumption \ref{assm:psi_smoothness}-\ref{assump:prior:pi0} hold. We can find constants $b,c>0$ in (\ref{def:taun}) such that for all $n$ large enough, if $\lambda_-(b)>0$, it holds
\[\PE\left[\Pi_{\rm db}^{(\theta\vert u)}\left(\{\theta\in\rset^q:\;\|\theta-\theta_\star\|_2 > \tau_n \}\vert \hat u, \D_n\right)\right] \leq  n^{-C}, \]
for some constant $C>0$. Furthermore, let $\Psi_n$ be the probability distribution of $\sqrt{n}(\vartheta-\theta_\star)$, where $\vartheta\sim \Pi_{\rm db}^{(\theta\vert u)}(\cdot\vert\hat u, \D_n)$, then
\[\lim_{n\to \infty}\;\PE\left[\| \Psi_n - \mathbf{Q}_n\|_\tv\right]  = 0\, ,\]
where $\mathbf{Q}_n$ denotes the Gaussian measure $\textbf{N}_q(\Delta_n, (MV_n^{-1} M)^{-1})$ on $\rset^q$, with 
\[\Delta_n \eqdef -\frac{1}{\sqrt{n}}\sum_{i=1}^n \epsilon_i (MV_n^{-1} M)^{-1} M V_n^{-1} \mathsf{G}_{\theta_\star,u_\star}(X_i).\]
\end{theorem}

An important practical implication of the theorem is that the prior $\pi_0$, as well as the first stage estimator $\hat u$ have asymptotically vanishing impact on the posterior, which concentrates around $\theta_\star$ at the parametric rate (up to log terms). The theorem also implies that the $100(1-\alpha)\%$ posterior interval for the $j$-th component of $\theta$, say, is asymptotically equivalent to the confidence interval
\[\underbrace{\theta_{\star,j} + \frac{1}{\sqrt{n}}\Delta_{n,j}}_{\check\theta_j} \pm \frac{\sigma_j z_{1-\alpha/2}}{\sqrt{n}},\]
where $\sigma_j$ is the square root of the $j$-th diagonal of $(MV_n^{-1}M)^{-1}$, and $\Delta_{n,j}$ is the $j$-th component of $\Delta_n$, and $z_a$ is the $a$ quantile of the standard Gaussian distribution. However, these confidence intervals may not have good frequentist coverage. To guarantee that the Bayesian and the frequentist inference match requires taking 
\[V_n = \sigma_\epsilon^2 \bbE[\sG_{\theta_\star, u_\star}(X)\sG_{\theta_\star, u_\star}(X)^\top]. \]
With this particular choice, we readily see that the estimator $\check\theta_j$ introduced above is asymptotically equivalent to the debiased frequentist estimator (see (Equation \eqref{eq:exp_main_2})) and therefore, for $n$ large the Bayesian and frequentist inference fully match.

\section{Numerical Experiments}
\label{sec:simulation}
In this section, we present a simulation study to investigate the finite-sample performance of the estimators introduced in Section~\ref{sec:method}, with particular emphasis on illustrating the necessity of the proposed debiasing step for valid inference on the PDE parameter. We consider a one--dimensional elliptic diffusion equation: 
\[ -\nabla\cdot(a_\theta(x)\nabla u(x)) +\theta_3 u(x) = f(x), \;\;\; x\in (0,1),\] 
with boundary conditions $u(0) = u(1) = 0$, where the function $a_\theta$ is parametric of the form
\[a_\theta(x) = \exp\left(\theta_1 \cos(\pi x) + \theta_2\cos(2\pi x)\right).\]
The source function $f$ is assumed known, and taken in this example as $f(x) = 5.0 +2.5 \sin(3.0\pi x)$. This equation is a canonical model for a wide range of physical phenomena in which an unknown medium property must be recovered from indirect measurements. For instance, in thermal conductivity estimation, $u$ is the temperature field of a material with location-dependent thermal conductivity $a_\theta$. %The problem of recovering $a$ (that is, $\theta$) from noisy measurements of $u$ is also important in material science and geothermal prospecting. 
This PDE falls in our general setting with $\theta= (\theta_1,\theta_2,\theta_3)^\top$, 
\[ K_0 = 0, \;\; K_1 = (\mathbb{I},\partial_x, \partial_{xx})^\top,\;\; \mbox{ and }\;\; \Psi_\theta(x,(z_1,z_2,z_3)) = \theta_3 z_1 - a_\theta'(x) z_2 - a_\theta(x) z_3 - f(x), \]
which is infinitely differentiable with bounded derivatives of any compact set. Hence, Assumption \ref{assm:psi_smoothness} holds. 
%In contrast to the transport and advection examples, the map $\theta\mapsto a_\theta$ is nonlinear, so both the plug-in problem and the debiased score below are nonlinear in $\theta$. 

%and we derive the expression the representer of the Gateaux derivative of $u\mapsto \nabla_\theta\cR(\theta,u)$ from Proposition \ref{prop:def_G}: for $j=1,2$,
%We check that for $j=1,2$
%\[\frac{\partial}{\partial\theta_j} \cR(\theta,u) =-\int_0^1(a_{\theta,j}(x) u'(x))'r_\theta(x)\omega(x)\rmd x,\;\;\mbox{ and }\;\; \frac{\partial}{\partial\theta_3}\cR(\theta,u)  = \int_0^1u(x) r_\theta(x)\omega(x)\rmd x,\]
%where $a_{\theta,j}(x) \eqdef \partial a_\theta(x)/\partial \theta_j$, $j=1,2$, and $r_\theta(x) = \Psi_\theta(x,K_1 u(x)) = -(a_\theta(x)u'(x))' +\theta_3 u(x) - f(x)$, and we reading from Proposition \ref{prop:def_G}, the expression of the representer of the Gateaux derivative of $u\mapsto %\nabla_\theta\cR(\theta,u)$ are given by
%\[\mathsf{G}_{\theta,u,j}(x) = \frac{\rmd}{\rmd x}\left(\omega(x) \frac{\partial}{\partial {\theta_j}}\left(a_\theta'(x) r_\theta(x)\right)\right)  - \frac{\rmd^2}{\rmd x^2}\left(\omega(x) \frac{\partial}{\partial{\theta_j}}\left(a_\theta(x) r_\theta(x)\right)\right),\]
%where $r_\theta(x) =  -(a_\theta(x)u'(x))' - f(x)$.
%\DM{I think $\textcolor{red}{\theta_3 \ \omega(x) u(x)}$ is missing from the expression of $G_{\theta, u, 3}(x)$. This term is coming from 
$$
%K_{11}^\dagger\left(\omega \ \partial_{\theta_3} (\Psi_\theta \circ K_1(u)) \partial_{z_1}(\Psi_\theta \circ K_1(u)) \right)(x) = \theta_3 K_{11}^\dagger(\omega \ u )(x) = \theta_3 \omega(x) u(x) \,.
$$
%The last equality follows from the fact that $K_{11}$ is the identity operator. 
%}

\subsection{Data-generating process} The forward solution $u_\theta$ has no closed form and is approximated by the finite
element method (piecewise-linear elements on a uniform mesh of $4000$ elements, see e.g. \cite{brenner2008mathematical}). We set
the true parameter to $\theta_\star=(0.5,0.3,2.0)^\top$ and generate observations as follows:
\begin{enumerate}
  \item generate $X_i\stackrel{\text{i.i.d.}}{\sim}\mathrm{Unif}(0,1)$, $1\le i\le n$;
  \item set $Y_i=u_{\theta_\star}(X_i)+\varepsilon_i$, with
        $\varepsilon_i\stackrel{\text{i.i.d.}}{\sim}N(0,\sigma^2)$, $\sigma=0.02$.
\end{enumerate}
We report results for two sample sizes, $n\in\{500,\,5000\}$.

\subsection{Implementation of the first-stage estimators} 
Following Section~2.1 we define the PDE risk with the boundary-vanishing weight $\omega(x)=[x(1-x)]^2$,
\begin{equation}
  \mathcal{R}(\theta,u)
  \;=\;\tfrac12\int_0^1 \omega(x)\,r_\theta(x)^2\,\mathrm{d}x,
  \qquad r_\theta=\theta_3 u -\bigl(a_\theta u'\bigr)'-f .
  \label{eq:diff-risk}
\end{equation}
Throughout the example $u_\star$ is approximated by a fully connected neural network with scalar input, three hidden layers of width $64$, and scalar output, using the modified rectified power unit activation of order $k=5$, $\mathrm{mRePU}_k$, which is $C^{k-1}=C^4$ and hence guarantees continuity of $u$ through its fourth-order derivative --- required for the representer $G_{\theta,u}$. Derivatives are evaluated by automatic differentiation in single precision. The Sobolev penalty and the expectations defining $\mathcal{R}(\theta,u)$ are approximated by Monte Carlo collocation using $4000$ points drawn uniformly on $(0,1)$. To estimate $\widehat u$ we minimise the empirical squared loss augmented with a $W^{4,2}$ Sobolev penalty (on derivatives up to order four) and a Dirichlet boundary penalty enforcing $u(0)=u(1)=0$; the Sobolev weight is selected by cross-validation on the PDE residual separately for each sample size, giving $\lambda_{\mathrm{sob}}=10^{-11}\times (n/500)^{-10/11}$. The PINN estimator adds the profiled PDE penalty with weight $\lambda_{\mathrm{pde}}=0.1$. Networks are trained with Adam (learning rate $2\times10^{-3}$ for $\widehat u$, $10^{-3}$ for the PINN objective) with global-norm gradient clipping, for $10{,}000$ epochs, with early stopping (patience $1000$) for the PINN estimator. The plug-in estimator $\widehat\theta^{\mathrm{PI}}$ is obtained by solving this nonlinear least-squares problem with $u$ replaced by the first-stage estimator $\widehat u$, using a trust-region, initialised at $\theta = (0, 0, 0)^\top$.

\subsection{Construction of the debiased estimator} 
From Proposition~3.4, we worked out the expression of the representer
$G_{\theta,u}=(G_{\theta,u,1},G_{\theta,u,2},G_{\theta,u,3})^\top$ of the G\^ateaux derivative of
$u\mapsto\nabla_\theta\mathcal{R}(\theta,u)$ that we can write as follows. Set $b_j(x) \eqdef \cos(j\pi x)$, and $B_j[u](x)\eqdef  (b_j(x)a_\theta(x) u'(x))'$, $j=1,2$. Then for $j=1,2$,
\begin{equation}
  G_{\theta,u,j}(x) = \left(a_\theta(x) \left(\omega(x)B_j[u](x)\right)'\right)' - \theta_3\omega(x) B_j[u](x) - \left(b_j(x)a_\theta(x)\left(\omega(x)r_\theta(x)\right)'\right)',
  \label{eq:diff-G-12}
\end{equation}
and
\begin{equation}
  G_{\theta,u,3}(x) = - \left(a_\theta(x)\left(\omega(x)u(x)\right)'\right)' +\theta_3 \omega(x) u(x) + \omega(x) r_\theta(x).
  \label{eq:diff-G-3}
\end{equation}
Therefore, the empirical debiased score in (2.8) takes the form
\begin{equation}
  S_n^{\mathrm{db}}(\theta,\widehat u)
  \;=\;\nabla_\theta\mathcal{R}(\theta,\widehat u)
      \;+\;\frac1{n}\sum_{i=1}^n
        \bigl\{Y_i-\widehat u(X_i)\bigr\}\,G_{\theta,\widehat u}(X_i).
  \label{eq:diff-score}
\end{equation}
We then obtain $\widehat\theta^{\mathrm{db}}$ by minimising
$\|S_n^{\mathrm{db}}(\theta,\widehat u)\|_2^2$ over a neighbourhood
$\Theta_n=\{\theta:\|\theta-\widehat\theta^{\mathrm{PI}}\|_2\le\delta_n\}$ of the plug-in,
using a multi-start quasi-Newton search polished by a derivative-free step to guard against
shallow stationary points; here $\delta_n=3.0$.% TODO: reconcile delta_n with 4.1/4.2 (0.75/0.12).

For the Bayesian estimator we target the debiased posterior
with the scale matrix $V_n$ taken as 
$V_n=\widehat\sigma^2\,\tfrac1n\sum_i G_{\theta,\widehat u}(X_i)G_{\theta,\widehat u}(X_i)^\top$, and with the prior $\pi_0$ is taken as the Gaussian distribution $N(0,(1/\log(n))I_3)$. Sampling uses the adaptive random-walk Metropolis scheme of \cite{haarioetal01} with \cite{atchade:rosenthal05} scale adaptation, run for $60{,}000$ iterations with $30{,}000$ discarded as burn-in and thinning $10$.

\subsection{Monte Carlo design and results}
We repeat the experiment over $B=100$ independent Monte Carlo runs for each sample size. In every run we record the four estimators
$\widehat\theta^{\mathrm{PINN}}$, $\widehat\theta^{\mathrm{PI}}$,
$\widehat\theta^{\mathrm{db}}$, and the debiased posterior mean. 
The results are displayed on Figure~\ref{fig:Diff:1}-\ref{fig:Diff:2}. The biases and the coverages probabilities are also given in Table \ref{tab:ellip3-mc}-\ref{tab:ellip3-cov}. %display, for each coordinate, the boxplots of the four estimators, the running $95\%$
%credible/confidence-interval coverage, the debiased posterior density from a single run
%(with the point estimates and the truth overlaid), and the Q--Q plot of the standardised
%debiased estimator against $N(0,1)$. 
It is clear that both $\widehat\theta^{\mathrm{PI}}$
and $\widehat\theta^{\mathrm{PINN}}$ exhibit substantial bias in all coordinates, that is markedly reduced by the debiased estimator $\widehat\theta^{\mathrm{db}}$. The asymptotic normality of $\widehat\theta^{\mathrm{db}}$ is supported by the Q--Q plots, and its $95\%$ intervals attain approximately nominal coverage. The middle columns in Figure~\ref{fig:Diff:1}-\ref{fig:Diff:2} show the output from one MCMC run from the posterior distribution. The result is approximately Gaussian and close to the Bernstein-von Mises limit (in dashed line).

\begin{figure}[h!]
 \centering
        \includegraphics[width=\textwidth]{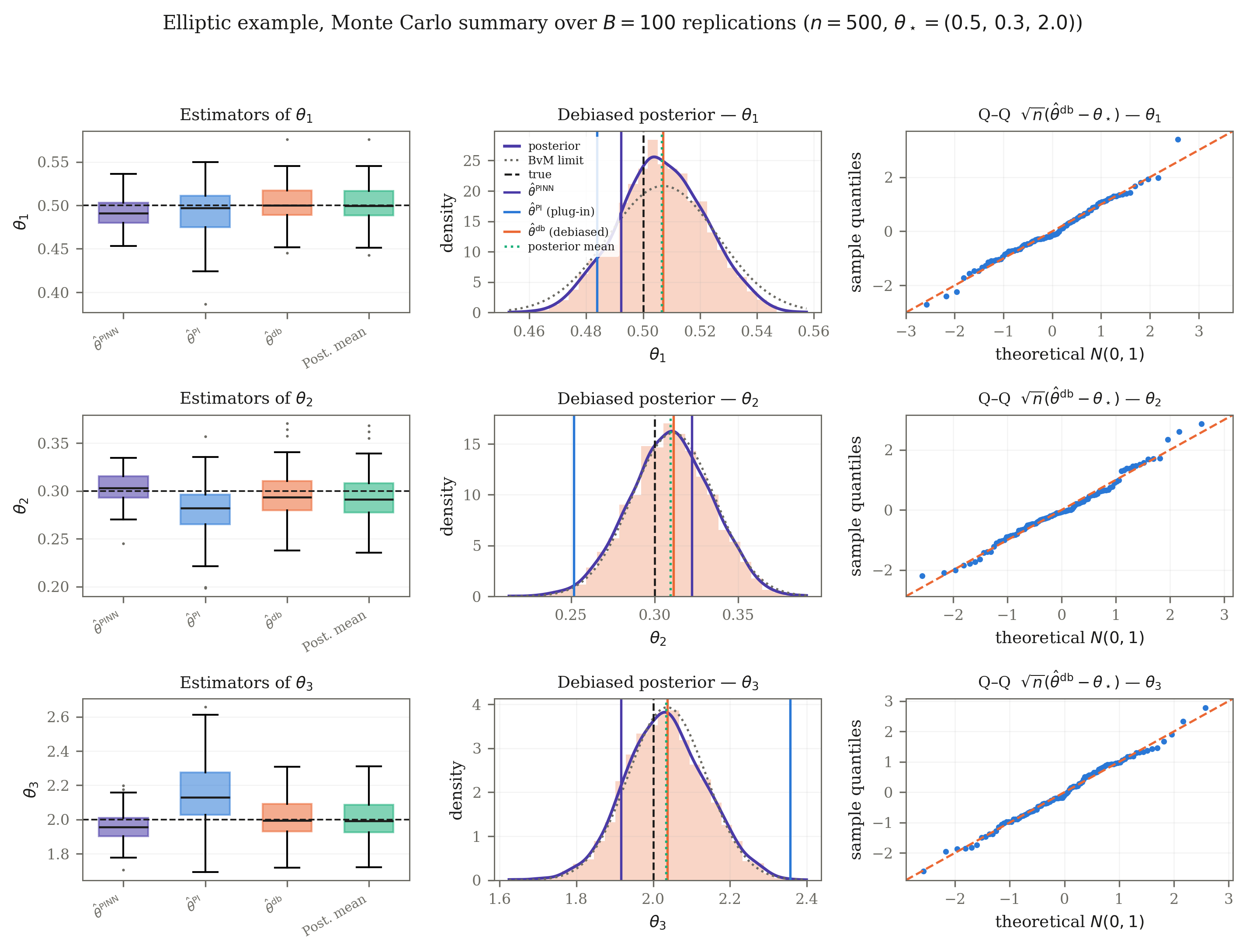}
        \caption{Diffusion example with $n=500$. Monte Carlo summary over $B=100$ replications
    ($\theta_\star=(0.5,0.3,2.0)$). Rows: $\theta_1$, $\theta_2$, $\theta_3$. Columns: boxplots of $\widehat\theta^{\mathrm{PINN}},\widehat\theta^{\mathrm{PI}},
    \widehat\theta^{\mathrm{db}}$ and the posterior mean; debiased posterior density from a single run with the BvM limit, the estimators and the true value overlaid; Q--Q plot of 
    $\sqrt{n}(\widehat\theta^{\mathrm{db}} -\theta_\star)$ against $N(0,1)$.}
        \label{fig:Diff:1}
\end{figure}

    \vspace{0.5cm}

\begin{figure}[h!]
        \centering
        \includegraphics[width=\textwidth]{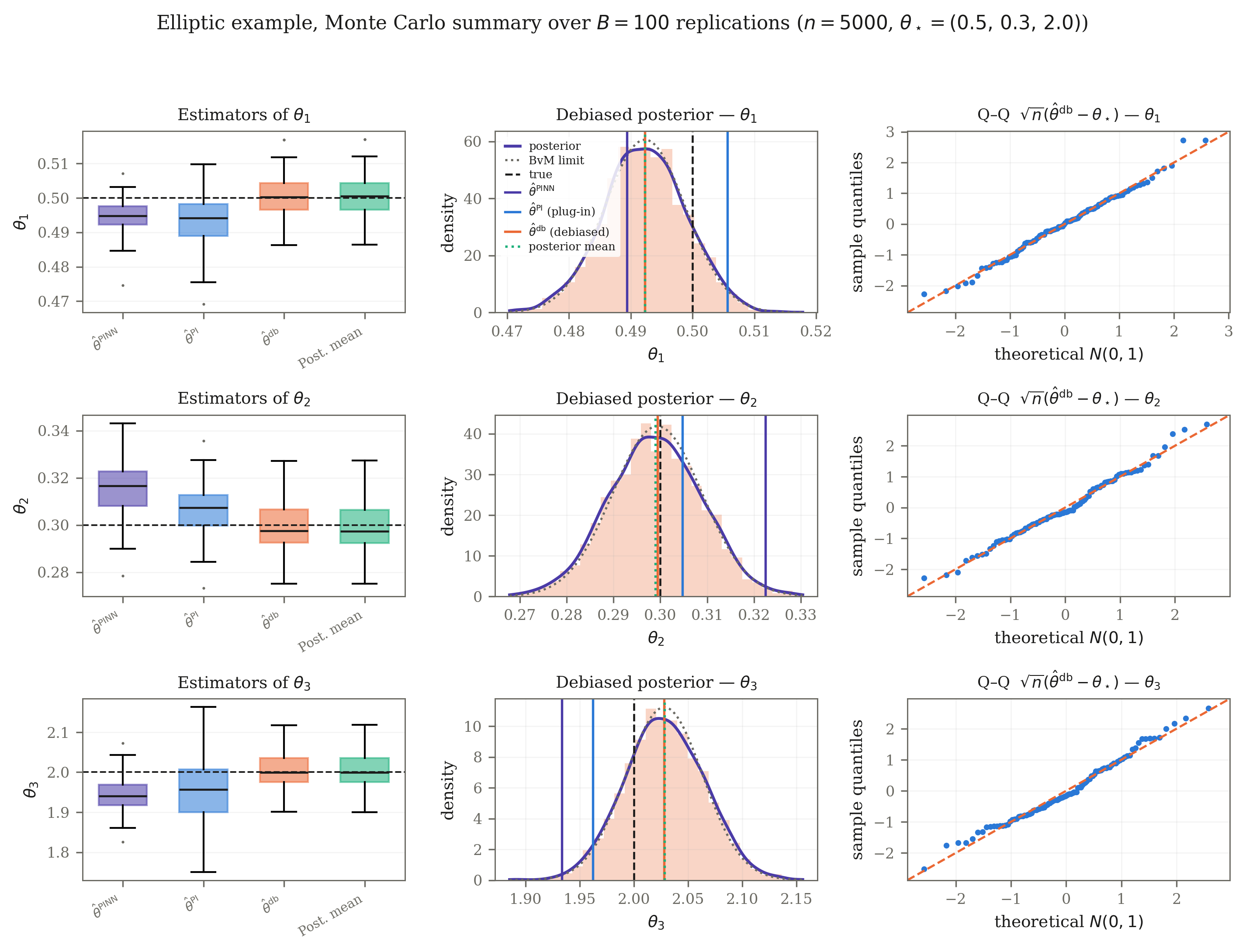}
            \caption{Diffusion example with $n=5000$. Monte Carlo summary over $B=100$ replications
    ($\theta_\star=(0.5,0.3,2.0)$). Rows: $\theta_1$, $\theta_2$, $\theta_3$. Columns: boxplots of $\widehat\theta^{\mathrm{PINN}},\widehat\theta^{\mathrm{PI}},
    \widehat\theta^{\mathrm{db}}$ and the posterior mean; debiased posterior density from a single run with the BvM limit, the estimators and the true value overlaid; Q--Q plot of 
    $\sqrt{n}(\widehat\theta^{\mathrm{db}} -\theta_\star)$ against $N(0,1)$.}
    \label{fig:Diff:2}
\end{figure}

\begin{table}[htbp]
\centering
\caption{Monte Carlo summary (based on $B=100$ replications) for the three-parameter elliptic example. True $\theta_\star=(0.5,0.3,2.0)$. Entries are the bias (mean minus truth) and the standard deviation of each estimator.}
\label{tab:ellip3-mc}
\begin{tabular}{lcccccc}
\toprule
& \multicolumn{2}{c}{$\theta_1 = 0.5$} & \multicolumn{2}{c}{$\theta_2 = 0.3$} & \multicolumn{2}{c}{$\theta_3 = 2.0$}\\
\cmidrule(lr){2-3}\cmidrule(lr){4-5}\cmidrule(lr){6-7}
Estimator & bias & sd & bias & sd & bias & sd\\
\midrule
\multicolumn{7}{l}{\emph{$n=500$}}\\
PINN & $-0.0086$ & $0.0155$ & $+0.0034$ & $0.0164$ & $-0.0428$ & $0.0916$\\
Plug-in & $-0.0078$ & $0.0273$ & $-0.0202$ & $0.0274$ & $+0.1498$ & $0.1931$\\
Debiased & $+0.0030$ & $0.0214$ & $-0.0045$ & $0.0264$ & $+0.0041$ & $0.1097$\\
Posterior mean & $+0.0024$ & $0.0216$ & $-0.0071$ & $0.0266$ & $+0.0020$ & $0.1094$\\
\midrule
\multicolumn{7}{l}{\emph{$n=5000$}}\\
PINN & $-0.0054$ & $0.0044$ & $+0.0151$ & $0.0113$ & $-0.0558$ & $0.0426$\\
Plug-in & $-0.0069$ & $0.0069$ & $+0.0067$ & $0.0108$ & $-0.0483$ & $0.0857$\\
Debiased & $+0.0002$ & $0.0061$ & $-0.0010$ & $0.0104$ & $+0.0066$ & $0.0417$\\
Posterior mean & $+0.0003$ & $0.0061$ & $-0.0012$ & $0.0105$ & $+0.0060$ & $0.0419$\\
\bottomrule
\end{tabular}
\end{table}

\begin{table}[htbp]
\centering
\caption{Empirical coverage (based on $B=100$ replications) of the nominal $95\%$ frequentist confidence and Bayesian credible intervals.}
\label{tab:ellip3-cov}
\begin{tabular}{lcccccc}
\toprule
& \multicolumn{3}{c}{$n=500$} & \multicolumn{3}{c}{$n=5000$}\\
\cmidrule(lr){2-4} \cmidrule(lr){5-7}
& $\theta_1$ & $\theta_2$ & $\theta_3$ & $\theta_1$ & $\theta_2$ & $\theta_3$\\
\midrule
Frequentist (debiased) & $94\%$ & $95\%$ & $94\%$ & $97\%$ & $92\%$ & $88\%$\\
Bayesian (credible) & $90\%$ & $94\%$ & $93\%$ & $97\%$ & $93\%$ & $88\%$\\
\bottomrule
\end{tabular}
\end{table}

\subsection{Robustness to tuning parameter}
The practical use of PINN typically requires tedious cross-validation in choosing the penalty term $\lambda_{\rm sob}$. One advantage of the debiased procedure is an improved robustness to the choice of $\lambda_{\rm sob}$ compared to the first stage estimator $\hat \theta^{{\rm PI}}$ as we show in this section. We set $n=500$, and compare the biases of the first stage estimator $\hat \theta^{{\rm PI}}$ and the debiased procedure $\hat \theta^{{\rm db}}$ across six values of $\lambda_{\mathrm{sob}}$ spanning five
orders of magnitude. For each $\lambda_{\mathrm{sob}}$ we repeat the simulation $B=100$. The results (showing $\sqrt n\times$ bias) are summarized on Figure \ref{fig:rob:lambda}, and show that the debiased procedures are markedly more robust to the choice of $\lambda_{\rm sob}$ across several orders of magnitude.

\begin{figure}[h!]
        \centering
\includegraphics[width=\textwidth]{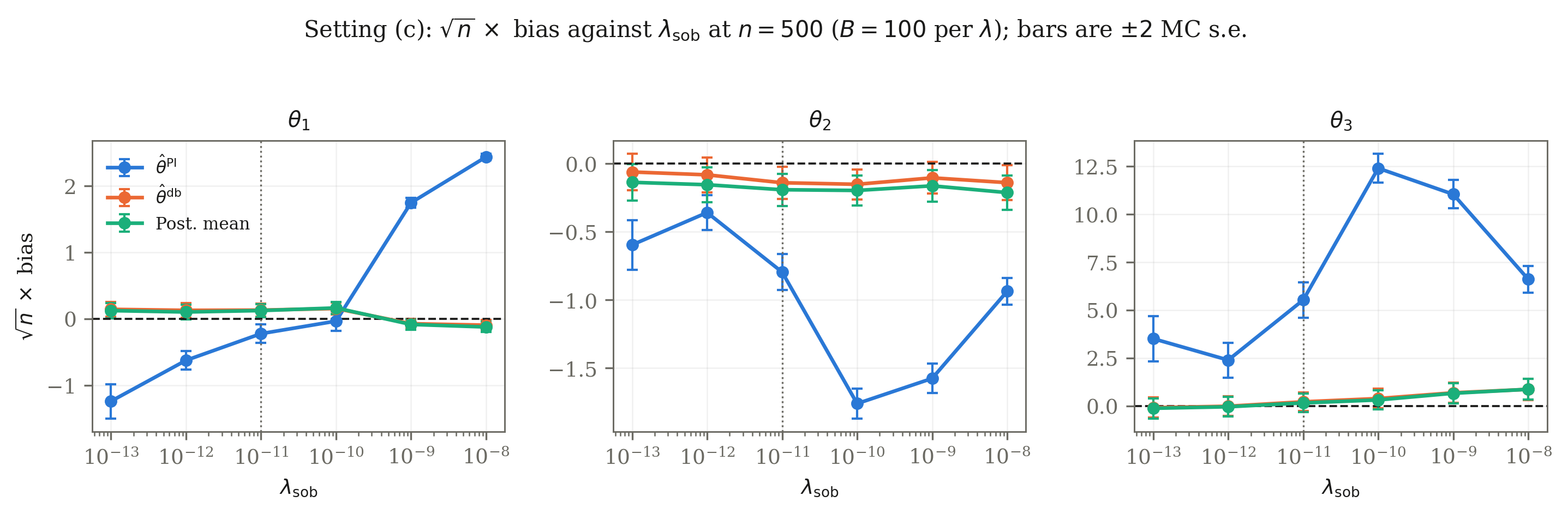}
\caption{Diffusion PDE example with $n=500$. We report the estimation bias ($\sqrt{n}\times {\rm Bias}$) of the plug-in estimator $\hat\theta^{\rm PI}$, the debiased estimator $\hat\theta^{\rm db}$ and the debiased posterior mean  across the three parameters $(\theta_1,\theta_2,\theta_3)$, for different values of the regularization parameter $\lambda_{\rm sob}$. Based on $B=100$ replications.}
\label{fig:rob:lambda}
\end{figure}

\section{Conclusion and future work}
\label{sec:conclusion}
In this paper, we developed a debiasing framework for statistical inference in PDE-constrained inverse problems when the unknown PDE solution is estimated using flexible neural-network methods such as PINNs. Standard plug-in or PINN-based estimators of the finite-dimensional PDE parameter generally inherit the slower nonparametric convergence rate of the estimated solution, leading to biased and inefficient inference. To overcome this difficulty, we constructed an influence-function-based debiased score whose first-order sensitivity to errors in the nuisance function is eliminated. The resulting estimator is $\sqrt n$-consistent and asymptotically normal without requiring undersmoothing of the neural-network estimator. We also proposed a Bayesian analogue based on a debiased quasi-likelihood and established a Bernstein-von Mises theorem showing that the Bayesian inference based on debiased posterior distribution agrees asymptotically with the corresponding frequentist inference. As a by-product, we obtained near-minimax-optimal convergence rates for estimating a Sobolev-smooth regression function and its derivatives simultaneously using neural networks, which may be of independent interest.

However, several interesting directions remain open for future investigation. First, although our debiased estimator attains the parametric rate, it is not yet clear whether its limiting variance is semiparametrically efficient. A natural benchmark is the likelihood-based estimator defined through the exact forward map, but such an estimator is often computationally infeasible because it requires repeatedly solving the PDE. If the corresponding limiting variances are unequal, it would be interesting to understand whether this gap reflects a genuine computational-information-theoretic tradeoff, or whether there exists a computationally feasible procedure that also attains the semiparametric efficiency bound. Second, our theory focuses on finite-dimensional parameters, whereas many scientific inverse problems involve nonparametric or infinite-dimensional unknowns (e.g., spatially varying thermal conductivity). Extending the proposed debiasing and posterior-correction framework to such infinite-dimensional targets is therefore an important direction. Third, another natural extension is to high-dimensional parameter regimes, where the number of candidate PDE terms or physical parameters may grow with the sample size. Developing debiased inference procedures that also perform variable selection would be especially relevant for PDE discovery, where the goal is not only to estimate parameters but also to identify the active terms in the governing equation. 
%\DM{Any direction from a Bayesian perspective?} \Yves{Not specifically. We can tackle these extensions either in a frequentist or Bayesian perspective.}

\section{Proof of Theorem \ref{thm:main_freq}}
\label{appn} %% if no title is needed, leave empty \section*{}.
\noindent
Recall that the debiased estimator $\hat \theta^{\rm db}$ is defined as: 
$$
\htdb= \argmin_{\theta \in \Theta_n} \|\sS_n^{\rm db}(\theta, \hat u)\|_2^2
$$
where $\hat u$ is the neural-network based estimator obtained via minimizing the squared error loss with Sobolev penalty (Equation \eqref{eq:u_est_NN}) and $\Theta_n$ is a shrinking neighborhood around $\hat \theta^{\rm PI}$. 
Define the events $(\Omega_{n, 1}, \Omega_{n, 2})$ as: 
\begin{align*}
\Omega_{n, 1} & = \text{Minimizer of }\|\sS_n^{\rm db}(\theta, \hat u)\|_2^2 \text{ is attained at the interior of }\Theta_n, \\
\Omega_{n, 2} & = \text{The minimum singular value of } \nabla_\theta \sS_n^{\rm db}(\htdb, \hat u) \text{ is strictly positive}. 
\end{align*}
It is established in Lemma S.5.1 and S.5.2 of the supplementary document respectively
% \ref{lem:min_interior} and Lemma \ref{lem:zero_root_db} respectively 
that $\bbP(\Omega_{n, j}) \to 1$ as $n \uparrow \infty$ for $j \in \{1, 2\}$ respectively. Define $\Omega_n = \Omega_{n, 1} \cap \Omega_{n, 2}$. Then $\bbP(\Omega_n) \to 1$ as $n \uparrow \infty$. Therefore, it is enough to establish asymptotic normality of $\sqrt{n}(\htdb - \theta_*)$ on $\Omega_n$. First of all, on $\Omega_n$, we have: 
% Then on $\Omega_n = \Omega_{n, 1} \cap \Omega_{n, 2}$, we have: 
$$
0 = \sS^{\rm db}_n(\htdb, \hat u) = \sfS^{\rm db}_n(\theta_*, u_*) + \left(\sfS^{\rm db}_n(\hat \theta^{\rm db} , \hat u) - \sfS^{\rm db}_n(\theta_*, u_*)\right)
$$
% \begin{align*}
% \hat \theta^{\rm db} & = \argzero_{\theta \in \Theta} \sfS_n(\theta, \hat u) = \argzero_{\theta \in \Theta} \left\{(\nabla_\theta \cR)(\theta,\hat u) + \frac{1}{n}\sum_{i=1}^n(Y_i - \hat u(X_i)) \mathsf{G}_{\theta, \hat u}(X_i)\right\}  \,.
% \end{align*}
% This immediately implies: 
% \begin{align*}
%     0 = \sfS_n(\hat \theta^{\rm db} , \hat u) = \sfS_n(\theta_*, u_*) + \left(\sfS_n(\hat \theta^{\rm db} , \hat u) - \sfS_n(\theta_*, u_*)\right)
% \end{align*}
Now using the fact $\nabla_\theta \cR(\theta_*, u_*) = 0$, we have: 
$$
\sfS^{\rm db}_n(\theta_*, u_*) = \nabla_\theta \cR(\theta_*, u_*) + \frac1n \sum_i (Y_i - u_*(X_i)) \mathsf{G}_{\theta_*, u_*}(X_i) = \frac1n \sum_i \eps_i \mathsf{G}_{\theta_*, u_*}(X_i) \,.
$$
Therefore, along with the previous equation, we have: 
\begin{equation}
    \label{eq:exp_main_1}
    0 = \frac1n \sum_i \eps_i \mathsf{G}_{\theta_*, u_*}(X_i) + \left(\sfS^{\rm db}_n(\hat \theta^{\rm db} , \hat u) - \sfS^{\rm db}_n(\theta_*, u_*)\right)
\end{equation}
We can expand the second term on the right hand side of the above equation as follows:  
% \begin{align*}
%     & \sfS_n(\hat \theta^{\rm db} , \hat u) - \sfS_n(\theta_*, u_*) \\
%     & = \underbrace{\nabla_\theta \cR(\htdb, \hat u) - \nabla_\theta \cR(\htdb, u_*)}_{:= T_1} + \underbrace{\nabla_\theta \cR(\htdb, u_*) - \nabla_\theta \cR(\theta_*, u_*)}_{:=T_2} + \underbrace{\frac1n \sum_{i = 1}^n \eps_i \left(\mathsf{G}_{\htdb, \hat u}(X_i) - \mathsf{G}_{\theta_*, u_*}(X_i)\right)}_{:=T_3} \\
%     & \qquad + \underbrace{\frac1n \sum_i (u_*(X_i) - \hat u(X_i))\left(\mathsf{G}_{\htdb, \hat u}(X_i) - \mathsf{G}_{\theta_*, u_*}(X_i)\right)}_{:=T_4} \\
%     & \qquad + \underbrace{\frac1n \sum_i (u_*(X_i) - \hat u(X_i))\mathsf{G}_{\theta_*, u_*}(X_i)}_{:=T_5} \\
%     & \triangleq \sum_{j = 1}^5 T_j \,.
% \end{align*}
\begin{align*}
    & \sfS^{\rm db}_n(\hat \theta^{\rm db} , \hat u) - \sfS^{\rm db}_n(\theta_*, u_*) \\
    & = \underbrace{\nabla_\theta \cR(\htdb, \hat u) - \nabla_\theta \cR(\htdb, u_*)}_{:= T_1} + \underbrace{\nabla_\theta \cR(\htdb, u_*) - \nabla_\theta \cR(\theta_*, u_*)}_{:=T_2} \\
    & \qquad + \underbrace{\frac1n \sum_{i = 1}^n \eps_i \left(\mathsf{G}_{\htdb, \hat u}(X_i) - \mathsf{G}_{\theta_*, u_*}(X_i)\right)}_{:=T_3} + \underbrace{\frac1n \sum_i (u_*(X_i) - \hat u(X_i))\mathsf{G}_{\htdb, \hat u}(X_i)}_{:=T_4} \\
    & \triangleq \sum_{j = 1}^4 T_j \,.
\end{align*}
This expansion, along with Equation \eqref{eq:exp_main_1}, implies: 
\begin{equation}
\label{eq:exp_main_1_1}
0 =  \frac1n \sum_i \eps_i \mathsf{G}_{\theta_*, u_*}(X_i) + \sum_{j = 1}^4 T_j \,.
\end{equation}
Let us start with $T_2$. We will perform one step Taylor-step expansion: 
\begin{align}
    \label{eq:T2_exp_1}
    T_2 = \nabla_\theta \cR(\htdb, u_*) - \nabla_\theta \cR(\theta_*, u_*) & = \left(\int_0^1 \nabla^2_\theta \cR(\theta_\star + t(\htdb - \theta_\star), u_*) \ dt\right)(\htdb - \theta_*) \notag \\
    & := H_n(\htdb - \theta_*)  
\end{align}
Since $\widehat\theta^{\rm db}\in\Theta_n$, $\delta_n\to0$, and $\widehat\theta^{\rm PI}\to\theta^\star$ in probability, we have $\widehat\theta^{\rm db}\to\theta^\star$ in probability. Hence, by the continuity of $\nabla_\theta^2\cR(\theta,u^\star)$,
\[
H_n \overset{p}{\longrightarrow} M:=\nabla_\theta^2\cR(\theta^\star,u^\star) \,.
\]
By Assumption \ref{assm:curvature_pde_risk}, $M$ is nonsingular, so $H_n^{-1}=M^{-1}+o_p(1)$ and consequently $H_n$ is invertible for all large $n$ with probability going to $1$.
As a consequence, from Equation \eqref{eq:exp_main_1_1}, we have: 
\begin{equation}
    \label{eq:exp_main_2}
    \sqrt{n}(\htdb - \theta_*) = -H_n^{-1}\left\{\frac{1}{\sqrt{n}} \sum_i \eps_i \mathsf{G}_{\theta_*, u_*}(X_i) + \sqrt{n}\sum_{j \neq 2} T_j\right\}
\end{equation}
In the rest of the proof, we show that $\sqrt{n}(T_1+ T_3 + T_4) = o_p(1)$. This, along with the fact that $H_n^{-1} \to M^{-1}$ in probability, will immediately imply via an application of the central limit theorem and Slutsky's theorem: 
$$
\sqrt{n}(\htdb - \theta_*) \implies \cN(0, \sigma_\eps^2 M^{-1}\Sigma M^{-1}), 
$$
where
$$
M = \nabla^2_\theta \cR(\theta_*, u_*), \qquad \Sigma = \bbE[\mathsf{G}_{\theta_*, u_*}(X)\mathsf{G}_{\theta_*, u_*}(X))^\top] \,.
$$
We first show that $\sqrt{n}(T_1 + T_4) = o_p(1)$. 
Henceforth, we use the notation $\bP(f(X))= \int_\Omega f(x)\rmd x$. 
By Taylor expansion, we have: 
\begin{align*}
& \nabla_\theta \cR(\htdb, \hat u) - \nabla_\theta \cR(\htdb, u_*) \\
& = \int_0^1 \int_\Omega (\hat u - u_*)(x) \ \mathsf{G}_{\htdb, u_* + \alpha(\hat u - u_*)}(x) \ dx \ d\alpha \\
& = \int_0^1 \bP\left[(\hat u(X) - u_*(X))\mathsf{G}_{\htdb, u_* + \alpha(\hat u - u_*)}(X)\right] \ d\alpha \\
& =  \bP\left[(\hat u(X) - u_*(X))\mathsf{G}_{\htdb, \hat u}(X)\right] \\
& \qquad + \int_0^1 \bP\left[(\hat u(X) - u_*(X))\left(\mathsf{G}_{\htdb, u_* + \alpha(\hat u - u_*)}(X) - \mathsf{G}_{\htdb, \hat u}(X)\right)\right] \ d\alpha
\end{align*}
Furthermore, we can center $T_4$ as: 
\begin{align*}
    T_4 & = \bbP_n((u_*(X) - \hat u(X))\mathsf{G}_{\htdb, \hat u}(X)) \\
    & = (\bbP_n - \bP)((u_*(X) - \hat u(X))\mathsf{G}_{\htdb, \hat u}(X)) + \bP((u_*(X) - \hat u(X))\mathsf{G}_{\htdb, \hat u}(X)) \,.
\end{align*}
As a consequence, we have: 
\begin{align*}
    T_1 + T_4 & =  (\bbP_n - \bP)((u_*(X) - \hat u(X))\mathsf{G}_{\htdb, \hat u}(X))  \\
    & \qquad \qquad + \int_0^1 \bP\left[(\hat u(X) - u_*(X))\left(\mathsf{G}_{\htdb, u_* + \alpha(\hat u - u_*)}(X) - \mathsf{G}_{\htdb, \hat u}(X)\right)\right] \ d\alpha \\
    & = T_{141} + T_{142} \,.
\end{align*}
Let us first argue that $\sqrt{n}|T_{141}| = o_p(1)$, where it is enough to argue that $\sqrt{n}|T_{141, j}| = o_p(1)$ for all $1 \le j \le q$. Towards that end, we use the following facts: 
\begin{enumerate}
    \item Define $\cU_n$ as: 
    \begin{equation}
        \label{eq:def_un}
        \cU_n = \left\{u: \|u\|_{W^{\beta, 2}} \le H_1 + \|u_*\|_{W^{\beta, 2}}, \  \|\partial^{\bj}u  - \partial^{\bj} u_*\|_2 \le \delta_{n, |\bj|} \ \forall \ |\bj| \le 2\tau\right\}
    \end{equation}
   where $H_1$ is same as defined in Equation S.4.5. of the supplementary document, and 
    $$
    \delta_{n, |\bj|} = C\left(\frac{n}{\log{n}}\right)^{-\frac{\beta - |\bj|}{2\beta + d}}(\log{n})^{1 - \frac{|\bj|}{\beta}} \,.
    $$
    Then by Theorem \ref{thm:deep_sobolev}, $\bbP(\hat u \notin \cU_n) \to 0$. 

    \item $\|u\|_\infty, \|\mathsf{G}_{\theta, u, j}\|_\infty \le B$ for all $u \in \cU_n$ (the first conclusion follows from Theorem \ref{thm:deep_sobolev} and the second conclusion follows from part (1) of Proposition \ref{prop:def_G} along with Theorem \ref{thm:deep_sobolev}). 

    \item $\|\mathsf{G}_{\theta, u, j} - \mathsf{G}_{\theta', u, j}\|_\infty \le L \|\theta - \theta'\|_2$ (follows from part (2) of Proposition \ref{prop:def_G}). 

    \item $\|\mathsf{G}_{\theta, u, j} - \mathsf{G}_{\theta, u', j}\|_\infty \le L \|u - u'\|_\infty^{\alpha_*}$ for all $u \in \cU_n$ (follows from part (2) of Proposition \ref{prop:def_G}). 
\end{enumerate}
We start with the following decomposition: 
\begin{align*}
    \bbP(\sqrt{n}|T_{141 ,j}| > t) & = \bbP(\sqrt{n}T_{141 ,j} > t,\hat u \in \cU_n) + \bbP(\sqrt{n}|T_{141 ,j}| > t, \hat u \notin \cU_n) \\
    & \le \bbP\left(\sup_{\substack{\theta \in \Theta \\ u \in \cU_n }}\sqrt{n}|(\bbP_n - \bP)(u- u_*) \mathsf{G}_{u, \theta, j}| > t, \hat u \in \cU_n\right) + \bbP(\hat u \notin \cU_n) \\
    & \le \bbP\left(\sup_{\substack{\theta \in \Theta \\ u \in \cU_n }}\sqrt{n}|(\bbP_n - \bP)(u- u_*) \mathsf{G}_{u, \theta, j}| > t\right) + \bbP(\hat u \notin \cU_n)
\end{align*}
As $\lim_{n \uparrow \infty} \bbP(\hat u \notin \cU_n) = 0$, it is enough to show: 
$$
\sup_{\theta \in \Theta, u \in \cU_n }|\bbG_n(u- u_*) \mathsf{G}_{u, \theta, j}| = o_p(1) \,, \quad \bbG_n = \sqrt{n}(\bbP_n - \bP) \,.
$$
This essentially follows from the stochastic equicontinuity of an empirical process over the Donsker class. Let 
$$
\cM_j = \{m_{u, \theta} := (u- u_*) \bG_{u, \theta, j}: u \in \cU, \ \theta \in \Theta\} \,.
$$
where $\cU = \{u \in W^{\beta, 2}: \|u\|_{W^{\beta, 2}} \le \|u_\star\|_{W^{\beta, 2}} + H_1\}$. Clearly, by definition $\cU_n \subseteq \cU$. We next argue that $\cM_j $ is Donsker under our assumption $\beta - 2\tau > d$. Towards that end, let $\{u_1, \dots, u_{N_1}\}$ constitutes and $(\eps/6BL)^{1/\alpha_*}$ cover of $\cU$ with respect to $L_\infty$ norm and $\{\theta_1, \dots, \theta_{N_2}\}$ constitutes an $\eps/6BL$ cover of $\Theta$ with respect to $L_2$ norm. Then given $(u, \theta)$, there exists $(u_i, \theta_j)$ such that $\|u - u_i\|_\infty \le (\eps/6BL)^{1/\alpha*}$ and $\|\theta - \theta_j\|_2 \le \eps/6BL$. As a consequence, 
\begin{align*}
    \|m_{u, \theta} - m_{u_i, \theta_j}\|_\infty & \le \|(u - u_*)\mathsf{G}_{u, \theta, j} - (u_i - u_*)\mathsf{G}_{u_i, \theta_j, j}\|_\infty \\
    & \le \|(u - u_i)\mathsf{G}_{u, \theta, j}\|_\infty + \|(u_i - u_*)(\mathsf{G}_{u, \theta, j} - \mathsf{G}_{u, \theta_j, j})\|_\infty  \\
    & \hspace{10em}+ \|(u_i - u_*)(\mathsf{G}_{u, \theta_j, j} - \mathsf{G}_{u_i, \theta_j, j})\|_\infty \\
    & \le B \|u - u_i\|_\infty + 2BL \|\theta - \theta_j\|_2 + 2BL \|u  - u_i\|_\infty ^{\alpha_*}  \ \ \le \eps \,.
\end{align*}
Here we assume without loss of generality $B, L \ge 1$. Therefore, 
\begin{align*}
\log{\cN(\cM_j, L_\infty, \eps)} & \le \log{\cN(\cU, L_\infty, (\eps/6BL)^{1/\alpha_*})} + \log{\cN(\Theta, L_2, \eps/6BL)} \\
& \le C_1 \eps^{-\frac{d}{\beta \alpha_*}} + C_2 q\log{\frac{1}{\eps}} \,.
\end{align*}
Here we used the covering number of the Sobolev function class (e.g., see Theorem 5.2 of \cite{birman1967piecewise}). 
By Assumption \ref{assm:dgp}, $\beta \alpha_* > d/2$, and consequently $\cM_j$ is Donsker. Now define 
$$
\cM_{n, j} = \{m_{u, \theta} \in \cM_j: \ u \in \cU_n\}. 
$$
From Proposition \ref{prop:def_G}, we have $\|\sG_{u, \theta, j}\|_\infty \le B$ uniformly over $u \in \cU, \theta \in \Theta$. Therefore, for $m_{u, \theta} \in \cM_{n, j}$, we have: 
\begin{align*}
    \|m_{u, \theta}\|_2 = \|(u - u_*)\sG_{u, \theta}\|_2 \le \|\sG_{u, \theta}\|_\infty \|u - u_\star\|_2 \le B \delta_{n, \mathbf{0}} \,.
\end{align*}
Therefore, with respect to the variance semi-metric $\rho(f, g) = \sqrt{\var_\bP(f - g)} \le \|f - g\|_{L_2(\bP)}$, we have, we have from equicontinuity of the Donsker class (e.g. see Lemma 2.3.11 of \cite{van1996weak} or Lemma 19.24 of \cite{van2000asymptotic}), 
$$
\sup_{m_{u, \theta} \in \cM_{n, j}}|\bbG_n(m_{u, \theta})| \le \sup_{m_{u, \theta} \in \cM_j: \rho(m_{u, \theta}, 0) \le B\delta_{n, \mathbf{0}}}|\bbG_n(m_{u, \theta})| = o_p(1) \,.
$$
This concludes that $\sqrt{n}T_{141} = o_p(1)$. 

We next show that $\sqrt{n}|T_{142, j}| = o_p(1)$. First, by definition of $\cU_n$, we have $\|\hat u - u_*\|_2 \le \delta_{n, 0}$. Therefore, an application of Theorem S.2.2 of the supplementary document yields: 
$$
\|\hat u - u_*\|_\infty \le C \|\hat u - u_*\|_{W^{\beta, 2}(\Omega)}^{\frac{d}{2\beta}} \|\hat u - u_*\|_2^{1 - \frac{d}{2\beta}} \le C_1 \delta_{n ,0}^{1 - \frac{d}{2\beta}} \,.
$$
Here the last line follows from the fact that both $\hat u$ and $u_*$ has bounded $W^{\beta, 2}$ norm; that $u_*$ has bounded $W^{\beta, 2}$ follows from our Assumption and that $\hat u$ has bounded $W^{\beta, 2}$ follows from Theorem \ref{thm:deep_sobolev}. Now, by the Cauchy-Schwarz inequality, we have: 
\begin{align*}
    & \bP\left[\left|(\hat u(X) - u_*(X))\left(\mathsf{G}_{\htdb, u_* + \alpha(\hat u - u_*), j}(X) - \mathsf{G}_{\htdb, \hat u, j}(X)\right)\right|\right] \\
    & \le \|\hat u - u_*\|_2 \left\|\mathsf{G}_{\htdb, u_* + \alpha(\hat u - u_*), j}(X) - \mathsf{G}_{\htdb, \hat u, j}(X)\right\|_2 \\
    & \le \|\hat u - u_*\|_2 \|\hat u - u_*\|_\infty^{\alpha_*} \\
    & \le C_1 \|\hat u - u_*\|_2^{1 + \alpha_*\left(1 - \frac{d}{2\beta}\right)} \le L \delta_{n, 0}^{1 + \alpha_*\left(1 - \frac{d}{2\beta}\right)} \,.
\end{align*}
Now, from the definition of $\delta_{n, 0}$ we have $\delta_{n, 0} = n^{-\beta/(2\beta + d)}(\log{n})^a$. Furthermore, we have: 
$$
1 + \alpha_*\left(1 - \frac{d}{2\beta}\right) = 1 + \frac{\beta - 2\tau - d/2}{\beta - d/2} \frac{2\beta - d}{2\beta} = 1 + \frac{2\beta - 4\tau - d}{2\beta} = \frac{4\beta - 4\tau - d}{2\beta} \,.
$$
This implies: 
\begin{align*}
\sqrt{n}\delta_{n, 0}^{1 + \alpha_*\left(1 - \frac{d}{2\beta}\right)} & = n^{\frac{1}{2} - \frac{\beta}{2\beta + d}\frac{4\beta - 4\tau - d}{2\beta}}\left(\log{n}\right)^{a\left( \frac{4\beta - 4\tau - d}{2\beta}\right)} \\
& = n^{\frac12\left(1 -\frac{4\beta - 4\tau - d}{2\beta + d} \right)}\left(\log{n}\right)^{a\left( \frac{4\beta - 4\tau - d}{2\beta}\right)} \\
& = n^{-\frac{\beta - 2\tau - d}{2\beta + d}}\left(\log{n}\right)^{a\left( \frac{4\beta - 4\tau - d}{2\beta}\right)} = o(1) \,.
\end{align*}
Here, the last conclusion follows from the fact that $\beta > 2 \tau + d$ (Assumption \ref{assm:dgp}). Therefore, $\sqrt{n}T_{142, j} = o_p(1)$. 
\\\\
\noindent
Finally, we show that $\sqrt{n}T_3 = o_p(1)$ in two steps. First, we show that $\sqrt{n}T_3 = O_p(1)$. As we have already estalished $\sqrt{n}(T_1 + T_4) = o_p(1)$, then from Equation \eqref{eq:exp_main_2}, this conclusion will immediately imply  $\sqrt{n}(\htdb - \theta_*) = O_p(1)$. Towards that end, we use Lemma S.7.1 of the supplementary document, which essentially follows from Dudley's integral bound. We use this lemma with $\cF = \cM'_n$, where 
$$
\cM'_n = \left\{m_{u, \theta}(x) = \sG_{\theta, u, j}(x) - \sG_{\theta_*, u_*}(x): \quad u \in \cU_n, \theta \in \Theta\right\} \,.
$$
Clearly, by an application of Proposition \ref{prop:def_G}, we have: 
$$
\|\sG_{\theta, u} - \sG_{\theta, u'}\|_\infty \le C\|u - u'\|_\infty^{\alpha_*} \,, \quad \forall \ u, u' \in \cU_n \,, \quad \forall \theta \in \Theta \,,
$$
and 
$$
\|\sG_{\theta, u} - \sG_{\theta', u}\|_\infty \le C' \|\theta - \theta'\|_2, \quad \forall \ u \in \cU_n \,, \quad \forall \theta, \theta' \in \Theta \,,
$$
for some constant $C, C' > 0$ independent of $(\theta, u)$, as soon as $\Theta$ is a compact subset of $\reals^d$. Following a similar argument as in the previous part of the proof, given any $\tau > 0$, suppose $\{u_1, \dots, u_{N_1}\}$ be the $(\tau/2C)^{1/\alpha_*}$ cover of $\cU_n$ (w.r.t. $L_\infty$ norm) and $\{\theta_1, \dots, \theta_{N_2}\}$ be the $(\tau/2C')$ cover of $\Theta$ w.r.t. $L_2$ norm. Then we claim that 
$$
\cM_n(\tau) := \{\sG_{\theta_i, u_j} - \sG_{\theta_*, u_*}: 1 \le i \le N_2, 1 \le j \le N_1\}
$$
is a $\tau$-cover of $\cM_n$. This essentially follows from a simple triangle inequality, because, given any $(u, \theta)$, let $(u_i, \theta_j)$ be their corresponding center of cover, i.e. $\|u - u_i\|_{L_\infty} \le (\tau/2C)^{1/\alpha_*}$ and $\|\theta - \theta_j\|_2 \le \tau/2C'$. Then, 
\begin{align*}
    \left\|(\sG_{\theta, u} - \sG_{\theta_*, u_*}) - (\sG_{\theta_j, u_i} - \sG_{\theta_*, u_*})\right\|_{L_\infty} & = \|\sG_{\theta, u} - \sG_{\theta_j, u_i}\|_{L\infty} \\
    & \le \|\sG_{\theta, u} - \sG_{\theta_j, u}\|_{L\infty} + \|\sG_{\theta_j, u} - \sG_{\theta_j, u_i}\|_{L\infty} \\
    & \le C'\|\theta - \theta_j\|_2 + C\|u - u_i\|_\infty^{\alpha_*} \\
    & \le C' \times \frac{\tau}{2C'} + C \times \frac{\tau}{2C} \le \tau \,.
\end{align*}
Consequently, we have: 
\begin{align*}
\log{\cN(\cM_n, \tau, L_\infty)} & \le \log{\cN(\Theta, \tau/2C', L_2)} + \log{\cN(\cU_n, (\tau/2C)^{1/\alpha_*}, L_\infty)} \\
& \le C_1 q \log{(C_2/\tau)} + C_3 \tau^{-\frac{d}{\beta \alpha_*}} \,.
\end{align*}
Therefore, using the fact $\beta \alpha_* > d/2$ (Assumption \ref{assm:dgp}), we have: 
$$
\int_0^B \sqrt{1 + \log{\cN(\cM_n, \tau, L_\infty)}} \ d\tau = B' < \infty \,.
$$
and consequently by Lemma S.7.1 of the supplementary docment,
\begin{equation}
\label{eq:maximal_eps_global}
\bbE\left[\sup_{u \in \cU_n, \theta \in \Theta} \left|\frac{1}{\sqrt{n}}\sum_i \eps_i (\sG_{\theta, u}(X_i) - \sG_{\theta_*, u_*}(X_i))\right|\right] \le B''
\end{equation}
for some constant $B'' > 0$. We will now use this to show $\sqrt{n}T_3 = O_p(1)$, as follows: 
\begin{align*}
    & \bbP\left(\left|\frac{1}{\sqrt{n}}\sum_{i = 1}^n \eps_i \left(\mathsf{G}_{\htdb, \hat u}(X_i) - \mathsf{G}_{\theta_*, u_*}(X_i)\right)\right| \ge t\right) \\
    & \le \bbP\left(\left|\frac{1}{\sqrt{n}}\sum_{i = 1}^n \eps_i \left(\mathsf{G}_{\htdb, \hat u}(X_i) - \mathsf{G}_{\theta_*, u_*}(X_i)\right)\right| \ge t, u \in \cU_n\right) + \bbP(u \in \cU_n^c) \\
    & \le  \bbP\left(\sup_{u \in \cU_n, \theta \in \Theta} \left|\frac{1}{\sqrt{n}}\sum_{i = 1}^n \eps_i \left(\mathsf{G}_{\theta, u}(X_i) - \mathsf{G}_{\theta_*, u_*}(X_i)\right)\right| \ge t, u \in \cU_n\right) + \bbP(u \in \cU_n^c) \\
    & \le \frac{1}{t}\bbE\left[\sup_{u \in \cU_n, \theta \in \Theta} \left|\frac{1}{\sqrt{n}}\sum_i \eps_i (\sG_{\theta, u}(X_i) - \sG_{\theta_*, u_*}(X_i))\right|\right] + \bbP(u \in \cU_n^c) \\
    & \le \frac{B''}{t}+ \bbP(u \in \cU_n^c)
\end{align*}
Therefore, upon observing the fact that $\bbP(u \in \cU_n^c) \to 0$ as $n \uparrow \infty$, we conclude that $\sqrt{n}T_3 = O_p(1)$, which, in turn, implies $\sqrt{n}(\htdb - \theta_*) = O_p(1)$. 

\medskip
\noindent
Now we use this fact to show that in fact $\sqrt{n}T_3 = o_p(1)$. First, define $\check \Theta_n = \{\theta: \|\theta - \theta_*\|_2 \le \sqrt{\log{n}/n}\}$. From the $\sqrt{n}$-consistency of $\htdb$, we have $\bbP(\htdb \notin \check \Theta_n) \to 0$ as $n \uparrow \infty$. Therefore, following the exact calculation as in the previous display, we obtain: 
\begin{align*}
     & \bbP\left(\left|\frac{1}{\sqrt{n}}\sum_{i = 1}^n \eps_i \left(\mathsf{G}_{\htdb, \hat u}(X_i) - \mathsf{G}_{\theta_*, u_*}(X_i)\right)\right| \ge t\right) \\
     & \le \frac{1}{t}\bbE\left[\sup_{u \in \cU_n, \theta \in \check\Theta_n} \left|\frac{1}{\sqrt{n}}\sum_i \eps_i (\sG_{\theta, u}(X_i) - \sG_{\theta_*, u_*}(X_i))\right|\right] + \bbP(\{\hat u \in \cU_n^c\}\mbox{ or } \{ \htdb \in \check\Theta_n^c\}) \,.
\end{align*}
The second term, by definition, is $o(1)$. Therefore, if we can show that the expected suprema of $(\cU_n, \check \Theta_n)$ is $o(1)$, we are done. Towards that end, we will use Lemma \ref{lem:dudley}. Define a function class $\cG_{n, j}$ as: 
$$
\cG_{n, j} = \{\sG_{\theta, u, j} - \sG_{\theta_\star, u_\star, j}: \ \theta \in \hat \Theta_n, u \in \cU_n\} \,.
$$
It is immediate from Proposition \ref{prop:def_G} that 
\begin{align*}
    \|\sG_{\theta, u} - \sG_{\theta_\star, u_\star}\|_\infty & \le C\left(\|u - u_\star\|_\infty^{\alpha_*} + \|\theta - \theta_*\|_2\right) \\
    & \le C\left(\delta_{n, \mathbf{0}}^{\alpha_*(1 - d/2\beta)} + \sqrt{\frac{\log{n}}{n}}\right) := b_n \,.
\end{align*}
By the same argument as before, we can conclude $\cG_{n, j}$ is a Donsker funciton class with 
$$
\log{\cN(u, \cG_{n, j}, L_\infty)} \le C(u^{-d/(\beta \alpha_*)} + q\log{(1/u)})\,.
$$
Therefore, an application of Lemma S.7.1 of the supplementary document yields: 
\begin{align*}
    & \bbE\left[\sup_{u \in \cU_n, \theta \in \check\Theta_n} \left|\frac{1}{\sqrt{n}}\sum_i \eps_i (\sG_{\theta, u}(X_i) - \sG_{\theta_*, u_*}(X_i))\right|\right]  \\
    & \lesssim \int_0^{C b_n} \sqrt{1 + \log{\cN(u, \cG_{n, j}, L_\infty)}} \ du \\
    & \lesssim b_n^{1 - d/(2\beta \alpha_*)} + b_n \sqrt{\log{(C/b_n)}} \longrightarrow 0 \,.
\end{align*}
as $b_n \to 0$. 
This completes the proof of the fact that $\sqrt{n}T_3 = o_p(1)$, which completes the proof of the theorem. 
% As we have already established $\sqrt{n}(T_1 + T_4) = o_p(1)$, to argue $\sqrt{n}$-consistency of $\htdb$ from Equation \eqref{eq:exp_main_2}, all we need to argue is $\sqrt{n}T_3 = O_p(1)$. This is immediate as: 

\bibliographystyle{plain}
\bibliography{biblio}

\newpage 
\begin{appendix}
\renewcommand{\thesection}{S.\arabic{section}}
\section{Additional numerical illustrations}

\subsection{$(1+1)$ transport PDE}\label{sec:transport_freq}

We begin with the $1+1$ transport equation
\begin{align*}
    u^\star_t(x,t) - \theta^\star u^\star_x(x,t) = 0, \qquad (x,t) \in [0,1]^2.
\end{align*}
Observe that this PDE is a special case of our general formulation in
\eqref{eq:pde_model_2}, with
\begin{align*}
    K_0 = \partial_t, \qquad K_1 = \partial_x, \qquad \Psi_\theta(x,z) = -\theta z.
\end{align*}
\\\\
\noindent
{\bf Data-generating process.}
To generate data from a model that satisfies the PDE exactly, we take the true solution to be
\begin{align*}
    u^\star(x,t) = \sin\!\big(2\pi(x+\theta^\star t)\big)
                 + \cos\!\big(\pi(x+\theta^\star t)\big),
    \qquad \text{with} \quad \theta^\star = 2.
\end{align*}
It is immediate that $u^\star_t(x,t) = 2 u^\star_x(x,t)$. We generate
observations according to the following procedure:
\begin{enumerate}
    \item First, we generate $(X_i, T_i) \overset{\mathrm{i.i.d.}}{\sim}
    \mathrm{Unif}([0,1]^2)$, $1 \le i \le n$, where we consider three sample
    sizes $n \in \{500,\, 3000,\, 5000\}$.
    \item Next, for $1 \le i \le n$, we set $Y_i = u^\star(X_i, T_i) +
    \varepsilon_i$, where $\varepsilon_i \overset{\mathrm{i.i.d.}}{\sim}
    N(0, \sigma^2)$. In our experiments, we take $\sigma = 0.5$.
\end{enumerate}
Following Section \ref{sec:setup}, we define the PDE risk using the
boundary-vanishing weight $\omega(x,t) = x^2(1-x)^2 t^2 (1-t)^2$.
Consequently, for a fixed $u$, the minimizer $\theta(u)$ of
$\mathcal{R}(\theta, u)$ has the following closed form:
\begin{align*}
    \theta(u) =
    \frac{\int_{[0,1]^2} \omega(x,t)\, u_t(x,t) u_x(x,t)\, dx\, dt}
         {\int_{[0,1]^2} \omega(x,t)\, u_x(x,t)^2\, dx\, dt}.
\end{align*}
Accordingly, the plug-in estimator $\hat{\theta}^{\mathrm{PI}}$ (resp.\ the
PINN estimator $\hat{\theta}_{\mathrm{PINN}}$) is obtained by replacing $u$
with the first-stage estimator $\widehat{u}$ (resp.\ the PINN estimator
$\widehat{u}_{\mathrm{PINN}}$). In the implementation, the above integrals are
approximated by a midpoint quadrature rule on a $300 \times 300$ grid over
$[0,1]^2$.
\\\\
\noindent
{\bf Implementation of the first-stage estimators.}
For both the plug-in and PINN procedures, the solution $u^\star$ is
approximated by a fully connected neural network with three hidden layers,
width $64$, scalar output, and RePU activation of order $3$, namely
$\sigma(z) = (z_+)^3$. Since our methodology requires accurate evaluation of
first- and second-order derivatives, all computations are carried out using
automatic differentiation in double precision. The expectations with respect to
the distribution of $(X,T)$ that appear in the Sobolev penalty
\eqref{eq:u_est_NN} and in the definition of $\mathcal{R}(\theta,u)$ (for
example, in \eqref{eq:pde_risk} and others) are approximated numerically by
Monte Carlo collocation, using $1024$ fresh draws from $\mathrm{Unif}([0,1]^2)$
per epoch for the Sobolev penalty and $4000$ per epoch for the profiled PDE
penalty. For estimating $(\widehat{u}_{\mathrm{PINN}},
\hat{\theta}_{\mathrm{PINN}})$, we use $\lambda_n = 10^{-5}$ and
$\widetilde{\lambda}_n = 10$ (see \eqref{eq:u_est_NN_PINN} for their
definitions), as these choices provide a good approximation to $u^\star$ and
its derivatives. For estimating $\widehat{u}$ via \eqref{eq:u_est_NN}, we again
set $\lambda_n = 10^{-5}$. Thus, the tuning parameters are kept fixed
throughout the experiment. The neural networks are trained in \texttt{pytorch}
using the Adam optimizer with learning rate $10^{-3}$, for at most $20{,}000$
epochs. We further employ early stopping with patience $300$, based on a
validation set, to stabilize training and prevent overfitting.
\\\\
\noindent
{\bf Construction of the debiased estimator.}
In the present transport model, the score admits a particularly simple closed form. Indeed,
\begin{align*}
    \nabla_\theta \mathcal{R}(\theta, u)
    = -\int_{[0,1]^2} \omega(x,t)\, u_t(x,t) u_x(x,t)\, dx\, dt
      + \theta \int_{[0,1]^2} \omega(x,t)\, u_x(x,t)^2\, dx\, dt.
\end{align*}
Writing
\begin{align*}
    A(u) = \int_{[0,1]^2} \omega(x,t)\, u_t(x,t) u_x(x,t)\, dx\, dt,
    \qquad
    B(u) = \int_{[0,1]^2} \omega(x,t)\, u_x(x,t)^2\, dx\, dt,
\end{align*}
we have
\begin{align*}
    \nabla_\theta \mathcal{R}(\theta, u) = -A(u) + \theta B(u).
\end{align*}
Moreover, the representer $G_{\theta,u}$ of the G\^{a}teaux derivative of
$\nabla_\theta \mathcal{R}(\theta, u)$ reduces in this case to
\begin{align*}
    G_{\theta,u}(x,t)
    = \partial_t\big(\omega(x,t) u_x(x,t)\big)
    + \partial_x\big(\omega(x,t) u_t(x,t)\big)
    - 2\theta\, \partial_x\big(\omega(x,t) u_x(x,t)\big).
\end{align*}
Hence, defining
\begin{align*}
    a_u(x,t) = \partial_t\big(\omega(x,t) u_x(x,t)\big)
             + \partial_x\big(\omega(x,t) u_t(x,t)\big),
    \qquad
    b_u(x,t) = \partial_x\big(\omega(x,t) u_x(x,t)\big),
\end{align*}
we obtain
\begin{align*}
    G_{\theta,u}(x,t) = a_u(x,t) - 2\theta\, b_u(x,t).
\end{align*}
Therefore, the empirical debiased score in \eqref{eq:debiased_score} takes the
form
\begin{align*}
    S^{\mathrm{db}}_n(\theta, \widehat{u})
    = -\widehat{A} + \widehat{B}\theta
    + \frac{1}{n} \sum_{i=1}^{n}
      \big\{ Y_i - \widehat{u}(X_i, T_i) \big\}
      \big( \widehat{a}_i - 2\theta\, \widehat{b}_i \big),
\end{align*}
where $\widehat{A} = A(\widehat{u})$, $\widehat{B} = B(\widehat{u})$, and $\widehat{a}_i, \widehat{b}_i$ denote the corresponding evaluations of $a_{\widehat{u}}$ and $b_{\widehat{u}}$ at the observations. We emphasize that no sample splitting is employed: the same $n$ observations are used both to fit $\widehat{u}$ and to evaluate the empirical debiased score. Since this score is affine in $\theta$, the second-stage root is explicit. In accordance with \eqref{eq:def_htdb}, and to guard against rare finite-sample instability, we project the resulting root onto the neighborhood
\begin{align*}
    \Theta_n = \big[ \hat{\theta}^{\mathrm{PI}} - \delta_n,\,
                     \hat{\theta}^{\mathrm{PI}} + \delta_n \big],
    \qquad \delta_n = 0.75.
\end{align*}
This yields the final debiased estimator $\htdb$.
\\\\
\noindent
{\bf Monte Carlo design and summary measures.}
We repeat the experiment
over $100$ independent Monte Carlo runs
for each sample size $n \in \{500,\, 3000,\, 5000\}$, and in each run, we record $\hat{\theta}^{\mathrm{PI}}$, $\hat{\theta}_{\mathrm{PINN}}$, and
$\htdb$. The results are displayed in Figure~\ref{fig:transport_11_freq}. The top row presents kernel density estimates of the $\sqrt{n}$-scaled errors $\sqrt{n}(\widehat{\theta} - \theta^\star)$ for the three estimators, overlaid with the limiting distribution $N(0, \sigma^2_\varepsilon M^{-1} \Sigma M^{-1})$ of Theorem~\ref{thm:main_freq}, whose variance is computed at $(\theta^\star, u^\star)$ by numerical quadrature. The bottom panel presents boxplots of the three estimators for each sample size.

It is clear from the figure that both $\hat{\theta}^{\mathrm{PI}}$ and $\hat{\theta}_{\mathrm{PINN}}$ exhibit a substantial bias that does not vanish on the $\sqrt{n}$ scale, reflecting the slower nonparametric rate they inherit from the first stage. In contrast, the proposed debiasing step substantially reduces the bias, and the distribution of $\sqrt{n}(\htdb - \theta^\star)$ approaches the limiting law as $n$ grows, corroborating our theoretical findings.

\begin{figure}[t]
    \centering
    \includegraphics[width=\textwidth]{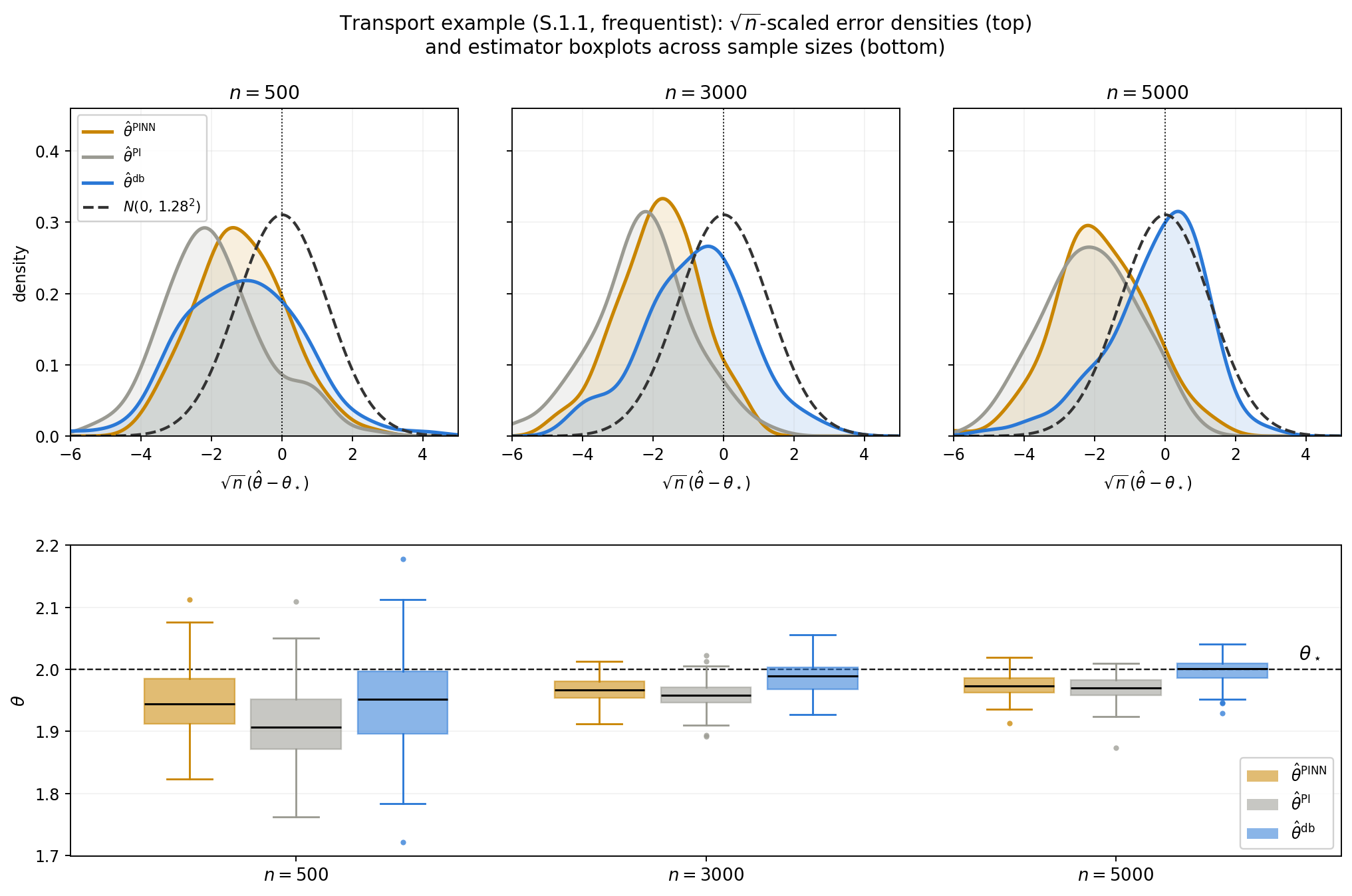}
    \caption{The top row presents kernel density estimates of
    $\sqrt{n}(\widehat{\theta} - \theta^\star)$ for
    $\hat{\theta}_{\mathrm{PINN}}$, $\hat{\theta}^{\mathrm{PI}}$, and
    $\htdb$ over $100$ Monte Carlo runs for each
    $n \in \{500, 3000, 5000\}$, together with the limiting distribution of
    Theorem~\ref{thm:main_freq} (dashed). The bottom panel presents boxplots
    of the three estimators for each sample size; the dashed line marks
    $\theta^\star = 2$.}
    \label{fig:transport_11_freq}
\end{figure}

\subsection{A linear advection-reaction equation}\label{sec:advection_freq}

We next consider the first-order linear PDE
\begin{align*}
    u^\star(x,t) - \theta^\star_1 u^\star_x(x,t) - \theta^\star_2 u^\star_t(x,t) = 0,
    \qquad (x,t)\in[0,1]^2 .
\end{align*}
Observe that this PDE is a special case of our general formulation in
\eqref{eq:pde_model_2}, with
\begin{align*}
    K_0 = I, \qquad K_1 = (\partial_x,\partial_t), \qquad
    \Psi_\theta(z_1,z_2) = -\theta_1 z_1 - \theta_2 z_2 .
\end{align*}
\noindent
{\bf True parameters.}
Throughout the experiments, we set the true parameter to
$\theta^\star=(\theta^\star_1,\theta^\star_2)=(0.9,1.2)$. A solution of the
above PDE, from which we generate the observations, is
\begin{align*}
    u^\star(x,t) = \exp\!\Big(\frac{\theta^\star_1 x + \theta^\star_2 t}
        {(\theta^\star_1)^2+(\theta^\star_2)^2}\Big)\,
        g\big(\theta^\star_2 x - \theta^\star_1 t\big),
    \qquad g(s) = \delta_0 + a\sin(\omega_1 s) + b\cos(\omega_2 s),
\end{align*}
with $\delta_0=0.30$, $a=0.80$, $b=0.60$, $\omega_1=1$ and $\omega_2=1.5$.
\begin{figure}[htbp]
  \centering
  \includegraphics[width=0.9\textwidth]{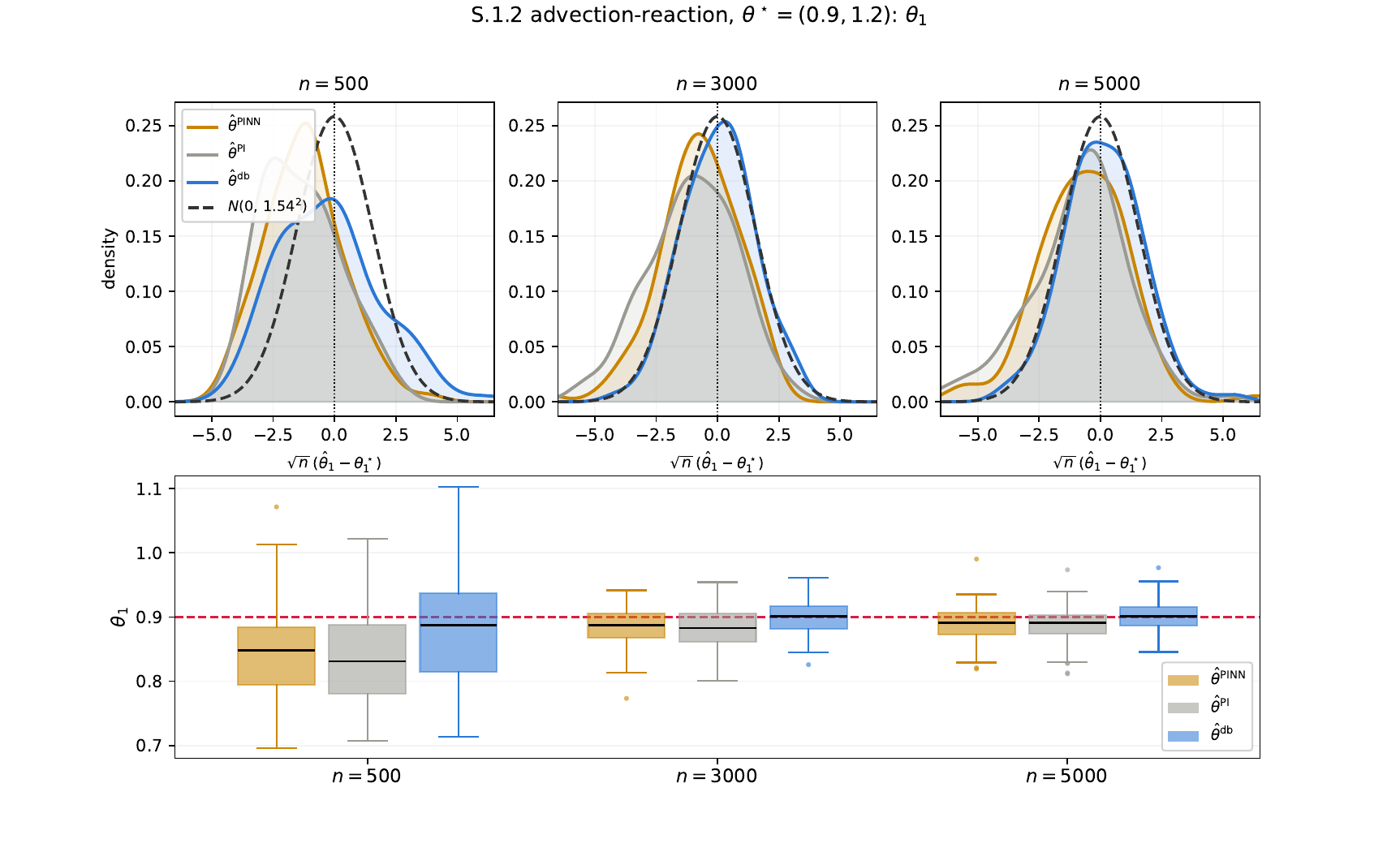}\\[1.5ex]
  \includegraphics[width=\textwidth]{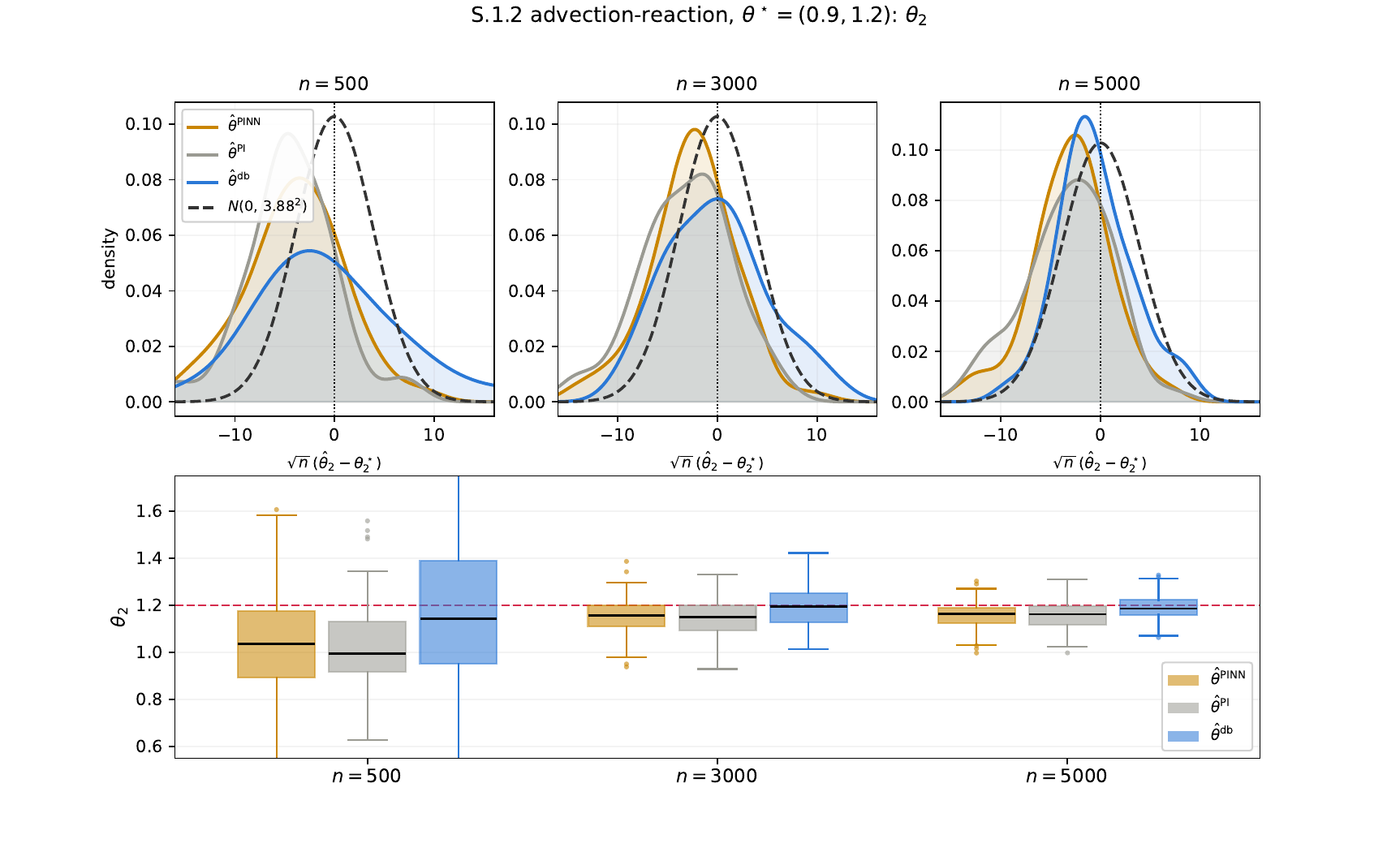}
%  \caption{Advection-reaction example, $\theta^\star=(0.9,1.2)$, frequentist
%  debiasing (Algorithm~\ref{alg:two_step_freq}), $n\in\{500,3000,5000\}$ over $100$
%  replications. Top block: component $\theta_1$; bottom block: component
%  $\theta_2$. Within each block, the upper row shows histograms of the
%  studentized debiased estimator
%  $\sqrt{n}(\htdb_j-\theta_j^\star)$ against the
%  $\mathcal{N}(0,1)$ reference density, and the lower row shows boxplots of the plug-in $\hat\theta^{\rm PI}_j$, the PINN estimate
%  $\hat\theta_{\mathrm{PINN}, j}$, and the debiased estimate
%  $\htdb_j$; dashed lines mark the truth
%  ($\theta_1^\star=0.9$, $\theta_2^\star=1.2$).}
 \caption{Advection–reaction example, \(\theta^\star=(0.9,1.2)\), frequentist debiasing (Algorithm~\ref{alg:two_step_freq}), \(n\in\{500,3000,5000\}\) over 100 replications. Top block: component \(\theta_1\); bottom block: component \(\theta_2\). Within each block, the upper row shows a smoothed histogram of \(\sqrt{n}(\widehat{\theta}_j-\theta_j^\star)\) for the PINN, plug-in, and debiased estimators, overlaid with the reference density \(N(0,1.54^2)\) for \(j=1\) and \(N(0,3.88^2)\) for \(j=2\). The lower row shows boxplots of the plug-in estimator \(\widehat{\theta}^{\mathrm{PI}}_j\), the PINN estimate \(\widehat{\theta}_{\mathrm{PINN},j}\), and the debiased estimate \(\widehat{\theta}^{\mathrm{db}}_j\); dashed lines mark the truth \((\theta_1^\star=0.9,\theta_2^\star=1.2)\).
}
  \label{fig:s12-AR}
\end{figure}
\\\\
\noindent
{\bf Data-generating process.}
We generate observations according to the following procedure:
\begin{enumerate}
    \item First, we generate $(X_i,T_i)\overset{\mathrm{i.i.d.}}{\sim}
    \mathrm{Unif}([0,1]^2)$, $1\le i\le n$, with $n \in \{500, 3000, 5000\}$.
    \item Next, for $1\le i\le n$, we set $Y_i = u^\star(X_i,T_i)+\varepsilon_i$,
    where $\varepsilon_i\overset{\mathrm{i.i.d.}}{\sim} N(0,\sigma^2)$ and
    $\sigma=0.35$.
\end{enumerate}
Following Section \ref{sec:debias_method}, we define the PDE risk using the
boundary-vanishing weight $\omega(x,t)=x^2(1-x)^2t^2(1-t)^2$. Consequently, for
a fixed $u$, the risk takes the form
\begin{align*}
    \mathcal R(\theta,u) = \frac12\int_{[0,1]^2}\omega(x,t)
        \big(u - \theta_1 u_x - \theta_2 u_t\big)^2\,dx\,dt .
\end{align*}
Writing
\begin{align*}
    A(u) = \int_{[0,1]^2}\omega\, u\begin{pmatrix}u_x\\ u_t\end{pmatrix}dx\,dt,
    \qquad
    B(u) = \int_{[0,1]^2}\omega
        \begin{pmatrix}u_x^2 & u_x u_t\\ u_x u_t & u_t^2\end{pmatrix}dx\,dt,
\end{align*}
the minimizer $\theta(u)$ of $\mathcal R(\theta,u)$ is given in closed form by $\theta(u)=B(u)^{-1}A(u)$, provided $B(u)$ is nonsingular. Accordingly, the plug-in estimator $\widehat{\theta}_{\mathrm{PI}}$ (resp.\ the PINN estimator $\widehat{\theta}_{\mathrm{PINN}}$) is obtained by replacing $u$ with the first-stage estimator $\widehat{u}$ (resp.\ $\widehat{u}_{\mathrm{PINN}}$).
\\\\
\noindent
{\bf Implementation of the first-stage estimators.}
We use $n$ observations to fit the first-stage estimator $\widehat{u}$ and to evaluate the empirical debiased score for each $n \in \{500, 3000, 5000\}$. For both the plug-in and PINN procedures, $u^\star$ is approximated by a fully connected neural network with three hidden layers, width $64$, scalar output, and a RePU activation function of order $3$, namely $\sigma(z)=(z_+)^3$. Since our methodology requires accurate evaluation of first- and second-order derivatives, all computations are carried out using automatic differentiation in double precision. The expectations with respect to the distribution of $(X,T)$ that appear in the Sobolev penalty \eqref{eq:u_est_NN} and in the definition of $\mathcal R(\theta,u)$ are approximated by Monte Carlo collocation using $4000$ points drawn uniformly over $[0,1]^2$ at each epoch. To estimate $\widehat{u}$, we minimize the empirical squared loss along with a Sobolev penalty $\|u\|^2_{W^{3,2}}$ with 
$\lambda_{\mathrm{sob}}=10^{-6}$; the PINN estimator adds another PDE penalty with $\lambda_{\mathrm{pde}}=10$. The networks are trained in \texttt{pytorch} using the Adam optimizer with learning rate $2\times10^{-3}$, for at most $3000$ epochs, with early stopping (patience $500$) based on a $20\%$ validation set. To compute $\widehat{\theta}_{\mathrm{PI}}$, we approximate the integrals defining $A(\widehat{u})$ and $B(\widehat{u})$ using a midpoint quadrature rule on a $300\times300$ grid over $[0,1]^2$. 
\\\\
\noindent
{\bf Construction of the debiased estimator.}
In the present model, the score admits the closed-form representation
\begin{align*}
    \nabla_\theta \mathcal R(\theta,u) = -A(u) + B(u)\theta .
\end{align*}
Moreover, the representer $G_{\theta,u}=(G_{\theta,u,1},G_{\theta,u,2})^\top$ of
the G\^{a}teaux derivative of $\nabla_\theta\mathcal R(\theta,u)$ with respect
to $u$ can be written as
\begin{align*}
    G_{\theta,u,1}(x,t) &= \omega_x u
        - 2\theta_1\,\partial_x\big(\omega u_x\big)
        - \theta_2\big\{\partial_x(\omega u_t) + \partial_t(\omega u_x)\big\},\\
    G_{\theta,u,2}(x,t) &= \omega_t u
        - \theta_1\big\{\partial_x(\omega u_t) + \partial_t(\omega u_x)\big\}
        - 2\theta_2\,\partial_t\big(\omega u_t\big).
\end{align*}
Therefore, the empirical debiased score in \eqref{eq:debiased_score} takes the
form
\begin{align*}
    S^{\mathrm{db}}_n(\theta,\widehat{u}) = -\widehat{A} + \widehat{B}\theta
    + \frac1n\sum_{i=1}^n\big\{Y_i - \widehat{u}(X_i,T_i)\big\}
        G_{\theta,\widehat{u}}(X_i,T_i),
\end{align*}
where $\widehat{A}=A(\widehat{u})$ and $\widehat{B}=B(\widehat{u})$. Since $G_{\theta,\widehat{u}}$ is affine in $\theta$, the debiased score is itself affine in $\theta$, so the root of $\sS_n^{\rm db}(\theta, \hat u)$ with respect to $\theta$ is available in closed form. In accordance with \eqref{eq:def_htdb}, and to guard against rare finite-sample instability, we project the resulting root onto the Euclidean ball centered at $\widehat{\theta}_{\mathrm{PI}}$ with radius $\delta_n=0.36\sqrt{7000/n}$. This constraint is active in only a few of the $n=500$ replications and inactive for $n\in\{3000,5000\}$. This yields the final debiased estimator $\widehat{\theta}_{\mathrm{db}}$.
\\\\
\noindent
{\bf Monte Carlo design and summary measures.}
For each sample size $n\in\{500,3000,5000\}$, we repeat the experiment over $100$ independent Monte Carlo runs, and in each run we record $\hat{\theta}^{\mathrm{PI}}$, $\hat{\theta}_{\mathrm{PINN}}$ and $\htdb$. The results are displayed in Figure \ref{fig:s12-AR}, one for each coordinate of $\theta$. Each figure shows, in the top row, kernel density estimates of the $\sqrt{n}$-scaled errors $\sqrt{n}(\hat{\theta}_j-\theta^\star_j)$ for the three estimators at $n\in\{500,3000,5000\}$, overlaid with the limiting Gaussian law of Theorem~\ref{thm:main_freq}; and in the bottom row, boxplots of the three estimators across the three sample sizes. Both $\hat{\theta}^{\mathrm{PI}}$ and $\hat \theta_{\rm PINN}$ are biased downward in both coordinates, whereas the debiased estimator $\htdb$ is approximately centered at $\theta^\star$: the density of $\sqrt{n}(\htdb - \theta^\star)$ concentrates around the limiting law as $n$ grows, and the median of $\htdb$ agrees closely with $\theta^\star$ at $n=3000$ and $n=5000$. Its $95\%$ confidence intervals attain approximately nominal coverage at these sample sizes (between $95\%$ and $98\%$), with mild under-coverage at $n=500$ (around $90\%$), reflecting the larger first-stage error in the small-sample regime.

\section{Some results on Sobolev spaces}
In this section, we present some results on the Sobolev function space. As before, we use the notation $W^{\beta, p}(\Omega)$ to denote the set of functions $f: \Omega \to \reals$ with finite Sobolev norms, where $\Omega \subseteq \reals^d$, and we use the notation $C^{k, \gamma}(\Omega)$ to denote the set of H\"{o}lder functions with index $k+ \gamma$, where $k$ is a non-negative integer and $\gamma \in (0, 1)$. We use the notation $C^K(\Omega)$ to denote $C^{K, 0}(\Omega)$, i.e., the set of all $K$ times differentiable functions with uniformly bounded derivatives. 
In our theoretical results, we use some key results for Sobolev function spaces, which we describe below: 

\begin{theorem}[Sobolev embedding theorem]
    \label{thm:sob_emb}
    Let $\Omega = (0, 1)^d$ and $f \in W^{\beta, p}(\Omega)$ for some integer $\beta \ge 1$. If $\beta - d/p$ is not an integer, then   
    $$
    \|f\|_{C^{K, \gamma}(\Omega)} \le C \|f\|_{W^{\beta, p}(\Omega)}, \quad K = \left\lfloor \beta - \frac{d}{p} \right\rfloor, \ \gamma = \beta - \frac{d}{p} - K \,,
    $$
    for some constant $C > 0$ independent of $f$. If $\beta - d/p$ is an integer, then: 
    $$
     \|f\|_{C^{K}(\Omega)} \le C \|f\|_{W^{\beta, p}(\Omega)}, \quad K = \beta - \frac{d}{p} - 1 \,,
    $$
    for some constant $C > 0$ independent of $f$. 
\end{theorem}
\begin{proof}
    For a detailed proof, see Theorem 9.9, Corollary 9.11, 9.13, and 9.15 of \cite{brezis2011functional}. The embedding theorem is first proved for $\Omega = \reals^d$ (Theorem 9.9 and Corollary 9.13 of \cite{brezis2011functional}), and then the results are extended to any $\Omega \subseteq \reals^d$ with smooth boundary (Corollary 9.15 of \cite{brezis2011functional}) using the argument of extension operator (Theorem 9.7 of \cite{brezis2011functional}). For this extension argument to work, we require $\Omega$ to satisfy some smoothness conditions (see Theorem 5 of Chapter 6 of \cite{stein1970singular}), which are satisfied by $(0, 1)^d$ (Example 2 of Section 3.3 of \cite{stein1970singular}). 
\end{proof}

\begin{theorem}[Gagliardo--Nirenberg interpolation]
\label{thm:interpolation}
Let $\Omega=(0,1)^d$, let $\beta\in\mathbb N$, and let $0\le j<\beta$. Let $1\le p_0,p_1\le\infty$, $1\le r\le\infty$, and let $\theta\in[j/\beta,1]$ satisfy
\begin{equation}
\frac1r=\frac jd+\theta\left(\frac1{p_1}-\frac{\beta}{d}\right)+\frac{1-\theta}{p_0}.
\label{eq:GN-general-relation}
\end{equation}
Then, for every multi-index $\alpha$ with $|\alpha|=j$ and every $f\in W^{\beta,p_1}(\Omega)\cap L^{p_0}(\Omega)$, there exists a constant $C=C(d,\beta,j,p_0,p_1,r,\theta,\Omega)$, such that
\begin{equation}
\|\partial^\alpha f\|_{L^r(\Omega)} \le C \|f\|_{W^{\beta,p_1}(\Omega)}^{\theta} \|f\|_{L^{p_0}(\Omega)}^{1-\theta}.
\label{eq:GN-general}
\end{equation}
as long as $\theta < 1$. In particular, taking $p_0=p_1=p$, we obtain that if 
\begin{equation}
\frac1r = \frac{j-\theta\beta}{d} + \frac1p,
\label{eq:GN-special-relation}
\end{equation}
then
\begin{equation}
\|\partial^\alpha f\|_{L^r(\Omega)} \le C \|f\|_{W^{\beta,p}(\Omega)}^\theta \|f\|_{L^p(\Omega)}^{1-\theta}.
\label{eq:GN-special}
\end{equation}
If $1<p<\infty$ and $\beta-j-d/p$ is a non-negative integer, then the endpoint $\theta=1$ is excluded, and the conclusion holds for $\theta\in[j/\beta,1)$.
\end{theorem}
\begin{proof}
We first recall the full Gagliardo--Nirenberg inequality on $\mathbb R^d$ \cite{nirenberg1959elliptic}.
If $g\in W^{\beta,p_1}(\mathbb R^d)\cap L^{p_0}(\mathbb R^d)$ and if
$\theta$ satisfies \eqref{eq:GN-general-relation}, then
\begin{equation}
\|\partial^\alpha g\|_{L^r(\mathbb R^d)}
\le
C
\|g\|_{W^{\beta,p_1}(\mathbb R^d)}^{\theta}
\|g\|_{L^{p_0}(\mathbb R^d)}^{1-\theta}.
\label{eq:GN-Rd}
\end{equation}
The usual endpoint restriction is needed in the critical Sobolev case; in particular, when
$r=\infty$ and $p_1<\infty$, we use the non-critical case $\theta<1$. Since $\Omega=(0,1)^d$ is a bounded Lipschitz domain, there exists an extension operator
$E$ (see Theorem 5 of Chapter 6 of \cite{stein1970singular}) such that, for every integer $\beta$ and every $1\le p_1\le\infty$,
\[
\|Ef\|_{W^{\beta,p_1}(\mathbb R^d)}
\le
C
\|f\|_{W^{\beta,p_1}(\Omega)}.
\]
Moreover, the extension can be chosen so that
\[
\|Ef\|_{L^{p_0}(\mathbb R^d)}
\le
C
\|f\|_{L^{p_0}(\Omega)}
\]
for every $1\le p_0\le\infty$. Applying Equation \eqref{eq:GN-Rd} to $g=Ef$ gives
\[
\|\partial^\alpha f\|_{L^r(\Omega)}
\le
\|\partial^\alpha Ef\|_{L^r(\mathbb R^d)}
\le
C
\|Ef\|_{W^{\beta,p_1}(\mathbb R^d)}^\theta
\|Ef\|_{L^{p_0}(\mathbb R^d)}^{1-\theta}.
\]
Using the two extension bounds yields
\[
\|\partial^\alpha f\|_{L^r(\Omega)}
\le
C
\|f\|_{W^{\beta,p_1}(\Omega)}^\theta
\|f\|_{L^{p_0}(\Omega)}^{1-\theta}.
\]
This proves \eqref{eq:GN-general}. Finally, if $p_0=p_1=p$, then \eqref{eq:GN-general-relation} becomes
\[
\frac1r
=
\frac jd
+
\theta\left(\frac1p-\frac{\beta}{d}\right)
+
\frac{1-\theta}{p}
=
\frac{j-\theta\beta}{d}
+
\frac1p,
\]
which gives \eqref{eq:GN-special-relation} and hence the special case
\eqref{eq:GN-special}.
\end{proof}

\section{Proof of Proposition \ref{prop:def_G}}
\label{sec:proof:prop:def_G}
\begin{proof}
Take $u\in W^{\beta, 2}(\Omega, L)$, the space of all Sobolev functions with Sobolev norm at most $L$. Therefore, and since $\beta-2\tau> d$ (Assumption \ref{assm:dgp}), by the Sobolev embedding theorem (Theorem \ref{thm:sob_emb}), $\|\partial^{\balpha} u\|_{L_\infty}  \leq L'$ for some constant $L'$ uniformly over $(u, \balpha)$ for all $\balpha$ with $|\balpha| \le 2\tau$. 
Consequently, by Assumption \ref{assm:psi_smoothness}, there exists a nonnegative continuous function $\theta\mapsto c_1(\theta)$ such that: 
$$
|(\Psi_\theta \circ K_1 u)(x)| = |\Psi(x, K_1(u)(x), \theta)| \le c_1(\theta), \quad \forall \ x \in [0, 1]^d, \ \ u \in W^{\beta, 2}(\Omega, L) \,.
$$
as $\|K_{1j}(u)\|_\infty \le L'$. Furthermore, by another application of Assumption \ref{assm:psi_smoothness}, we also have: 
$$
|\partial_{\theta_j} (\Psi_\theta \circ K_1 u)(x))| = |\partial_{\theta_j} \Psi(x, K_1(u)(x), \theta)|\le c_2(\theta) 
$$
for some nonnegative continuous function $c_2(\cdot)$. As a consequence, 
$$
\left|\left\{ K_0 u(x) + (\Psi_\theta\circ K_1 u)(x)\right\} \partial_{\theta_j} (\Psi_\theta \circ K_1 u)(x)\right|  \le c_2(\theta)(L' + c_1(\theta)), 
$$
for all $x \in \Omega$ and $u \in W^{\beta, 2}(\Omega, L)$.  
Therefore, we can justify interchanging the derivative and the integral via the dominated convergence theorem (DCT) to obtain: 
\[
(\nabla_\theta \cR)(\theta,u) =  \int_\Omega \left\{ K_0 u(x) + (\Psi_\theta\circ K_1 u)(x)\right\} \nabla_\theta (\Psi_\theta\circ K_1 u)(x)\omega(x)\rmd x,\;\;\;\theta\in\rset^q.
\]
For any $h\in W^{\beta,2}(\Omega)$,  the Gateaux derivative of $(\nabla_\theta \cR)(\theta,u)$ with respect to $u$ in the direction $h$ is defined as:
\[
\rmd (\nabla_\theta \cR)(\theta,u)[h] =\lim_{t\to 0}\frac{1}{t}\left[(\nabla_\theta \cR)(\theta,u + t h ) - (\nabla_\theta \cR)(\theta,u)\right].
\]
To obtain this derivative explicitly, we write
\begin{multline*} 
\frac{(\nabla_\theta \cR)(\theta,u + t h ) -(\nabla_\theta \cR)(\theta,u)}{t} = \int_\Omega K_0 h(x)\nabla_\theta (\Psi_\theta\circ K_1 (u + th))(x)\omega(x)\rmd x \\
+ \int_\Omega \frac{\left[(\Psi_\theta\circ K_1 (u + th))(x) - (\Psi_\theta\circ K_1 u)(x)\right]}{t}\nabla_\theta (\Psi_\theta\circ K_1 (u + th))(x)\omega(x)\rmd x \\
+ \int_\Omega \left\{ K_0 u(x) + \Psi_\theta\circ K_1 u(x)\right\} \frac{\left[ \nabla_\theta (\Psi_\theta\circ K_1 (u+th))(x) - \nabla_\theta(\Psi_\theta\circ K_1 u)(x)\right]}{t}\omega(x)\rmd x.
\end{multline*}
Under Assumption \ref{assm:psi_smoothness} we can carry the limit as $t\to 0$ under the integral, which gives 
\begin{multline*}
\rmd (\nabla_\theta \cR)(\theta,u)[h] =\lim_{t \to 0}\frac{(\nabla_\theta \cR)(\theta,u + t h ) -(\nabla_\theta \cR)(\theta,u)}{t} \\
= \int_\Omega \left\{ K_0 h(x) + \nabla_z(\Psi_\theta\circ K_1 u)(x)^\top K_1 h(x)\right\} \nabla_\theta (\Psi_\theta\circ K_1 u)(x)\omega(x)\rmd x \\
+ \int_\Omega \left\{ K_0 u(x) + (\Psi_\theta\circ K_1 u)(x)\right\} \nabla^{(2)}_{z,\theta} (\Psi_\theta\circ K_1 u)(x)K_1 h(x) \omega(x)\rmd x.
\end{multline*}
Let $K_0^\dagger$ denote the adjoint of $K_0$  as defined in Equation \eqref{eq:adjoint:def} and \eqref{eq:adjoint}.  
%With a smoothness assumption (twice continuous differentiability should be enough) on $\theta\mapsto \Psi_\theta(x)$ for all $x\in\Omega$ and since  $K_0 u, K_1 u\in C^{\beta-\tau}(\Omega)$,  for $u\in C^\beta(\Omega)$ we can differentiate under the integral and write
%Similarly, for all $\theta\in\rset^q$, and $u\in C^\beta(\Omega)$, we can take the Gateaux derivative of $(\nabla_\theta \cR)(\theta,u)$ with respect to $u$ in the direction $h$ and write
%where $\nabla^{(2)}_{x,\theta}\Psi_\theta(x)\in\rset^{d\times p}$ with $i$-th row given by the partial derivatives wrt the components of $x$ of $(\nabla_\theta \Psi_\theta (x))_i$. 
%\DM{We can generalize it to $d_1 \times d_2$. Yves: It should have been $d\times p$, not $d\times d$.}
%Let $K_0^\dagger$ denote the adjoint operator of $K_0$, well-defined on $C^\beta_0(\Omega)$: for $f,g\inC^\beta(\Omega)$ with $\beta>\tau$ such that $f\inC^\beta_0(\Omega)$ ($f$ vanishes on the boundary of $\Omega$): 
%\[\int_\Omega f(x) (K_0 g)(x)\rmd x = \int_\Omega g(x) (K_0^\dagger f)(x)\rmd x.\] 
We extend $K_0^\dagger$ to multivalued functions component-wise. First, we note that since $z \mapsto \partial_{\theta_j} \Psi(x, z, \theta) \in C^\beta$ by Assumption \ref{assm:psi_smoothness}, and $u\in W^{\beta,2}(\Omega,L)$, each coordinate function of $x \mapsto \partial_{\theta_j} (\Psi_\theta \circ K_1(u))(x) \in W^{\beta-\tau,2}(\Omega)$, and consequently $x \mapsto \partial_{\theta_j} (\Psi_\theta \circ K_1(u))(x) \  \omega(x) \in W^{\beta-\tau,2}_0(\Omega)$ (by Assumption \ref{assm:w}, See Theorem 2 Section 5.5 of \cite{evans:pde}). Using this, we  obtain
\begin{equation}
\label{adj:eq:1} 
\int_\Omega K_0 h(x) \nabla_\theta (\Psi_\theta \circ K_1(u))(x)\omega(x)\rmd x = \int_\Omega h(x) K_0^\dagger\left(\omega\nabla_\theta (\Psi_\theta \circ K_1(u))\right)(x)\rmd x.\end{equation}
Let us define $K_1^\dagger\eqdef (K_{1,1}^\dagger,\cdots,K_{1,p}^\dagger)^\top$, where $K_{1,j}^\dagger$ denotes the adjoint of $K_{1,j}$. If $A:\Omega \to\rset^{q\times p}$ is a matrix-valued function where each component belongs to $W_0^{\tau,2}(\Omega)$, we define $K_1^\dagger A:\;\Omega\to\rset^q$ as $(K_1^\dagger A)_i \eqdef \sum_{k=1}^p K_{1,k}^\dagger A_{ik}$. With this definition, we check that
\[ \int_\Omega A(x) K_1 u(x)\rmd x = \int_\Omega u(x) K_1^\dagger A(x)\rmd x.\] 
Since the components of $z\mapsto\nabla_\theta \Psi(x, z, \theta)$ and $z\mapsto\nabla^{(2)}_{z, \theta} \Psi(x, z, \theta)$ belong to $C^\beta(\Omega)$ (Assumption \ref{assm:psi_smoothness}), we have for all $u\in W^{\beta,2}(\Omega,L)$,
\begin{align*}
    \partial_{\theta_j}(\Psi_\theta \circ K_1(u)) \ \partial_{z_k} (\Psi_\theta \circ K_1(u)) \in W^{\beta-\tau,2}(\Omega), & \quad \forall \ 1\leq j \le q, 1 \le k\leq p, \\
    \{K_0 u + (\Psi_\theta \circ K_1(u))\} \partial^2_{z_k \theta_j}(\Psi_\theta \circ K_1(u))\in W^{\beta-\tau,2}(\Omega) & \quad \forall \ 1\leq j\leq q, 1\leq k\leq p. 
\end{align*}
Therefore, it follows that
\begin{multline}
\label{adj:eq:2}
\int_\Omega \nabla_z(\Psi_\theta \circ K_1(u))(x)^\top K_1 h(x)  \nabla_\theta (\Psi_\theta \circ K_1(u))(x)\omega(x)\rmd x \\
= \int_\Omega h(x) K_1^\dagger\left(\omega\left\{\nabla_\theta (\Psi_\theta \circ K_1(u)) \nabla_z(\Psi_\theta \circ K_1(u))^\top\right\}\right)(x)\rmd x,
\end{multline}
and
\begin{multline}
\label{adj:eq:3}
\int_\Omega \left\{ K_0 u(x) + \Psi_\theta(K_1 u(x))\right\} \nabla^{(2)}_{z,\theta}(\Psi_\theta \circ K_1(u))(x)K_1 h(x) \omega(x)\rmd x  \\
= \int_\Omega h(x)K_1^{\dagger}\left(\omega \left\{ K_0 u + (\Psi_\theta \circ K_1(u))\right\} \nabla^{(2)}_{z,\theta} (\Psi_\theta \circ K_1(u))\right)(x)\rmd x.
\end{multline}
In conclusion, combining (\ref{adj:eq:1})-(\ref{adj:eq:3}) for $\theta\in\rset^q$, and $u\in W^{\beta,2}(\Omega,L)$, we can find $\mathsf{G}_{\theta,u}:\Omega \to\rset^q$, $\mathsf{G}_{\theta,u}\in L^2(\Omega,\rset^q)$, such that the Gateaux derivative in the direction $h\in W^{\beta,2}(\Omega)$ can be expressed as: 
\[\rmd (\nabla_\theta \cR)(\theta,u)[h]  = \int_\Omega h(x) \mathsf{G}_{\theta,u}(x)\rmd x,\]
where
\begin{multline}
\label{eq:def_G}
\mathsf{G}_{\theta,u}(x) =K_0^\dagger\left(\omega\nabla_\theta(\Psi_\theta \circ K_1(u))\right)(x) \\
+ K_1^\dagger\left(\omega\left\{\nabla_\theta (\Psi_\theta \circ K_1(u)) \nabla_z (\Psi_\theta \circ K_1(u))^\top\right\}\right)(x) \\
+K_1^{\dagger}\left(\omega \left\{ K_0 u + (\Psi_\theta \circ K_1(u))\right\} \nabla^{(2)}_{z,\theta} (\Psi_\theta \circ K_1(u))\right)(x).
\end{multline}
Note that, since $K_0^\dagger$ and $K_1^\dagger$ are also differential operators of order $\tau$, $x\mapsto \mathsf{G}_{\theta,u}(x)$ has derivatives up to order $\beta-2\tau$, and $\mathsf{G}_{\theta,u}\in W^{\beta-2\tau,2}(\Omega)$, and Sobolev embedding yields $\|G_{\theta,u}\|_\infty \le c_0(\theta)$, for some function $c_0$ that is bounded on compact sets. This completes the first part of the proof of Proposition \ref{prop:def_G}. 
\\\\
\underline{{\bf Part (2):}} For $u,u'\in W^{\beta,2}(\Omega,L)$, $\theta,\theta'\in\rset^q$, we write
\begin{equation}\label{eq:decom:diffG0}
\mathsf{G}_{\theta,u}(x) - \mathsf{G}_{\theta',u'}(x) = \mathsf{G}_{\theta,u}(x) - \mathsf{G}_{\theta,u'}(x) + \mathsf{G}_{\theta,u'}(x) - \mathsf{G}_{\theta',u'}(x).    
\end{equation}
Using (\ref{eq:def_G}), we write the first difference on the RHS of the last display as 
\begin{align} 
\label{eq:decom:diffG1}
& \mathsf{G}_{\theta,u}(x) - \mathsf{G}_{\theta,u'}(x) \\
& = K_0^{\dagger}\left(\omega\left\{\nabla_\theta (\Psi_\theta \circ K_1(u)) - \nabla_\theta (\Psi_\theta \circ K_1(u'))\right\}\right)(x) \nonumber\\
&+ K_1^{\dagger}\left(\omega\left\{\left(\nabla_\theta (\Psi_\theta \circ K_1(u)) -  \nabla_\theta (\Psi_\theta \circ K_1(u'))\right)\nabla_z (\Psi_\theta \circ K_1(u))^\top\right\}\right)(x)\nonumber\\
& + K_1^{\dagger}\left(\omega \nabla_\theta (\Psi_\theta \circ K_1(u'))\left\{ \nabla_z (\Psi_\theta \circ K_1(u)) - \nabla_z (\Psi_\theta \circ K_1(u'))\right\}^\top \right)(x)\nonumber\\
&+K_1^{\dagger}\left(\omega \left\{ K_0 u - K_0 u' + (\Psi_\theta \circ K_1(u)) - (\Psi_\theta \circ K_1(u'))\right\} \nabla^{(2)}_{z,\theta} (\Psi_\theta \circ K_1(u))\right)(x)\nonumber\\
&+K_1^{\dagger}\left(\omega \left\{ K_0 u' + (\Psi_\theta \circ K_1(u'))\right\} \left\{\nabla^{(2)}_{z,\theta} (\Psi_\theta \circ K_1(u)) - \nabla^{(2)}_{z,\theta} (\Psi_\theta \circ K_1(u'))\right\}\right)(x).
\end{align}
Given $h\in W^{\tau,2}(\Omega)$, by the properties of the adjoint, 
\begin{align*}
   & \int_\Omega h(x) K_0^{\dagger}\left(\omega\left\{\nabla_\theta (\Psi_\theta \circ K_1(u)) - \nabla_\theta (\Psi_\theta \circ K_1(u'))\right\}\right)(x)\rmd x \\
   & \qquad \qquad = \int_\Omega K_0(h)(x) \left\{\nabla_\theta (\Psi_\theta \circ K_1(u)) - \nabla_\theta (\Psi_\theta \circ K_1(u'))\right\}\omega(x)\rmd x. 
\end{align*}
Since $K_0$ and $K_1$ are of order $\tau$,
\begin{multline*}
    \left\|\int_\Omega h(x) K_0^{\dagger}\left(\omega\left\{\nabla_\theta (\Psi_\theta \circ K_1(u)) - \nabla_\theta (\Psi_\theta \circ K_1(u'))\right\}\right)(x)\rmd x\right\|_2 \\
    \leq c_0\times \sqrt{q}\times \|h\|_{W^{\tau,2}} \max_{1\leq j\leq q}\; \sqrt{\int_\Omega |\partial_{\theta_j} (\Psi_\theta \circ K_1(u))(x) - \partial_{\theta_j} (\Psi_\theta \circ K_1(u'))(x)|^2\omega(x)\rmd x},
\end{multline*}
for some constant $c_0$ that depends only on the operator $K_0$ and $\|\omega\|_\infty$. With $u,u'\in W^{\beta,2}(\Omega,L)$, and $\beta > 2\tau + d/2$, by Sobolev embedding $\max(\|K_1 u\|_\infty,\|K_1 u'\|_\infty) \leq c_0 L$ for some constant $c_0$. Using this with Assumption \ref{assm:psi_smoothness}, we obtain 
\[
|\partial_{\theta_j} (\Psi_\theta \circ K_1(u))(x) - \partial_{\theta_j} (\Psi_\theta \circ K_1(u'))(x)| \leq c(\theta) \|K_1(u - u')(x)\|_2
\]
As a result, since $K_1$ is or order $\tau$, we get
\begin{align*}
& \left\|\int_\Omega h(x) K_0^{\dagger}\left(\omega\left\{\nabla_\theta (\Psi_\theta \circ K_1(u)) - \nabla_\theta (\Psi_\theta \circ K_1(u'))\right\}\right)(x)\rmd x\right\|_2 \\
& \leq c_1(\theta) \times \|h\|_{W^{\tau,2}} \times \| u - u'\|_{W^{\tau,2}},
\end{align*}
for some continuous function of $\theta$ $c_1$. The remaining terms in (\ref{eq:decom:diffG1}) are handled similarly. We decomposition the second difference on the RHS of (\ref{eq:decom:diffG0}) similarly:
\begin{align} 
\label{eq:decom:diffG2}
& \mathsf{G}_{\theta,u'}(x) - \mathsf{G}_{\theta',u'}(x) \\
= & \ K_0^{\dagger}\left(\omega\left\{\nabla_\theta (\Psi_\theta \circ K_1(u')) - \nabla_{\theta} (\Psi_{\theta'} \circ K_1(u'))\right\}\right)(x) \nonumber\\
&+ K_1^{\dagger}\left(\omega\left\{\left(\nabla_\theta (\Psi_\theta \circ K_1(u'))-  \nabla_{\theta} (\Psi_{\theta'} \circ K_1(u'))\right)\nabla_z (\Psi_\theta \circ K_1(u'))^\top\right\}\right)(x)\nonumber\\
& + K_1^{\dagger}\left(\omega \ \nabla_{\theta} (\Psi_{\theta'} \circ K_1(u'))\left\{ \nabla_z (\Psi_\theta \circ K_1(u'))- \nabla_z(\Psi_{\theta'} \circ K_1(u'))\right\}^\top \right)(x)\nonumber\\
&+K_1^{\dagger}\left(\omega \left\{(\Psi_\theta \circ K_1(u')) - (\Psi_{\theta'} \circ K_1(u'))\right\} \nabla^{(2)}_{z,\theta} (\Psi_\theta \circ K_1(u'))\right)(x)\nonumber\\
&+K_1^{\dagger}\left(\omega \left\{ K_0 u' + (\Psi_{\theta'} \circ K_1(u'))\right\} \left\{\nabla^{(2)}_{z,\theta} (\Psi_\theta \circ K_1(u')) - \nabla^{(2)}_{z,\theta} (\Psi_{\theta'} \circ K_1(u' ))\right\}\right)(x).
\end{align}
We then proceed similarly as above. We show the details for the first term. The other terms are handled similarly. Given $h\in W^{\tau,2}(\Omega)$, by the properties of the adjoint,
\begin{multline*} 
\int_\Omega h(x) K_0^{\dagger}\left(\omega\left\{\nabla_\theta (\Psi_\theta \circ K_1(u)) - \nabla_{\theta} (\Psi_{\theta'} \circ K_1(u))\right\}\right)(x)\rmd x \\
= \int_\Omega K_0 h(x) \left\{\nabla_\theta (\Psi_\theta \circ K_1(u)) - \nabla_{\theta} (\Psi_{\theta'} \circ K_1(u))\right\}\omega(x)\rmd x.
\end{multline*}
Using Assumption \ref{assm:psi_smoothness}, and since $\|K_1u\|_\infty \leq c_0 L$, we conclude that for some $\bar\theta\in \mathbf{B}_r$ between $\theta$ and $\theta'$,
\[ 
\|\nabla_\theta (\Psi_\theta \circ K_1(u)) - \nabla_{\theta} (\Psi_{\theta'} \circ K_1(u))\|_2 \leq \|\nabla^{(2)}_\theta (\Psi_{\bar \theta} \circ K_1(u))\|_2\times \|\theta-\theta'\|_2\leq C \|\theta-\theta'\|_2.
\]
Therefore
\[
\left\|\int_\Omega h(x) K_0^{\dagger}\left(\omega\left\{\nabla_\theta (\Psi_\theta \circ K_1(u)) - \nabla_{\theta} (\Psi_{\theta'} \circ K_1(u))\right\}\right)(x)\rmd x\right\|_2 \leq C\sqrt{q} \|\theta-\theta'\|_2 \|h\|_{W^{\tau,2}}.
\]
Hence, the conclusion follows.

\medskip
\noindent
\underline{{\bf Part (3):}} We next show that the representer function $\sG_{\theta, u}(x)$ satisfies for any two $u, u' \in W^{\beta, 2}(L)$: 
$$
\|\sG_{\theta, u, j} - \sG_{\theta, u', j}\|_{L_\infty(\Omega)} \le C \|u - u'\|^{\alpha_*}_{L_\infty(\Omega)} \quad \forall \ 1 \le j \le q, \quad \text{ with } \ \alpha_* = \frac{\beta - 2\tau - d/2}{\beta - d/2} \,.
$$
We use the shorthand $\|\cdot\|_\infty$ for $\|\cdot\|_{L_\infty(\Omega)}$. By Equation \eqref{eq:def_G}, we have: 
\begin{align*}
& \mathsf{G}_{\theta,u, j}(x) \\
& =K_0^\dagger\left(\omega \ \partial_{\theta_j} (\Psi_\theta \circ K_1(u))\right)(x) + e_j^\top K_1^\dagger\left(\omega\left\{\nabla_\theta (\Psi_\theta \circ K_1(u)) \nabla_z (\Psi_\theta \circ K_1(u))^\top\right\}\right)(x) \\
& \qquad \qquad \qquad \qquad +e_j^\top K_1^{\dagger}\left(\omega \left\{ K_0 u + (\Psi_\theta \circ K_1(u))\right\} \nabla^{(2)}_{z,\theta} (\Psi_\theta \circ K_1(u))\right)(x) \\
& := \mathsf{G}^{(1)}_{\theta,u, j}(x) + \mathsf{G}^{(2)}_{\theta,u, j}(x) + \mathsf{G}^{(3)}_{\theta,u, j}(x) \,.
\end{align*}
As a consequence, we have: 
$$
\|\mathsf{G}_{\theta,u, j} - \mathsf{G}_{\theta,u', j}\|_{\infty} \le \sum_{k = 1}^3 \|\mathsf{G}^{(k)}_{\theta,u, j} - \mathsf{G}^{(k)}_{\theta,u', j}\|_{\infty}
$$
Therefore, we need to bound $\|\mathsf{G}^{(k)}_{\theta,u, j} - \mathsf{G}^{(k)}_{\theta,u', j}\|_{\infty}$ for $1 \le k \le 3$. Let us start with $k = 1$. Note that for $u \in W^{\beta, 2}(\Omega, L)$, the function $x \mapsto \omega(x) \ \partial_{\theta_j} (\Psi_\theta \circ K_1(u))(x) \in W^{\beta - \tau, 2}(\Omega, L_1)$ for some $L_1 > 0$, as the differential operator $K_1$ has order atmost $\tau$. Furthermore, both $K_0$ and $K_0^\dagger$ are differential operators with order at most $\tau$. Also, by Theorem \ref{thm:sob_emb}, we know $u \in C^{2\tau}$ as $\beta - 2\tau > d$. Therefore, there exists a compact set $\Omega_1 \subseteq \reals^p$, such that $K_1(u)(x) \in \Omega_1$ for any $u \in W^{\beta, 2}(\Omega, L)$. 
Therefore applying Theorem \ref{thm:interpolation} on $\omega \ \partial_{\theta_j} (\Psi_\theta \circ K_1(u)) - \omega \ \partial_{\theta_j} (\Psi_\theta \circ K_1(u'))$, we obtain: 
\allowdisplaybreaks
\begin{align*}
& \left\|K_0^\dagger\left(\omega \ \partial_{\theta_j} (\Psi_\theta \circ K_1(u)) - \omega \ \partial_{\theta_j} (\Psi_\theta \circ K_1(u'))\right)(x)\right\|_\infty \\
& \le C \|\omega \ \partial_{\theta_j} (\Psi_\theta \circ K_1(u))- \omega \ \partial_{\theta_j} (\Psi_\theta \circ K_1(u'))\|^{\frac{\tau}{\beta - \tau - d/2}}_{W^{\beta-\tau, 2}} \\
& \qquad \qquad \qquad \qquad \times \|\omega \ \partial_{\theta_j} (\Psi_\theta \circ K_1(u)) - \omega \ \partial_{\theta_j} (\Psi_\theta \circ K_1(u'))\|^{1 - \frac{\tau}{\beta - \tau - d/2}}_{\infty} \\
& \le C_1 \|\omega\|^{\frac{\beta - 2\tau - d/2}{\beta - \tau - d/2}}_\infty \|\partial_{\theta_j} (\Psi_\theta \circ K_1(u)) - \partial_{\theta_j} (\Psi_\theta \circ K_1(u'))\|^{\frac{\beta - 2\tau - d/2}{\beta - \tau - d/2}}_{\infty} \\
& \le C_2 \|\partial_{\theta_j}(\Psi_\theta \circ K_1(u)) - \partial_{\theta_j} (\Psi_\theta \circ K_1(u'))\|^{\frac{\beta - 2\tau - d/2}{\beta - \tau - d/2}}_{\infty} \\
& \le C_2 \sqrt{p}\left(\sup_{x \in \Omega_1, \theta \in \Theta} \|\nabla^{(2)}_{z \theta}\Psi(x, z, \theta)\|_{\rm op}\right)^{\frac{\beta - 2\tau - d/2}{\beta - \tau - d/2}} \ \max_{1 \le j \le p} \|K_{1, j}(u - u')\|_\infty^{\frac{\beta - 2\tau - d/2}{\beta - \tau - d/2}} \\
& \le C_3 \ \max_{1 \le j \le p} \|K_{1, j}(u - u')\|_\infty^{\frac{\beta - 2\tau - d/2}{\beta - \tau - d/2}} \,. \\
% & \le C_1 \|w\|_\infty^{\frac{\beta - 2\tau}{\beta - \tau}}\left(\sup_{x, \theta} \|\nabla_{x \theta}\Psi_{\theta}(x)\|_{\rm op}\right)^{\frac{\beta - 2\tau}{\beta - \tau}}\left(\sup_x \sum_j (K_{1j}(u  - u')(x))^2\right)^{\frac{\beta - 2\tau}{2(\beta - \tau)}} \\
% & \le C_1 \|w\|_\infty^{\frac{\beta - 2\tau}{\beta - \tau}}\left(\sup_{x, \theta} \|\nabla_{x \theta}\Psi_{\theta}(x)\|_{\rm op}\right)^{\frac{\beta - 2\tau}{\beta - \tau}}\left(\sum_j \|K_{1j}(u  - u')\|_\infty^2\right)^{\frac{\beta - 2\tau}{2(\beta - \tau)}} \\
% & \le C_1 \|w\|_\infty^{\frac{\beta - 2\tau}{\beta - \tau}}\left(\sup_{x, \theta} \|\nabla_{x \theta}\Psi_{\theta}(x)\|_{\rm op}\right)^{\frac{\beta - 2\tau}{\beta - \tau}}\left(pC_2^2 \|u - u'\|_\infty^{\frac{2(\beta - \tau)}{\beta}}\right)^{\frac{\beta - 2\tau}{2(\beta - \tau)}} \\
% & \le C_1 p^{\frac{\beta - 2\tau}{2(\beta - \tau)}}C_2^{\frac{\beta - 2\tau}{\beta - \tau}} \|w\|_\infty^{\frac{\beta - 2\tau}{\beta - \tau}}\left(\sup_{x, \theta} \|\nabla_{x \theta}\Psi_{\theta}(x)\|_{\rm op}\right)^{\frac{\beta - 2\tau}{\beta - \tau}}\|u  - u_0\|_\infty^{\frac{\beta - 2\tau}{\beta}}  \\
% & := C_{\mathsf{G}, 1} \|u  - u_0\|_\infty^{\frac{\beta - 2\tau}{\beta}} \,.
\end{align*}
We next apply Theorem \ref{thm:interpolation} one more time, on the functions $K_{1j}(u- u')$. Note that $u - u' \in W^{\beta, 2}(\Omega, 2L)$ and $K_{1j}$ is a differential operator of order atmost $\tau$. Therefore, we have for any $1 \le j \le p$: 
\begin{align*}
    \|K_{1, j}(u - u')\|_\infty^{\frac{\beta - 2\tau - d/2}{\beta - \tau - d/2}} & \le C_4 \|u - u'\|^{\frac{\tau}{\beta - d/2}}_{W^{\beta, 2}} \|u - u'\|_\infty^{\frac{\beta - \tau - d/2}{\beta  -d/2}}  \le C_5 \|u - u'\|_\infty^{\frac{\beta - \tau - d/2}{\beta  -d/2}} \,.
\end{align*}
Using this we have: 
\begin{align*}
     \|\mathsf{G}^{(1)}_{\theta,u, j} - \mathsf{G}^{(1)}_{\theta,u', j}\|_{\infty} & = \left\|K_0^\dagger\left(\omega \ \partial_{\theta_j} (\Psi_\theta \circ K_1(u)) - \omega \ \partial_{\theta_j} (\Psi_\theta \circ K_1(u'))\right)(x)\right\|_\infty \\
     & \le C_6 \|u - u'\|_\infty^{\frac{\beta - 2\tau - d/2}{\beta - d/2}} \,.
\end{align*}
Now, we analyze the second term $\mathsf{G}^{(2)}_{\theta,u, j}$. By definition, we have: 
\begin{align*}
    & \mathsf{G}^{(2)}_{\theta,u, j} - \mathsf{G}^{(2)}_{\theta,u', j} \\
    & = e_j^\top K_1^\dagger\left(\omega \ \left\{\nabla_\theta (\Psi_\theta \circ K_1(u)) \nabla_z(\Psi_\theta \circ K_1(u))^\top \right. \right. \\
    & \qquad \qquad \qquad \left. \left. - \nabla_\theta (\Psi_\theta \circ K_1(u'))\nabla_z(\Psi_\theta \circ K_1(u'))^\top\right\}\right) \\
    & = \sum_{k = 1}^p K_{1k}^\dagger \left(\omega \ \left\{\partial_{\theta_j} (\Psi_\theta \circ K_1(u))\partial_{z_k}(\Psi_\theta \circ K_1(u)) - \partial_{\theta_j} (\Psi_\theta \circ K_1(u'))\partial_{z_k}(\Psi_\theta \circ K_1(u'))\right\}\right)
\end{align*}
Therefore, we immediately have by the triangle inequality of $L_\infty$ norm:
\begin{align*}
& \|\mathsf{G}^{(2)}_{\theta,u, j} - \mathsf{G}^{(2)}_{\theta,u', j}\|_\infty \\
& \le \sum_{k = 1}^p \left\|K_{1k}^\dagger \left(\omega \ \left\{\partial_{\theta_j} (\Psi_\theta \circ K_1(u))\partial_{z_k}(\Psi_\theta \circ K_1(u)) \right. \right. \right. \\
& \qquad \qquad \qquad \qquad \qquad \left. \left. \left. - \partial_{\theta_j} (\Psi_\theta \circ K_1(u'))\partial_{z_k}(\Psi_\theta \circ K_1(u'))\right\}\right)\right\|_\infty \,.
\end{align*}
Hence, it is enough to show that each of the summands on the RHS is upper bounded by some constant times $\|u - u'\|_\infty^{\alpha_*}$. Towards that end, fix some $1 \le k \le p$. By Assumption \ref{assm:psi_smoothness}, we know: 
$$
x \mapsto \omega(x) \ \partial_{\theta_j} (\Psi_\theta \circ K_1(u))(x) \ \partial_{z_k}(\Psi_\theta \circ K_1(u))(x) \in W^{\beta - \tau, 2}(\Omega, L_1), 
$$
for some $L_1 > 0$. Therefore, we have: 
\allowdisplaybreaks
\begin{align*}
    & \left\|K_{1k}^\dagger \left(\omega \ \left\{\partial_{\theta_j} (\Psi_\theta \circ K_1(u))\partial_{z_k}(\Psi_\theta \circ K_1(u)) - \partial_{\theta_j} (\Psi_\theta \circ K_1(u'))\partial_{z_k}(\Psi_\theta \circ K_1(u'))\right\}\right)\right\|_\infty \\
    & \le C \left\|\omega \ \left\{\partial_{\theta_j} (\Psi_\theta \circ K_1(u))\partial_{z_k}(\Psi_\theta \circ K_1(u)) \right. \right. \\
    & \qquad \qquad \left. \left. - \partial_{\theta_j} (\Psi_\theta \circ K_1(u'))\partial_{z_k}(\Psi_\theta \circ K_1(u'))\right\}\right\|^{\frac{\tau}{\beta  - \tau - d/2}}_{W^{\beta - \tau, 2}} \\
    & \qquad \qquad \qquad \times \left\|\omega \ \left\{\partial_{\theta_j} (\Psi_\theta \circ K_1(u))\partial_{z_k}(\Psi_\theta \circ K_1(u)) \right. \right. \\
    & \qquad \qquad \qquad \qquad \qquad \left. \left. - \partial_{\theta_j} (\Psi_\theta \circ K_1(u'))\partial_{z_k}(\Psi_\theta \circ K_1(u'))\right\}\right\|^{\frac{\beta - 2\tau - d/2}{\beta - \tau - d/2}}_\infty \\
    & \le C_1 \left\|\omega \ \left\{\partial_{\theta_j} (\Psi_\theta \circ K_1(u))\partial_{z_k}(\Psi_\theta \circ K_1(u)) \right. \right. \\
    & \qquad \qquad \left. \left. - \partial_{\theta_j} (\Psi_\theta \circ K_1(u'))\partial_{z_k}(\Psi_\theta \circ K_1(u'))\right\}\right\|^{\frac{\beta - 2\tau - d/2}{\beta - \tau - d/2}}_\infty \\
    & \le C_1 \|\omega\|_\infty^{\frac{\beta - 2\tau - d/2}{\beta - \tau - d/2}} \left\|\partial_{\theta_j} (\Psi_\theta \circ K_1(u))\partial_{z_k}(\Psi_\theta \circ K_1(u)) \right. \\
    & \qquad \qquad \qquad \qquad \qquad \left. - \partial_{\theta_j} (\Psi_\theta \circ K_1(u'))\partial_{z_k}(\Psi_\theta \circ K_1(u'))\right\|^{\frac{\beta - 2\tau - d/2}{\beta - \tau - d/2}}_\infty  \\
    & \le C_2 \left\|\partial_{\theta_j} (\Psi_\theta \circ K_1(u))\partial_{z_k}(\Psi_\theta \circ K_1(u)) - \partial_{\theta_j} (\Psi_\theta \circ K_1(u'))\partial_{z_k}(\Psi_\theta \circ K_1(u'))\right\|^{\frac{\beta - 2\tau - d/2}{\beta - \tau - d/2}}_\infty  \\
    & \le C_2 \left\|\partial_{\theta_j} (\Psi_\theta \circ K_1(u))\partial_{z_k}(\Psi_\theta \circ K_1(u)) - \partial_{\theta_j} (\Psi_\theta \circ K_1(u))\partial_{z_k}(\Psi_\theta \circ K_1(u'))\right\|^{\frac{\beta - 2\tau - d/2}{\beta - \tau - d/2}}_\infty \\
    & \quad  + C_2 \left\|\partial_{\theta_j} (\Psi_\theta \circ K_1(u))\partial_{z_k}(\Psi_\theta \circ K_1(u')) - \partial_{\theta_j} (\Psi_\theta \circ K_1(u'))\partial_{z_k}(\Psi_\theta \circ K_1(u'))\right\|^{\frac{\beta - 2\tau - d/2}{\beta - \tau - d/2}}_\infty \\
    & \le C_2 \left(\sup_{x \in \Omega_1, \theta \in \Theta} |\partial_{\theta_j} \Psi(x, K_1(u)(x), \theta)|\right)^{\frac{\beta - 2\tau - d/2}{\beta - \tau - d/2}} \\
    & \qquad \qquad \times \|\partial_{z_k}(\Psi_\theta \circ K_1(u)) - \partial_{z_k}(\Psi_\theta \circ K_1(u'))\|_\infty^{\frac{\beta - 2\tau - d/2}{\beta - \tau - d/2}} \\
    & \qquad \qquad \qquad + C_2\left(\sup_{x \in \Omega_1, \theta \in \Theta} |\partial_{z_k} \Psi(x, K_1(u')(x), \theta)|\right)^{\frac{\beta - 2\tau - d/2}{\beta - \tau - d/2}} \\
    & \qquad \qquad \qquad \qquad \qquad \times \|\partial_{\theta_j} (\Psi_\theta \circ K_1(u))  - \partial_{\theta_j} (\Psi_\theta \circ K_1(u'))\|_\infty^{\frac{\beta - 2\tau - d/2}{\beta - \tau - d/2}} \\ 
    & \le C_3\left(\|\partial_{z_k}(\Psi_\theta \circ K_1(u)) - \partial_{z_k}(\Psi_\theta \circ K_1(u'))\|_\infty^{\frac{\beta - 2\tau - d/2}{\beta - \tau - d/2}} \right. \\
    & \qquad \qquad \qquad \left. + \|\partial_{\theta_j} (\Psi_\theta \circ K_1(u))  - \partial_{\theta_j} (\Psi_\theta \circ K_1(u'))\|_\infty^{\frac{\beta - 2\tau - d/2}{\beta - \tau - d/2}} \right) \,.
    % & \le (\|w\|_\infty\nabla_{\theta_j} \Psi_\theta(K_1 u))^{\frac{\beta - 2\tau}{\beta - \tau}}\|\left\{\nabla_{x_k}\Psi_\theta(K_1 u)^\top - \nabla_{x_k}\Psi_\theta(K_1 u')^\top\right\}\|^{\frac{\beta - 2\tau}{\beta - \tau}}_\infty \\
    % & \quad + \|w\|_\infty^{\frac{\beta - 2\tau}{\beta - \tau}}\sum_{k = 1}^p \|\nabla_{x_k}\Psi_\theta(K_1 u')\|_\infty^{\frac{\beta - 2\tau}{\beta - \tau}}\|\left\{\nabla_{\theta_j} \Psi_\theta(K_1 u)  - \nabla_{\theta_j} \Psi_\theta(K_1 u')\right\}\|^{\frac{\beta - 2\tau}{\beta - \tau}}_\infty \\
    % & \le p (\|w\|_\infty \times \sup_j \|\nabla_{\theta_j} \Psi_\theta(K_1 u)\|_\infty \times \sup_{x, \theta} \|\nabla^2_x \Psi_\theta(x)\|_{\rm op})^{^{\frac{\beta - 2\tau}{\beta - \tau}}} \|K_1(u - u')\|_{2, \infty}^{\frac{\beta - 2\tau}{\beta - \tau}} \\
    % & \quad + p (\|w\|_\infty \times \sup_k \|\nabla_{x_k} \Psi_\theta(K_1 u)\|_\infty \times \sup_{x, \theta} \|\nabla_{x \theta} \Psi_{\theta}(x)\|_{\rm op})^{^{\frac{\beta - 2\tau}{\beta - \tau}}}\|K_1(u - u')\|_{2, \infty}^{\frac{\beta - 2\tau}{\beta - \tau}} \\
    % & \le C_{\mathsf{G}, 2}\|u - u'\|_\infty^{\frac{\beta - 2\tau}{\beta}} \,.
\end{align*}
We next use the same technique as before. Note that by Lipschitz property of $\partial_{z_k} \Psi(x, z, \theta)$ and $\partial_{\theta_j}\Psi(x, z, \theta)$, we have: 
\begin{align*}
    & \|\partial_{z_k}(\Psi_\theta \circ K_1(u)) - \partial_{z_k}(\Psi_\theta \circ K_1(u'))\|_\infty \\
    & \qquad \le p \sup_{z} \|\nabla^{(2)}_z \Psi(x, z, \theta) \|_2 \ \max_{1 \le j\le p} \|K_{1j}(u - u')\|_{\infty} \\
    & \qquad \le C_4\max_{1 \le j\le p} \|K_{1j}(u - u')\|_{\infty} \,,\\
    & \|\partial_{\theta_j} (\Psi_\theta \circ K_1(u))  - \partial_{\theta_j} (\Psi_\theta \circ K_1(u'))\|_\infty \\
    & \qquad \le p \sup_{x \in \Omega_1} \|\nabla^{(2)}_{z\theta} \Psi(x, z, \theta)\|_2 \ \max_{1 \le j\le p} \|K_{1j}(u - u')\|_{\infty} \\
    & \qquad \le C_6\max_{1 \le j\le p} \|K_{1j}(u - u')\|_{\infty} \,.
\end{align*}
Also, we have already shown that: 
$$
  \|K_{1, j}(u - u')\|_\infty^{\frac{\beta - 2\tau - d/2}{\beta - \tau - d/2}} \le C_5 \|u - u'\|_\infty^{\frac{\beta - \tau - d/2}{\beta  -d/2}} \,.
$$
Hence, combining this, we get: 
$$
\|\mathsf{G}^{(2)}_{\theta,u, j} - \mathsf{G}^{(2)}_{\theta,u', j}\|_\infty \le C_6 \|u - u'\|_\infty^{\frac{\beta - 2\tau - d/2}{\beta  -d/2}} \,.
$$
The bound for $\mathsf{G}^{(3)}_{\theta,u, j}$ is similar and hence skipped for brevity. 
It remains to bound \(G_{\theta,u,j}-G_{\theta',u,j}\) in \(L^\infty\).
Using the decomposition analogous to \eqref{eq:decom:diffG1}, and using
the assumed uniform \(C^{\beta-\tau}\)-regularity in the \(z\)-variable
together with Lipschitz dependence on \(\theta\), each coefficient
appearing inside \(K_0^\dagger\) and \(K_1^\dagger\) differs by
\(O(\|\theta-\theta'\|_2)\) in \(W^{\beta-\tau,2}\). Applying the Sobolev
embedding after the differential operators \(K_0^\dagger,K_1^\dagger\)
therefore gives
\[
\|G_{\theta,u,j}-G_{\theta',u,j}\|_\infty
\le
C\|\theta-\theta'\|_2.
\]
Combining this with the preceding \(u\)-difference bound proves
\eqref{stab:G:3}.
This completes the proof. 

\end{proof}

\section{Proofs of Theorem \ref{thm:deep_sobolev}, Corollary \ref{cor:deep_sobolev_PINN} and Lemma \ref{lem:consistency_PI}}\label{proof:thm:deep_sobolev}
\subsection{Proof of Theorem \ref{thm:deep_sobolev}}
Recall the definition of the function class $\tilde \cF_{\rm NN}$ as specified in Assumption \ref{assm:function_class}. For $u \in \tilde \cF_{\rm NN}$ we have $\|u\|_\infty \le m$ and $\|u\|_{W^{\beta, 2}} < \infty$.
As $\hat u$ is the minimizer of Equation \eqref{eq:u_est_NN} over $\tilde \cF_{\rm NN}$, we have, for any fixed non-zero $\tilde u_n \in F_{\rm NN}$: 
\begin{align*}
\frac1n \sum_i (Y_i - \hat u(X_i))^2 + \lambda_n \|\hat u\|_{W^{\beta, 2}}^2 \le \frac1n \sum_i (Y_i - \tilde u_n(X_i))^2 + \lambda_n \|\tilde u_n\|_{W^{\beta, 2}}^2 
\end{align*}
Some simple algebra yields: 
\begin{align*}
    \|\hat u - u_*\|_n^2  & \le \|u_* - \tilde u_n\|_n^2 + \frac{2}{n}\sum_i \eps_i (\hat u(X_i) - \tilde u_n(X_i)) + \lambda_n (\|\tilde u_n\|_{W^{\beta, 2}}^2 - \|\hat u\|_{W^{\beta, 2}}^2) \\
    & = \|u_* - \tilde u_n\|_n^2 + \frac{2}{n}\sum_i \eps_i \frac{(\hat u(X_i) - \tilde u_n(X_i))}{\|\tilde u_n\|_{W^{\beta, 2}} + \|\hat u\|_{W^{\beta, 2}}}(\|\tilde u_n\|_{W^{\beta, 2}} + \|\hat u\|_{W^{\beta, 2}}) \\
    & \hspace{10em}+ \lambda_n (\|\tilde u_n\|_{W^{\beta, 2}}^2 - \|\hat u\|_{W^{\beta, 2}}^2) 
\end{align*}
Now, let us define another function class: 
$$
\cG_n = \left\{g: g = \frac{u - \tilde u_n}{\|u\|_{W^{\beta, 2}} + \|\tilde u_n\|_{W^{\beta, 2}}}, \quad u \in \tilde \cF_{\rm NN} \right\}
$$
Furthermore, we set $g = 0$ when both $u = \tilde u_n \equiv 0$. From the Sobolev embedding theorem (Theorem \ref{thm:sob_emb}) we have: 
$$
\|g\|_\infty \le C \|g\|_{W^{\beta, 2}} \le C \frac{\|u - \tilde u_n\|_{W^{\beta, 2}}}{\|u\|_{{W^{\beta, 2}}} + \|\tilde u_n\|_{W^{\beta, 2}}} \le C \,.
$$
Hence the functions in $\cG_n$ are uniformly bounded. 
%Clearly, the functions in $\cG_n$ are uniformly bounded by $2m/c_*$ where $c_* = \|\tilde u_n\|_{W^{\beta, 2}}$. 
Therefore, by Lemma 4 of \cite{fan2024factor}, we have, with probability $\ge 1 - c_4 e^{\log{\log{n}}-t}$ (for any $t > 0, \eps > 0$ and some constant $c_4 > 0$): 
$$
\frac{1}{n}\sum_i \eps_i g(X_i) \le c_5\left(\|g\|_n + \frac{1}{n}\right) \sqrt{\frac{\mathrm{Pdim}(\cG_n)\log{n}}{n} + \frac{t}{n}}, \quad \forall \ g \in \cG_n \,.
$$
where $\mathrm{Pdim}(\cG_n)$ denotes the Pseudo-dimension of $\cG_n$, orignially introduced in \cite{pollard1990empirical}, see also Definition 2 of \cite{bartlett2019nearly}. As a consequence, choosing $t = \mathrm{Pdim}(\cG_n)\log{n}$, we obtain that for all $g \in \cG_n$: 
$$
\frac{1}{n}\sum_i \eps_i g(X_i) \le \sqrt{2}c_5\left(\|g\|_n + \frac{1}{n}\right) \sqrt{\frac{\mathrm{Pdim}(\cG_n)\log{n}}{n}}, 
$$
with probability $\ge  1 - c_4e^{\log{\log{n}} - \mathrm{Pdim}(\cG_n)\log{n}}$. 
The above inequality implies that we have on an event $\e_n$ with probability $\ge 1 - c_4e^{\log{\log{n}} - \mathrm{Pdim}(\cG_n)\log{n}}$: 
\begin{align}
\label{eq:mb_2}
    & \|\hat u - u_*\|_n^2  \notag \\
    & \le \|u_* - \tilde u_n\|_n^2 + \sqrt{2}c_5\left(\|\hat u - \tilde u_n\|_n + \frac{\|\hat u\|_{W^{\beta, 2}} + \|\tilde u_n\|_{W^{\beta, 2}}}{n}\right) \sqrt{\frac{\mathrm{Pdim}(\cG_n)\log{n}}{n}} \notag \\
    & \hspace{20em} + \lambda_n (\|\tilde u_n\|_{W^{\beta, 2}}^2 - \|\hat u\|_{W^{\beta, 2}}^2).
\end{align}
We will later show that $\|\hat u\|_{W^{\beta, 2}} \le H_1$ for some constant $H_1 > 0$. Let us now assume this for the time being, which yields: 
\begin{align*}
     \|\hat u - u_*\|_n^2  & \le \|u_* - \tilde u_n\|_n^2 + \sqrt{2}c_5 \|\hat u - \tilde u_n\|_n\sqrt{\frac{\mathrm{Pdim}(\cG_n)\log{n}}{n}} \\
     & \qquad \qquad \qquad+ \frac{\sqrt{2}c_5(H_1 + c_*)}{n}\sqrt{\frac{\mathrm{Pdim}(\cG_n)\log{n}}{n}} + 2\lambda_n c_*^2 \\
     & \le  \|u_* - \tilde u_n\|_n^2 + \frac{\|\tilde u_n - \hat u\|_n^2}{4} + 2c_5^2 \frac{\mathrm{Pdim}(\cG_n)\log{n}}{n} \\
     & \qquad \qquad \qquad+ \frac{\sqrt{2}c_5(H_1 + c_*)}{n}\sqrt{\frac{\mathrm{Pdim}(\cG_n)\log{n}}{n}} + 2\lambda_n c_*^2 \\
     & \le  \|u_* - \tilde u_n\|_n^2 + \frac{\|\tilde u_n - u_*\|_n^2}{2} + \frac{\|\hat u - u_*\|_n^2}{2} + 2c_5^2 \frac{\mathrm{Pdim}(\cG_n)\log{n}}{n} \\
     & \qquad \qquad \qquad+ \frac{\sqrt{2}c_5(H_1 + c_*)}{n}\sqrt{\frac{\mathrm{Pdim}(\cG_n)\log{n}}{n}} + 2\lambda_n c_*^2 \\
     \implies \frac12\|\hat u - u_*\|_n^2 & \le \frac{3}{2}\|\tilde u_n - u_*\|_n^2 + 2c_5^2 \frac{\mathrm{Pdim}(\cG_n)\log{n}}{n} \\
     & \qquad \qquad \qquad+ \frac{\sqrt{2}c_5(H_1 + c_*)}{n}\sqrt{\frac{\mathrm{Pdim}(\cG_n)\log{n}}{n}} + 2\lambda_n c_*^2 \\
     \implies \|\hat u - u_*\|_n^2  & \le 3\|\tilde u_n - u_*\|_n^2 + 4c_5^2 \frac{\mathrm{Pdim}(\cG_n)\log{n}}{n} \\
      & \qquad \qquad \qquad+ \frac{2\sqrt{2}c_5(H_1 + c_*)}{n}\sqrt{\frac{\mathrm{Pdim}(\cG_n)\log{n}}{n}} + 4\lambda_n c_*^2
\end{align*}
Now if we choose $\lambda_n = K(\mathrm{Pdim}(\cG_n)\log{n}/n)$ for some constant $K > 0$, then we have: 
\begin{equation}
\label{eq:basic_bound_dnn}
\|\hat u - u_*\|_n^2  \le 3\|\tilde u_n - u_*\|_n^2 + (4c_5^2 + 4KH_1 + 1)\frac{\mathrm{Pdim}(\cG_n)\log{n}}{n}
\end{equation}
where we ignore the third summand on the RHS of the previous equation (essentially we have bounded it by $1$) because it is of much smaller order than the second term (smaller by an order of $O(n^{1/2})$, as we will argue later when we will establish a bound on $\mathrm{Pdim}(\cG_n)$).

The next part of the proof balances approximation error and stochastic error. Recall that the above bound is true for any $\tilde u_n \in \tilde \cF_{\rm NN}$. We now choose $\tilde u_n$ carefully. By our approximation theorem (Theorem \ref{thm:RePU-spline-approx} together with Remark \ref{rem:approx_beta}), we know there exists $\tilde u_{n, \mathrm{NN}} \in \cF_{\rm NN}$ such that: 
$$
\|\tilde u_{n, \mathrm{NN}} - u_*\|_2^2 \le C_a N^{-\frac{2\beta}{d}}, \qquad \|\tilde u_{n, \mathrm{NN}} - u_*\|_{W^{\beta, 2}} = c_1 < \infty \,.
$$
Furthermore, by an application of the interpolation theorem (Theorem \ref{thm:interpolation}), we have: 
\begin{align}
\label{eq:approx-linf-corrected}
\|\tilde u_{n, \mathrm{NN}} - u_*\|_\infty & \le C \|\tilde u_{n, \mathrm{NN}} - u_*\|^{\frac{d}{2\beta}}_{W^{\beta, 2}} \|\tilde u_{n, \mathrm{NN}}- u_*\|_2^{1 - \frac{d}{2\beta}} \notag \\
& \le C_1 \|\tilde u_{n, \mathrm{NN}} - u_*\|_2^{1 - \frac{d}{2\beta}} \le C_1 N^{-\frac{(2\beta - d)}{2d}} \,.
\end{align}
%Furthermore, by Theorem \ref{thm:interpolation}, we have
%\begin{align}
%\|\widetilde u_{n,\mathrm{NN}}-u^\star\|_\infty
%&\le
%C
%\|\widetilde u_{n,\mathrm{NN}}-u^\star\|_{W^{\beta,2}}^{d/(2\beta)}
%\|\widetilde u_{n,\mathrm{NN}}-u^\star\|_{2}^{1-d/(2\beta)}
%\notag \\
%&\le
%C\|\widetilde u_{n,\mathrm{NN}}-u^\star\|_{2}^{1-d/(2\beta)}
%\le
%CN^{-(2\beta-d)/(2d)} .
%\end{align}
Here we used that $\beta>d/2$, which follows from Assumption \ref{assm:dgp}.
Therefore, by Bernstein's inequality, we have with probability $\ge 1 - e^{-t}$: 
$$
\|\tilde u_{n, \mathrm{NN}} - u_*\|_n^2 \le \|\tilde u_{n, \mathrm{NN}}- u_*\|_2^2 + \|\tilde u_{n, \mathrm{NN}} - u_*\|_2\|\tilde u_{n, \mathrm{NN}} - u_*\|_\infty\sqrt{\frac{2t}{n}} +  \frac{\|\tilde u_{n, \mathrm{NN}} - u_*\|_\infty^2t}{3n} \,.
$$
With $t = C\log{n}$, we have on an event $\cE'_n$ with probability $\ge 1 - n^{-C}$: 
\begin{align*}
    \|\tilde u_{n, \mathrm{NN}} - u_*\|_n^2 & \le \|\tilde u_{n, \mathrm{NN}} - u_*\|_2^2 + \|\tilde u_{n, \mathrm{NN}} - u_*\|_2\|\tilde u_{n, \mathrm{NN}} - u_*\|_\infty\sqrt{\frac{2C\log{n}}{n}} \\
    & \hspace{17em} +  \frac{C\|\tilde u_{n, \mathrm{NN}} - u_*\|_\infty^2\log{n}}{3n} \\
    & \le C_2\left(N^{-\frac{2\beta}{d}} + N^{-\left(\frac{2\beta}{d} - \frac12\right)}\sqrt{\frac{\log{n}}{n}} + N^{-\frac{(2\beta - d)}{d}}\frac{\log{n}}{n}\right) \,.
\end{align*}
Now, by the fact $m > 2\|u_*\|_\infty$, we have $T_m(u_*(x)) = u_*(x)$. Therefore, setting $\tilde u_n = T_m \circ \tilde u_{n, \mathrm{NN}} \in \tilde \cF_{\rm NN}$ and the fact that $T_m$ is $L_T$-Lipschitz yields: 
$$
\|\tilde u_n - u_*\|_n^2 = \|T_m(\tilde u_{n, \mathrm{NN}}) - T_m(u_*)\|_n^2 \le L^2_T\|\tilde u_{n, \mathrm{NN}} - u_*\|_n^2 
$$
Combining this with Equation \eqref{eq:basic_bound_dnn}, we can conclude that on the event $\e_n \cap \e_n'$: 
\begin{align*}
\|\hat u - u_*\|_n^2 & \le C_2\left(N^{-\frac{2\beta}{d}} + N^{-\left(\frac{2\beta}{d} - \frac12\right)}\sqrt{\frac{\log{n}}{n}} + N^{-\frac{(2\beta - d)}{d}}\frac{\log{n}}{n}\right)  \\
& \hspace{9em}+ (4c_5^2 + 4KH + 1)\frac{\mathrm{Pdim}(\cG_N)\log{n}}{n} \,.
\end{align*}
Now, to bound the stochastic error, we need a bound on $\mathrm{Pdim}(\cG_n)$, for which we will use Theorem 8.7 of \cite{anthony1999neural}. 
%First of all, observe that for any $u \in \cF_{\rm NN}$ and a fixed $\tilde u_n \in \cF_{\rm NN}$, the function $u - \tilde u_n$ can be realized by a neural network with width $2N$, depth $L+1$, total number of trainable weights $2W + 2$. 
First, observe that for any $u\in \widetilde{\mathcal F}_{\rm NN}$ and the fixed approximating network $\widetilde u_n\in\widetilde{\mathcal F}_{\rm NN}$, the function $u-\widetilde u_n$ can be realized by a neural network with width $2N$, depth $L+1$, and total number of trainable weights at most $2W+2$. Although the scaling factor in the definition of $\mathcal G_n$ depends on $u$, the class $\mathcal G_n$ is contained in the output-rescaled class
$$
\cG_n \subseteq \cH_n: = \left\{a(u-\widetilde u_n):u\in\widetilde{\mathcal F}_{\rm NN},\ a \ge 0\right\}.
$$
This larger class $\cH_n$ can be realized by a) computing $u$ and $\tilde u_n$ paralley and then subtracting them, and b) adjusting one additional unrestricted scalar weight $a$ at the final layer. 
Therefore, the functions in $\cH_n$ belongs the collection of neural networks with peicewise polynomial activation function with atmost $2W + 3$ active weights and as before, the depth is atmost $L+1$. Therefore, its pseudo-dimension is still of order $O(W\log W)$ for fixed depth.
Consequently, $\operatorname{Pdim}(\mathcal G_n) \le C W\log W$. 
% Now, the following transformation: 
% $$
% (u - \tilde u_n) \longrightarrow \frac{u - \tilde u_n}{\|u\|_{W^{\beta, 2}} + \|\tilde u_n\|_{W^{\beta, 2}}}
% $$
% is just rescaling $(u -\tilde u_n)$, and consequently can be realized by a neural network with the same architecture as $u - \tilde u_n$. Therefore, all functions of $\cG_n$ can be realized by a neural network with piecewise polynomial activation functions, width $2N$, depth $L+1$, and total number of trainable weights $2W + 2$. 
As a consequence, by Theorem 8.7 of \cite{anthony1999neural}, we have: 
\begin{align*}
\mathrm{Pdim}(\cG_n) & \le \mathrm{Pdim}(\cH_n) \\
& \le C ((2W+3)(L+1)^2 + (2W+3)L\log{((2W+3)(L+1))}) \\
& \le C_{\rm stoc} W \log{W}
\end{align*}
where the second inequality holds for all large $n$ as we allow $W$ to increase with $n$ while $L$ is held fixed. Now, as per our specification of $\cF_{\rm NN}$ in Assumption \ref{assm:function_class}, we have $W \asymp N$, which implies: 
\begin{align*}
\|\hat u - u_*\|_n^2 & \le C'_a\left(N^{-\frac{2\beta}{d}} + N^{-\left(\frac{2\beta}{d} - \frac12\right)}\sqrt{\frac{\log{n}}{n}} + N^{-\frac{(2\beta - d)}{d}}\frac{\log{n}}{n}\right)  \\
& \hspace{9em}+  c_{1, s}(4c_5^2 + 4KH + 1)\frac{N \log{(c_{2, s} N)}\log{n}}{n}\,.
\end{align*}
Therefore, by the choice $N = (n/\log{n})^{d/(2\beta + d)}$, as specified in Assumption \ref{assm:function_class}, we obtain on $\e_n \cap \e_n'$: 
$$
\|\hat u - u_*\|_n^2 \le C_a'\left(\left(\frac{n}{\log{n}}\right)^{-\frac{2\beta}{2\beta + d}} + \left(\frac{n}{\log{n}}\right)^{-\frac{2\beta}{d}} \right) + C\log{n} \left(\frac{n}{\log{n}}\right)^{-\frac{2\beta}{2\beta + d}} \,.
$$
Now, using the fact that $2\beta/(2\beta + d) < 2\beta/d$, 
% $$
% \frac12 + \frac{2\beta - d}{2\beta + d} > \frac{2\beta}{2\beta + d} \,\,, \qquad 1 + \frac{2(\beta - d)}{2\beta + d} > \frac{2\beta}{2\beta + d} \,\,,
% $$
we conclude: 
$$
\|\hat u - u_*\|_n^2 \le C'\log{n} \left(\frac{n}{\log{n}}\right)^{-\frac{2\beta}{2\beta + d}}
$$
Next we show that our choice of $\lambda_n$ and $N$, as specified in Theorem \ref{thm:deep_sobolev} and Assumption \ref{assm:function_class} respectively, automatically implies $\|\hat u\|_{W^{\beta, 2}} \le H_1$ for some constant $H_1$. Starting from Equation \eqref{eq:mb_2}, applying the inequality $2ab \le (a^2 + b^2)$ as before, we obtain: 
\begin{align*}
    \frac12\|\hat u - u_*\|_n^2 + \lambda_n \|\hat u\|^2_{W^{\beta, 2}} & \le \frac{3}{2}\|\tilde u_n - u_*\|_n^2 + 2c_5^2 \frac{\mathrm{Pdim}(\cG_N)\log{n}}{n} \\
    & \qquad + \sqrt{2}c_5\left(\|\hat u\|_{W^{\beta, 2}} + \|\tilde u_n\|_{W^{\beta, 2}}\right) \frac{1}{n}\sqrt{\frac{\mathrm{Pdim}(\cG_N)\log{n}}{n}} \\
    & \qquad + \lambda_n \|\tilde u_n\|^2_{W^{\beta, 2}} \,.
\end{align*}
Now, by our choice of $N \asymp (n/\log{n})^{d/(2\beta + d)}$ and if we set 
\begin{equation}
\label{eq:choice_lambda}
\lambda_n = K \frac{\mathrm{Pdim}(\cG_N)\log{n}}{n}, 
\end{equation}
the we have: 
$$
\frac{3}{2}\|\tilde u_n - u_*\|_n^2 + 2c_5^2 \frac{\mathrm{Pdim}(\cG_N)\log{n}}{n} \le K_1 \lambda_n
$$
for some constant $K_1$ for all large $n$. Therefore, we have: 
\begin{align*}
    \lambda_n \|\hat u\|^2_{W^{\beta, 2}} \le \frac12\|\hat u - u_*\|_n^2 + \lambda_n \|\hat u\|^2_{W^{\beta, 2}} & \le K_1 \lambda_n +  \sqrt{2}c_5\left(\|\hat u\|_{W^{\beta, 2}} + \|\tilde u_n\|_{W^{\beta, 2}}\right)\frac{\sqrt{\lambda_n}}{n\sqrt{K}}\\
    & + \lambda_n \|\tilde u_n\|^2_{W^{\beta, 2}} \,.
\end{align*}
As a consequence, we have: 
\begin{align*}
\|\hat u\|^2_{W^{\beta, 2}} & \le K_1 + \frac{\sqrt{2}c_5}{n\sqrt{K\lambda_n}}\|\hat u\|_{W^{\beta, 2}} + \frac{\sqrt{2}c_5}{n\sqrt{K\lambda_n}}\|\tilde u_n\|_{W^{\beta, 2}} + \|\tilde u_n\|^2_{W^{\beta, 2}} \\
& \le K_1 + \frac{\|\hat u\|_{W^{\beta, 2}}^2}{4} + \frac{4c_5^2}{n^2 K \lambda_n} + \frac{5\|\tilde u_n\|_{W^{\beta, 2}}^2}{4} \hspace{.1in} \left[\because ab \le a^2 + \frac{b^2}{4}\right]
\end{align*}
which further implies: 
$$
\|\hat u\|^2_{W^{\beta, 2}} \le \frac{4}{3}\left(K_1 + \frac{4c_5^2}{n^2 K \lambda_n}\right) + \frac{5\|\tilde u_n\|_{W^{\beta, 2}}^2}{3} \le \frac{4}{3}\left(K_1 + \frac{4c_5^2}{n^2 K \lambda_n}\right) + \frac{5}{4}c^2_* 
$$
Furthermore, by the definition of $\lambda_n$, we have:  
$$
\lambda_n \asymp \left(\frac{n}{\log{n}}\right)^{-\frac{2\beta}{2\beta + d}}\log{n} \implies n^2 \lambda_n \to \infty \,.
$$
As a consequence, we have for all large $n$: 
$$
\frac{4c_5^2}{n^2 K \lambda_n} \le 1\,,
$$
which yields: 
\begin{equation}
\label{eq:def_H1}
\|\hat u\|^2_{W^{\beta, 2}} \le \frac{4}{3}\left(K_1 + 1\right) +\frac{5}{4}c^2_* := H_1^2 \,.
\end{equation}

\noindent
\underline{\bf Going from empirical norm to population norm:}
Our next goal is to derive a bound on the population norm $\|\hat u-u_*\|_{L_2(P_X)}$. Towards that end, our main tool is Theorem 14.1 of \cite{wainwright2019high}.  However, for our purpose, it is better to use the more tuned version, i.e., Lemma 3 of \cite{fan2024factor}, which implies:
$$
\left|\|\hat u - u_*\|_n^2 - \|\hat u - u_*\|_2^2\right| \le \frac12\left(\|\hat u - u_*\|_2^2 + C_{st} \frac{\mathrm{Pdim}(\tilde \cF_{\rm NN} - u_*) \log{n}}{n}\right)
$$
on an event $\e_{n, 1}$ with probability $\ge 1 - e^{-c_1 \mathrm{Pdim}(\tilde \cF_{\rm NN}) \log{n}}$, $\tilde \cF_{\rm NN} - u_* = \{u - u_*: u \in \tilde \cF_{\rm NN}\}$.  As shifting by a constant function does not change the pseudo dimension, we have $\mathrm{Pdim}(\tilde \cF_{\rm NN} - u_*) = \mathrm{Pdim}(\tilde \cF_{\rm NN})$. 
Therefore, on the event $\e_n \cap \e_{n, 1} \cap \e_n'$, we have: 
$$
\|\hat u - u_*\|_2^2 \le 2\|\hat u - u_*\|_n^2 + C_{st} \frac{\mathrm{Pdim}(\tilde \cF_{\rm NN}) \log{n}}{n} \lesssim \left(\frac{n}{\log{n}}\right)^{-\frac{2\beta}{2 \beta + d}} \log{n} \,.
$$

\noindent
\underline{\bf From function estimation to derivative estimation.}
So far we have established that on the event $\e_n \cap \e_n' \cap \e_{n, 1}$, the estimator $\hat u$ satisfies: 
$$
\|\hat u - u_*\|_2^2 \le C\log{n} \left(\frac{n}{\log{n}}\right)^{-\frac{2\beta}{2\beta + d}} \,. 
$$
and $\|\hat u\|_{W^{\beta, 2}} \le H_1$. Our next goal is to derive the rate of estimation of $\|\partial^{\bk}\tilde u_n - \partial^{\bk} u_*\|_2^2$ for $|\bk| \le \beta$. 
Applying the special case
\eqref{eq:GN-special} of Theorem~\ref{thm:interpolation} with $p_0=p_1=r=2, j = |\bk|$ and $\theta = |\bk|/\beta$ we get
\[
\|\partial^\bk(\hat u-u^\star)\|_2
\le
C
\|\hat u-u^\star\|_{W^{\beta,2}}^{|\bk|/\beta}
\|\hat u-u^\star\|_2^{1-|\bk|/\beta}.
\]
Therefore,
\begin{align}
\label{eq:derivative-rate-corrected}
\|\partial^\bk(\hat u-u^\star)\|_{L^2(P_X)}
&=
\|\partial^\bk(\hat u-u^\star)\|_{L^2([0,1]^d)}
\notag \\
&\le
C
\left\{
\log n
\left(\frac{n}{\log n}\right)^{-2\beta/(2\beta+d)}
\right\}^{(1-|\bk|/\beta)/2}
\notag \\
&\le
C
(\log n)^{(1-|\bk|/\beta)/2}
\left(\frac{n}{\log n}\right)^{-(\beta-|\bk|)/(2\beta+d)}.
\end{align}
This completes the proof. 
%\DM{Check whether the approximation error in Belomestny's paper can be extended to real $\beta$.} 
\\\\
\noindent
\underline{\bf Conclusions hold for the untruncated neural-network part as well:}
So far, we have established that, for every multi-index $k$ satisfying
$0\le |\bk|<\beta$, the truncated estimator $\hat u\in\widetilde{\mathcal F}_{\rm NN}$
satisfies, with high probability,
\[
\|\partial^\bk(\hat u-u^\star)\|_{L^2(P_X)}
\le
C_k
(\log n)^{(1-|k|/\beta)/2}
\left(\frac{n}{\log n}\right)^{-(\beta-|k|)/(2\beta+d)}.
\]
Since $\hat u\in\widetilde{\mathcal F}_{\rm NN}$, there exists
$\hat u_{\rm NN}\in\mathcal F_{\rm NN}$ such that $\hat u=T_m\circ \hat u_{\rm NN}$. We now show that, on the same high-probability event, the truncation is inactive for all sufficiently
large $n$. By Theorem \ref{thm:interpolation}, applied to
$f=\hat u-u^\star$, we have
\[
\begin{aligned}
\|\hat u-u^\star\|_\infty
&\le
C
\|\hat u-u^\star\|_{W^{\beta,2}}^{d/(2\beta)}
\|\hat u-u^\star\|_2^{1-d/(2\beta)}
\\
&\le
C
\|\hat u-u^\star\|_2^{1-d/(2\beta)}
\\
&\le
C
\left\{
\log n
\left(\frac{n}{\log n}\right)^{-2\beta/(2\beta+d)}
\right\}^{(1-d/(2\beta))/2}
=o(1).
\end{aligned}
\]
Here we used the bound
$
\|\hat u-u^\star\|_{W^{\beta,2}}
\le
H_1+\|u^\star\|_{W^{\beta,2}}
\le C,
$
and the fact that $\beta>d/2$. Since $m>2\|u^\star\|_\infty$, choose $\delta>0$ such that $m-\delta>2\|u^\star\|_\infty$. 
Then, for all sufficiently large $n$ on the high-probability event,
\[
\|\hat u\|_\infty
\le
\|u^\star\|_\infty+\|\hat u-u^\star\|_\infty
<
m-\delta.
\]
By the construction of the smooth truncation map $T_m$, we have $|T_m(t)|<m-\delta$ implies $T_m(t)=t$. Consequently,
\[
\hat u(x)=T_m(\hat u_{\rm NN}(x))=\hat u_{\rm NN}(x),
\quad x\in[0,1]^d,
\]
for all sufficiently large $n$ on the same event. Therefore, the derivative bounds proved above
for $\hat u$ also hold for the underlying untruncated neural-network function
$\hat u_{\rm NN}$:
\[
\|\partial^\bk(\hat u_{\rm NN}-u^\star)\|_{L^2(P_X)}
\le
C_k
(\log n)^{(1-|k|/\beta)/2}
\left(\frac{n}{\log n}\right)^{-(\beta-|\bk|)/(2\beta+d)},
\qquad 0\le |\bk|<\beta.
\]
This completes the proof.

\subsection{Proof of Corollary \ref{cor:deep_sobolev_PINN}}
The proof is identical to that of Theorem~\ref{thm:deep_sobolev}, except
for the appearance of the additional nonnegative penalty
\(\widetilde\lambda_n\mathcal R(u)\). Let \(\widetilde u_n\in
\widetilde{\mathcal F}_{\rm NN}\) be the approximating network used in
the proof of Theorem~\ref{thm:deep_sobolev}. By the optimality of
\(\hat u_{\rm PINN}\),
\begin{align*}
&\frac1n\sum_{i=1}^n
\{Y_i-\hat u_{\rm PINN}(X_i)\}^2
+
\lambda_n\|\hat u_{\rm PINN}\|_{W^{\beta,2}}^2
+
\widetilde\lambda_n\mathcal R(\hat u_{\rm PINN})
\\
&\qquad \qquad \qquad \le
\frac1n\sum_{i=1}^n
\{Y_i-\widetilde u_n(X_i)\}^2
+
\lambda_n\|\widetilde u_n\|_{W^{\beta,2}}^2
+
\widetilde\lambda_n\mathcal R(\widetilde u_n).
\end{align*}
Since \(\mathcal R(\hat u_{\rm PINN})\ge 0\), the leftmost PINN penalty
can be dropped. Thus, the basic inequality used in
Theorem~\ref{thm:deep_sobolev} remains unchanged, except for the
additional term
\(\widetilde\lambda_n\mathcal R(\widetilde u_n)\) on the right hand side.
It remains to be bound this term. Since
\[
\mathcal R(u)=\min_{\theta\in\Theta} \ \cR(\theta,u),
\]
and \(\cR(\theta_\star,u_\star)=0\), we have: 
\begin{align*}
    \mathcal R(\widetilde u_n) & \le \cR(\theta_\star,\widetilde u_n) \\
    & = \int_{\Omega}\left\{K_0(\tilde u_n)(x) + \Psi_{\theta_*}(x,K_1 \tilde u_n(x)) \right\}^2\omega(x)\rmd x\\
&= \int_{\Omega}\left\{K_0(\tilde u_n-u_\star)(x)  + \Psi_{\theta_\star}(x,K_1 \tilde u_n(x)) - \Psi_{\theta_\star}(x,K_1 u_\star(x))\right\}^2\omega(x) \ dx \\
& \le 2\|\omega\|_\infty \|\tilde u_n-u_\star\|_{W^{\tau, 2}}^2 + 2\int_{\Omega} \left(\Psi_{\theta_\star}(x,K_1 \tilde u_n(x)) - \Psi_{\theta_\star}(x,K_1 u_\star(x))\right)^2 \ \omega(x)\rmd x\\
& \le 2\|\omega\|_\infty \|\tilde u_n-u_\star\|_{W^{\tau, 2}}^2 + 2L^2_{\Psi_{\theta_*}} \int_\Omega \|K_1 (u_\star - \tilde u_n)(x)\|_2^2 \ \omega(x) \ dx \\
& \le C_1 \|\tilde u_n-u_\star\|_{W^{\tau, 2}}^2 \,.
\end{align*}
Here we have used the fact that by Assumption \ref{assm:psi_smoothness}, $\Psi(x,z,\theta^\star)$ is  Lipschitz in $z$, uniformly over $x\in\overline\Omega$ and over bounded $z$-sets. Since $K_1\widetilde u_n$ and $K_1u^\star$ are uniformly bounded with respect to $L_\infty$ norm using Sobolev embedding theorem (Theorem \ref{thm:sob_emb}), we have: 
\[
\|\Psi_{\theta^\star}(\cdot,K_1\widetilde u_n) - \Psi_{\theta^\star}(\cdot,K_1u^\star)\|_2 \le L_{\Psi_{\theta_*}} \|K_1(\widetilde u_n-u^\star)\|_2.
\]
% Using the definition of \(R\), the PDE satisfied by \(u_\star\), and the
% local Lipschitz property of \(\Psi_{\theta_\star}\), we obtain
% \[
% \begin{aligned}
% \mathcal R(\widetilde u_n)
% &\le
% C\int_\Omega
% \left[
% K_0(\widetilde u_n-u_\star)(x)
% +
% \Psi_{\theta_\star}(K_1\widetilde u_n(x))
% -
% \Psi_{\theta_\star}(K_1u_\star(x))
% \right]^2
% \omega(x)\,dx
% \\
% &\le
% C\|\widetilde u_n-u_\star\|_{W^{\tau,2}}^2.
% \end{aligned}
% \]
By the Sobolev approximation result used in
Theorem~\ref{thm:deep_sobolev},
\[
\|\widetilde u_n-u_\star\|_{W^{\tau,2}}^2
\lesssim
N^{-2(\beta-\tau)/d}.
\]
Therefore,
\[
\widetilde\lambda_n\mathcal R(\widetilde u_n)
\lesssim
\widetilde\lambda_n N^{-2(\beta-\tau)/d}
\lesssim
\lambda_n.
\]
This is of the same order as the terms already present in the proof of
Theorem~\ref{thm:deep_sobolev}. Hence, all subsequent steps of that proof
go through unchanged, possibly with different constants. This yields the
claimed Sobolev-norm bound and the derivative estimation rates for
\(\hat u_{\rm PINN}\).

\subsection{Proof of Lemma \ref{lem:consistency_PI}}

\begin{proof}
    Our argument largely follows the idea of the argmin continuous mapping theorem for the M-estimation problem (see Theorem 3.2.2 of the Second edition of \cite{van1996weak}). Fix some $\delta > 0$. By Assumption \ref{assm:curvature_pde_risk}, we have: 
    \begin{equation}
    \label{eq:basic_curvature}
    \inf_{\theta: \|\theta - \theta_*\|_2 > \delta} \ \cR(\theta, u_*) = c(\delta) > 0, \qquad \inf_{\theta: \|\theta - \theta_*\|_2 \le \delta_0} \lambda_{\min}\left(\nabla^2_\theta \cR(\theta, u_*)\right) \ge c_0 \,.
    \end{equation}
    Furthermore, by Theorem \ref{thm:sob_emb}, we know that for $u \in W^{\beta, 2}(\Omega, L)$, $u \in C^{2\tau}(\Omega, L')$ as long as $\beta - 2\tau > d/2$, which is satisfied as per Assumption \ref{assm:dgp}. As a consequence, $|K_0 u(x)| \le L'$ and $\|K_1 u(x)\|_2 \le L'\sqrt{p}$ for all $x \in (0, 1)^d \equiv \Omega$. As the function $\Psi_{\theta}(x,z)$ is continuous, it is uniformly bounded on the compact set (Assumption \ref{assm:psi_smoothness}). Therefore, 
    $$
    \left|K_0 u(x) + \Psi_\theta (x,K_1 u(x))\right| \le L'', \quad \forall \ x \in \Omega. 
    $$
    Now, given any two $u, u' \in W^{\beta, 2}(\Omega, L)$, we have: 
    % On the other hand, given two function $u, u' \in W^{\beta, 2}(\Omega, L)$, we have: 
    % \begin{align*}
    % |K_0(u - u')(x)| \vee \|K_1(u - u')(x)\|_2 & \le \sqrt{p} \|u - u'\|_{C^{\tau}(\Omega, L')} \\
    % & \le C \sqrt{p} \|\hat u - u'\|_2^{} \\
    % & \le C \sqrt{p} \|u - u'\|^{\frac{\tau + \frac{d}{2}}{\beta}}_{W^{\beta, 2}(\Omega, L)}\|u - u'\|_2^{1- \frac{\tau + \frac{d}{2}}{\beta}}\,,
    % \end{align*}
    % where the last inequality follows from Theorem \ref{thm:interpolation}. This immediately implies: 
    \allowdisplaybreaks
    \begin{align*}
        & |\cR(\theta, u) - \cR(\theta, u')| \\
        & \le \frac12 \int_\Omega\left|\left(K_0 u(x) + \Psi_\theta (x,K_1 u(x))\right)^2 - \left(K_0 u'(x) + \Psi_\theta (x,K_1 u'(x))\right)^2\right| \ dx  \\
        & \le \frac12 \int_\Omega \left|K_0 (u-u')(x) + \Psi_\theta (K_1 u(x)) - \Psi_\theta(K_1 u'(x))\right| \\
        & \qquad \qquad \qquad \qquad \times \left|K_0 (u+u')(x) + \Psi_\theta (x,K_1 u(x)) + \Psi_\theta(x,K_1 u'(x))\right| \ dx \\
        & \le L'' \int_\Omega \left|K_0 (u-u')(x) + \Psi_\theta (x,K_1 u(x)) - \Psi_\theta(x,K_1 u'(x))\right| \ dx \\
        & \le L''\left[\|K_0(u - u')\|_{L_2(P_X)} + \| \Psi_\theta (x,K_1 u(x)) - \Psi_\theta(x,K_1 u'(x))\|_{L_2(P_X)}\right] \\
        & \le L''\left[\|K_0(u - u')\|_{L_2(P_X)} \right. \\
        & \qquad \left. + \sqrt{p}\sup_{(x,z):\; x\in\Omega,\|z\| \le L'\sqrt{p}}\;\|\nabla_z \Psi_\theta(x,z)\|_2\ \max_{1 \le j \le p} \|K_{1, j}(u - u')\|_{L_2(P_X)}\right]
    \end{align*}
If we set $u = u_*$ and $u' = \hat u$, where $\hat u$ is obtained as in Equation \eqref{eq:u_est_NN}, then, by Theorem \ref{thm:deep_sobolev}, we have: 
$$
\|K_0(\hat u - u_*)\|_2 \vee \max_{1 \le j \le p}\|K_{1, j}(\hat u - u_*)\|_{L_2(P_X)} \ = O_p\left(n^{-\frac{\beta - \tau}{2\beta + d}}(\log{n})^{a(\beta, d, \tau)}\right) \,.
$$
Furthemore, as $\Theta$ is compact, by continuity of $\nabla_\theta \Psi_\theta(x)$ with respect to $\theta$, we have: 
$$
\sup_{\theta \in \Theta}\;\sup_{(x,z):\; x\in\Omega,\|z\| \le L'\sqrt{p}}\;\|\nabla_z \Psi_\theta(x,z)\|_2 \le C^\Psi_1 < \infty\,.
$$
Consequently, we have 
$$
\sup_{\theta \in \Theta} \left|\cR(\theta, \hat u) - \cR(\theta, u_*)\right| = O_p\left(n^{-\frac{\beta - \tau}{2\beta + d}}(\log{n})^{a(\beta, d, \tau)}\right) = o_p(1) \,.
$$
Now, recall that, by definition, $\hat \theta^{\rm PI} = \argmin_{\theta \in \Theta} \cR(\theta, \hat u)$. We have: 
\begin{align*}
    & \bbP\left(\|\hat \theta^{\rm PI} - \theta_*\|_2 > \delta\right) \\
    & \le \bbP\left(\inf_{\theta \in B_{\delta}(\theta_*)} \cR(\theta, \hat u) >\inf_{\theta \in B_\delta(\theta_*)^c} \ \cR(\theta, \hat u)\right) \\
    & = \bbP\left(\inf_{\theta \in B_{\delta}(\theta_*)} \left\{\cR(\theta, \hat u) - \cR(\theta, u_*) + \cR(\theta, u_*)\right\} > \right. \\
    & \hspace{9em}\left. \inf_{\theta \in B_\delta(\theta_*)^c} \ \left\{\cR(\theta, \hat u) - \cR(\theta, u_*) + \cR(\theta, u_*)\right\}\right) \\
    & \le \bbP\left(\inf_{\theta \in B_{\delta}(\theta_*)} \cR(\theta, u_*) + \sup_{\theta \in \Theta} \left|\cR(\theta, \hat u) - \cR(\theta, u_*)\right| > \right. \\
    & \hspace{9em} \left. \inf_{\theta \in B_\delta(\theta_*)^c} \ \cR(\theta, u_*) - \sup_{\theta \in \Theta} \left|\cR(\theta, \hat u) - \cR(\theta, u_*)\right| \right) \\
    & =\bbP\left(2\sup_{\theta \in \Theta} \left|\cR(\theta, \hat u) - \cR(\theta, u_*)\right| > \inf_{\theta \in B_\delta(\theta_*)^c} \ \cR(\theta, u_*) \right) \\
    & \le \bbP\left(\sup_{\theta \in \Theta} \left|\cR(\theta, \hat u) - \cR(\theta, u_*)\right| > \frac{c(\delta)}{2}\right) \longrightarrow 0 \,.
\end{align*}
Therefore, $\hat \theta^{\rm PI}$ is a consistent estimator of $\theta_*$. Furthermore, define a sequence $r_n$ such that: 
$$
r_n = o\left(n^{\frac{\beta - \tau}{2\beta + d}}(\log{n})^{-a(\beta, d, \tau)}\right)
$$
The exact same calculation, along with the second fact of Equation \eqref{eq:basic_curvature}, yields $\|\hat \theta^{\rm PI} - \theta_*\|_2 = o_p(r_n^{-1/2})$. This completes the proof. 
\end{proof}

\section{Supporting technical results for the Proof of Theorem \ref{thm:main_freq}}\label{thm:main_freq:supp:mat}

\begin{lemma}
    \label{lem:min_interior}
    $\Theta_n = \{\theta: \|\theta - \hat \theta^{\rm PI}\|_2 \le \delta_n\}$ for any sequence $\{\delta_n\}\downarrow 0$ satisfying: 
    $$
    \delta_n \gg \frac{1}{\sqrt{n}} + \log{n}\left(
    \frac{n}{\log{n}}\right)^{-\frac{\beta}{2\beta + d}} \,.
    $$
    and $\bbP(\theta_* \in \Theta_n) \to 1$ as $n \uparrow \infty$ (e.g., one may take $\delta_n = r_n^{-1/2}$ as in Lemma \ref{lem:consistency_PI}). 
    Then under Assumption \ref{assm:dgp}-\ref{assm:w} and $\Theta$ is a compact subset of $\reals^q$, the minimizer of $\|\sS^{\rm db}_n(\theta, \hat u)\|_2^2$ on $\Theta_n$ lies in the interior of $\Theta_n$ with probability approaching to $1$. 
\end{lemma}

\begin{proof}
    The goal here is to show that the minimizer of $\|\sS_n^{\rm db}(\theta, \hat u)\|_2^2$ lies on the interior of $\Theta_n$. Since $\theta^\star$ lies in the interior of $\Theta$, there exists $\eta_\theta > 0$ such that $\bar B(\theta^\star,2\eta_\Theta)\subset \Theta$. By Lemma~\ref{lem:consistency_PI}, $\hat\theta^{PI}\to \theta^\star$ in probability, and by assumption $\delta_n\downarrow 0$. Hence the event
    $$
    \cE_{\Theta,n} := \left\{\|\hat\theta^{PI}-\theta^\star\|_2+\delta_n < \eta_\Theta\right\} \cap \left\{\|\hat u\|_{W^{\beta, 2}} \le H_1\right\}
    $$
    for some constant $H_1$
    satisfies $\bbP(\cE_{\Theta,n})\to 1$ (where the bound on $\hat u$ follows from Theorem \ref{thm:deep_sobolev}) and on \(\cE_{\Theta,n}\), $\hat\theta^{PI}$ also lies in the interior of $\Theta$. Since \(\hat\theta^{PI}\) minimizes the differentiable map \(\theta\mapsto R(\theta,\hat u)\) over \(\Theta\), we have from the first order condition $\nabla_\theta R(\hat\theta^{PI},\hat u)=0$ on \(\mathcal E_{\Theta,n}\).
    We now prove the conclusion of Lemma \ref{lem:min_interior} on this event. More precisely, we will show that with probability is going to $1$,
    $$
    \|\sS_n^{\rm db}(\hat \theta^{\rm PI}, \hat u)\|_2^2 < \|\sS_n^{\rm db}(\hat \theta^{\rm PI} + v\delta_n, \hat u)\|_2^2, \quad \text{ for any }v \in S^{q -1} \,.
    $$
    Now, by definition of the debiased score function: 
    \begin{align*}
    \sS_n^{\rm db}(\hat \theta^{\rm PI}, \hat u) & = \underbrace{\nabla_\theta \cR(\hat \theta^{\rm PI}, \hat u)}_{=0} + \frac1n \sum_{i = 1}^n (y_i - \hat u(x_i))\sG_{\hat \theta^{\rm PI}, \hat u}(x_i) \\
    & = \frac1n \sum_{i = 1}^n (y_i - \hat u(x_i))\sG_{\hat \theta^{\rm PI}, \hat u}(x_i)
    % & = \frac1n \sum_{i = 1}^n \eps_i \sG_{\hat \theta^{\rm PI}, \hat u}(x_i) + \frac1n \sum_{i = 1}^n (u_*(x_i) - \hat u(x_i))\sG_{\hat \theta^{\rm PI}, \hat u}(x_i) \\
    % & = \frac1n \sum_{i = 1}^n \eps_i \sG_{\theta_*, u_*}(x_i) +  \frac1n \sum_{i = 1}^n \eps_i (\sG_{\hat \theta^{\rm PI}, \hat u}(x_i)-\sG_{\theta_*, u_*}(x_i))  + \frac1n \sum_{i = 1}^n (u_*(x_i) - \hat u(x_i))\sG_{\hat \theta^{\rm PI}, \hat u}(x_i)
    \end{align*}
    On the other hand: 
    \allowdisplaybreaks
    \begin{align*}
        & \sS_n^{\rm db}(\hat \theta^{\rm PI} + \delta_n v, \hat u) \\
        & = \nabla_\theta \cR(\hat \theta^{\rm PI} + \delta_n v, \hat u) + \frac1n \sum_{i = 1}^n (y_i - \hat u(x_i))\sG_{\hat \theta^{\rm PI} + \delta_n v, \hat u}(x_i) \\
        & = \nabla_\theta \cR(\hat \theta^{\rm PI} , \hat u) + \delta_n \int_0^1 \nabla_\theta^2 \cR(\hat \theta_{\rm PI} + t\delta_n v, \hat u)v \ dt + \frac1n \sum_{i = 1}^n (y_i - \hat u(x_i))\sG_{\hat \theta^{\rm PI} + \delta_n v, \hat u}(x_i) \\
        % & = \delta_n \nabla^2_\theta \cR(\bar \theta, \hat u)v + \frac1n \sum_{i = 1}^n (y_i - \hat u(x_i))\sG_{\hat \theta^{\rm PI} + \delta_n v, \hat u} \\
        & = \delta_n \int_0^1 \nabla_\theta^2 \cR(\hat \theta_{\rm PI} + t\delta_n v, \hat u)v \ dt + \frac1n \sum_{i = 1}^n (y_i - \hat u(x_i))(\sG_{\hat \theta^{\rm PI}+ \delta_n v, \hat u} - \sG_{\hat \theta^{\rm PI}, \hat u}(x_i))  \\
        & \hspace{25em} + \frac1n \sum_{i = 1}^n (y_i - \hat u(x_i))\sG_{\hat \theta^{\rm PI}, \hat u} \\
        & = \underbrace{\delta_n H_{n, v}(\hat \theta_{\rm PI}, \hat u)v + \frac1n \sum_{i = 1}^n (y_i - \hat u(x_i))(\sG_{\hat \theta^{\rm PI}+ \delta_n v, \hat u} - \sG_{\hat \theta^{\rm PI}, \hat u}(x_i))}_{:= \mathfrak{X}} \\
        & \hspace{25em}+ \frac1n \sum_{i = 1}^n (y_i - \hat u(x_i))\sG_{\hat \theta^{\rm PI}, \hat u}
    \end{align*}
Therefore, we have: 
$$
  \sS^{\rm db}(\hat \theta^{\rm PI} + v \delta_n, \hat u) = \mathfrak{X} +   \sS^{\rm db}(\hat \theta^{\rm PI}, \hat u) 
$$
As $\|y\| > 2\|x\| \implies \|x + y\|_2^2 > \|x\|_2^2$, all we need to show is: 
$$
\|\mathfrak{X}\|_2 > 2\left\|\frac1n \sum_{i = 1}^n (y_i - \hat u(x_i))\sG_{\hat \theta^{\rm PI}, \hat u}(x_i)\right\|_2
$$
A sufficient condition for that is: 
\begin{align}
\label{eq:lower_bound_1}
   \delta_n \|H_{n, v}(\hat \theta_{\rm PI}, \hat u)v\|_2 & > \left\|\frac1n \sum_{i = 1}^n (y_i - \hat u(x_i))(\sG_{\hat \theta^{\rm PI}+ \delta_n v, \hat u} - \sG_{\hat \theta^{\rm PI}, \hat u}(x_i))\right\|_2 \notag \\
   & \hspace{10em} + 2\left\|\frac1n \sum_{i = 1}^n (y_i - \hat u(x_i))\sG_{\hat \theta^{\rm PI}, \hat u}(x_i)\right\|_2 \,. 
\end{align}
In the rest of the proof, we show that the RHS is $o_p(\delta_n)$. For the second term of the RHS, we have: 
\begin{align*}
\left|\frac1n \sum_{i = 1}^n (y_i - \hat u(x_i))\sG_{\hat \theta^{\rm PI}, \hat u}(x_i)\right| & \le \left|\frac1n \sum_{i = 1}^n \eps_i\sG_{\hat \theta^{\rm PI}, \hat u}(x_i)\right| + \left|\frac1n \sum_{i = 1}^n (u_*(x_i) - \hat u(x_i))\sG_{\hat \theta^{\rm PI}, \hat u}(x_i)\right| \\
& \le \sup_{\substack{\theta \in \Theta \\ u \in W^{\beta, 2}(\Omega, L)}} \left|\frac1n \sum_{i = 1}^n \eps_i\sG_{\theta, u}(x_i)\right| \\
& \qquad \qquad + \sup_{\theta \in \Theta} \left|\frac1n \sum_{i = 1}^n (u_*(x_i) - \hat u(x_i))\sG_{\theta, \hat u}(x_i)\right|
\end{align*}
Given that the function class $\{\sG_{\theta, u}: \theta \in \Theta, u \in W^{\beta, 2}(\Omega,L)\}$ (where we can take $L = \|u_*\|_{W^{\beta, 2}} + H_1$) is Donsker function class as soon as $d/(\beta \alpha_*) < 2$, by the exact same empirical process calculation that used to establish Equation \eqref{eq:maximal_eps_global}, we have: 
$$
\sup_{\substack{\theta \in \Theta \\ u \in W^{\beta, 2}(\Omega, L)}} \left|\frac1n \sum_{i = 1}^n \eps_i\sG_{\theta, u}(x_i)\right| = O_p\left(\frac{1}{\sqrt{n}}\right) = o_p(\delta_n)\,.
$$
For the other term, we have by 
\begin{align*}
& \sup_{\theta \in \Theta} \left|\frac1n \sum_{i = 1}^n (u_*(x_i) - \hat u(x_i))\sG_{\theta, \hat u}(x_i)\right| \\
& \le \sup_{\substack{\theta \in \Theta \\ u \in W^{\beta, 2}(\Omega, L)}}|(\bbP_n - P)(u_*(x_i) - u(x_i))\sG_{\theta, u}(x_i)| + \sup_{\theta \in \Theta}|\bP(u_*(x_i) - \hat u(x_i))\sG_{\theta, \hat u}(x_i)|
\end{align*}
The first term on the right hand side is $O_p(n^{-1/2}) = o_p(\delta_n)$, which again follows from the fact that the function class $(u_* - u)\sG_{\theta, u}$ is Donsker over $\Theta \otimes W^{\beta, 2}(\Omega, L)$. For the second term on the RHS, a simple application of the Cauchy-Schwarz inequality yields: 
\begin{align*}
\sup_{\theta \in \Theta}|\bP(u_*(x_i) - \hat u(x_i))\sG_{\theta, \hat u}(x_i)| & \le \|\hat u - u_*\|_2 \sup_{\substack{\theta \in \Theta \\ u \in W^{\beta, 2}(\Omega, L)}}\|\sG_{\theta, u}\|_2 \\
& \lesssim \|\hat u - u_*\|_2 \\
& = O_p\left(\log{n}\left(
\frac{n}{\log{n}}\right)^{-\frac{\beta}{2\beta + d}}\right) = o_p(\delta_n) \,.
\end{align*}
Now, for the first term of the RHS of Equation \eqref{eq:lower_bound_1}, we have: 
\begin{align*}
& \left\|\frac1n \sum_{i = 1}^n (y_i - \hat u(x_i))(\sG_{\hat \theta^{\rm PI}+ \delta_n v, \hat u}(x_i) - \sG_{\hat \theta^{\rm PI}, \hat u}(x_i))\right\|_2 \\
& \le \left\|\frac1n \sum_{i = 1}^n (y_i - \hat u(x_i))\sG_{\hat \theta^{\rm PI}+ \delta_n v, \hat u}(x_i)\right\|_2  + \left\|\frac1n \sum_{i = 1}^n (y_i - \hat u(x_i))\sG_{\hat \theta^{\rm PI}, \hat u}(x_i)\right\|_2 
\end{align*}
We have already shown that the second term of the above equation is $o_p(\delta_n)$. The calculation for the first term is exactly similar by observing that: 
\begin{align*}
    & \left\|\frac1n \sum_{i = 1}^n (y_i - \hat u(x_i))\sG_{\hat \theta^{\rm PI}+ \delta_n v, \hat u}(x_i)\right\|_2 \\
    & = \left\|\frac1n \sum_{i = 1}^n (\eps_i + u_*(x_i) - \hat u(x_i))\sG_{\hat \theta^{\rm PI}+ \delta_n v, \hat u}(x_i)\right\|_2 \\
    & \le \left\|\frac1n \sum_{i = 1}^n \eps_i \sG_{\hat \theta^{\rm PI}+ \delta_n v, \hat u}(x_i)\right\|_2 + \left\|\frac1n \sum_{i = 1}^n (u_*(x_i) - \hat u(x_i))\sG_{\hat \theta^{\rm PI}+ \delta_n v, \hat u}(x_i)\right\|_2 \\
    & \le \sup_{\substack{\theta \in \Theta \\ u \in W^{\beta, 2}(\Omega, L)}} \left\|\frac1n \sum_{i = 1}^n \eps_i \sG_{\theta, u}(x_i)\right\|_2 + \sup_{\substack{\theta \in \Theta \\ u \in W^{\beta, 2}(\Omega, L)}}|(\bbP_n - P)(u_*(x_i) - u(x_i))\sG_{\theta, u}(x_i)| \\
    & \hspace{10em}+ \sup_{\theta \in \Theta}|\bP(u_*(x_i) - \hat u(x_i))\sG_{\theta, \hat u}(x_i)| \,.
\end{align*}
which is exactly what we have shown $o_p(\delta_n)$ before. This completes the proof that the RHS of Equation \eqref{eq:lower_bound_1} is $o_p(\delta_n)$. Now for the LHS of Equation \eqref{eq:lower_bound_1}, we have: 
$$
\delta_n \|\nabla_\theta^2 H_{n, v}(\hat \theta^{\rm PI}, \hat u)v\|_2 \ge \delta_n \lambda_-, \qquad \text{with probability going to 1}\,,
$$
where, $\lambda_-$ is as defined in Assumption \ref{assm:curvature_pde_risk}, as $\hat u$ converges to $u_*$ with respect to $W^{2\tau, 2}$ norm and $\hat \theta_{\rm PI}$ converges to $\theta_*$ with respect to $L_2$ norm. As a consequence, the minimizer of the squared norm of $\sS_n^{\rm db}(\theta, u)$ will be in the interior of $\Theta_n$ with probability approaching to $1$.  
\end{proof}

\begin{lemma}
    \label{lem:zero_root_db}
    Under the same setup as of Lemma \ref{lem:consistency_PI} and \ref{lem:min_interior}, the debiased estimator $\htdb$, as defined in Equation \ref{eq:def_htdb}, satisfies $\hat \sS_n^{\rm db}(\htdb, \hat u) = 0$ with probability going to $1$ as $n \uparrow \infty$. 
\end{lemma}

\begin{proof}
For  notational simplicity, define $J_n(\theta) = \nabla_\theta \sS_n^{\rm db}(\theta, \hat u)$. Note that, $\hat \theta^{\rm db}$ satisfies: 
$$
\hat \theta^{\rm db} = \argmin_{\theta \in \Theta_n} \|\sS_n^{\rm db}(\theta, \hat u)\|_2^2 \implies 2\ J_n(\hat \theta^{\rm db}) \ \sS_n^{\rm db}(\hat \theta^{\rm db}, \hat u) = 0 \,.
$$
Therefore, if we can prove that with probability going to $1$, $\sigma_{\min}(J_n(\hat \theta^{\rm db})) \ge c_-  > 0$, then on this event $\sS_n^{\rm db}(\hat \theta^{\rm db}, \hat u) = 0$ and we are done. Therefore, our goal is to show that with probability going to $1$, $\sigma_{\min}(J_n(\hat \theta^{\rm db})) \ge c_-  > 0$. We start with the following decomposition:  
\begin{align*}
& J_n(\hat \theta^{\rm db}) \\
& = \nabla^2_\theta \cR(\htdb, \hat u) + \frac1n \sum_i (y_i - \hat u(x_i))\nabla_\theta \sG_{\htdb, \hat u}(x_i) \\
& = \nabla^2_\theta \cR(\hat \theta^{\rm PI}, \hat u) + \left(\nabla^2_\theta \cR(\htdb, \hat u) - \nabla^2_\theta \cR(\hat \theta^{\rm PI}, \hat u)\right) + \frac1n \sum_i (y_i - \hat u(x_i))\nabla_\theta \sG_{\htdb, \hat u}(x_i) \,.
\end{align*}
As a consequence for any $v \in \bbS^{q-1}$: 
\begin{align*}
    \|J_n(\hat \theta^{\rm db}) v\|_2 & \ge \|\nabla^2_\theta \cR(\hat \theta^{\rm PI}, \hat u)v\|_2 - \|\nabla^2_\theta \cR(\htdb, \hat u) - \nabla^2_\theta \cR(\hat \theta^{\rm PI}, \hat u)\|_{\rm op} \\
& \qquad \quad \qquad \qquad \qquad \qquad \qquad  - \left\|\frac1n \sum_i (y_i - \hat u(x_i))\nabla_\theta \sG_{\htdb, \hat u}(x_i) \right\|_{\rm op} \,.
\end{align*}
% Therefore: 
% \begin{align*}
% \lambda_{\min}(\nabla_\theta \sS_n^{\rm db}(\theta, \hat u)) & \ge \lambda_{\min}(\nabla^2_\theta \cR(\hat \theta^{\rm PI}, \hat u)) - \|\nabla^2_\theta \cR(\htdb, \hat u) - \nabla^2_\theta \cR(\hat \theta^{\rm PI}, \hat u)\|_{\rm op} \\
% & \qquad \quad \qquad \qquad \qquad \qquad \qquad \qquad - \left\|\frac1n \sum_i (y_i - \hat u(x_i))\nabla_\theta \sG_{\htdb, \hat u}(x_i) \right\|_{\rm op} \,.
% \end{align*}
First, we show that
\[
\left\|
\frac1n\sum_{i=1}^n
\{Y_i-\hat u(X_i)\}
\nabla_\theta \sG_{\hat\theta^{\rm db},\hat u}(X_i)
\right\|_{\mathrm{op}}
=o_p(1).
\]
Towards that end, for \(1\le j,k\le q\), define two function classes: 
\begin{align}
\label{def:H_jkn}
    \mathcal H_{jk,n} & = \left\{x\mapsto \partial_{\theta_k}\sG_{\theta,u,j}(x): \theta\in\Theta,\ u\in\mathcal U_n\right\} 
\end{align}
\begin{align}
    \label{def:A_jkn}
    \mathcal A_{jk,n} & = \left\{x\mapsto (u(x)-u_\star(x))\partial_{\theta_k}\sG_{\theta,u,j}(x): \theta\in\Theta,\ u\in\mathcal U_n\right\}.
\end{align}
where $\cU_n$ is same as defined in Equation \eqref{eq:def_un}. By the definition of \(\mathcal U_n\) (see the proof of Theorem \ref{thm:main_freq}), there exists a deterministic
constant \(C>0\) such that, with probability tending to one, $\hat u  \in \cU_n, \|\hat u-u_\star\|_{L^2(P)}\le \delta_{n,0}$ and $\delta_{n, 0} \downarrow 0$. We will henceforth work on this event. Using \(Y_i=u_\star(X_i)+\epsilon_i\), we have: 
\begin{align*}
& \frac1n\sum_{i=1}^n (Y_i-\hat u(X_i))\ \partial_{\theta_k} \sG_{\htdb, \hat u, j}(X_i) \\
& = \frac1n\sum_{i=1}^n \epsilon_i \  \partial_{\theta_k} \sG_{\htdb, \hat u, j}(X_i) - \frac1n\sum_{i=1}^n (\hat u(X_i) - u_\star(X_i)) \  \partial_{\theta_k} \sG_{\htdb, \hat u, j}(X_i)\\
&=: A_{n,jk} - B_{n,jk}.
\end{align*}
Now, 
$$
|A_{n,jk}| \le \frac1{\sqrt n} \sup_{h\in\mathcal H_{jk,n}} \left|\frac1{\sqrt n}\sum_{i=1}^n \epsilon_i h(X_i)\right|.
$$
By the Donsker property of the class \(\cH_{jk,n}\) (Lemma \ref{lem:donsker-dthetaG}) and Lemma \ref{lem:dudley},
\[
\sup_{h\in\mathcal H_{jk,n}}
\left|
\frac1{\sqrt n}\sum_{i=1}^n \epsilon_i h(X_i)
\right|
=
O_p(1) \implies A_{n,jk}=O_p(n^{-1/2})=o_p(1).
\]
Next let us consider \(B_{n,jk}\). Let $f_{jk,n}(x)
=
\{\hat u(x) - u_\star(x)\}H_{jk,n}(x)$. Then \(f_{jk,n}\in\mathcal A_{jk,n}\). Hence
\[
B_{n,jk} = \bP\left((\hat u - u_\star) \  \partial_{\theta_k} \sG_{\htdb, \hat u, j}\right) +  \left(\bbP_n-\bP\right)\left((\hat u - u_\star) \  \partial_{\theta_k} \sG_{\htdb, \hat u, j}\right)\,.
\]
For, the first term in the above equation, we can bound it as: 
\begin{align*}
    \left|\bP\left((\hat u - u_\star) \  \partial_{\theta_k} \sG_{\htdb, \hat u, j}\right)\right| & \le \left\|\partial_{\theta_k} \sG_{\htdb, \hat u, j}\right\|_\infty \ \|\hat u - u_*\|_1 \\
    & \le \sup_{h \in \cH_{jk, n}}\|h\|_\infty \  \|\hat u - u_*\|_2 \,.
\end{align*}
As $h$ is bounded uniformly over $\cH_{jk, n}$, we have 
$$
 \left|\bP\left((\hat u - u_\star) \  \partial_{\theta_k} \sG_{\htdb, \hat u, j}\right)\right| \le C\|\hat u - u_*\|_2 \le C\delta_{n, 0} = o(1) ,.
$$
Hence the first summand of $B_{jk, n}$ is $o_p(1)$. For the second term 
$$
\left|\left(\bbP_n-\bP\right)\left((\hat u - u_\star) \  \partial_{\theta_k} \sG_{\htdb, \hat u, j}\right)\right| \le \frac{1}{\sqrt{n}}\sup_{a_{jk, n} \in \cA_{jk, n}} \ \left|\bbG_n(a_{jk, n})\right| \,.
$$
where $\bbG_n = \sqrt{n}(\bbP_n - \bP)$. 
By the Donsker property of the class \(\mathcal A_{jk,n}\) (established in Lemma \ref{lem:donsker-dthetaG}), $\sup_{a_{jk, n}\in\mathcal A_{jk,n}}|\mathbb G_n a_{jk, n}| = O_p(1).$ Therefore
\[
\left|\left(\bbP_n-\bP\right)\left((\hat u - u_\star) \  \partial_{\theta_k} \sG_{\htdb, \hat u, j}\right)\right| = O_p(n^{-1/2}) = o_p(1).
\]
Therefore $B_{n,jk}=o_p(1)$ for all $1 \le j, k \le q$ and consquently $\max_{1 \le j, k \le q}|B_{jk, n}| = o_p(1)$ as $q$ is fixed. Define $\bB_n \in \reals^{q \times q}$ with $\bB_{n, jk} = B_{jk, n}$. We have: 
$$
\left\|\frac1n\sum_{i=1}^n
(Y_i-\hat u(X_i)) \nabla_\theta \sG_{\hat\theta^{\rm db},\hat u}(X_i)\right\|_{\mathrm{op}} = \|\bB_n\|_{\rm op} \le C_q \max_{1 \le j, k, \le q} |B_{jk, n}| = o_p(1) \,. 
$$
On the other hand: 
$$
\|\nabla^2_\theta \cR(\htdb, \hat u) - \nabla^2_\theta \cR(\hat \theta^{\rm PI}, \hat u)\|_{\rm op} \le \sup_{\theta \in \Theta_n} \|\nabla^2_\theta \cR(\theta, \hat u) - \nabla^2_\theta \cR(\hat \theta^{\rm PI}, \hat u)\|_{\rm op} = o_p(1) \,,
$$
where the last conclusion follows from the continuity of $\theta \mapsto \nabla_\theta^2R(\theta,\hat u)$ under Assumption \ref{assm:psi_smoothness}, together with the fact that $\Theta_n$ is a shrinking neighborhood of $\hat\theta^{\rm PI}$.
% where the last conclusion follows from the fact that $\nabla^2 \cR(\theta, \hat u)$ is continuous with respect to $\theta$ and $\Theta_n$ is a shrinking neighborhood around $\hat \theta_{\rm PI}$ (Assumption \ref{assm:psi_smoothness}). 
As a consequence, 
$$
\|J_n(\htdb)v\|_2 \ge \|\nabla_\theta^2 \cR(\hat \theta^{\rm PI}, \hat u)v\|_2 - o_p(1) 
$$
Now, as $\lambda_{\min}(\nabla_\theta^2 \cR(\hat \theta^{\rm PI}, \hat u)) \ge \lambda_-$ (Assumption \ref{assm:curvature_pde_risk}), we have with probability going to $1$ $\|J_n(\htdb)v\|_2 \ge \lambda_-/2$ for all $v \in \bbS^{q-1}$. This completes the proof.  
\end{proof}

\section{Proof of Theorem \ref{thm:cond:db:rate}}
\label{sec:proof:theory:Bayesian}
Throughout $\lambda_{\max}(A)$ (resp. $\lambda_{\min}(A)$) denotes the largest (resp. smallest) eigenvalue of the matrix $A$, and $\|A\|_{\rm op}$ its operator norm. If $A\in\rset^{q\times q}$ is symmetric positive definite, and $x,y\in\rset^q$, we will use the notations 
\[\pscal{x}{y}_A \eqdef x^\top A y, \;\;\;\;\mbox{ and }\;\;\;\;  \|x\|_A \eqdef  \sqrt{\pscal{x}{x}_A}.\]

For $r\geq 0$, we define $\mathbf{B}_r = \mathbf{B}_r(\theta_\star)\eqdef \left\{\theta\in\rset^q:\;\|\theta- \theta_\star\|_2 \leq r\right\}$.  Given $0\leq s \leq \beta$, and for some constant $L_1>0$, we set
\[r_n(s) \eqdef L_1 \left(\log{n}\right)^{\frac12\left(1 - \frac{s}{\beta}\right)}\left(\frac{n}{\log{n}}\right)^{-\frac{\beta - s}{2\beta + d}},\;\;\mbox{ and }\;\; r_n\eqdef r_n(\tau),\]
where $\tau$ is the PDE order as defined  in Assumption \ref{assm:dgp}. We note that $r_n(\beta)=L_1$. We define
\[\mathcal{U}_n \eqdef \left\{u\in W^{\beta,2}(\Omega):\;\; \|u - u_\star\|_{W^{s,2}(\Omega)} \leq r_n(s),\;  0\leq s\leq \beta\right\}.\]
We recall that the conclusion of Theorem \ref{thm:deep_sobolev} is that $\PP(\hat u\notin \mathcal{U}_n) \leq n^{-C}$, for some constants $L_1,C$. Throughout the proof, $0<C,C_0,C_1,C_2<\infty$ denote some generic constants independent of the sample size $n$, that we do not track (i.e. their actual value may change from one equation to the next). 
Let 
\[ T_n^{(1)}(\theta,u) \eqdef \mathsf{S}_n^{\rm db}(\theta,u) - \mathsf{S}^{\rm db}(\theta,u),\;\;\;\theta\in\rset^q,\; u\in W^{\beta,2}(\Omega).\]
For  some constant $C_1$, we set
\[\e_{n,0}\eqdef\left\{ \D_n:\;  \;\;\max_{1\leq j\leq q}\;\;\sqrt{n}|T_{n,j}^{(1)}(\theta_\star,u_\star)| \leq  C_1\sqrt{A\log(n)} \right\},\]
and
\[\e_{n,1}\eqdef\left\{ \D_n:\;  \sup_{u\in\mathcal{U}_n}\;\sup_{\theta\in\mathbf{B}_1(\theta_\star)}\;\;\max_{1\leq j\leq q}\;\;\sqrt{n}|T_{n,j}^{(1)}(\theta,u) - T_{n,j}^{(1)}(\theta_\star,u_\star)| \leq  C_1\sqrt{A\log(n)} \right\},\]
where $T_{n,j}^{(1)}$ denotes the $j$-th component of $T_{n}^{(1)}$. 
In these definitions, $A\geq 1$ is an absolute constant that can be taken as $A=1$. Its sole purpose is to control the sizes (hence the probabilities) of these events. Throughout we assume that the sample size $n$ is large enough such that
\begin{equation}\label{tech:cond:1pc} 
\sqrt{\frac{2q\log(2q)}{\bar v\log(n) + n\|V_n^{-1/2}M\|_{{\rm op}}^2}} \leq \min(1,r_n),\;\;\mbox{ and }\;\;\; r_n^2\leq  C_1 \sqrt{\frac{A q\log(n)}{n}},\;\;
\end{equation}
and furthermore, we can find a constant $C_0$ such that
\begin{equation}\label{tech:cond:2pc}
\max\left(C_1 A\lambda_{\max}(V_n^{-1})\log(n),\; \log\left(\bar v\log(n) + n \|V_n^{-1/2}M\|_{{\rm op}}^2\right)\right) \leq C_0\log(n).
\end{equation}
We note that, since by assumption $\beta>2\tau + d/2$, $r_n=r_n(\tau) \sim n^{-(\beta-\tau)/(2\beta +d)}$. Hence the conditions in (\ref{tech:cond:1pc}) are always true, for all $n$  large enough. (\ref{tech:cond:2pc}) also readily follows from Assumption \ref{assump:prior:pi0} for all $n$ large enough. By Lemma  \ref{lem:localization} below we can find a constant $b>0$ such that for all $\D_n\in \{\D_n:\; \hat u\in\mathcal{U}_n\} \cap \e_{n,0}\cap \e_{n,1}$
\[ \;\Pi^{(\theta\vert u)}(\mathbf{B}_b^c\vert \hat u,\D_n) \leq n^{-C},\] 
for some constant $C>0$. Using this value $b$ we then define
\[\e_n \eqdef \left\{ \D_n:\; \hat u\in\mathcal{U}_n\right\} \cap \e_{n,0} \cap \e_{n,1} \;\cap \; \e_{n,2},\]
where
\[\e_{n,2}\eqdef\left\{ \D_n:\;  \sup_{u\in\mathcal{U}_n}\;\sup_{\theta\in\mathbf{B}_b(\theta_\star)}\;\;\max_{1\leq j\leq q}\;\;\sqrt{n}|T_{n,j}^{(1)}(\theta,u) - T_{n,j}^{(1)}(\theta_\star,u_\star)| \leq  C_1 b\sqrt{A\log(n)} \right\},\]

First we argue that we can find $C_1$ in the definition of $\e_n$ such that for all $n$ large enough, $\PP(\D_n\notin \e_n)\leq n^{-C_3}$, for some constant $C_3$. Specifically, we can find a constant $C_2<\infty$ such that if in addition to (\ref{tech:cond:1pc})-(\ref{tech:cond:2pc}) $n$ satisfies
\begin{equation}\label{tech:cond:3pc} 
C\times L_1^{\frac{d}{2\beta}} \times r_n(0)^{1-{\frac{d}{2\beta}}} \leq \min\left(1,b^{1/\alpha_*}\right),\;\;\;\mbox{ and }\;\;\; n\geq C_2 A\log(n),
\end{equation}
where $C$ is as in Theorem \ref{thm:interpolation},  then $\PP(\D_n\notin \e_n)\leq n^{-C_3}$, for some constant $C_3$.
\\\\
\noindent
\underline{\bf Event $\left\{ \D_n:\; \hat u\in\mathcal{U}_n\right\}$:} By Theorem \ref{thm:deep_sobolev} that we can choose $L_1$ in the definition of $r_n(s)$ such that $\PP(\hat u\notin \mathcal{U}_n)\leq n^{-C}$ for some constant $C$. 
\\\\
\noindent
\underline{\bf Event $\e_{n,0}$:} We have
\[T_{n,j}^{(1)}(\theta_\star,u_\star) =\frac{1}{n}\sum_{i=1}^n \epsilon_i \mathsf{G}_{\theta_\star,u_\star,j}(X_i).\]
By Proposition \ref{prop:def_G}, 
\[ \|\mathsf{G}_{\theta_\star,u_\star,j}\|_\infty \leq c_0.\]
Therefore, for $m\geq 2$,
\[\PE\left(\left|\frac{\epsilon}{\sigma_\epsilon} \mathsf{G}_{\theta_\star,u_\star,j}(X)\right|^m\right) \leq c_0^m \PE\left(\left|\frac{\epsilon}{\sigma_\epsilon}\right|^m\right) \leq c_0^2 c_0^{m-2} \frac{m!}{2},\]
where we use that for a sub-Gaussian random variable $Z$ with mean $0$ and variance $1$, $\PE(|Z|^m)\leq m!/2$ for $m\geq 2$. By Bernstein's inequality (Lemma 5.7 of \cite{geer2000empirical}) with $K_0=R_0=\sigma_\epsilon c_0$, for any $a_1>0$ satisfying $a_1\leq \sigma_\epsilon c_0\sqrt{n}$,
\begin{align*}
\PP\left(\left|\frac{1}{\sqrt{n}}\sum_{i=1}^n \epsilon_i \mathsf{G}_{\theta_\star,u_\star,j}(X_i)\right|>a_1\right) & \leq 2\exp\left(-\frac{a_1^2}{2\left(\sigma_\epsilon^2 c_0^2 + \sigma_\epsilon c_0 a_1 n^{-1/2}\right)}\right) \\
& \leq 2\exp\left(-\frac{a_1^2}{4\sigma_\epsilon^2 c_0^2}\right).
\end{align*}
Setting $a_1 = 2\sigma_\epsilon c_0\sqrt{A\log n}$, the condition $a_1\leq\sigma_\epsilon c_0\sqrt{n}$ reduces to $n\geq 4A\log n$, which is implied by the second condition in (\ref{tech:cond:3pc}) by taking $C_2 \geq 4$. Hence, provided that $C_1 \geq 2\sigma_\epsilon c_0$, we have $\PP(\D_n\notin \e_{n,0})\leq 2/n^A$.
\\\\
\noindent
\underline{\bf Event $\e_{n,1} \;\cap \; \e_{n,2}$:} Under the stated conditions, Lemma \ref{lem:tailbbound:dbScore} implies that $\PP(\D_n\notin \e_{n,1})\leq Cn^{-A}$ and $\PP(\D_n\notin \e_{n,2})\leq Cn^{-A}$.

\subsection{Proof of the contraction rate of Theorem \ref{thm:cond:db:rate}}
For $r>0$, we set  $\mathbf{B}^c_{r} \eqdef \{\theta\in\rset^q:\;\|\theta - \theta_\star\|_2 > r\}$.  Thanks to Lemma  \ref{lem:localization} we have chosen  $b>\tau_n$ such that $$
\textbf{1}_{\e_n}(\D_n)\;\Pi^{(\theta\vert u)}(\mathbf{B}_b^c\vert \hat u,\D_n) \leq n^{-C}, 
$$
for some constant $C>0$.  Writing $\mathbf{B}^c_{\tau_n} = (\mathbf{B}_b\setminus \mathbf{B}_{\tau_n}) \cup \mathbf{B}_b^c$, we have
\[\Pi^{(\theta\vert u)}(\mathbf{B}^c_{\tau_n}\vert \hat u,\D_n) \leq \textbf{1}_{\e_n}(\cD_n)\left(\Pi^{(\theta\vert u)}(\mathbf{B}_b^c\vert \hat u,\D_n) + \Pi^{(\theta\vert u)}((\mathbf{B}_b\setminus \mathbf{B}_{\tau_n})\vert \hat u,\D_n)\right) + \textbf{1}_{\e_n^c}(\cD_n).\]
 Hence
\[\PE\left[\Pi^{(\theta\vert u)}(\mathbf{B}^c_{\tau_n}\vert \hat u,\D_n)\right] 
\leq \PE\left[\textbf{1}_{\e_n}(\cD_n) \Pi^{(\theta\vert u)}(\mathbf{B}_b\setminus \mathbf{B}_{\tau_n}\vert \hat u,\D_n) \right]  + \frac{1}{n^C}+ \PP\left(\D_n\notin \e_n\right). \]
We write
\begin{multline}\label{eq1:proof:thm:pc}
\Pi^{(\theta\vert u)}(\mathbf{B}_b\setminus \mathbf{B}_{\tau_n}\vert \hat  u,\D_n) \\
= \frac{(2\pi)^{-q/2}\int_{\mathbf{B}_b\setminus \mathbf{B}_{\tau_n}}e^{-\frac{n}{2}\left(\|\mathsf{S}_n^{\rm db}(\theta,\hat u)\|_{V_n^{-1}}^2 - \|\mathsf{S}_n^{\rm db}(\theta_\star,u_\star)\|_{V_n^{-1}}^2\right) + (\log\pi_0(\theta) - \log\pi_0(\theta_\star))}\rmd\theta}{(2\pi)^{-q/2}\int_{\rset^q}e^{-\frac{n}{2}\left(\|\mathsf{S}_n^{\rm db}(\theta,\hat u)\|_{V_n^{-1}}^2-\|\mathsf{S}_n^{\rm db}(\theta_\star,u_\star)\|_{V_n^{-1}}^2\right) + (\log\pi_0(\theta) - \log\pi_0(\theta_\star))}\rmd\theta}.
\end{multline}
It is shown in Lemma \ref{lem:control:NC} below that under the stated assumptions, and for $\cD_n\in\e_n$,
 \begin{multline} 
 \left(\frac{1}{2\pi}\right)^{q/2}\int_{\rset^q}e^{-\frac{n}{2}\left(\|\mathsf{S}_n^{\rm db}(\theta,\hat u)\|_{V_n^{-1}}^2-\|\mathsf{S}_n^{\rm db}(\theta_\star,u_\star)\|_{V_n^{-1}}^2\right)+(\log\pi_0(\theta) - \log\pi_0(\theta_\star))}\rmd\theta \\
 \geq \frac{1}{2} \exp\left(-c_0 n\lambda_{\max}(V_n^{-1}) \bar r_n^2 -\frac{q}{2}\log\left(\bar v\log n + n \|V_n^{-1/2}M\|_{{\rm op}}^2\right)\right),
 \end{multline}
for some constant $c_0$, where
 \[ \bar r_n \eqdef \max\left(r_n^{2},\; C_1 \sqrt{\frac{Aq\log(n)}{n}}\right) = C_1 \sqrt{\frac{Aq\log(n)}{n}}.\]
 We exploit that lower-bound together with (\ref{tech:cond:2pc}) to further reduce (\ref{eq1:proof:thm:pc}) to 
\begin{multline}\label{eq2:proof:thm:pc}
    \textbf{1}_{\e_n}(\D_n) \Pi^{(\theta\vert u)}(\mathbf{B}_b\setminus \mathbf{B}_{\tau_n}\vert \hat u,\D_n) \leq 2\exp\left(C_0 q\log(n)\right)\\
    \times (2\pi)^{-q/2}\int_{\mathbf{B}_b\setminus \mathbf{B}_{\tau_n}}e^{-\frac{n}{2}\left(\|\mathsf{S}_n^{\rm db}(\theta,\hat u)\|_{V_n^{-1}}^2 - \|\mathsf{S}_n^{\rm db}(\theta_\star,u_\star)\|_{V_n^{-1}}^2\right) +(\log\pi_0(\theta) - \log\pi_0(\theta_\star))}\rmd\theta,
\end{multline}
for some constant $C_0$. We write
\begin{multline*}
\|\mathsf{S}_n^{\rm db}(\theta,\hat u)\|_{V_n^{-1}}^2 - \|\mathsf{S}_n^{\rm db}(\theta_\star,u_\star)\|_{V_n^{-1}}^2 = \|\mathsf{S}_n^{\rm db}(\theta,\hat u) - \mathsf{S}_n^{\rm db}(\theta_\star,u_\star)\|_{V_n^{-1}}^2 \\ + 2\pscal{\mathsf{S}_n^{\rm db}(\theta_\star,u_\star)}{\mathsf{S}_n^{\rm db}(\theta,\hat u) - \mathsf{S}_n^{\rm db}(\theta_\star,u_\star)}_{V_n^{-1}}.
\end{multline*}
Using Lemma \ref{lem:control:NC},
\[\mathsf{S}_n^{\rm db}(\theta,u) - \mathsf{S}_n^{\rm db}(\theta_\star,u_\star) =  (\nabla_\theta \cR)(\theta,u_\star) +  \underbrace{T_n^{(1)}(\theta,u) - T_n^{(1)}(\theta_\star,u_\star)}_{R_n^{(1)}(\theta,u)} + \underbrace{ T_n^{(2)}(\theta,u)}_{R_n^{(2)}(\theta,u)},\]
Since $(\nabla_\theta\cR)(\theta_\star,u_\star)=0$, a first order Taylor expansion and the notation $M_\theta \eqdef \nabla_\theta^{(2)}\cR(\theta,u_\star)$ yields
\[\mathsf{S}_n^{\rm db}(\theta,u) - \mathsf{S}_n^{\rm db}(\theta_\star,u_\star) =  M_{\bar \theta}(\theta-\theta_\star) +  R_n^{(1)}(\theta,u) + R_n^{(2)}(\theta,u),\]
for some $\bar\theta$ on the segment between $\theta$ and $\theta_\star$. Therefore, 
\begin{align*}
    & \|\mathsf{S}_n^{\rm db}(\theta,u)\|_{V_n^{-1}}^2 - \|\mathsf{S}_n^{\rm db}(\theta_\star,u_\star)\|_{V_n^{-1}}^2 \\
    & \geq \|M_{\bar\theta}(\theta-\theta_\star)\|_{V_n^{-1}}^2 +2 \pscal{M_{\bar\theta}(\theta-\theta_\star)}{\mathsf{S}_n^{\rm db}(\theta_\star,u_\star) + R_n^{(1)}(\theta,u) + R_n^{(2)}(\theta,u)}_{V_n^{-1}} \\
    &-2\lambda_{\max}(V_n^{-1})\|\mathsf{S}_n^{\rm db}(\theta_\star,u_\star)\|_2 \times \|R_n^{(1)}(\theta,u) + R_n^{(2)}(\theta,u)\|_2.
\end{align*}
For  $\mathcal{\D}_n\in\e_n$,
\[\|\mathsf{S}_n^{\rm db}(\theta_\star,u_\star)\|_2 + \sup_{\theta\in\mathbf{B}_b}\;\sup_{u\in\mathcal{U}_n} \|R_n^{(1)}(\theta,u)\|_2 \leq  C_1(1 + b) \sqrt{\frac{Aq \log(n)}{n}},\]
and we apply (\ref{stab:G:2}) of Proposition \ref{prop:def_G} to get
\begin{multline*}
    \|R^{(2)}_n(\theta,u)\|_2 = \|T^{(2)}_n(\theta,u)\|_2 \leq c_1(\theta)\times  \| u - u_\star\|_{W^{\tau,2}}^2\\
\leq \left(\sup_{\theta:\;\|\theta-\theta_\star\|_2\leq b}\; c_1(\theta)\right)\;\times \; r_n^{2} \leq C_1 \sqrt{\frac{Aq \log(n)}{n}},
\end{multline*}
where the last inequality uses (\ref{tech:cond:1pc}).
Hence, with the definition of the curvature parameter $\lambda_-(b)$, we conclude that
\begin{multline*}  
\|\mathsf{S}_n^{\rm db}(\theta,u)\|_{V_n^{-1}}^2 - \|\mathsf{S}_n^{\rm db}(\theta_\star,u_\star)\|_{V_n^{-1}}^2 \geq \lambda_-(b)\|\theta-\theta_\star\|_2^2 \\
- C_0\sqrt{\lambda_+(b)\lambda_{\max}(V_n^{-1})} \times \sqrt{\frac{Aq\log(n)}{n}}\|\theta-\theta_\star\|_2 - C_0\lambda_{\max}(V_n^{-1})\times (1 + b)^2\frac{Aq\log(n)}{n},
\end{multline*}
for some constant $C_0$. If we choose $c$ in the definition of $\tau_n$ such that 
\[\tau_n\geq \left(4C_0\frac{\sqrt{\lambda_+(b)}}{\lambda_-(b)} + 2(1+b)\sqrt{\frac{C_0}{\lambda_-(b)}}\right)\sqrt{\frac{A\lambda_{\max}(V_n^{-1})q\log(n)}{n}},\]
then for $\|\theta-\theta_\star\|_2\geq \tau_n$,
\[ \frac{\lambda_-(b)}{4}\|\theta-\theta_\star\|_2^2 \geq C_0\lambda_{\max}(V_n^{-1})\times b^2\frac{Aq\log(n)}{n},\]
and
\[\frac{\lambda_-(b)}{4}\|\theta-\theta_\star\|_2^2 \geq C_0\sqrt{\lambda_+(b)\lambda_{\max}(V_n^{-1})} \times b\sqrt{\frac{Aq\log(n)}{n}}\|\theta-\theta_\star\|_2.\]
Hence, for $\|\theta - \theta_\star\|_2>\tau_n$,
\[ \|\mathsf{S}_n^{\rm db}(\theta,u)\|_{V_n^{-1}}^2 - \|\mathsf{S}_n^{\rm db}(\theta_\star,u_\star)\|_{V_n^{-1}}^2   \geq  \frac{\lambda_-(b)\tau_n^2}{2}.\] 
We are thus able to conclude from (\ref{eq2:proof:thm:pc}) that for $\D_n\in\e_n$ 
\begin{align}\label{eq3:proof:thm:pc}
   & \Pi^{(\theta\vert u)}(\mathbf{B}_b\setminus \mathbf{B}_{\tau_n}\vert u,\D_n) \notag \\  \leq & 2\exp\left(C_0 q\log(n)-\frac{n \lambda_-(b) \tau_n^2}{4}\right)\nonumber\\
    & \quad\quad \times  (2\pi)^{-q/2}\int_{\mathbf{B}_b\setminus \mathbf{B}_r}e^{\log\pi_0(\theta) - \log\pi_0(\theta_\star))}\rmd\theta\nonumber\\
    \leq & \exp\left(C_0 q\log(n)  -\frac{n \lambda_-(b) \tau_n^2}{4} \right)\nonumber\\
    &\times \left(\frac{1}{2\pi}\right)^{q/2}\int_{\rset^q}\exp\left(\pscal{\nabla\log\pi_0(\theta_\star)}{\theta-\theta_\star} - \frac{\underline{v}\log(n)}{2}\|\theta-\theta_\star\|_2^2\right)\rmd\theta\nonumber\\
    &\leq \exp\left(C_0 q\log(n)  -\frac{n \lambda_-(b) \tau_n^2}{4} + \frac{q}{2}\log\left(\frac{1}{\underline{v}\log(n)}\right) + \frac{\|\nabla\log\pi_0(\theta_\star)\|_2^2}{2\underline{v}\log(n)}\right)\nonumber\\
    & \leq 2e^{-C_0\log(n)q},
\end{align}
by taking $c$ in $\tau_n$ such that
\[ n \lambda_-(b) \tau_n^2 \geq 8C_0q\log(n) +  \frac{\|\nabla\log\pi_0(\theta_\star)\|_2^2}{2\underline{v}\log(n)},\]
where we recall that by Assumption \ref{assump:prior:pi0}, $\underline{v}\log(n)\geq 1$, and $\|\nabla\log\pi_0(\theta_\star)\|_2\leq \underline{v}\log(n)\sqrt{q}$.
We conclude that
\[\PE\left[\Pi^{(\theta\vert u)}(\mathbf{B}_{\tau_n}^c\vert \hat u,\D_n)\right] \leq 2e^{-C_0\log(n)q}  + \frac{1}{n^C}+ \PP\left(\D_n\notin \e_n\right).\]
We have seen in the preamble of the proof that the probability on the right-hand side in the last display is bounded by $1/n^{C}$. Hence the result.
\medskip

\subsection{Proof of the Bernstein-von Mises approximation of Theorem \ref{thm:cond:db:rate}}
Let 
\[\Psi_n(\rmd h) \propto\exp\left(-\frac{n}{2} \sS^{\rm db}_n(\theta_\star + n^{-1/2}h,\hat u)^\top V_n^{-1} \sS^{\rm db}_n(\theta_\star + n^{-1/2}h,\hat u)\right)\pi_0(\theta_\star + n^{-1/2}h)\rmd h,\] 
denote the distribution of $\sqrt{n}(\vartheta - \theta_\star)$, where $\vartheta\sim \Pi^{(\theta\vert u)}_{\rm db}(\cdot\vert \hat u,\D_n)$, $Q_n$ be the distribution of $\mathbf{N}_q(\Delta_n, (M V_n^{-1}M)^{-1})$. Let $\textbf{B}_n \eqdef\{h\in\rset^q:\;\|h\|_2 \leq \tau_n\sqrt{n}\}$, and  let $\Psi_n^{\textbf{B}_n}$ (resp. $Q_n^{\textbf{B}_n}$) denote the restriction of $\Psi_n$ (resp. $Q_n$) on $\textbf{B}_n$. That is, $\Psi_n^{\textbf{B}_n}(A) = \Psi_n(A\cap \textbf{B}_n)/\Psi_n(\textbf{B}_n)$, $A\subseteq \textbf{B}_n$, and similarly for $Q_{\textbf{B}_n}$. To save notations, we will also write $\Psi_n^{\textbf{B}_n}$ to denote the density of $\Psi_n^{\textbf{B}_n}$ with respect to the Lebesgue measure, and similarly for $Q_n^{\textbf{B}_n}$ -- which object we are referring to should be clear from the context. We therefore have
\[\PE\left(\|\Psi_n - Q_n\|_\tv\right) \leq  2\PE\left(\Psi_n(\textbf{B}_n^c)\right) + 2\PE\left(Q_n(\textbf{B}_n^c)\right) + \PE\left(\|\Psi_n^{\textbf{B}_n} - Q_n^{\textbf{B}_n}\|_\tv\right).\]
By definition, $\Psi_n(\mathbf{B}_n^c) = \Pi_{\rm db}^{(\theta\vert u)}(\{\theta:\;\|\theta-\theta_\star\|_2 >\tau_n\}\vert\hat u,\;\D_n)$. As a result, the first part of the theorem guarantees that
\[\PE\left(\Psi_n(\textbf{B}_n^c)\right) \leq n^{-C},\]
for some constant $C>0$. For the second term, we have
\begin{multline*} 
 \PE\left[Q_n(\textbf{B}_n^c)\right)]\leq \PP(\|\Delta_n\|_2>\tau_n\sqrt{n}/2) + \PE\left(\textbf{1}_{\{\|\Delta_n\|_2\leq \tau_n\sqrt{n}/2\}}\;Q_n(\textbf{B}_n^c)\right) \\
 \leq \PP\left(\|\Delta_n\|_2>\tau_n\sqrt{n}/2\right) + \PP\left(\|(MV_n^{-1}M)^{-1/2} Z\|_2 > \frac{\tau_n\sqrt{n}}{2}\right),\end{multline*} 
where $Z= (Z_1,\ldots,Z_q)\stackrel{i.i.d.}{\sim}\textbf{N}(0,1)$. Without any loss of generality, we can choose $c$ in $\tau_n$ large enough such that both terms in the last display converge to $0$, as $n\to\infty$. We deduce that 
\begin{equation}\label{eq:bvm:b0}
\lim_{n\to\infty} \PE\left[\|\Psi_n - Q_n\|_\tv\right] = \lim_{n\to\infty} \PE\left[\|\Psi_n^{\textbf{B}_n} - Q_n^{\textbf{B}_n}\|_\tv\right].
\end{equation}
For any two arbitrary densities $f_1,f_2$ on some measure space $\Xset$, with $f_2>0$, since $|a-b| = a+b -2\min(a,b)$ for all $a,b\geq 0$, we can write
\begin{align*} 
\int_{\Xset}|f_1(x) - f_2(x)|\rmd x &= 2\int_{\Xset}\left[1 - \min\left(1,\frac{f_1(x)}{f_2(x)}\right)\right] f_2(x)\rmd x\\
& = 2\int_{\Xset}\max\left(0,1-\frac{f_1(x)}{f_2(x)}\right)f_2(x)\rmd x.
\end{align*}
We apply this to $f_1 = \Psi_n^{\textbf{B}_n}$ and $f_2 = Q_n^{\textbf{B}_n}$. We set
\begin{align*}
L(h,u) \eqdef & \frac{n}{2}\left(\|\mathsf{S}_n^{\rm db}(\theta_\star + n^{-1/2}h,u)\|_{V_n^{-1}}^2 - \|\mathsf{S}_n^{\rm db}(\theta_\star,u_\star)\|_{V_n^{-1}}^2\right) \\
& \qquad \qquad \qquad \qquad - (\log(\pi_0(\theta_\star + n^{-1/2}h )) - \log(\pi_0(\theta_\star)),
\end{align*}
and 
\[\chi(h,u) \eqdef L(h,u) - \frac{1}{2}(h-\Delta_n)(M V_n^{-1}M)(h-\Delta_n).\]
Hence
\begin{align*} 
\frac{\Psi_n^{\textbf{B}_n}(h)}{Q_n^{\textbf{B}_n}(h)} &= \frac{e^{-\chi(h,\hat u)}\times \int_{\textbf{B}_n}e^{-\frac{1}{2}(h'-\Delta_n)(M V_n^{-1}M)(h'-\Delta_n)}\rmd h'}{\int_{\textbf{B}_n}e^{-L(\theta_\star + n^{-1/2}h',\hat u)}\rmd h'} \\
&  = \frac{\int_{\textbf{B}_n}e^{\chi(h',\hat u)-\chi(h,\hat u)} e^{-L(\theta_\star + n^{-1/2}h',\hat u)}\rmd h'}{\int_{\textbf{B}_n}e^{-L(\theta_\star + n^{-1/2}h',\hat u)}\rmd h'}.
\end{align*}
With the last display, and since $x\mapsto \max(0,x)$ is convex, we apply Jensen's inequality to obtain that for all $\theta\in \textbf{B}_n$,
\[\max\left(0, 1 - \frac{\Psi_n^{\textbf{B}_n}(h)}{Q_n^{\textbf{B}_n}(h)}\right) \leq \int_{\textbf{B}_n}\max\left(0,1-e^{\chi(h',\hat u)-\chi(h,\hat u)}\right)\Psi_n^{\textbf{B}_n}(h')\rmd h'. \]
And since $\max(0,1-e^x)\leq \min(1,|x|)$, we readily conclude that
\begin{multline*}
\|\Psi_n^{\textbf{B}_n} - Q_n^{\textbf{B}_n}\|_\tv\leq 2 \int_{\textbf{B}_n}\int_{\textbf{B}_n}\max\left(0,1-e^{\chi(h',\hat u)-\chi(h,\hat u)}\right)\Psi_n^{\textbf{B}_n}(h')\rmd h' Q_n^{\textbf{B}_n}(h)\rmd h \\
\leq 2\textbf{1}_{\bar \e_n^c}(\D_n) + 2\textbf{1}_{\bar \e_n}(\D_n)\;\sup_{h_1,h_2\in \textbf{B}_n}\;\left|\chi(h_1,\hat u)-\chi(h_2,\hat u)\right|,
\end{multline*}
where we define
\[\bar\e_n \eqdef \e_n \cap \e_{n,3},\]
where, with $b_n \eqdef r_n^{\alpha_*(1-d/(2\beta))}$,
\[\e_{n,3}\eqdef\left\{ \D_n:\;  \sup_{u\in\mathcal{U}_n}\;\sup_{\theta\in\mathbf{B}_{\tau_n}(\theta_\star)}\;\;\max_{1\leq j\leq q}\;\;\sqrt{n}|T_{n,j}^{(1)}(\theta,u) - T_{n,j}^{(1)}(\theta_\star,u_\star)| \leq  C_1 b_n\sqrt{A\log(n)} \right\}.\]
Hence it suffices control the probability of the event $\bar\e_n$, and to show that the term $\textbf{1}_{\bar \e_n}(\D_n) \sup_{h_1,h_2\in \textbf{B}_n} |\chi(h_1,\hat u) - \chi(h_2,\hat u)| \to 0$ in $L^1$. To save space, we introduce $\theta_h \eqdef \theta_\star + n^{-1/2} h$. We have
\begin{multline*}
\|\mathsf{S}_n^{\rm db}(\theta_h,u)\|_{V_n^{-1}}^2-\|\mathsf{S}_n^{\rm db}(\theta_\star,u_\star)\|_{V_n^{-1}}^2 = \|\mathsf{S}_n^{\rm db}(\theta_h,u) - \mathsf{S}_n^{\rm db}(\theta_\star,u_\star)\|_{V_n^{-1}}^2 \\
+ 2\pscal{\mathsf{S}_n^{\rm db}(\theta_\star,u_\star)}{\mathsf{S}_n^{\rm db}(\theta_h,u) - \mathsf{S}_n^{\rm db}(\theta_\star,u_\star)}_{V_n^{-1}},\end{multline*}
and furthermore, as shown in (\ref{eq10:lem:lb}) of Lemma \ref{lem:control:NC}
\[\mathsf{S}_n^{\rm db}(\theta,u) - \mathsf{S}_n^{\rm db}(\theta_\star,u_\star) = M(\theta - \theta_\star) + R_n(\theta,u),\;\; \mbox{ where }\;\; R_n(\theta,u) \eqdef  \sum_{k=1}^3 R_n^{(k)}(\theta,u),\]
with 
\begin{multline*} 
R_n^{(1)}(\theta,u) \eqdef T_n^{(1)}(\theta,u) - T_n^{(1)}(\theta_\star,u_\star),\; \\
R_n^{(2)}(\theta,u)\eqdef T_n^{(2)}(\theta,u),\; \;\; \mbox{ and }\;\;\;\; R_n^{(3)}(\theta,u) \eqdef \frac{1}{2} (\nabla^{(3)}_\theta\cR)(\bar\theta,u_\star)\cdot(\theta-\theta_\star,\theta-\theta_\star).\end{multline*}
Hence
\begin{align*}
& \|\mathsf{S}_n^{\rm db}(\theta_h,u)\|_{V_n^{-1}}^2-\|\mathsf{S}_n^{\rm db}(\theta_\star,u_\star)\|_{V_n^{-1}}^2 \\ 
& = (\theta_h - \theta_\star)^\top M V_n^{-1}M (\theta_h-\theta_\star) + 2(\theta_h-\theta_\star)^\top M V_n^{-1}R_n(\theta_h,u) \\
& + \|R_n(\theta_h,u)\|_{V_n^{-1}}^2 + 2\pscal{\mathsf{S}_n^{\rm db}(\theta_\star,u_\star)}{M(\theta_h-\theta_\star)}_{V_n^{-1}} + 2\pscal{\mathsf{S}_n^{\rm db}(\theta_\star,u_\star)}{R_n(\theta_h,u)}_{V_n^{-1}}.
\end{align*}
Therefore
\begin{align*}
L(h,u) =& \frac{n}{2}\left(\|\mathsf{S}_n^{\rm db}(\theta_h,u)\|_{V_n^{-1}}^2 - \|\mathsf{S}_n^{\rm db}(\theta_\star,u_\star)\|_{V_n^{-1}}^2\right) - (\log(\pi_0(\theta_h )) - \log(\pi_0(\theta_\star)) \\
  =&\frac{1}{2}h^\top (M V_n^{-1} M) h + \sqrt{n}h^\top M V_n^{-1}\mathsf{S}_n^{\rm db}(\theta_\star,u_\star)+ n(\theta_h-\theta_\star)^\top M V_n^{-1}R_n(\theta_h,u)\\
 & + \frac{n}{2}\left( \|R_n(\theta_h,u)\|_{V_n^{-1}}^2 +  2\pscal{\mathsf{S}_n^{\rm db}(\theta_\star,u_\star)}{R_n(\theta_h,u)}_{V_n^{-1}} \right)\\
 & - (\log(\pi_0(\theta_h )) - \log(\pi_0(\theta_\star)).
\end{align*}
Noting by completing the square that 
\[ \frac{1}{2}h^\top (M V_n^{-1} M) h + \sqrt{n}h^\top M V_n^{-1}\mathsf{S}_n^{\rm db}(\theta_\star,u_\star) = \frac{1}{2}(h-\Delta_n)^\top (MV_n^{-1}M)(h-\Delta_n) + C_n,\]
where $C_n$ does not depend on $h$, we conclude that
\begin{multline*} 
\chi(h,u) =  n(\theta_h-\theta_\star)^\top M V_n^{-1}R_n(\theta_h,u)\\
 + \frac{n}{2}\left( \|R_n(\theta_h,u)\|_{V_n^{-1}}^2 +  2\pscal{\mathsf{S}_n^{\rm db}(\theta_\star,u_\star)}{R_n(\theta_h,u)}_{V_n^{-1}} \right)\\
  - (\log(\pi_0(\theta_h )) - \log(\pi_0(\theta_\star)) + C_n.
 \end{multline*}
Taking the difference,
\begin{align*}
\chi(h,u) - \chi(h',u) & =  n\left[(\theta_h-\theta_\star)^\top M V_n^{-1}R_n(\theta_h,u) - (\theta_{h'}-\theta_\star)^\top M V_n^{-1}R_n(\theta_{h'},u) \right] \\
& + \frac{n}{2}\left[\|R_n(\theta_h,u)\|_{V_n^{-1}}^2-\|R_n(\theta_{h'},u)\|_{V_n^{-1}}^2 \right]
\\ 
& + n\left[\pscal{\mathsf{S}_n^{\rm db}(\theta_\star,u_\star)}{R_n(\theta_h,u) - R_n(\theta_{h'},u)}_{V_n^{-1}}\right] \\
& +(\log(\pi_0(\theta_{h'} )) - \log(\pi_0(\theta_\star)) - (\log(\pi_0(\theta_h )) - \log(\pi_0(\theta_\star)).
\end{align*}
We deal with each term separately.
\\\\
\noindent
\underline{\bf On terms $n(\theta_h-\theta_\star)^\top M V_n^{-1}R_n(\theta_h,u)$ and $n\|R_n(\theta_h,u)\|_{V_n^{-1}}^2$:}
By definition, $R_n^{(1)}(\theta,u) =T_n^{(1)}(\theta,u) - T_n^{(1)}(\theta_\star,u_\star)$. Therefore, for  $\D_n\in\bar \e_n$, and recalling $b_n \eqdef r_n^{\alpha_*(1-d/(2\beta))}$,
\[\sup_{u\in\mathcal{U}_n}\;\sup_{h\in\mathbf{B}_n}\;\|R_n^{(1)}(\theta_h,u)\|_2 \leq C_1 b_n\sqrt{\frac{Aq\log(n)}{n}},\;\;\mbox{ and }\;\; \|\mathsf{S}_n^{(\rm db)}(\theta_\star,u_\star)\|_2 \leq C_1 \sqrt{\frac{Aq\log(n)}{n}}.\]
We apply Proposition \ref{prop:def_G}-(2) to get for all $u\in\mathcal{U}_n$
\[\|R^{(2)}_n(\theta,u)\|_2 = \|T^{(2)}_n(\theta,u)\|_2 \leq c_0 \sqrt{q}\times  \| u - u_\star\|_{W^{\tau,2}}^{2}\leq c_0 r_n^2,\]
for some constant $c_0$. Assumption \ref{assm:psi_smoothness} implies that for  $\|\theta-\theta_\star\|_2\leq \tau_n$,
\[ \|R_n^{(3)}(\theta,u)\|_2 \leq c_0 \|\theta -\theta_\star\|_2^2 \leq c_0 \tau_n^2.\]
To summarize, we have
\begin{align*}
\sup_{u\in\mathcal{U}_n}\;\sup_{h\in\mathbf{B}_n}\;\|R_n(\theta_h,u)\|_2 & \leq C_1\max\left(b_n\sqrt{\frac{Aq\log(n)}{n}}, \tau_n^2, r_n^2\right) \\
& = C_1 \max\left(b_n\sqrt{\frac{Aq\log(n)}{n}}, r_n^2\right),
\end{align*}
which implies that 
\begin{equation} 
\label{eq:bvm:b1} \sup_{u\in\mathcal{U}_n}\;\sup_{h\in\mathbf{B}_n}\;n\left|(\theta_h-\theta_\star)^\top M V_n^{-1}R_n(\theta_h,u)\right| \leq C_1 \max(b_n\log(n), n\tau_n r_n^2) \to 0,
\end{equation}
and
\begin{equation}\label{eq:bvm:b2} 
\sup_{u\in\mathcal{U}_n}\;\sup_{h\in\mathbf{B}_n}\;n\|R_n(\theta_h,u)\|_{V_n^{-1}}^2 \leq C_1 \max\left(b_n^2\log(n),n r_n^4\right) \to 0,
\end{equation}
as $n\to\infty$.
\\\\
\noindent
\underline{\bf On the term $n\left[\pscal{\mathsf{S}_n^{\rm db}(\theta_\star,u_\star)}{R_n(\theta_h,u) - R_n(\theta_{h'},u)}_{V_n^{-1}}\right]$:} To handle this term, we add and subtract $R_n^{(2)}(\theta_\star,u)$. First, note as in the previous calculations that we readily have for $\D_n\in\bar \e_n$,
\begin{equation}\label{eq:bvm:b3}
\sup_{u\in\mathcal{U}_n}\;\sup_{h\in\mathbf{B}_n}\; n\left|\pscal{\mathsf{S}_n^{\rm db}(\theta_\star,u_\star)}{R_n^{(1)}(\theta_h,u)}_{V_n^{-1}}\right| \leq C_1 b_n\log(n)\to 0,
\end{equation}
as $n\to\infty$, and
\begin{equation}\label{eq:bvm:b4} 
\sup_{u\in\mathcal{U}_n}\;\sup_{h\in\mathbf{B}_n}\; n\left|\pscal{\mathsf{S}_n^{\rm db}(\theta_\star,u_\star)}{R_n^{(3)}(\theta_h,u)}_{V_n^{-1}}\right| \leq C_1 \sqrt{n\log(n)}\tau_n^2\to 0,
\end{equation}
as $n \to \infty$. We then write 
\[
R_n^{(2)}(\theta_h,u) - R_n^{(2)}(\theta_{h'},u) = R_n^{(2)}(\theta_h,u) - R_n^{(2)}(\theta_\star,u) - \left[ R_n^{(2)}(\theta_{h'},u) - R_n^{(2)}(\theta_\star,u)\right].
\]
Recall that $R_n^{(2)}(\theta,u)=  T_n^{(2)}(\theta,u)$, and going back to the definition of $T_n^{(2)}$ in Lemma \ref{prop:Sn:decomp},
\begin{align*} 
& R_n^{(2)}(\theta,u) - R_n^{(2)}(\theta_\star,u) \\
& =  \int_0^1 \int_\Omega \left(\mathsf{G}_{\theta,u+\alpha(u-u_\star)}(x) - \mathsf{G}_{\theta_\star,u+\alpha(u-u_\star)}(x)\right)(u(x) - u_\star(x))\rmd x\rmd \alpha \\
& \;\;\;\;\;\;\;- \int_\Omega \left(\mathsf{G}_{\theta,u}(x) - \mathsf{G}_{\theta_\star,u}(x)\right)(u(x) - u_\star(x))\rmd x\end{align*}
We then apply (\ref{stab:G:3}) of Proposition \ref{prop:def_G} to get, for all $\|\theta-\theta_\star\|_2\leq\tau_n$, and $u\in\mathcal{U}_n$,
\begin{align*}
& \|R_n^{(2)}(\theta,u) - R_n^{(2)}(\theta_\star,u)\|_2 \\
& \leq 
\sqrt{q}\max_{1\leq j\leq q}\max\left(\|\mathsf{G}_{\theta,u+\alpha(u-u_\star),j} - \mathsf{G}_{\theta_\star,u+\alpha(u-u_\star),j}\|_{L^\infty(\Omega)}, \right. \\
& \left. \qquad \qquad \qquad \qquad \|\mathsf{G}_{\theta,u,j} - \mathsf{G}_{\theta_\star,u,j}\|_{L^\infty(\Omega)}\right) \int_\Omega |u(x) - u_\star(x)|\rmd x\\
& \leq C_{0}\|\theta-\theta_\star\|_2 \| u - u_\star\|_{L^2(\Omega)} \\
& \leq C_0 \tau_n \times \| u - u_\star\|_{L^2(\Omega)} \leq C_0 \tau_n r_n(0).
\end{align*}
Hence
\begin{align} 
\label{eq:bvm:b5} 
& \sup_{u\in\mathcal{U}_n}\;\sup_{h\in\mathbf{B}_n}\; n\left|\pscal{\mathsf{S}_n^{\rm db}(\theta_\star,u_\star)}{R_n^{(2)}(\theta_h,u) - R_n^{(2)}(\theta_\star,u)}_{V_n^{-1}}\right| \notag \\
& \hspace{15em}\leq C_1 \sqrt{n\log(n)}\tau_n r_n(0)\to 0,
\end{align}
as $n\to\infty$.
\\\\
\noindent
\underline{\bf On term $|(\log(\pi_0(\theta_{h'} )) - \log(\pi_0(\theta_\star))|$:} By Assumption \ref{assump:prior:pi0},
\begin{align*} \left| (\log(\pi_0(\theta_{h'} )) - \log(\pi_0(\theta_\star))\right| & \leq \|\nabla\log\pi_0(\theta_\star)\|_2 \frac{\|h\|_2}{\sqrt{n}} + \frac{\bar v\log(n)}{2}\frac{\|h\|_2^2}{n}\\
& \leq \bar v\log(n)\sqrt{q}\frac{\|h\|_2}{\sqrt{n}} + \frac{\bar v\log(n)}{2}\frac{\|h\|_2^2}{n}.
\end{align*}
Hence
\begin{equation} \label{eq:bvm:b6} 
\sup_{h\in \mathbf{B}_n} \left| (\log(\pi_0(\theta_{h'} )) - \log(\pi_0(\theta_\star))\right| \leq \bar v\log(n)\left(\sqrt{q}\tau_n + \tau_n^2\right) \to 0,\end{equation}
as $n\to\infty$.
Combining (\ref{eq:bvm:b1})-(\ref{eq:bvm:b6}), we deduce that
\[
\textbf{1}_{\bar \e_n}(\D_n)\;\sup_{h_1,h_2\in \textbf{B}_n}\;\left|\chi(h_1,\hat u)-\chi(h_2,\hat u)\right| \leq C_0\max\left(b_n\log(n),n \tau_n r_n^2\right),
\]
and the term on the right-hand side of the last display converges to $0$, as $n\to\infty$.
We deduce that 
\[\PE\left[\|P_{\textbf{B}_n} - Q_{\textbf{B}_n}\|_\tv\right]\leq 2\PP\left(\D_n\notin\bar \e\right) +  C_0\max\left(b_n\log(n),n \tau_n r_n^2\right).\]
By Lemma \ref{lem:tailbbound:dbScore}, applied with $\rho=\tau_n$, and $\bar\rho = (C L_1^{d/(2\beta)} \times r_n(0)^{1-d/(2\beta)})^{\alpha_*}$, $\PP(\D_n\notin\bar \e_{n,3})\leq n^{-C}$, for some constant $C>0$. The theorem then follows from the last display, and (\ref{eq:bvm:b0}).

\subsection{Technical supporting lemmas for Theorem \ref{thm:cond:db:rate}}
We collect here the proofs of the supporting lemmas for Theorem \ref{thm:cond:db:rate}. We first establish a key decomposition of the debiased score $\mathsf{S}_n^{\rm db}$. 
\begin{lemma}\label{prop:Sn:decomp}
Assume Assumptions \ref{assm:psi_smoothness}-\ref{assm:w}. Define
\[ T_n^{(1)}(\theta,u) \eqdef \sS_n^{\rm db}(\theta,u) - \sS^{\rm db}(\theta,u),\]
\[ 
T_n^{(2)}(\theta,u) \eqdef \int_0^1 \int_\Omega \left(\mathsf{G}_{\theta,u+\alpha(u-u_\star)}(x) - \mathsf{G}_{\theta,u}(x)\right)(u(x) - u_\star(x))\rmd x\rmd \alpha.
\]
For all $\theta\in\rset^q$, and $u\in W^{\beta,2}(\Omega,L)$ it holds  
\[ \mathsf{S}_n^{\rm db}(\theta,u) =
 (\nabla_\theta\cR)(\theta,u_\star)  + T_n^{(1)}(\theta,u) + T_n^{(2)}(\theta,u).
\]
\end{lemma}
\begin{proof} 
By definition,
\begin{equation}\label{lem:Sn:decom:eq1}
\mathsf{S}_n^{\rm db}(\theta,u) = \mathsf{S}^{\rm db}(\theta,u) + T_n^{(1)}(\theta,u).\end{equation}
And 
\begin{align*} 
\mathsf{S}^{\rm db}(\theta,u) &= (\nabla_\theta\cR)(\theta,u) +   \int_\Omega \mathsf{G}_{\theta,u}(x)\cdot \left(u_\star(x) - u(x)\right)\rmd x\\
& = (\nabla_\theta\cR)(\theta,u_\star) - \left((\nabla_\theta\cR)(\theta,u_\star) - (\nabla_\theta\cR)(\theta,u)\right) \\
& + \int_\Omega \mathsf{G}_{\theta,u}(x)\cdot \left(u_\star(x) - u(x)\right)\rmd x.
\end{align*}
We write $(\nabla_\theta\cR)(\theta,u_\star) - (\nabla_\theta\cR)(\theta,u) = (\nabla_\theta\cR)(\theta,u + (u_\star-u)) - (\nabla_\theta\cR)(\theta,u)$ that we expand using a first order Taylor expansion with integral remainder to obtain
\[
    (\nabla_\theta\cR)(\theta,u_\star) - (\nabla_\theta\cR)(\theta,u) = 
     \int_0^1 \int_\Omega \mathsf{G}_{\theta,u+\alpha(u-u_\star)}(x)\cdot(u_\star(x) - u(x))\rmd x\rmd \alpha.
\]
We thus have
\[ \mathsf{S}^{\rm db}(\theta,u) = (\nabla_\theta\cR)(\theta,u_\star)  + T_n^{(2)}(\theta,u),\]
which we combine with (\ref{lem:Sn:decom:eq1}) to obtain the stated  result.
\end{proof}

The next lemma gives an important lower-bound on the normalizing constant of the posterior distribution. 
\begin{lemma}\label{lem:control:NC}
 Suppose that Assumption \ref{assm:psi_smoothness}-\ref{assm:w} holds. Suppose also that the prior distribution $\pi_0$ satisfies Assumption \ref{assump:prior:pi0}, and the sample size assumption (\ref{tech:cond:1pc}) holds. For $\D_n\in\e_{n,0} \cap \e_{n,1}$, and for all $u\in\mathcal{U}_n$, it holds
 \begin{multline} 
 \left(\frac{1}{2\pi}\right)^{q/2}\int_{\rset^q}e^{-\frac{n}{2}\left(\|\mathsf{S}_n^{\rm db}(\theta,u)\|_{V_n^{-1}}^2-\|\mathsf{S}_n^{\rm db}(\theta_\star,u_\star)\|_{V_n^{-1}}^2\right) + (\log\pi_0(\theta) - \log\pi_0(\theta_\star))}\rmd\theta \\
 \geq \frac{1}{2} \exp\left(-c_0 n\lambda_{\max}(V_n^{-1}) \bar r_n^2 -\frac{q}{2}\log\left(\bar v\log n + n \|V_n^{-1/2}M\|_{{\rm op}}^2\right)\right),
 \end{multline}
 for some constant $c_0$, where
 \[ \bar r_n \eqdef \max\left(r_n^{2},\; C_1 \sqrt{\frac{Aq\log(n)}{n}}\right),\]
with $A$ and $C_1$ as in the definition of $\e_{n,0} \cap \e_{n,1}$.
\end{lemma}
\begin{proof}
With $M$ as in (\ref{def:M}), we set $b_0\eqdef \sqrt{\frac{2q\log(2q)}{\bar v\log n+2n\|V_n^{-1/2}M\|_{{\rm op}}^2}}$, and define
\[\mathbf{B}_0 \eqdef\left\{\theta\in\rset^q:\;\|\theta - \theta_\star\|_2 \leq b_0\right\}.\]
Recall that under (\ref{tech:cond:1pc}) $b_0\leq \min(1,r_n)$. %Without any loss of generality we assume that the absolute constant $A$ in the definition of the event $\e_n(1,A)$ is taken such that $b_0 \leq C_1 \sqrt{\frac{Aq\log(n)}{n}} = \bar r_n$. 
Since $(\nabla_\theta \cR)(\theta_\star,u_\star)=0$, and $T_n^{(2)}(\theta_\star,u_\star)=0$, using Lemma \ref{prop:Sn:decomp}, we have
\[\mathsf{S}_n^{\rm db}(\theta,u) - \mathsf{S}_n^{\rm db}(\theta_\star,u_\star) =  (\nabla_\theta \cR)(\theta,u_\star) +  \underbrace{T_n^{(1)}(\theta,u) - T_n^{(1)}(\theta_\star,u_\star)}_{R_n^{(1)}(\theta,u)} + \underbrace{ T_n^{(2)}(\theta,u)}_{R_n^{(2)}(\theta,u)},\]
and for some $\bar\theta$ on the segment between $\theta$ and $\theta_\star$, we have
\[ (\nabla_\theta \cR)(\theta, u_\star) = (\nabla_\theta \cR)(\theta_\star,u_\star) + M(\theta-\theta_\star) + \frac{1}{2} (\nabla^{(3)}_\theta\cR)(\bar\theta,u_\star)\cdot(\theta-\theta_\star,\theta-\theta_\star),\]
where $M\eqdef (\nabla^{(2)}_\theta \cR)(\theta,u_\star)\vert_{\theta=\theta_\star}$,  as defined in (\ref{def:M}). For convenience (recall that $(\nabla_\theta \cR)(\theta_\star,u_\star)=0$) we write the last display as
\[(\nabla_\theta \cR)(\theta, u_\star) = M(\theta-\theta_\star) +   R_n^{(3)}(\theta,u).\]
In summary,
\begin{equation}\label{eq10:lem:lb}
\mathsf{S}_n^{\rm db}(\theta,u) - \mathsf{S}_n^{\rm db}(\theta_\star,u_\star) = M(\theta - \theta_\star) + R_n(\theta,u),\;\; \mbox{ where }\;\; R_n(\theta,u) \eqdef  \sum_{k=1}^3 R_n^{(k)}(\theta,u).
\end{equation}
We define
\[ R^{(4)}(\theta_\star,\theta) \eqdef \log\pi_0(\theta) - \log\pi_0(\theta_\star) - \pscal{\nabla\log\pi_0(\theta_\star)}{\theta-\theta_\star}.\] 
Since,
\begin{multline*}
\|\mathsf{S}_n^{\rm db}(\theta,u)\|_{V_n^{-1}}^2-\|\mathsf{S}_n^{\rm db}(\theta_\star,u_\star)\|_{V_n^{-1}}^2 = \|\mathsf{S}_n^{\rm db}(\theta,u) - \mathsf{S}_n^{\rm db}(\theta_\star,u_\star)\|_{V_n^{-1}}^2 \\
+ 2\pscal{\mathsf{S}_n^{\rm db}(\theta_\star,u_\star)}{\mathsf{S}_n^{\rm db}(\theta,u) - \mathsf{S}_n^{\rm db}(\theta_\star,u_\star)}_{V_n^{-1}},\end{multline*}
with $\mu\eqdef -\nabla\log\pi_0(\theta_\star) +   n MV_n^{-1}\mathsf{S}_n^{\rm db}(\theta_\star,u_\star)$, we have
\begin{multline}\label{eq1:lem:lb}
\frac{n}{2}\left(\|\mathsf{S}_n^{\rm db}(\theta,u)\|_{V_n^{-1}}^2-\|\mathsf{S}_n^{\rm db}(\theta_\star,u_\star)\|_{V_n^{-1}}^2\right) - \log\pi_0(\theta) + \log\pi_0(\theta_\star) \\
= \frac{n}{2}\|M(\theta-\theta_\star) + R_n(\theta,u)\|_{V_n^{-1}}^2 - R^{(4)}(\theta_\star,\theta)  + \pscal{\mu}{\theta - \theta_\star}\\
+ n \pscal{\mathsf{S}_n^{\rm db}(\theta_\star,u_\star)}{R_n(\theta,u)}_{V_n^{-1}}.
\end{multline}
By Assumption \ref{assump:prior:pi0}, $-R^{(4)}(\theta_\star,\theta) \leq \bar v\log n \|\theta - \theta_\star\|_2^2/2$. Expanding the term $\|M(\theta-\theta_\star) + R_n(\theta,u)\|_{V_n^{-1}}^2$, we deduce that for $\|\theta-\theta_\star\|_2\leq b_0$,
\begin{multline}\label{eq2:lem:lb}
\frac{n}{2}\left(\|\mathsf{S}_n^{\rm db}(\theta,u)\|_{V_n^{-1}}^2-\|\mathsf{S}_n^{\rm db}(\theta_\star,u_\star)\|_{V_n^{-1}}^2\right) - \log\pi_0(\theta) + \log\pi_0(\theta_\star) \\
\leq \frac{1}{2}\left(n\|V_n^{-1/2}M\|_{{\rm op}}^2  + \bar v\log n\right)\|\theta - \theta_\star\|_2^2 + \pscal{\mu}{\theta-\theta_\star} \\
+ \frac{n}{2} \|R_n(\theta,u)\|_{V_n^{-1}}^2 + n\|R_n(\theta,u)\|_{V_n^{-1}}\left(\|\mathsf{S}_n^{\rm db}(\theta_\star,u_\star)\|_{V_n^{-1}}  + b_0 \|V_n^{-1/2}M\|_{{\rm op}}\right).
\end{multline}
To continue, we need upper bounds on $\|R_n(\theta,u)\|_{V_n^{-1}}$ and $\|\mathsf{S}_n^{\rm db}(\theta_\star,u_\star)\|_{V_n^{-1}}$. For $\|\theta-\theta_\star\|_2\leq b_0\leq 1$, $u\in\mathcal{U}_n$ and $\D_n\in\e_{n,1}$, 
\[\|R_n^{(1)}(\theta,u)\|_{2}  \leq 2C_1 \sqrt{\frac{Aq\log(n)}{n}}.\]
For $\|\theta-\theta_\star\|_2\leq b_0$, and $u\in\mathcal{U}_n$, we apply Proposition \ref{prop:def_G}-(2) to get
\[\|R^{(2)}_n(\theta,u)\|_2 = \|T^{(2)}_n(\theta,u)\|_2 \leq c_0 \sqrt{q}\times  \| u - u_\star\|_{W^{\tau,2}}^{2}\leq c_0 r_n^2,\]
for some constant $c_0$. Assumption \ref{assm:psi_smoothness} implies that for  $\|\theta-\theta_\star\|_2\leq b_0$,
\[ \|R_n^{(3)}(\theta,u)\|_2 \leq c_0 \|\theta -\theta_\star\|_2^2 \leq c_0 b_0^2\leq c_0 r_n^2,\]
since $b_0\leq r_n$ by Assumption (\ref{tech:cond:1pc}). For $\D_n\in\e_{n,0}$,
\[ \|\mathsf{S}_n^{\rm db}(\theta_\star,u_\star)\|_2 = \|T_n^{(1)}(\theta_\star,u_\star)\|_2 \leq  C_1 \sqrt{\frac{Aq\log(n)}{n}}.\]
In summary, for $u\in\mathcal{U}_n$, $\|\theta-\theta_\star\|_2 \leq b_0$, and $\D_n\in\e_{n,0}\cap \e_{n,1}$, we have
\begin{align*}
& \frac{n}{2} \|R_n(\theta,u)\|_{V_n^{-1}}^2 + n\|R_n(\theta,u)\|_{V_n^{-1}}\left(\|\mathsf{S}_n^{\rm db}(\theta_\star,u_\star)\|_{V_n^{-1}}  + b_0 \|V_n^{-1/2}M\|_{{\rm op}}\right)  \\
& \hspace{25em} \leq c_0 n\lambda_{\max}(V_n^{-1}) \bar r_n^2,
\end{align*}
for some constant $c_0$. The last display together with (\ref{eq2:lem:lb}) shows that for  $\D_n\in\e_{n,0}\cap \e_{n,1}$ and $u\in\mathcal{U}_n$,
\begin{multline*} 
\int_{\rset^q}e^{-\frac{n}{2}\left(\|\mathsf{S}_n^{\rm db}(\theta,u)\|_{V_n^{-1}}^2-\|\mathsf{S}_n^{\rm db}(\theta_\star,u_\star)\|_{V_n^{-1}}^2\right) + \log\pi_0(\theta) - \log\pi_0(\theta_\star)}\rmd\theta \\
\geq e^{-c_0 n\lambda_{\max}(V_n^{-1}) \bar r_n^2} \int_{\mathbf{B}_0}e^{-\frac{\bar v\log(n) + n \|V_n^{-1/2}M\|_{{\rm op}}^2}{2}\|\theta-\theta_\star\|_2^2 -\pscal{\mu}{\theta-\theta_\star}}\rmd\theta \\
= e^{-c_0 n\lambda_{\max}(V_n^{-1}) \bar r_n^2} \int_{\{v\in\rset^q:\;\|v\|_2\leq b_0\}}e^{-\frac{\bar v\log(n) + n \|V_n^{-1/2}M\|_{{\rm op}}^2}{2}\|v\|_2^2 -\pscal{\mu}{v}}\rmd v.
\end{multline*}
By Jensen's inequality, and setting $\lambda\eqdef \bar v\log(n) + n \|V_n^{-1/2}M\|_{{\rm op}}^2$ to save space,
\begin{multline*} 
\int_{\{v\in\rset^q:\;\|v\|_2\leq b_0\}}e^{ -\pscal{\mu}{v}} \frac{e^{-\frac{\lambda}{2}\|v\|_2^2}}{\int_{\{v\in\rset^q:\;\|v\|_2\leq b_0\}} e^{-\frac{\lambda}{2}\|x\|_2^2}\rmd x}\rmd v \\
\geq \exp\left(-\int_{\{v\in\rset^q:\;\|v\|_2\leq b_0\}}\pscal{\mu}{v} \frac{e^{-\frac{\lambda}{2}\|v\|_2^2}}{\int_{\{v\in\rset^q:\;\|v\|_2\leq b_0\}} e^{-\frac{\lambda}{2}\|x\|_2^2}\rmd x}\rmd v\right) =1.
\end{multline*}
Hence 
for $\D_n\in\e_{n,0}\cap \e_{n,1}$, and $u\in\mathcal{U}_n$
\begin{multline*} 
\int_{\rset^q}e^{-\frac{n}{2}\left(\|\mathsf{S}_n^{\rm db}(\theta,u)\|_{V_n^{-1}}^2-\|\mathsf{S}_n^{\rm db}(\theta_\star,u_\star)\|_{V_n^{-1}}^2\right)+ \log\pi_0(\theta) - \log\pi_0(\theta_\star)}\rmd\theta \\
\geq e^{-c_0 n\lambda_{\max}(V_n^{-1}) \bar r_n^2} \int_{\{v\in\rset^q:\;\|v\|_2\leq b_0\}} e^{-\frac{\bar v\log(n) + 2n \|V_n^{-1/2}M\|_{{\rm op}}^2}{2}\|v\|_2^2}\rmd v \\
=e^{-c_0 n\lambda_{\max}(V_n^{-1}) \bar r_n^2} \left(\frac{2\pi}{\bar v\log(n) + n \|V_n^{-1/2}M\|_{{\rm op}}^2}\right)^{q/2} \times \\
\PP\left(\|Z\|_2 \leq b_0\sqrt{\bar v\log(n) +n \|V_n^{-1/2}M\|_{{\rm op}}^2}\right),
\end{multline*}
where $Z_j\stackrel{i.i.d.}\sim \textbf{N}(0,1)$. Since $b_0\sqrt{\bar v\log n +n \|V_n^{-1/2}M\|_{{\rm op}}^2}= \sqrt{2q\log(2q)}$, the probability in the last display is at least $1/2$. Hence, for $\D_n\in\e_{n,0}\cap \e_{n,1}$, and $u\in\mathcal{U}_n$
\begin{multline*} 
\left(\frac{1}{2\pi}\right)^{q/2}\int_{\rset^q}e^{-\frac{n}{2}\left(\|\mathsf{S}_n^{\rm db}(\theta,u)\|_{V_n^{-1}}^2-\|\mathsf{S}_n^{\rm db}(\theta_\star,u_\star)\|_{V_n^{-1}}^2\right)+ \log\pi_0(\theta) - \log\pi_0(\theta_\star)}\rmd\theta \\
\geq \frac{1}{2} e^{-c_0 n\lambda_{\max}(V_n^{-1}) \bar r_n^2 -q\log(\bar v\log(n) + n \|V_n^{-1/2}M\|_{{\rm op}}^2)/2}.
\end{multline*}
\end{proof}

\begin{lemma}\label{lem:localization}
Suppose that Assumption \ref{assm:psi_smoothness}-\ref{assm:w} holds. Suppose also that the prior distribution $\pi_0$ satisfies Assumption \ref{assump:prior:pi0}, and the sample size assumptions (\ref{tech:cond:1pc}) and (\ref{tech:cond:2pc}) hold. We can find $c>0$ such that for all $\D_n\in\e_{n,0}\cap \e_{n,1}$ and $u\in\mathcal{U}_n$, we have 
    \[ \Pi^{(\theta\vert u)}\left(\left\{\theta\in\rset^q:\; \|\theta-\theta_\star\|_2>c\sqrt{q}\right\}\vert u,\D_n\right) \leq  2\exp\left(-\frac{c^2\underline{v}\log(n)q}{64}\right).\]
\end{lemma}
\begin{proof}
Given $b>0$, we define $\mathbf{B}_b^c\eqdef  \{\theta\in\rset^q:\;\|\theta - \theta_\star\|_2 > b\}$. For  any measurable set $B\subseteq\rset^q$, using the definition of $\Pi^{(\theta\vert u)}$ and the lower bound provided by Lemma \ref{lem:control:NC}, for $u\in\mathcal{U}_n$, we have 
\begin{align}
\label{eq1:lem:loc} 
& \Pi^{(\theta\vert u)}(B\vert u,\D_n) \notag \\
& = \frac{(2\pi)^{-q/2}\int_{B}e^{-\frac{n}{2}\left(\|\mathsf{S}_n^{\rm db}(\theta, u)\|_{V_n^{-1}}^2 - \|\mathsf{S}_n^{\rm db}(\theta_\star,u_\star)\|_{V_n^{-1}}^2\right)+(\log\pi_0(\theta) - \log\pi_0(\theta_\star))}\rmd\theta}{(2\pi)^{-q/2}\int_{\rset^q}e^{-\frac{n}{2}\left(\|\mathsf{S}_n^{\rm db}(\theta,u)\|_{V_n^{-1}}^2 - \|\mathsf{S}_n^{\rm db}(\theta_\star,u_\star)\|_{V_n^{-1}}^2\right) + (\log\pi_0(\theta) - \log\pi_0(\theta_\star))}\rmd\theta} \nonumber\\
& \leq 2\exp\left(c_0 n\lambda_{\max}(V_n^{-1}) \bar r_n^2 + \frac{q}{2}\log(\bar v\log n + n \|V_n^{-1/2}M\|_{{\rm op}}^2)\right)\nonumber\\
& \quad \quad\times (2\pi)^{-q/2}\int_{B}e^{-\frac{n}{2}\left(\|\mathsf{S}_n^{\rm db}(\theta,u)\|_{V_n^{-1}}^2 - \|\mathsf{S}_n^{\rm db}(\theta_\star,u_\star)\|_{V_n^{-1}}^2\right)+(\log\pi_0(\theta) - \log\pi_0(\theta_\star))}\rmd\theta\nonumber\\
& \leq 2\exp\left(c_0 n\lambda_{\max}(V_n^{-1}) \bar r_n^2 + \frac{q}{2}\log(\bar v\log n + n \|V_n^{-1/2}M\|_{{\rm op}}^2) +\frac{n}{2}\|\mathsf{S}_n^{\rm db}(\theta_\star,u_\star)\|_{V_n^{-1}}^2\right)\nonumber\\
 & \quad\quad \times (2\pi)^{-q/2} \int_{C}e^{-\frac{\underline{v}\log n}{2}\|\theta-\theta_\star\|_2^2-\pscal{\nabla\log\pi_0(\theta_\star)}{\theta-\theta}}\rmd\theta,\end{align}
where we recall that
\[ \bar r_n \eqdef \max\left(r_n^2,\; C_1\sqrt{\frac{A q\log(n)}{n}}\right) = C_1\sqrt{\frac{A q\log(n)}{n}}.\]
Furthermore, for $\D_n\in\e_{n,0}$, $\|\mathsf{S}_n^{\rm db}(\theta_\star,u_\star)\|_2^2  = \|T_n^{(1)}(\theta_\star,u_\star)\|_2^2 \leq c_0 \bar r_n^2$, for some  constant $c_0$. Therefore using (\ref{tech:cond:2pc}),
\[c_0 n\lambda_{\max}(V_n^{-1}) \bar r_n^2 + \frac{q}{2}\log(\bar v\log n + n \|V_n^{-1/2}M\|_{{\rm op}}^2) +\frac{n}{2}\|\mathsf{S}_n^{\rm db}(\theta_\star,u_\star)\|_{V_n^{-1}}^2 \leq C_0 q\log(n),\]
for some constant $C_0$. Combined with (\ref{eq1:lem:loc}) we obtain
\begin{equation}\label{eq2:lem:loc} 
\Pi^{(\theta\vert u)}(B \vert u,\D_n) \leq 2\exp\left(C_0 q\log(n)\right)
 \times (2\pi)^{-q/2} \int_{B}e^{-\frac{\underline{v}\log n}{2}\|\theta-\theta_\star\|_2^2-\pscal{\nabla\log\pi_0(\theta_\star)}{\theta-\theta}}\rmd\theta.
 \end{equation}
Recall from Assumption \ref{assump:prior:pi0} that $\|\nabla\log\pi_0(\theta_\star)\|_2 \leq \underline{v}\log(n)\sqrt{q}$. Using the fact that $-a x^2 + bx\leq 0$ for all $x\geq b/a$, for $\|\theta - \theta_\star\|_2 > c\sqrt{q}$ with $c\geq 4$, we have
\begin{align*}
& -\frac{\underline{v}\log n}{4}\|\theta-\theta_\star\|_2^2-\pscal{\nabla\log\pi_0(\theta_\star)}{\theta-\theta_\star} \\
& \leq -\frac{\underline{v}\log n}{4}\|\theta-\theta_\star\|_2^2 + \|\nabla\log\pi_0(\theta_\star)\|_2 \|\theta-\theta_\star\|_2 \leq 0.
\end{align*}
Therefore with $b\geq c\sqrt{q}$ for $c\geq 4$,
\begin{multline*} 
(2\pi)^{-q/2}\int_{\{\|\theta - \theta_\star\|_2 > b\}}e^{-\frac{\underline{v}\log n}{2}\|\theta-\theta_\star\|_2^2-\pscal{\nabla\log\pi_0(\theta_\star)}{\theta-\theta}}\rmd\theta \\
\leq (2\pi)^{-q/2}\int_{\{\|\theta - \theta_\star\|_2 > b\}}e^{-\frac{\underline{v}\log n}{4}\|\theta-\theta_\star\|_2^2}\rmd\theta = \left(\frac{2}{\underline{v}\log n}\right)^{q/2}\PP\left(\|Z\|_2 > b\sqrt{\frac{\underline{v}\log n}{2}}\right),
\end{multline*}
where $Z_1,\ldots,Z_q\stackrel{iid}{\sim}\textbf{N}(0,1)$. With $b= c\sqrt{q}$, $c\geq 4$, a tail bound for chi-squares distributions (see e.g. the Hanson-Wright ineq. in \cite{gine2021mathematical}) gives
\[ \PP\left(\|Z\|_2 > b\sqrt{\frac{\underline{v}\log n}{2}}\right) = \PP\left(\|Z\|_2^2 > c^2q\frac{\underline{v}\log(n)}{2}\right) \leq e^{-c^2\underline{v}\log(n)q/32}.\]
Then with  $\underline{v}\log(n)\geq 1$, we conclude that for $\D_n\in\e_{n,0}\cap \e_{n,1}$,
\[\Pi^{(\theta\vert u)}(\mathbf{B}_b^c \vert u,\D_n) \leq 
    2\exp\left(c_0 q\log(n) -\frac{c^2\underline{v}\log(n)q}{32}\right)\leq 2\exp\left(-\frac{c^2\underline{v}\log(n)q}{64}\right),\]
by choosing $c$ appropriately. Hence the result.
\end{proof}

\medskip

\section{Empirical processes results}\label{sec:emp:proc:results}
For the convenience of the reader we collect here all the empirical process results that we use. 
% \Yves{Can you add here the empirical process results from Fan et al that you used in Theorem 1? And harmonize the notation and briefly comment the results to make the entire section read nicely.}
% Our proof of Theorem \ref{thm:main_freq} uses Theorem 5.2 of \cite{chernozhukov2014gaussian}, which we state here for the ease of the readers: 
% \begin{theorem}[Theorem 5.2 of \cite{chernozhukov2014gaussian}]
% \label{thm:cck}
%     Suppose $\cF$ is a collection of functions with a bounded envelope $F$ (i.e. $\|F\|_\infty$ is finite). Assume $\|f\|_2^2 \le \sigma^2$ for all $f \in \cF$. Set $\delta = \sigma/\|F\|_2$. Define the entropy integral $\cJ(\delta, \cF, F)$ as: 
%     $$
%     \cJ(\delta, \cF, F) = \int_0^\delta \sup_Q \ \sqrt{1 + \log{\cN(\cF, L_2(Q), \eps \ \|F\|_{Q, 2})}} \ d\eps \,.
%     $$
%     Then, we have: 
%     \begin{equation}
%     \label{eq:upper_bound_cck}
%     \bbE\left[\sup_{f \in \cF}|\bbG_n(f)|\right] \lesssim \cJ(\delta, \cF, F)\|F\|_2 + \frac{\|F\|_\infty \cJ^2(\delta, \cF, F)}{\delta^2 \sqrt{n}}
%     \end{equation}
% \end{theorem}
First, we need the following general result from the empirical process theory, which essentially follows from Dudley's integral bound: 
\begin{lemma}\label{lem:dudley}
    Suppose $\cF$ is a collection of functions on $\Omega$, the domain of $X$, such that $\|f\|_\infty \le B$ for some $B > 0$ for all $f \in \cF$. Furthermore, suppose $X$ and $\eps$ are independent random variables, with $\eps$ being a sub-Gaussian random variable. Then: 
    $$
    \bbE\left[\sup_{f \in \cF}\left|\frac{1}{\sqrt{n}}\sum_{i = 1}^n \eps_i f(X_i)\right|\right] \le C\int_0^B \sqrt{1 + \log{\cN(u, \cF, L_\infty)}} \ du
    $$
    where $\cN(u, \cF, \infty)$ is the covering number $\cF$ at resolution $u$ with respect to $L_\infty$ norm. In particular, if the covering number of $\cF$ satisfies: 
    $$
    \log{\cN(u, \cF, L_\infty)} \le C_1 u^{-\alpha}
    $$
    for some $\alpha < 2$, then we have: 
    $$
     \bbE\left[\sup_{f \in \cF}\left|\frac{1}{\sqrt{n}}\sum_{i = 1}^n \eps_i f(X_i)\right|\right] \le C_2
    $$
    for some constnat $C_2$ depending on $(C, C_1, B)$. 
\end{lemma}
\begin{proof}
See Proposition 2.5.1 of \cite{talagrand2022upper} or Theorem 2.3.6 of \cite{gine2021mathematical}. 
\end{proof}
\noindent

\subsection{Deviation bounds}\label{sec:dev:bound:Bay}
The proof of Theorem \ref{thm:cond:db:rate} employs some  concentration inequalities for empirical processes that we derive here, following Theorem 5.11 of \cite{geer2000empirical}. For any random variable $X$, and $\kappa>0$, its $\kappa$-Bernstein norm is defined as: 
\begin{equation}
\label{eq:bernstein_norm}
\psi_\kappa(X) \eqdef \sqrt{2\kappa^2 \bbE\left[e^{\frac{|X|}{\kappa}} - 1 - \frac{|X|}{\kappa}\right]} \,.
\end{equation}
Let $X,X_{1:n}\stackrel{i.i.d.}{\sim} P$, and $\PP_n$ their empirical measure.  It is immediate to check that if we can find $(K,R)$ such that
$$
\bbE[|X|^m] \le \frac12 m! K^{m-2} R^2 \ \ \forall \ \ m = 2, 3, \dots ,$$
then $\psi_{2K}(X) \le \sqrt{2}R$. Given $R,K<\infty$, let $\cG$ be a collection of real-valued functions such that $\PE(g(X))=0$ and $\psi_K(g(X)) \leq R$ for all $g\in\mathcal{G}$. Let $H_{B, K}(u, \cG)$ denote the bracketing entropy of $\cG$ with respect to the pseudo-metric $\psi_{K}$ (we refer the reader to \cite{geer2000empirical} for definition and more details). Then the following holds.
\begin{theorem}[Theorem 5.11 of \cite{geer2000empirical}]
\label{thm:vdg}
Given $R,K<\infty$, let $\cG$ be a collection of real-valued functions such that $\PE(g(X))=0$ and $\psi_K(g(X)) \leq R$ for all $g\in\mathcal{G}$. There  exists a universal constant $C\geq 1$ such that for any $a,C_0,C_1$ satisfying: 
\begin{equation}\label{thm:vdg:cond}
C_0R\times\max\left(\int_0^1\sqrt{H_{B, K}(uR, \cG)} \ du, \ 1 \right) \le a \le \min\left(\frac{C_1 R}{K}, 8 \right)\sqrt{n}R
\end{equation}
and $C^2 \le C_0^2/(C_1 + 1)$: 
$$
\bbP\left(\sup_{g \in \cG}\left|\sqrt{n}(\bbP_n - P)g\right| \ge a\right) \le Ce^{- \frac{a^2}{C^2 (C_1 + 1)R^2}} \,.
$$
\end{theorem}

We apply this result to obtain the follow. 

\begin{lemma}\label{lem:tailbbound:dbScore}
Assume Assumptions \ref{assm:psi_smoothness}-\ref{assm:w}, and $\beta\geq 2\tau +d$. Given $0\leq s \leq \beta$, and for some constant $L_1>0$, we set
\[r_n(s) \eqdef L_1 \left(\log{n}\right)^{\frac12\left(1 - \frac{s}{\beta}\right)}\left(\frac{n}{\log{n}}\right)^{-\frac{\beta - s}{2\beta + d}},\]and define
\[\mathcal{U}_n \eqdef \left\{u\in W^{\beta,2}(\Omega):\;\; \|u - u_\star\|_{W^{s,2}(\Omega)} \leq r_n(s),\;  0\leq s\leq \beta\right\}.\]
Fix $\bar\rho>0$. We can find constants $1\leq C_1,C_2<\infty$ such that for all $\rho\in(0,\bar\rho]$, $A\geq 2$, and all $n$ such that 
\begin{equation}\label{cond:n:lem:dev}
C L_1^{d/(2\beta)} \times r_n(0)^{1-d/(2\beta)} \leq \min(1, \bar \rho^{1/\alpha_*}), \qquad\text{and}\qquad n \geq C_1 A\log(n),
\end{equation}
where $C$ is as in Theorem \ref{thm:interpolation}, it holds
\[\PP\left(\sup_{u\in\mathcal{U}_n}\; \;\sup_{\theta\in\mathbf{B}_\rho(\theta_\star)}\;\;\max_{1\leq j\leq q}\;\;\sqrt{n}|T_{n,j}^{(1)}(\theta,u) - T_{n,j}^{(1)}(\theta_\star,u_\star)| > C_2\bar \rho\sqrt{A\log(n)} \right) \leq \frac{1}{n^{A/2}}.\]
\end{lemma}
\begin{proof}
Throughout the proof, $C$ will denote a positive constant that we do not track. Fix a component $1\leq j\leq q$. By Proposition \ref{prop:def_G}, we can find $c_1=c_1(\bar \rho,L,\|\theta_\star\|_2,\|u_\star\|_{W^{\beta,2}})$ such that for all $u\in\cA$, $\theta\in\mathbf{B}_{\rho}(\theta_\star)$, $\|\mathsf{G}_{\theta,u,j}\|_\infty \leq c_1$ and 
\[ 
\|\mathsf{G}_{\theta,u,j} - \mathsf{G}_{\theta_\star,u_\star,j}\|_\infty \leq c_1 (\|\theta-\theta_\star\|_2 + \|u-u_\star\|_\infty^{\alpha_*}) \leq c_1 (\bar \rho  + \|u-u_\star\|_\infty^{\alpha_*}) .\]
Furthermore, by the Gagliardo-Nirenberg interpolation (see Theorem \ref{thm:interpolation}) (setting $r=\infty$,  $j=0$, and $\theta= d/(2\beta)$ in that theorem):
\[\|u-u_\star\|_\infty \leq C\|u-u_\star\|_{L^2(\Omega)}^{1-d/(2\beta)}\|u-u_\star\|_{W^{\beta,2}(\Omega)}^{d/(2\beta)} \leq C r_n^{1-d/(2\beta)}(0)\times r_n(\beta)^{d/(2\beta)} \leq \bar \rho^{1/\alpha_*},\]
where the last inequality uses the first condition in (\ref{cond:n:lem:dev}). It follows that for all $u\in\cA$, $\theta\in\mathbf{B}_\rho(\theta_\star)$,
\begin{equation}
    \label{eq:bounds:famil:G}
\|\mathsf{G}_{\theta,u,j}\|_\infty \leq c_1, \;\; \mbox{ and }\;\; \|\mathsf{G}_{\theta,u,j} - \mathsf{G}_{\theta_\star,u_\star,j}\|_\infty \leq 2c_1\bar \rho.
\end{equation}
Since $Y_i = u_\star(X_i)+\epsilon_i$, the debiased score satisfies $\sS_n^{\rm db}(\theta,u) - \sS^{\rm db}(\theta,u) = \frac{1}{n}\sum_{i=1}^n\epsilon_i\mathsf{G}_{\theta,u}(X_i) - \PE[\epsilon\mathsf{G}_{\theta,u}(X)] + \frac{1}{n}\sum_{i=1}^n(u_\star-u)(X_i)\mathsf{G}_{\theta,u}(X_i) - \PE[(u_\star-u)(X)\mathsf{G}_{\theta,u}(X)]$.  Since $\PE[\epsilon\mathsf{G}_{\theta,u}(X)]=0$ (as $\epsilon\perp X$), splitting the first sum at $\mathsf{G}_{\theta_\star,u_\star}$ gives
\begin{align*} 
\sS_{n}^{{\rm db}}(\theta,u) - \sS^{\rm db}(\theta,u) & = \frac{1}{n}\sum_{i=1}^n \epsilon_i \mathsf{G}_{\theta_\star,u_\star}(X_i) \\
& + \frac{1}{n}\sum_{i=1}^n \epsilon_i \left(\mathsf{G}_{\theta,u}(X_i) - \mathsf{G}_{\theta_\star,u_\star}(X_i)\right) \\
& + \frac{1}{n}\sum_{i=1}^n (u_\star-u)(X_i)\mathsf{G}_{\theta,u}(X_i) - \PE\left[(u_\star-u)(X)\mathsf{G}_{\theta,u}(X)\right].
\end{align*}
Therefore,
\begin{align*} 
& T_{n,j}^{(1)}(\theta,u) - T_{n,j}^{(1)}(\theta_\star,u_\star)  \\
& = \frac{1}{n}\sum_{i=1}^n \epsilon_i \left(\mathsf{G}_{\theta,u}(X_i) - \mathsf{G}_{\theta_\star,u_\star}(X_i)\right) \\
& + \frac{1}{n}\sum_{i=1}^n (u_\star-u)(X_i)\mathsf{G}_{\theta,u}(X_i) - \PE\left[(u_\star-u)(X)\mathsf{G}_{\theta,u}(X)\right],
\end{align*}
and
\begin{align*}
& \sup_{u\in\mathcal{U}_n}\; \;\sup_{\theta\in\mathbf{B}_\rho(\theta_\star)}\;\;\sqrt{n}|T_{n,j}^{(1)}(\theta,u) - T_{n,j}^{(1)}(\theta_\star,u_\star)| \\
& \leq  \sigma_\epsilon\underbrace{\sup_{g\in\mathcal{G}_1} \left|\frac{1}{\sqrt{n}}\sum_{i=1}^ng(\epsilon_i,X_i)\right|}_{H_1} + c_1 \underbrace{\sup_{g\in\mathcal{G}_2} \left|\frac{1}{\sqrt{n}}\sum_{i=1}^ng(X_i)\right|}_{H_2},
\end{align*}
where
\[\mathcal{G}_1\eqdef\left\{(\eta,x)\mapsto \frac{\eta}{\sigma_\epsilon} \left(\mathsf{G}_{\theta,u,j}(x) - \mathsf{G}_{\theta_\star,u_\star,j}(x)\right),\;\;\theta\in\mathbf{B}_\rho(\theta_\star),\;u\in\mathcal{U}_n\right\},\]
\[\mathcal{G}_2\eqdef\left\{x\mapsto (u_\star-u)(x)\frac{\mathsf{G}_{\theta,u,j}(x)}{c_1}-\PE\left[(u_\star-u)(X)\frac{\mathsf{G}_{\theta,u,j}(X)}{c_1}\right],\;\;\theta\in\mathbf{B}_\rho(\theta_\star),\;u\in\mathcal{U}_n\right\}.\]

\medskip
\noindent\textbf{Bounding $H_1$.}\; Since $\epsilon/\sigma_\epsilon$ is sub-Gaussian with unit variance and (\ref{eq:bounds:famil:G}) gives $\|\mathsf{G}_{\theta,u,j}-\mathsf{G}_{\theta_\star,u_\star,j}\|_\infty\leq 2c_1\bar\rho$, for all $g\in\mathcal{G}_1$ and $m\geq 2$,
\[\PE\left[|g(\epsilon,X)|^m\right] \leq (2c_1\bar\rho)^m \frac{m!}{2}.\]
Therefore, as commenting on at the beginning of Section \ref{sec:dev:bound:Bay}, choosing $K_1\eqdef 4c_1\rho$, and $R_1\eqdef 2\sqrt{2}\,c_1\bar\rho$,  we have
\[\psi_{K_1}(g(\epsilon,X))\leq R_1, \quad\mbox{for all }g\in\mathcal{G}_1.\]

\smallskip
\noindent\textit{Bracketing entropy of $\mathcal{G}_1$.}\; For $\delta\in(0,R_1)$, let $\{u^{\ell_1}\}$ be a $(\delta/(8c_1))^{1/\alpha_\star}$-cover of $\mathcal{U}_n$ in $\|\cdot\|_\infty$ with $N_u(\delta)$ points, and $\{\theta^{\ell_2}\}$ be a $\delta/(8c_1)$-cover of $\mathbf{B}_\rho(\theta_\star)$ in $\|\cdot\|_2$ with $N_\theta(\delta)\leq(1+16c_1\bar\rho/\delta)^q$ points. Define brackets
\begin{multline*} 
L^{\ell_1,\ell_2}(\eta,x) \eqdef \frac{\eta}{\sigma_\epsilon}\left(\mathsf{G}_{\theta^{\ell_2},u^{\ell_1},j}(x) - \mathsf{G}_{\theta_\star,u_\star,j}(x)\right) -\frac{\delta|\eta|}{4\sigma_\epsilon},\quad \\
U^{\ell_1,\ell_2}(\eta,x) \eqdef \frac{\eta}{\sigma_\epsilon}\left(\mathsf{G}_{\theta^{\ell_2},u^{\ell_1},j}(x) - \mathsf{G}_{\theta_\star,u_\star,j}(x)\right) +\frac{\delta|\eta|}{4\sigma_\epsilon}.
\end{multline*}
For any $(u,\theta)\in\cA\times\mathbf{B}_\rho(\theta_\star)$, pick $(\ell_1,\ell_2)$ with $\|\theta-\theta^{\ell_2}\|_2\leq\delta/(8c_1)$ and $\|u-u^{\ell_1}\|_\infty\leq(\delta/(8c_1))^{1/\alpha_*}$. Then by Proposition~\ref{prop:def_G},
\[\|\mathsf{G}_{\theta,u,j}-\mathsf{G}_{\theta^{\ell_2},u^{\ell_1},j}\|_\infty \leq c_1\!\left(\frac{\delta}{8c_1}+\frac{\delta}{8c_1}\right)=\frac{\delta}{4},\]
which gives $L^{\ell_1,\ell_2}(\eta,x)\leq g_{\theta,u}(\eta,x)\leq U^{\ell_1,\ell_2}(\eta,x)$ for all $\eta,x$. Moreover, $U^{\ell_1,\ell_2}-L^{\ell_1,\ell_2}=\frac{\delta}{2\sigma_\epsilon}|\eta|$, so for $m\geq 2$,
\[\PE\left[\left|U^{\ell_1,\ell_2}(\epsilon,X)-L^{\ell_1,\ell_2}(\epsilon,X)\right|^m\right]\leq\left(\frac{\delta}{2}\right)^m\frac{m!}{2} \leq (2c_1\rho)^{m-2}\left(\frac{\delta}{2}\right)^2\frac{m!}{2},\]
which implies that $\psi_{K_1}(U^{\ell_1,\ell_2}-L^{\ell_1,\ell_2})\leq\sqrt{\delta/2} \leq  \delta$ for $\delta\leq R_1$. Since $\mathcal{U}_n\subseteq\{u:\|u-u_\star\|_{W^{\beta,2}}\leq r_n(\beta)\}$, the Birman--Solomyak theorem (Theorem~5.2 of \cite{birman1967piecewise}) gives the entropy bound $\log N_u(\delta)\leq C_{\rm BS}\,r_n(\beta)^{d/\beta}\,\delta^{-d/(\beta\alpha_*)}$ for an absolute constant $C_{\rm BS}$, and therefore
\begin{equation}\label{eq:entropy:G1}
H_{B,K_1}(\delta,\mathcal{G}_1)\leq q\log\!\left(1+\frac{16c_1\bar\rho}{\delta}\right)+C_{\rm BS}\,r_n(\beta)^{d/\beta}\,\left(\frac{1}{\delta}\right)^{d/(\beta\alpha_*)}.
\end{equation}

Using (\ref{eq:entropy:G1}) and $R_1=2\sqrt{2}c_1\bar\rho$, so that $16c_1\bar\rho/R_1=4\sqrt{2}$, we bound the entropy integral by splitting:
\begin{align*}
& \int_0^1\!\sqrt{H_{B,K_1}(uR_1,\mathcal{G}_1)}\,du \\
&\leq \sqrt{q}\int_0^1\!\sqrt{\log\!\left(1+\frac{4\sqrt{2}}{u}\right)}\,du
  +\sqrt{C_{\rm BS}}\,L_1^{d/(2\beta)}\int_0^1\!(uR_1)^{-d/(2\beta\alpha_*)}\,du \\
&= C''\sqrt{q}
  +\frac{\sqrt{C_{\rm BS}}\,L_1^{d/(2\beta)}\,(2\sqrt{2}c_1\bar\rho)^{1-d/(2\beta\alpha_*)}}{1-d/(2\beta\alpha_*)},
\end{align*}
where the first integral is a finite constant $C''$, $L\eqdef r_n(\beta)$, and where the second integral converges because $\beta\alpha_*>d/2$. Using $(A+B)\leq \sqrt{2(A^2+B^2)}$ valid for all $A,B\geq 0$, gives 
\[
\int_0^1\sqrt{H_{B,K_1}(uR_1,\mathcal{G}_1)}\,du\leq C'\sqrt{q+V_1(\bar\rho,L_1)},
\]
with
\[V_1(\bar\rho,L_1) = \left(\frac{\sqrt{C_{\rm BS}}\,L_1^{d/(2\beta)}\,(2\sqrt{2}c_1\bar\rho)^{1-d/(2\beta\alpha_*)}}{1-d/(2\beta\alpha_*)}\right)^2
\] 
We set
\[a \eqdef C'C_0 R_1\sqrt{\left(q+V_1(\bar\rho,L)\right)A\log n},\]
where $C_0 = 2C$ with $C$ as in Theorem \ref{thm:vdg}, and $C'$ can be taken satisfying $C'\geq 1$. The lower bound in (\ref{thm:vdg:cond}) is then satisfied by construction. For the upper bound, we choose $C_1=1$. since $K_1=4c_1\bar\rho$, $R_1=2\sqrt{2}c_1\bar\rho$, and $C_1R_1/K_1=\sqrt{2}/2$, we need $a\leq \sqrt{2}/2\cdot R_1\sqrt{n}$, i.e.,
\[C'C_0\sqrt{\left(q+V_1(\bar\rho,L)\right)A\log n}\leq \frac{\sqrt{2}}{2}\sqrt{n},\]
which can be written as $n\geq C_1 A\log(n)$, for some constant $C_1$. With $A\geq 1$, $C_0^2\geq C^2(C_1+1)/A$ so that $a_2^2/(C^2(C_1+1)R_1^2)\geq A\log n$, and Theorem \ref{thm:vdg} yields
\[\PP\left(\sigma_\epsilon\sup_{g\in\mathcal{G}_1}\left|\frac{1}{\sqrt{n}}\sum_{i=1}^ng(\epsilon_i,X_i)\right|>a\right)\leq C\exp\!\left(-\frac{a_2^2}{C^2(C_1+1)R_1^2}\right)\leq\frac{C}{n^A}\leq\frac{2}{n^A}.\]

\medskip
\noindent\textbf{Bounding $H_2$.}\; Each $g\in\mathcal{G}_2$ is centered ($\PE[g(X)]=0$ by definition) and bounded: $\|g\|_\infty\leq 2\|u-u_\star\|_\infty\leq 2 C\times r_n(0)^{1-d/(2\beta)}r_n(\beta)^{d/(2\beta)} \leq 2\bar \rho$, by Gagliardo-Nirenberg, for $n$ large enough to satisfy (\ref{cond:n:lem:dev}). Bounded random variables satisfy $\PE[|g(X)|^m]\leq (2\bar\rho)^m m!/2$, so $\psi_{4\bar\rho}(g(X))\leq 2\sqrt{2}\bar\rho$. 

\smallskip
\noindent\textit{Bracketing entropy of $\mathcal{G}_2$.}\; With $R_1 = 2\sqrt{2}\bar\rho$, $K_1 = 4\bar\rho$, and for $\delta\in(0,R_1)$, let $\{u^{\ell_1}\}$ be a $(\delta/(32))^{1/\alpha_\star}$-cover of $\mathcal{U}_n$ in $\|\cdot\|_\infty$ with $N_u(\delta)$ points, and $\{\theta^{\ell_2}\}$ be a $\delta/(32)$-cover of $\mathbf{B}_\rho(\theta_\star)$ in $\|\cdot\|_2$ with $N_\theta(\delta)\leq(1+64\rho/\delta)^q$ points. For any $(u,\theta)\in\cA\times\mathbf{B}_\rho(\theta_\star)$, pick $(\ell_1,\ell_2)$ with $\|\theta-\theta^{\ell_2}\|_2\leq\delta/(32)$ and $\|u-u^{\ell_1}\|_\infty\leq(\delta/(32))^{1/\alpha_*}$. Then
\begin{align*}
    & \|(u-u_\star)\mathsf{G}_{\theta,u,j} - (u^{\ell_1} - u_\star)\mathsf{G}_{\theta^{\ell_2},u^{\ell_1},j}\|_\infty \\
    & \leq  \|(u-u_\star)(\mathsf{G}_{\theta,u,j} - \mathsf{G}_{\theta^{\ell_2},u^{\ell_1},j})\|_\infty + \|(u-u^{\ell_1})\mathsf{G}_{\theta^{\ell_2},u^{\ell_1},j}\|_\infty \\
    & \leq c_1\|u-u_\star\|_\infty\left(\frac{\delta}{32} + \|u^{\ell_1} - u\|_\infty^{\alpha_\star}\right) + c_1 \left(\frac{\delta}{32}\right)^{1/\alpha_*}\\
    & \leq c_1\left(\frac{\delta}{16} + \frac{\delta}{16}\right) \leq c_1\frac{\delta}{8},
\end{align*}
where we use the fact for $u\in\mathcal{U}_n$, $\|u-u_\star\|_\infty \leq 1$, as a consequence of (\ref{cond:n:lem:dev}). This result implies that $\{(u^{\ell_1} - u_\star)\mathsf{G}_{\theta^{\ell_2},u^{\ell_1},j},\;\ell_1,\ell_2\}$ is a $(\delta/4)$-cover for $\mathcal{G}_2$, in the uniform norm, leading to brackets of size $\delta/2$ with $(4\bar\rho)$-Bernstein norm of at most $\delta$. A similar bracketing construction as above  with the Birman--Solomyak bound gives
\[H_{B,2}(\delta,\mathcal{G}_2)\leq q\log\!\left(1+\frac{64\bar \rho}{\delta}\right)+C_{\rm BS}\,L_1^{d/\beta}\,\delta^{-d/(\beta\alpha_*)}.\]
We carry the same entropy calculations to obtain 
\[ \int_0^{2\sqrt{2}\bar\rho}\sqrt{H_{B,2}(v,\mathcal{G}_2)}\,dv\leq C'\sqrt{q+V_1(\bar\rho,L)},\]
for some constant $C'$. Therefore, with
\[a'\eqdef C'C_0 (2\sqrt{2}\bar\rho) \sqrt{\left(q+V_1(\bar\rho,L)\right)A\log n},\]
and with the same sample size condition as above, applying Theorem \ref{thm:vdg} to $\mathcal{G}_2$ gives
\[\PP\left(\sup_{g\in\mathcal{G}_2}\left|\frac{1}{\sqrt{n}}\sum_{i=1}^ng(X_i)\right|>a'\right)\leq\frac{2}{n^A}.\]

\medskip
\noindent\textbf{Conclusion.}\; By the triangle inequality and a union bound over $j=1,\ldots,q$, we can find $C_2$ such that
\[\PP\left(\max_{1\leq j\leq q}\;\sup_{u,\theta}\; \sqrt{n}|T_{n,j}^{(1)}(\theta,u) - T_{n,j}^{(1)}(\theta_\star,u_\star)|>C_2\bar\rho\sqrt{A\log(n)}\right)\leq q\cdot\frac{6}{n^A}\leq\frac{1}{n^{A/2}},\]
where the last inequality uses $6q\leq n^{A/2}$, i.e., $n\geq(6q)^{2/A}$, which is implied by the second condition in (\ref{cond:n:lem:dev}) for large enough $C_1$.  
\end{proof}

\section{Neural network based approximation on Sobolev space}
In this section, we present a theorem about the optimal approximation error of a function $f \in W^{\beta, 2}((0, 1)^d, L)$ using a RePU($\nu$) neural network (where $\nu \ge \beta$) in terms of the number of parameters. Our main proof idea is closely related to that of \cite{belomestny2023simultaneous}, where the authors established the approximation error for H\"{o}lder smooth functions, i.e., $C^{\beta}(\Omega, L)$. The key differences being: i) our function class $W^{\beta, 2}$ which is larger than $C^{\beta}$, and ii) we use RePU($\nu$) activation function for some $\nu \ge \beta$ to ensure our neural networks have finite $W^{\beta, 2}$ norm. 
Throughout this section, Sobolev derivatives are understood in the weak sense.

\begin{theorem}[Approximation of $W^{\nu,2}((0,1)^d)$ functions by RePU($\nu$) networks]
\label{thm:RePU-spline-approx}
Let $d\ge 1$ and let $\nu\in\mathbb{N}$ with $\nu \ge 3$.  
Let $\Omega=(0,1)^d$ and suppose $f\in W^{\nu,2}(\Omega)$. Then given any $K \in \bbN$, there exists a RePU($\nu$) neural network $\cN_K$ with $\mathrm{depth}(\mathcal N_K) \le c_1(1 + \lceil \log_2 d\rceil)$, $\mathrm{width}(\mathcal N_K)
\le c_2(K+\nu)^d$, $\mathrm{sparsity}(\mathcal N_K)\le c_3 (K+\nu)^d$, such that: 
\begin{equation}
\label{eq:approx-rate}
\|f-\mathcal N_K\|_{W^{r,2}(\Omega)}
\;\le\;
C_{\nu,d}\, K^{-(\nu-r)}\,\|f\|_{W^{\nu,2}(\Omega)},
\qquad r=0,1,\dots,\nu,
\end{equation}
Here the constants $c_1, c_2, c_3, C_{\nu, d}$ depends on $(\nu, d)$, but not on $K$. 
% For any integer $k\ge 1$, define the mesh size $h:=1/k$ and let
% \[
% n:=k+p.
% \]
% Then there exists a RePU($\nu$) neural network $\mathcal N_k:\Omega\to\mathbb{R}$ such that
% \begin{equation}
% \label{eq:approx-rate}
% \|f-\mathcal N_k\|_{W^{r,2}(\Omega)}
% \;\le\;
% C_{p,d}\, h^{\,p-r}\,\|f\|_{W^{p,2}(\Omega)},
% \qquad r=0,1,\dots,p,
% \end{equation}
% where $C_{p,d}>0$ depends only on $(p,d)$.

% Moreover, the network $\mathcal N_k$ satisfies the following quantitative bounds:
% \begin{itemize}
% \item \textbf{Depth (RePU layers):}
% \[
% \mathrm{depth}(\mathcal N_k)=1+\lceil \log_2 d\rceil.
% \]
% \item \textbf{Maximum width:}
% \[
% \mathrm{width}(\mathcal N_k)
% \;\asymp\;
% \lfloor p/2\rfloor\,(k+p)^d.
% \]
% \item \textbf{Sparsity (number of nonzero parameters):}
% \[
% \mathrm{nnz}(\mathcal N_k)=\Theta\!\big((k+p)^d\big),
% \]
% up to a multiplicative constant depending only on $p$.
% \item \textbf{Magnitude of weights:}
% \begin{align*}
% \text{(i) spline-evaluation layers:} &\quad O(k^p),\\
% \text{(ii) multiplier layers:} &\quad O_p(1),\\
% \text{(iii) final linear layer:} &\quad |c_{\mathbf j}|\ \text{given by the spline coefficients of } f.
% \end{align*}
% \end{itemize}
\end{theorem}

\begin{proof}
The proof proceeds in six steps. 
\\\\
\noindent
{\bf Step 1: Sobolev approximation by tensor-product splines.}
Let $\{B^{K,\nu}_j\}_{j=1}^{K+\nu}$ denote the univariate B-spline basis of degree $\nu$ on $(0,1)$ associated with the open uniform knot sequence of spacing $h=1/K$. Let
$$
\cS_{K,\nu}^{(1)} := {\rm span}\left\{B_j^{K,\nu}:j\in[K+\nu]\right\}
$$
denote the corresponding univariate spline space. 
Define the $d$-dimensional tensor-product spline space
\[
S_{K,\nu} := \bigotimes_{i=1}^d \mathcal S_{K,\nu}^{(1)} =
\mathrm{span}\Big\{
\Phi_{\mathbf j}(x)
:=
\prod_{i=1}^d B^{K,\nu}_{j_i}(x_i)
\;\Big|\;
\mathbf j=(j_1,\dots,j_d)\in[K+\nu]^d
\Big\}.
\]
By classical spline approximation theory (e.g. see Theorem 6.25 and Corollary 6.26 of \cite{schumaker2007spline}), there exists a linear operator $Q_K:W^{\nu,2}(\Omega)\to S_{K,\nu}$ such that
\[
\|f-Q_K f\|_{W^{r,2}(\Omega)}
\le
C_{\nu,d}\, h^{\nu-r}\,\|f\|_{W^{\nu,2}(\Omega)},
\qquad r=0,1,\dots,\nu.
\]
Writing
\[
Q_K f(x) = \sum_{\mathbf j\in[K+\nu]^d} c_{\mathbf j}\,\Phi_{\mathbf j}(x),
\]
we obtain the spline approximant that will be implemented by the neural network.
\\\\
\noindent
{\bf Step 2: Exact RePU representation of univariate B-splines:}
% Each univariate B-spline admits a truncated-power representation:
% for every $j\in[K+\nu]$,
% \[
% B^{K,\nu}_j(x) = \sum_{m=0}^{\nu+1}\alpha_{j,m}\,(x-t_{j+m})_+^\nu,
% \]
% where $\{t_{j+m}\}$ are neighboring knots and the coefficients satisfy
% \[
% |\alpha_{j,m}|\le C_\nu \,h^{-\nu}=C_\nu\,K^\nu.
% \]
% For boundary B-splines associated with repeated knots in the open knot sequence, the same representation is understood in the limiting divided-difference sense; equivalently, these boundary splines are ordinary polynomials on the boundary cells and can be represented by the same finite collection of shifted RePU\((\nu)\) units.
% The coefficient bounds remain $ C_\nu \,h^{-\nu}$. Since $(x-t)_+^\nu=\sigma_\nu(x-t)$, each $B^{K,\nu}_j(x)$ is computed exactly by a one-hidden-layer RePU($\nu$) subnetwork with width $\nu+2$. Applying this construction to each coordinate $x_i$ and each index $j$, a single RePU layer of width $W_0=d(K+\nu)(\nu+2)$ computes all univariate spline values $\{B^{K,\nu}_j(x_i)\}_{i,j}$.
Let the univariate knots be $0 = \xi_0 < \xi_1 < \cdots < \xi_K = 1$, where $\xi_r = rh$ and $h = 1/K$. Every spline $s \in S^{(1)}_{K,\nu}$ admits the truncated-power representation (e.g., see \cite[Chapter VIII]{de1978practical})
$$
s(x) = p_s(x) + \frac{1}{\nu!}\sum_{r=1}^{K-1}\Delta_r(s)\,(x-\xi_r)_+^{\nu}, \qquad x \in [0,1],
$$
where $p_s \in \mathcal{P}_\nu$ is the polynomial that agrees with $s$ on the first knot interval (Here $\mathcal P_\nu$ denotes the space of polynomials of degree at most $\nu$) and $\Delta_r(s) = D^\nu s(\xi_r+) - D^\nu s(\xi_r-)$. 
Now for $s = B_j^{K,\nu}$, $B_j^{K,\nu}$ has support on at most $\nu + 1$ consecutive knot intervals, at most $\nu + 2$ of the coefficients $\Delta_r(B_j^{K,\nu})$ are nonzero, along with
$$
\max_r \left|\Delta_r(B_j^{K,\nu})\right| \le C_\nu h^{-\nu}.
$$
To represent the first polynomial term, let us fix distinct constants $a_\ell = \ell + 1$, $0 \le \ell \le \nu$. Since $\{(x + a_\ell)^\nu : 0 \le \ell \le \nu\}$ forms a basis of $\mathcal{P}_\nu$, there exist coefficients $c_{j,\ell}$
such that
$$
p_j(x) = \sum_{\ell=0}^{\nu} \gamma_{j,\ell}\,(x + a_\ell)_+^{\nu}, \qquad x \in [0,1],
$$
with $\max_\ell |\gamma_{j,\ell}| \le C_\nu h^{-\nu}$. Therefore we have: 
$$
B_j^{K,\nu}(x) = \sum_{\ell=0}^{\nu} \gamma_{j,\ell}\,\sigma_\nu(x + a_\ell) + \frac{1}{\nu!}\sum_{r \in I_j}\Delta_r(B_j^{K,\nu})\,\sigma_\nu(x - \xi_r),
$$
where $|I_j| \le \nu + 2$. Thus we can represent $B_j^{K,\nu}$ exactly on $[0,1]$ by a one-hidden-layer $\mathrm{RePU}(\nu)$ network of width at most $2\nu + 3$. Representing all univariate spline basis functions $\{B^{K,\nu}_j(x_i)\}_{i,j}$ in parallel therefore requires a single RePU layer of width at most $W_0 = d(K + \nu)(2\nu + 3)$. 
\\\\
\noindent
{\bf Step 3: Exact one-hidden-layer RePU($\nu$) of $f(x, y) = xy$:} Recall that RePU($\nu$) activation function is defined as $\sigma_\nu(t):=(t_+)^\nu=\max\{t,0\}^\nu$. We now construct a one-hidden-layer network $\mathrm{Mult}_\nu(u,v)$ such that
\[
\mathrm{Mult}_\nu(u,v)=uv,\qquad \forall (u,v)\in[0,1]^2,
\]
using $m=4\lfloor \nu/2\rfloor$ hidden units, with input weights in $\{-1,0,1\}$, and with biases and output weights bounded by constants depending only on $\nu$. For any real numbers $u,v$, we have the following identity
\begin{equation}
\label{eq:polarization}
uv=\frac{(u+v)^2-(u-v)^2}{4}.
\end{equation}
Thus it suffices to represent $t\mapsto t^2$ on $t\in[-1,2]$ using RePU($\nu$) units. Let $q = \lfloor \nu/2 \rfloor \ge 1$ as $\nu \ge 3$. Choose $q$ distinct shifts $a_i:=2+i$ for $1 \le i \le q$. For any $t\in[-1,2]$ and any $i\in[q]$, we have $a_i\pm t\ge 2+i-2\ge 1$, hence $(a_i\pm t)_+^\nu=(a_i\pm t)^\nu$. Define the \emph{evenized shifted feature}
% \begin{equation}
% \label{eq:even-feature}
% E_{a}(t):=(a+t)_+^p+(a-t)_+^p.
% \end{equation}
% Expanding by the binomial theorem gives, for every integer $p\ge 1$,
\begin{equation}
\label{eq:even-expansion}
E_{a_i}(t) = (a_i+t)_+^\nu+(a_i-t)_+^\nu =(a_i+t)^\nu+(a_i-t)^\nu = 2\sum_{j=0}^{\lfloor \nu/2\rfloor}\binom{\nu}{2j}\,a_i^{\nu-2j}\,t^{2j}.
\end{equation}
In particular, $E_{a_i}(t)$ is a polynomial in $t^2$ of degree $q$, with coefficients depending on $a_i$.
Let $c_1,\dots,c_q$ be chosen to satisfy the linear constraints
\begin{equation}
\label{eq:linear-system-c}
\sum_{i=1}^q c_i\,a_i^{\,\nu-2j}=
\begin{cases}
\displaystyle \frac{1}{2\binom{\nu}{2}}, & j=1,\\[6pt]
0, & j=2,3,\dots,q.
\end{cases}
\end{equation}
Since the $a_i$'s are distinct and the matrix $\big(a_i^{\nu-2j}\big)_{1\le j\le q,\,1\le i\le q}$ is a (generalized) Vandermonde matrix, \eqref{eq:linear-system-c} has a unique solution. Multiplying \eqref{eq:even-expansion} by $c_i$ and summing over $i$ yields, for all $t\in[-1,2]$,
\begin{align}
S(t) := \sum_{i=1}^q c_i E_{a_i}(t)
&=
2\sum_{j=0}^{q}\binom{\nu}{2j}\Big(\sum_{i=1}^q c_i a_i^{\nu-2j}\Big)t^{2j} \nonumber\\
&=
2\binom{\nu}{2}\Big(\sum_{i=1}^q c_i a_i^{\nu-2}\Big)t^{2}
\;+\;
2\binom{\nu}{0}\Big(\sum_{i=1}^q c_i a_i^{\nu}\Big) \nonumber\\
&= t^2 + C_\nu,
\label{eq:t2-plus-const}
\end{align}
where $C_\nu:=2\sum_{i=1}^q c_i a_i^{\nu}$ is a constant depending only on $p$ (and our chosen shifts).
Thus, we have produced an exact representation of $t^2$ up to an additive constant. Now, for $(u,v)\in[0,1]^2$,
\[
S(u+v)-S(u-v)=(u+v)^2-(u-v)^2=4uv.
\]
Therefore,
\begin{equation}
\label{eq:mult-final}
uv=\frac{S(u+v)-S(u-v)}{4}
=\frac{1}{4}\sum_{i=1}^q c_i\Big(E_{a_i}(u+v)-E_{a_i}(u-v)\Big).
\end{equation}
For each $i\in[q]$, expanding the four terms in \eqref{eq:mult-final} yields:
\[
E_{a_i}(u+v)=(a_i+u+v)_+^\nu+(a_i-u-v)_+^\nu,
\qquad
E_{a_i}(u-v)=(a_i+u-v)_+^\nu+(a_i-u+v)_+^\nu.
\]
Hence our one layer RePU network to implement the product is: 
\begin{align}
\label{eq:mult-4q}
& \mathrm{Mult}_\nu(u,v) \notag\\
& \qquad = \frac{1}{4}\sum_{i=1}^q c_i\left[(a_i+u+v)_+^\nu+(a_i-u-v)_+^\nu-(a_i+u-v)_+^\nu-(a_i-u+v)_+^\nu\right].
\end{align}
This is precisely a one-hidden-layer RePU($\nu$) network with $m=4q=4\lfloor \nu/2\rfloor$ hidden units with weights bounded by $C_\nu$ for some constant depending on $\nu$. 
% \smallskip
% \noindent\textbf{Step 6: Bounds on weights and biases.}
% By construction, each hidden neuron has input weights in $\{-1,0,1\}$, hence bounded by $1$ in magnitude.  
% Moreover, the biases are $a_i=2+i\le 2+q\le 2+\lfloor p/2\rfloor$, hence $O(p)$.
% Finally, the output weights are $\pm c_i/4$, where $(c_i)_{i=1}^q$ is the unique solution to \eqref{eq:linear-system-c}. Since $q$ is fixed once $p$ is fixed and the nodes $a_i=2+i$ are fixed constants, we have
% \[
% \max_{1\le i\le q}|c_i| \le C_p'
% \qquad\Longrightarrow\qquad
% \max_{i,r}|w_{i,r}|\le \frac{C_p'}{4},
% \]
% for some constant $C_p'$ depending only on $p$.
% Thus all multiplier-network weights and biases are bounded by $O_p(1)$ (more precisely, by constants depending only on $p$).
% \smallskip
% \noindent
% Combining the steps above shows that $\mathrm{Mult}_\nu(u,v)$ defined by \eqref{eq:mult-4q}
% computes $uv$ exactly on $[0,1]^2$ using a single hidden layer of width $m=4\lfloor p/2\rfloor$.
\\\\
\noindent
{\bf Step 4: Extension to multiple products:} We now extend the previous construction to the product of the form $f(x_1, \dots, x_d) = x_1x_2\cdots x_d$ by composing $\mathrm{Mult}_\nu$ in a balanced binary tree. Let $L:=\lceil \log_2 d\rceil$ and set $D:=2^L$. Pad the input with ones:
\[
\tilde x_i:=
\begin{cases}
x_i, & 1\le i\le d,\\
1, & d<i\le D,
\end{cases}
\]
so that $\prod_{i=1}^d x_i=\prod_{i=1}^D \tilde x_i$. Define the intermediate variables $z^{(0)}_i := \tilde x_i$ for $1 \le i \le D$ and for each level $\ell=0,1,\dots,L-1$,
\[
z^{(\ell+1)}_j
:=\mathrm{Mult}_\nu\!\big(z^{(\ell)}_{2j-1},\,z^{(\ell)}_{2j}\big),
\qquad j=1,\dots, D/2^{\ell+1}.
\]
When composing the multiplier blocks across different levels of the tree, the linear readout producing each intermediate product can be absorbed into the affine transformation of the next RePU layer. 
More precisely, each intermediate product $z_j^{(\ell)}$ is a linear readout of the RePU features produced at level $\ell$. When $z_j^{(\ell)}$ is used as an input to a multiplier at level $\ell+1$, every preactivation in the next multiplier is an affine function of $z_j^{(\ell)}$ and hence, after substitution, an affine function of the RePU features from level $\ell$. Thus the linear readout map can be composed with the affine input map of the next RePU layer and need not be counted as a separate nonlinear layer. Therefore, the above recursive construction is realized as an ordinary feedforward RePU$(\nu)$ network, up to changing constants in the width and sparsity bounds.
Since $z^{(\ell)}_j\in[0,1]$ for all $\ell,j$ (products of numbers in $[0,1]$ stay in $[0,1]$), the multiplier is always applied within its exactness domain. After $L$ levels, we obtain a single scalar
\[
z^{(L)}_1=\prod_{i=1}^D \tilde x_i=\prod_{i=1}^d x_i.
\]
This construction uses exactly $L=\lceil\log_2 d\rceil$ multiplication stages (each stage applies $\mathrm{Mult}_\nu$ in parallel), hence the number of RePU($\nu$) \emph{nonlinear} layers is $1+L$. 
At level $\ell$, there are $D/2^{\ell+1}$ copies of $\mathrm{Mult}_\nu$ in parallel; since each copy uses $m=4\lfloor \nu/2\rfloor$ hidden units, the RePU width at level $\ell$ is $mD/2^{\ell + 1}$ and the maximum width occurs at $\ell=0$, namely $mD/2\le m2^{L-1}\asymp m d$.
Thus $\prod_{i=1}^d x_i$ can be computed exactly on $[0,1]^d$ by composing the two-input RePU($\nu$) multiplier in a balanced binary tree of depth $\lceil\log_2 d\rceil$, padding with ones to the next power of two. Using this balanced binary-tree construction with the one-hidden-layer RePU($\nu$) multiplier, the function $f(x_1,\dots,x_d)=\prod_{i=1}^d x_i$ on $[0,1]^d$ is computed exactly by a RePU($\nu$) network of depth $\lceil\log_2 d\rceil$, maximum width $O(\nu\,d)$, with $O(\nu\,d)$ active (nonzero) parameters, whose input weights lie in $\{-1,0,1\}$, hidden biases are bounded by $2+\lfloor \nu/2\rfloor$, and all output weights are bounded by a constant $C_\nu$ depending only on $\nu$.
\\\\
\noindent
{\bf Step 5: Computing all possible tensor products of spline basis:} Starting from the input $x=(x_1,\dots,x_d)$, we first compute all univariate spline values $\{B^{K,\nu}_{j}(x_\ell):1\le \ell\le d,\ 1\le j\le K+\nu\}$. Then we use the balanced-tree multiplier construction of previous step to compute all tensor products $\{\Phi_{\mathbf j}(x)\}_{\bj \in [K+\nu]^d}$.
These $d$ collections of size $(K+\nu)$ are then combined using the exact two-input RePU($\nu$) multiplier $\mathrm{Mult}_\nu$ in a balanced binary-tree fashion: at each merge level, the current feature blocks are paired and all their entrywise products
(i.e.\ outer products) are computed in parallel.
Repeating this procedure for $\lceil\log_2 d\rceil$ merge levels yields, simultaneously and exactly, the full tensor product
feature map $\{\Phi_{\mathbf j}(x)\}_{\bj \in [K+\nu]^d}$. 
Consequently, the mapping from $x$ to all $(K+\nu)^d$ tensor-product spline basis functions can be realized by a RePU($\nu$)
network of depth at most $1+\lceil\log_2 d\rceil$, with the final multiplication layer dominating the width on the order of
$\lfloor \nu/2\rfloor (K+\nu)^d$, and with the total number of active (nonzero) parameters also scaling on the order of
$(K+\nu)^d$, up to constants depending only on $\nu$.
Therefore, starting from $x=(x_1,x_2,\cdots,x_d)$, the full tensor-product feature map $\{\Phi_{\mathbf j}(x)\}_{\bj \in [K+\nu]^d}$ can be computed exactly by a RePU($\nu$) neural network of depth at most $1+\lceil\log_2 d\rceil$, maximum width at most $C_{\nu, d}\,(K+\nu)^d$, and total number of active (nonzero) parameters at most $C_{\nu, d}'\,(K+\nu)^d$, where $C_{\nu, d},C_{\nu, d}'>0$ depend only on $(\nu, d)$ (and not on $K$).
% \paragraph{Step 3: Exact multiplication via RePU($\nu$) networks.}
% Because B-splines form a partition of unity and are nonnegative,
% \[
% 0\le B^{K,p}_j(x)\le 1.
% \]
% Hence all tensor-product basis functions $\Phi_{\mathbf j}(x)$ take values in $[0,1]$.
% There exists a one-hidden-layer RePU($\nu$) \emph{multiplier network}
% $\mathrm{Mult}_\nu(u,v)$ that computes $uv$ exactly on $[0,1]^2$.
% This network uses $m=4\lfloor p/2\rfloor$ hidden units,
% input weights in $\{-1,0,1\}$, biases bounded by $O_p(1)$,
% and output weights bounded by $O_p(1)$.
% Using a balanced binary tree of such multiplier blocks,
% all tensor products $\Phi_{\mathbf j}(x)=\prod_{i=1}^d B^{K,p}_{j_i}(x_i)$
% are computed in $L=\lceil\log_2 d\rceil$ additional RePU layers.
% The last multiplication layer has width of order $m\,n^d$,
% which dominates the network width.
\\\\
\noindent
{\bf Step 6: Final linear readout:}
Adding a final linear layer with weights $\{c_{\mathbf j}\}_{\mathbf j\in[K+\nu]^d}$ produces the scalar-valued network output
\[
\mathcal N_K(x)=\sum_{\mathbf j\in[K+\nu]^d} c_{\mathbf j}\,\Phi_{\mathbf j}(x) =Q_K f(x).
\]
Combining this identity with the spline approximation bound from Step~1 yields \eqref{eq:approx-rate}.
% \paragraph{Network complexity.}
% The depth is $1+L=1+\lceil\log_2 d\rceil$ RePU layers.
% The maximal width is attained in the final multiplication layer and is
% $\Theta(\lfloor p/2\rfloor\,n^d)$.
% The number of nonzero parameters scales as $\Theta(n^d)$.
% Weight magnitudes scale as $O(k^p)$ only in the spline-evaluation layer,
% while all multiplier-layer weights are $O_p(1)$.
% This completes the proof.
\end{proof}

\begin{remark}[Effect of lower smoothness $\beta< \nu$]
\label{rem:approx_beta}
If the target function satisfies only $f\in W^{\beta,2}((0,1)^d)$ for some $\beta<\nu$, then the \emph{network architecture and construction remain unchanged}, while only the approximation statement is modified. Indeed, we may still choose spline degree $\nu$ and use the same RePU($\nu$) network to compute all univariate B-splines, their tensor-product basis functions, and the final linear combination. The only difference is in the spline approximation step: the spline (or quasi-interpolant) $Q_K f$ now satisfies
\[
\|f-Q_K f\|_{W^{r,2}((0,1)^d)} \;\le\; C_{\beta,\nu,d}\, K^{-(\beta-r)}\,\|f\|_{W^{\beta,2}((0,1)^d)},
\qquad 0\le r\le \beta,
\]
so the approximation rate is governed by the actual smoothness $\beta$ rather than the spline degree $\nu$.
In particular, taking $\nu>\beta$ does not improve the convergence order but only increases constants and network size, while the depth, width, and sparsity bounds derived for the RePU$(\nu)$ realization continue to hold, up to constants depending on $(\beta,\nu,d)$.
\end{remark}

\section{Additional technical lemmas and their proofs}
\begin{lemma}[Smooth truncation function]
    \label{lem:truncation}
Fix an integer $p\ge 1$ and let $m>0$ and a small $0<\delta<m$. Define a function $S_p: \reals \mapsto \reals$ as: 
$$
S_p(t):= \begin{cases}
0, & t\le 0,\\[4pt]
B_p^{-1}\displaystyle\int_0^t s^p(1-s)^p\,ds, & 0<t<1,\\[8pt]
1, & t\ge 1,
\end{cases}
$$
where $B_p:=\int_0^1 s^p(1-s)^p\,ds$. 
Now, using this $S_p$, define another function $h_{m, p, \delta}(r)$ for $r\ge 0$ as: 
$$
h_{m, p, \delta}(x):= \begin{cases}
x, & 0\le x\le m-\delta,\\[4pt] 
x + (m-x) S_p(\frac{x-(m-\delta)}{\delta}), & m-\delta\le x \le m,\\[8pt]
m, & x \ge m.
\end{cases}
$$
Finally define $T_m \equiv T_{m, p, \delta}:\reals \mapsto \reals$ as 
$$
T_{m, p, \delta}(x) = \sign(x)h_{m, p, \delta}(|x|) \,,
$$
where we use the convention $\sign(0) = 0$. 
Then the following properties hold: 
\begin{enumerate}
    \item $T_{m, p, \delta}\in C^p(\mathbb R)$, odd, monotone nondecreasing on $\mathbb R$, and is a piecewise polynomial with finitely many pieces and degree at most $2p+2$.
    \item For all $x\in\mathbb R$, $|T_{m, p, \delta}(x)|\le m$, 
    and
    $$
    T_{m, p, \delta}(x)=x \quad\text{for } |x|\le m-\delta, \qquad T_{m, p, \delta}(x)=m\,\operatorname{sign}(x) \quad\text{for } |x|\ge m.
    $$
    As a consequence $|T_{m, p, \delta}(x)|\le m-\delta$ implies $T_{m, p, \delta}(x)=x$. 
    \item There exists a finite constant $L_p<\infty$, depending only on $p$, such that $T_{m, p, \delta}(x)$ is $L_p$ Lipschitz on $\reals$, where one may take
    $$
    L_p = \max\left\{1,\, \sup_{0\le t\le 1}\left[1-S_p(t)+(1-t)S_p'(t)\right]\right\}.
    $$
\end{enumerate}
\end{lemma}

\begin{proof}
We start by recording some elementary properties of $S_p$. On $(0,1)$, from the definition of $S_p$, we have $S_p'(t)=t^p(1-t)^p/B_p$. Hence $S_p$ is nondecreasing, satisfies $0\le S_p\le 1$, and is a polynomial of degree
$2p+1$ on $(0,1)$. Moreover, since $S_p'$ vanishes to order $p$ at both endpoints,
\[
S_p^{(k)}(0)=S_p^{(k)}(1)=0,
\qquad 1\le k\le p.
\]
Consequently, we have $S_p\in C^p(\mathbb R)$. Furthermore, by construction, $h_{m, p, \delta}(x)=x$ for all $0 \le x \le m-\delta$ and $h_{m, p, \delta}(x)= m$ for all $x \ge m$. 
On the transition interval $m-\delta\le  x \le m$, define
$$
t=\frac{x-(m-\delta)}{\delta} \implies t \in[0,1].
$$
Then $h_{m, p, \delta}(x) = x+(m-x)S_p(t)$.  Since $t$ is a linear function of $x$ and $S_p$ is a polynomial of degree $2p+1$ on $(0,1)$, it is immediate that this transition piece is a polynomial in $x$ of degree at most $2p+2$. Therefore, $h_{m, p, \delta}$ is a piecewise polynomial with finitely many pieces, and so is $T_{m, p, \delta}$. As a consequence, 
\begin{align*}
    |x|\le m-\delta & \implies h_{m, p, \delta}(|x|)=|x| \implies T_{m, p, \delta}(x)=\sign(x)|x|=x \,,\\
    |x|\ge m & \implies h_{m, p, \delta}(|x|)=m \implies T_{m, p, \delta}(x)=m\,\sign(x) \,.
\end{align*}
The uniform bound $|T_{m, p, \delta}(x)|\le m$ also follows from the definition.
We next prove $C^p$ smoothness. As $T_{m, p, \delta}(x) = x$ for $|x| \le m-\delta$ and $T_{m, p, \delta}(x) = m\sign(x)$ for $|x| \ge m$, we only need to establish the $C^p$ smoothness of $T_{m, p, \delta}(x)$ on the region $m-\delta \le |x| \le m$. Let us first consider the interior. When $m-\delta < x < m$, then we have: 
$$
T_{m, p, \delta}(x) = h_{m, p, \delta}(x) = x + (m - x)S_p\left(\frac{x - (m-\delta)}{\delta}\right) \,.  
$$
As $S_p(t) \in C^p$ on $(0, 1)$, it is immediate that $T_{m, p, \delta}(x) \in C^p$ for $m-\delta < x < m$. By exactly analogous argument, we can conclude that $T_{m, p, \delta}(x) \in C^p$ for $-m < x < -m + \delta$. Therefore, we only need to consider the endpoints of the intervals, i.e. $\{-m, -m + \delta, m-\delta, m\}$. 
% Let us consider $x = m$. Now as $T_{m, p, \delta}(x) = m$ for all $x \ge m$, it is immediate that right-hand-derivative of $T_{m, p, \delta}(x) = 0$ for all order. On the other hand, the left-hand derivative of $T_{m, p, \delta}(x)$ at $x = m$ is 
% $$
% T^{(1)_-}_{m, p, \delta}(m) = \lim_{h \downarrow 0} \frac{T_{m, p, \delta}(m) - T_{m, p, \delta}(m - h)}{h} = 1 - S_p(1) + (m - x) \lim_{h \downarrow 0} \frac{
% S_p(1) - S_p\left(1 - \frac{h}{\delta}\right)}{h}
% $$
% Since $0<\delta<m$, the function
% $h_{m, p, \delta}$ equals $r$ on a neighborhood of $0$. Hence the odd extension
% $T_{m, p, \delta}$ equals $x$ on the neighborhood $|x|<m-\delta$ of the origin,
% and therefore no loss of smoothness occurs at $x=0$.
%It remains to check the matching at $r=m-\delta$ and $r=m$. 
At $x=m-\delta$, we have $t=0$ and $h_{m, p, \delta}(x) = x+(m-x)S_p(t)$. 
Since $S_p(0)=0$ and $S_p^{(k)}(0)=0$ for $1\le k\le p$, all derivatives up to order $p$ of the correction term $(m - x)S_p((x - (m-\delta))/\delta)$ vanish at $x=m-\delta$. Hence, the transition piece matches the left piece
$x\mapsto x$ together with all derivatives up to order $p$. Similarly, at $x=m$, we have $t=1$ and may write $h_{m, p, \delta}(x) = m-(m-x)(1-S_p(t))$
Since $1-S_p(1)=0$ and $S_p^{(k)}(1)=0$ for $1\le k\le p$, all derivatives up to order $p$ of $(m-x)(1 - S_p((x - (m-\delta))/\delta))$ vanish at $x=m$. Hence, the transition piece matches the constant right piece $x\mapsto m$ together with all derivatives up to order $p$. Therefore $h_{m, p, \delta}\in C^p([0,\infty))$, and by the preceding observation at the origin, $T_{m, p, \delta}\in C^p(\reals)$.

We next prove monotonicity. It is enough to show that $h_{m, p, \delta}$ is
nondecreasing on $[0,\infty)$. This is immediate on $[0,m-\delta]$ and on $[m,\infty)$. On the transition interval $m-\delta<x<m$, with
$t=(x-(m-\delta))/\delta$, we have
\begin{equation}
\label{eq:deriv_Tm}
\frac{d}{dx}h_{m, p, \delta}(x) = 1-S_p(t)+\frac{m-x}{\delta}S_p'(t) = 1-S_p(t)+(1-t)S_p'(t).
\end{equation}
Because $0\le S_p(t)\le 1$, $S_p'(t)\ge 0$, and $1-t\ge 0$, this derivative is
nonnegative. Thus $h_{m, p, \delta}$ is nondecreasing. Since $T_{m, p, \delta}(x)=h_{m, p, \delta}(x)$ for all $x \ge 0$, $T_{m, p, \delta}(x)$ is non-decreasing on $[0, \infty)$. Furthermore, as $T_{m, p, \delta}(x)=-h_{m, p, \delta}(-x)$ for all $x < 0$ and consequently, it is non-decreasing on $(-\infty, 0)$ as well. Therefore, it is a monotone non-decreasing function on $\reals$.   

The Lipschitz bound of $T_{m, p, \delta}(x)$ essentially follows from Equation \eqref{eq:deriv_Tm}. On the transition interval, the derivative is $1-S_p(t)+(1-t)S_p'(t)$, while it is $1$ on $[-m+\delta, m - \delta]$ and $0$ on $[m, \infty) \cup (-\infty, -m]$. Thus
$$
\sup_{x\in\mathbb R}\left|(T_{m, p, \delta})'(x)\right| \le L_p :=\max\left\{1,\,\sup_{0\le t\le 1}\left[1-S_p(t)+(1-t)S_p'(t)\right]\right\}<\infty.
$$
This completes the proof. 
% By the mean value theorem,
% \[
% |T_{m, p, \delta}(x)-T_{m, p, \delta}(y)|
% \le L_p |x-y|,
% \qquad x,y\in\mathbb R.
% \]

% Finally, we prove the identity-recovery property. By oddness, it suffices to consider
% $x\ge 0$. If $0\le x\le m-\delta$, then $T_{m, p, \delta}(x)=x$. If
% $m-\delta<x<m$, then $t\in(0,1)$ and $S_p(t)>0$, so
% \[
% T_{m, p, \delta}(x)
% =
% x+(m-x)S_p(t)
% >
% x
% >
% m-\delta.
% \]
% If $x\ge m$, then $T_{m, p, \delta}(x)=m>m-\delta$. Hence
% $|T_{m, p, \delta}(x)|\le m-\delta$ can occur only when $|x|\le m-\delta$,
% and in that case $T_{m, p, \delta}(x)=x$. This proves the claim.
\end{proof}

%\begin{lemma}
%\label{lem:smoothness_comp_psi}
%   Fix $\theta\in\rset^q$, and $u=(u_1,\ldots,u_p)\in W^{\beta,2}(\Omega, L)$. Then, under the smoothness assumption of $\Psi$ (Assumption \ref{assm:psi_smoothness}), we have the following: 
%    \begin{align*}
%        \Psi_\theta\circ u & \in W^{\tau,s}(\Omega), \\
%        \frac{\partial\Psi_\theta}{\partial\theta_j}\circ u & \in W^{\tau,s}(\Omega), \quad \forall \ 1\leq j\leq d \\
%        \frac{\partial^2}{\partial x_k\partial\theta_j}\Psi_\theta\circ u & \in W^{\tau,s}(\Omega), \quad \forall 1\leq j\leq d, 1 \le k \le p \,.
%    \end{align*}
%\end{lemma}

\begin{lemma}[Donsker property of $\partial_\theta G_{\theta,u}$]
\label{lem:donsker-dthetaG}
Suppose Assumptions \ref{assm:dgp}, \ref{assm:psi_smoothness} and \ref{assm:w} holds and $\Theta\subset\mathbb R^q$ is compact. Let $U_n$ be the deterministic set used in the proof of Theorem~\ref{thm:main_freq},
namely
\[
U_n
=
\left\{
u:\ \|u\|_{W^{\beta,2}}\le C_U,\quad
\|\partial^\ell(u-u^\star)\|_2\le \delta_{n,|\ell|}
\ \text{for all }|\ell|\le 2\tau
\right\},
\]
for some constant $C_U<\infty$. Recall the definition of $\cH_{jk, n}$ and $\cA_{jk, n}$ from Equation \eqref{def:H_jkn} and \eqref{def:A_jkn} in Lemma \ref{lem:zero_root_db} for $1 \le j, k, \le q$. 
Then $\mathcal H_{jk,n}$ and $\mathcal A_{jk,n}$ are uniformly bounded and $P$-Donsker, uniformly in $n$. More precisely, for all sufficiently small $\varepsilon>0$,
\[
\log N\bigl(\varepsilon,\mathcal H_{jk,n},\|\cdot\|_\infty\bigr)
\lesssim
\varepsilon^{-d/(\beta\alpha_\star)}
+
q\log(1/\varepsilon),
\]
and the same bound holds for $\mathcal A_{jk,n}$.
\end{lemma}

\begin{proof}
Throughout the proof, constants may change from line to line but do not depend on $n$, $\theta$, or $u\in U_n$. For notational simplicity, define $r_{\theta,u}(x)=K_0(u)(x)+ (\Psi_\theta \circ K_1(u))(x)$. For a component $G_{\theta,u,j}$, Proposition~\ref{prop:def_G} gives
\begin{align*}
G_{\theta,u,j} & = K_0^\dagger\{\omega  \ \partial_{\theta_j}(\Psi_\theta \circ K_1(u))\} +  \sum_{\ell=1}^p K_{1,\ell}^\dagger\left[\omega \ \partial_{\theta_j}(\Psi_\theta \circ K_1(u)) \ \partial_{z_\ell}(\Psi_\theta \circ K_1(u))\right] \\
& \hspace{18em}+ \sum_{\ell=1}^p K_{1,\ell}^\dagger\left[\omega \ r_{\theta,u} \  \partial_{z_\ell\theta_j}^2(\Psi_\theta \circ K_1(u))\right].
\end{align*}
Differentiating this display with respect to $\theta_k$ gives $H_{\theta,u}^{jk} := \partial_{\theta_k}G_{\theta,u,j}$, 
where
\begin{align}
\label{eq:Hjk-formula}
H_{\theta,u}^{jk} = & K_0^\dagger\left[\omega \ \partial_{\theta_j\theta_k}^2(\Psi_\theta \circ K_1(u))\right]\notag \\
& + \sum_{\ell=1}^p K_{1,\ell}^\dagger \left[\omega \ \left\{\partial_{\theta_j\theta_k}^2(\Psi_\theta \circ K_1(u))  \ \partial_{z_\ell}(\Psi_\theta \circ K_1(u)) \right. \right. \notag \\
& \qquad \qquad \qquad \left. \left. + \partial_{\theta_j}(\Psi_\theta \circ K_1(u)) \ \partial_{z_\ell\theta_k}^2(\Psi_\theta \circ K_1(u)) \right\}\right] \notag \\
&+ \sum_{\ell=1}^p K_{1,\ell}^\dagger \left[\omega \ \left\{\partial_{\theta_k}(\Psi_\theta \circ K_1(u))\, \partial_{z_\ell\theta_j}^2(\Psi_\theta \circ K_1(u)) + r_{\theta,u}\,\partial_{z_\ell\theta_j\theta_k}^3(\Psi_\theta \circ K_1(u))\right\}\right].
\end{align}
We first prove uniform boundedness. Since $K_0$ and the components of $K_1$ have
order at most $\tau$ and $u\in U_n$ is uniformly bounded in $W^{\beta,2}$,
$$
\|K_0u\|_{W^{\beta-\tau,2}} + \|K_1u\|_{W^{\beta-\tau,2}} \le C.
$$
By Assumption \ref{assm:dgp}, $\beta-2\tau>d$, hence $\beta-\tau>d/2$. Using the smoothness of $\Psi_\theta$ and its derivatives (Assumption \ref{assm:psi_smoothness}), the fact that $\omega\in C_0^\beta(\Omega)$ (Assumption \ref{assm:w}) and $K_0(u), K_{1,j}(u) \in W^{\beta -\tau, 2}$, every coefficient inside $K_0^\dagger$ and $K_{1,\ell}^\dagger$ in \eqref{eq:Hjk-formula} is uniformly bounded in $W^{\beta-\tau,2}$ norm. Since the adjoint operators $K_0^\dagger$ and
$K_{1,\ell}^\dagger$ have order at most $\tau$, it follows that
\[
\sup_{\theta\in\Theta,\ u\in U_n}
\|H_{\theta,u}^{jk}\|_{W^{\beta-2\tau,2}}
\le C.
\]
Because $\beta-2\tau>d>d/2$, Sobolev embedding (Theorem \ref{thm:sob_emb}) gives
\begin{equation}
\sup_{\theta\in\Theta,\ u\in U_n}
\|H_{\theta,u}^{jk}\|_\infty
\le C.
\label{eq:Hjk-unif-bdd}
\end{equation}
Now, from Theorem \ref{thm:interpolation} we know that if $L$ is a differential operator of order at most $m$ and
$s-m>d/2$, then, for functions uniformly bounded in $W^{s,2}$,
\begin{equation}
\|Lf\|_\infty
\le
C\|f\|_{W^{s,2}}^{m/(s-d/2)}
\|f\|_\infty^{1-m/(s-d/2)}.
\label{eq:interp}
\end{equation}
Applying \eqref{eq:interp} with $s=\beta$ and $m=\tau$ gives, for any differential
operator $L$ of order at most $\tau$,
\begin{equation}
\|L(u-u')\|_\infty
\le
C\|u-u'\|_\infty^\nu,
\qquad
\nu
=
\frac{\beta-\tau-d/2}{\beta-d/2},
\label{eq:u-interp}
\end{equation}
because $u,u'\in U_n$ are uniformly bounded in $W^{\beta,2}$.
Consider one of the coefficient functions appearing inside an adjoint operator in
\eqref{eq:Hjk-formula}; denote it generically by $a_{\theta,u}$. By our previous argument, we have: 
\begin{equation*}
\sup_{\theta\in\Theta,\ u\in U_n}
\|a_{\theta,u}\|_{W^{\beta-\tau,2}}
\le C.
\end{equation*}
Moreover, using the Lipschitz bounds of the $\Psi$ and its partial derivatives in the $z$-variable (Assumption \ref{assm:psi_smoothness}), together with \eqref{eq:u-interp}, we have, for fixed $\theta$,
\begin{equation}
\label{eq:a_theta_u_bound}
\|a_{\theta,u}-a_{\theta,u'}\|_\infty
\le
C\left(
\|K_0(u-u')\|_\infty
+
\|K_1(u-u')\|_\infty
\right)
\le
C\|u-u'\|_\infty^\nu .
\end{equation}
Now apply \eqref{eq:interp} again, this time with $s=\beta-\tau$ and $m=\tau$,
to the difference $a_{\theta,u}-a_{\theta,u'}$. Since the latter is uniformly
bounded in $W^{\beta-\tau,2}$ by \eqref{eq:Hjk-formula}--\eqref{eq:Hjk-unif-bdd},
we obtain
\begin{equation}
\label{eq:k_a_theta_u_bound}
\|K^\dagger(a_{\theta,u}-a_{\theta,u'})\|_\infty
\le
C\|a_{\theta,u}-a_{\theta,u'}\|_\infty^\rho,
\qquad
\rho
=
\frac{\beta-2\tau-d/2}{\beta-\tau-d/2},
\end{equation}
where $K^\dagger$ denotes any one of $K_0^\dagger,K_{1,1}^\dagger,\ldots,K_{1,p}^\dagger$.
Combining \eqref{eq:Hjk-formula}, \eqref{eq:Hjk-unif-bdd}, \eqref{eq:a_theta_u_bound} and \eqref{eq:k_a_theta_u_bound}, we have: 
\begin{equation}
\|H_{\theta,u}^{jk}-H_{\theta,u'}^{jk}\|_\infty
\le
C\|u-u'\|_\infty^{\rho\nu}
=
C\|u-u'\|_\infty^{\alpha_\star},
\label{eq:Hjk-u-holder}
\end{equation}
as
\[
\rho\nu
=
\frac{\beta-2\tau-d/2}{\beta-\tau-d/2}
\cdot
\frac{\beta-\tau-d/2}{\beta-d/2}
=
\frac{\beta-2\tau-d/2}{\beta-d/2}
=
\alpha_\star.
\]
Furthermore, by Assumption \ref{assm:psi_smoothness}, $a_{\theta, u}$ are Lipschitz continuous in $\theta$, which implies: 
\begin{equation}
\label{eq:Hjk-theta-holder}
\|H_{\theta,u}^{jk}-H_{\theta',u}^{jk}\|_\infty \le C\|\theta-\theta'\|_2.
\end{equation}
Combining \eqref{eq:Hjk-u-holder} and \eqref{eq:Hjk-theta-holder}, we have
\begin{equation}
\label{eq:Hjk-holder}
\|H_{\theta,u}^{jk}-H_{\theta',u'}^{jk}\|_\infty \le C\left(\|u-u'\|_\infty^{\alpha_\star} + \|\theta-\theta'\|_2\right)
\end{equation}
We now bound the covering number of $\mathcal H_{jk,n}$. Fix
$\varepsilon\in(0,1)$, take an $L_\infty$-cover of $U_n$ with radius $r_u = (\varepsilon/2C)^{1/\alpha_\star}$ and a Euclidean cover of $\Theta$ with radius $r_\theta = (\varepsilon/2C)$. Since $U_n$ is contained in a fixed ball of $W^{\beta,2}$, the Sobolev $L_\infty$ entropy bound gives
\[
\log N(r_u,U_n,\|\cdot\|_\infty)
\lesssim
r_u^{-d/\beta}
=
\varepsilon^{-d/(\beta\alpha_\star)}.
\]
Since $\Theta\subset\mathbb R^q$ is compact, we also have: 
\[
\log N(r_\theta,\Theta,\|\cdot\|_2)
\lesssim
q\log(1/\varepsilon).
\]
The product cover together with \eqref{eq:Hjk-holder} gives
\begin{equation}
\label{eq:Hjk-entropy}
\log N\bigl(\varepsilon,\mathcal H_{jk,n},\|\cdot\|_\infty\bigr)
\lesssim
\varepsilon^{-d/(\beta\alpha_\star)}
+
q\log(1/\varepsilon).
\end{equation}
We now show that the same entropy bound holds for $\mathcal A_{jk,n}$. With a slight abuse of notation, let us now define $a_{\theta,u}(x) = (u(x)-u^\star(x))H_{\theta,u}^{jk}(x)$. Then, we have: 
\[
\begin{aligned}
\|a_{\theta,u}-a_{\theta',u'}\|_\infty
&\le
\|(u-u')H_{\theta,u}^{jk}\|_\infty
+
\|(u'-u^\star)(H_{\theta,u}^{jk}-H_{\theta',u'}^{jk})\|_\infty
\\
&\le
C\|u-u'\|_\infty
+
C\|H_{\theta,u}^{jk}-H_{\theta',u'}^{jk}\|_\infty
\\
&\le
C\left(
\|u-u'\|_\infty^{\alpha_\star}
+
\|\theta-\theta'\|_2
\right),
\end{aligned}
\]
where we used $\alpha_\star\le 1$ and the uniform boundedness of $u,u^\star$ and
$H_{\theta,u}^{jk}$. Therefore, by the same calculation as before: 
\begin{equation}
\label{eq:Ajk-entropy}
\log N\bigl(\varepsilon,\mathcal A_{jk,n},\|\cdot\|_\infty\bigr)
\lesssim
\varepsilon^{-d/(\beta\alpha_\star)}
+
q\log(1/\varepsilon).
\end{equation}
As $d/(\beta \alpha_\star) < 2$, the function class is Donsker. This completes the proof. 
% It remains to verify the entropy exponent. Let
% \[
% \gamma=\frac{d}{\beta\alpha_\star}.
% \]
% We need $\gamma<2$, equivalently $\beta\alpha_\star>d/2$. This follows from
% Assumption 1:
% \[
% \beta\alpha_\star-\frac d2
% =
% \frac{
% \beta(\beta-2\tau-d)+d^2/4
% }{
% \beta-d/2
% }
% >0,
% \]
% because $\beta-2\tau>d$. Hence $\gamma<2$. Therefore
% \[
% \int_0^1
% \sqrt{
% 1+\log N(\varepsilon,\mathcal H_{jk,n},\|\cdot\|_\infty)
% }
% \,d\varepsilon
% <\infty,
% \]
% and the same integral is finite for $\mathcal A_{jk,n}$. Since
% $L_2(Q)$ is dominated by $L_\infty$ for every probability measure $Q$, the same
% entropy integral bound holds uniformly in $L_2(Q)$. Thus, by the standard
% uniform entropy criterion for bounded classes, both $\mathcal H_{jk,n}$ and
% $\mathcal A_{jk,n}$ are $P$-Donsker.
% In particular,
% \[
% \sup_{h\in\mathcal H_{jk,n}}
% \left|
% \frac1{\sqrt n}\sum_{i=1}^n \epsilon_i h(X_i)
% \right|
% =
% O_p(1),
% \]
% for independent sub-Gaussian multipliers $\epsilon_i$, and
% \[
% \sup_{a\in\mathcal A_{jk,n}}
% |\mathbb G_n a|
% =
% O_p(1).
% \]
% This proves the claim.
\end{proof}

\end{appendix}

\end{document}